\documentclass[12pt,oneside]{book}

\newcommand{\myTitle}{Eulerian series, zeta functions and the arithmetic of partitions}%Series and the Arithmetic of Partitions}
\newcommand{\myName}{Robert Schneider}
\newcommand{\myPreviousDegrees}{
	B. S., University of Kentucky, 2012\\
	M. Sc., Emory University, 2016
}
\newcommand{\myDegree}{Doctor of Philosophy}
\newcommand{\myDocument}{dissertation} %switch to "thesis" if you're getting a masters degree
\newcommand{\myDepartment}{Mathematics}
\newcommand{\advisor}{Ken Ono}

\newcommand{\myYear}{2018}

\usepackage{amsmath, amsfonts, amssymb, amsthm, amscd}
\usepackage{graphicx}
\usepackage{color}
\usepackage[all]{xy}
\usepackage{tikz,tikz-cd,graphicx}
\usepackage{verbatim}
\usepackage{hyperref}
\usepackage{multicol}
\usepackage{multirow}
\usepackage{rotating}
\usepackage{fancyhdr,tabularx,cite,mathrsfs,caption}
\usepackage{fixltx2e}

\usepackage{amssymb}
\usepackage[OT2,T1]{fontenc}
\DeclareSymbolFont{cyrletters}{OT2}{wncyr}{m}{n}
\DeclareMathSymbol{\Sha}{\mathalpha}{cyrletters}{"58}
\usepackage{amssymb,amsmath,amsthm, hyperref, verbatim, amsfonts, fontenc}
\newtheorem*{prop*}{Proposition}

\newtheorem{example}[subsection]{Example}
\newtheorem*{corollary*}{Corollary}

\newtheorem{question}{Question}

\newcommand{\C}{\mathbb{C}}
\newcommand{\Q}{\mathbb{Q}}
\newcommand{\N}{\mathbb{N}}
\newcommand{\Z}{\mathbb{Z}}

\newcommand{\R}{\mathbb{R}}

\newcommand{\floor}[1]{\left \lfloor #1\right \rfloor}
\newcommand{\Comment}[1]{}

\newtheorem{theorem}[subsection]{Theorem}
\newtheorem{coro}[subsection]{Corollary}
\theoremstyle{remark}
\newtheorem*{remark}{Remark}
\newcommand{\rra}{\rightarrow}
\newcommand{\inv}{^{-1}}
\newcommand{\brk}[1]{ \left\lbrace #1 \right\rbrace }
\newcommand{\pwr}[1]{\left( #1 \right)}

\DeclareMathOperator{\Real}{Re}

\usepackage{amsmath}
\usepackage{amsfonts}
\usepackage{amssymb}
\usepackage{amsthm}
\usepackage{hyperref}
\usepackage{enumerate}
\usepackage{cleveref}
\usepackage{mathrsfs}
\usepackage[top=0.9in, bottom=1.15in, left=1.1in, right=1.1in]{geometry}
 \usepackage{hyperref}
\hypersetup{colorlinks=true,linkcolor=blue,citecolor=magenta}
\Crefname{conjecture}{Conjecture}{Conjectures}

\theoremstyle{plain}
\theoremstyle{plain}
\renewcommand{\P}{\mathbb{P}}

\renewcommand{\d}{\mathrm{d}}

\newcommand{\leg}[2]{\genfrac{(}{)}{}{}{#1}{#2}}
\usepackage{tikz}

\usepackage{paralist}
\usepackage{cite}

\usepackage{setspace}
\makeatletter
\let\ps@plain\ps@myheadings
\makeatother

\title{\myTitle}
\author{\myName}
\date{}

\usepackage{amsthm}

\newtheorem{definition}{Definition}[section]
\newtheorem{proposition}{Proposition}[section]
\newtheorem{lemma}{Lemma}[section]
\newtheorem{corollary}{Corollary}[section]

\newtheorem*{rmk}{Remark}

\begin{document}

% Adjust page numbers so they start after the contents.  
% Options for removing page numbers from places they don't belong
\pagestyle{empty}
%\input{emoryPages/distributionAgreement}
%\input{emoryPages/approvalSheet}
%\input{emoryPages/abstractCoverPage}
% % % % % % % % %
% %Cover Page % % 
% % % % % % % % %

\vspace*{\fill}

\begin{center}

\myTitle

\bigskip
\bigskip
\bigskip

By

\bigskip
\bigskip
\bigskip

\myName\\
\myPreviousDegrees

\bigskip
\bigskip
\bigskip

Advisor: \advisor, Ph.D.

\bigskip
\bigskip
\bigskip
\bigskip
\bigskip
\bigskip

A \myDocument\ submitted to the Faculty of the \\
James T. Laney School of Graduate Studies of Emory University \\
in partial fulfillment of the requirements for the degree of \\
\myDegree \\
in \myDepartment \\
\myYear

\end{center}

\vspace*{\fill}

%\newpage
%
%
%\vspace*{\fill}
%\begin{center}{\it ``All of analysis will one day be subsumed by the theory of partitions.''} \end{center}
%
%\  \  \  \  \  \  \  \  \  \  \  \  \  \  \  \  \  \  \  \  \  \  \  \  \  \  \  \  \  \  \  \  \  \  \  \  \  \  \  \  \  \  \  \  \  \  \  \  \  \  \  \  \  \  \  \  \  \  \  \  \  \  \  \  \  \  \  \  \  \  \  \  \  \  \  \  \  -- J. J. Sylvester\footnote{As paraphrased to the author by G. E. Andrews}
%\vspace*{\fill}
%\vspace*{\fill}\newpage
%
%\vspace*{\fill}
%\begin{center}{For Marci and Max} \end{center}
%\vspace*{\fill}
%\vspace*{\fill}
%\newpage

% % % % % % %
% %Abstract % %
% % % % % % %
\newpage 
\begin{center}
Abstract 

\bigskip

\myTitle \\
By \myName

\end{center}

\bigskip

%There is a 350 word limit.  This page can be single spaced if needed.
\noindent %In this thesis, we prove that 
In this dissertation we prove theorems at the intersection of the additive and multiplicative branches of number theory, bringing together ideas from partition theory, $q$-series, algebra, modular forms and analytic number theory. We present a natural multiplicative theory of integer partitions (which are usually considered in terms of addition), and explore new classes of partition-theoretic zeta functions and Dirichlet series --- as well as ``Eulerian'' $q$-hypergeometric series --- enjoying many interesting relations. We find a number of theorems of classical number theory and analysis arise as particular cases of extremely general combinatorial structure laws. 

Among our applications, we prove explicit formulas for the coefficients of the $q$-bracket of Bloch-Okounkov, a partition-theoretic operator from statistical physics related to quasi-modular forms; we prove partition formulas for arithmetic densities of certain subsets of the integers, giving $q$-series formulas to evaluate the Riemann zeta function; we study $q$-hypergeometric series related to quantum modular forms and the ``strange'' function of Kontsevich; and we show how Ramanujan's odd-order mock theta functions (and, more generally, the universal mock theta function $g_3$ of Gordon-McIntosh) arise from the reciprocal of the Jacobi triple product via the $q$-bracket operator, connecting also to unimodal sequences in combinatorics and quantum modular-like phenomena.

\newpage

\vspace*{\fill}
%\begin{center}
\noindent {\it ``Partitions constitute the sphere in which analysis lives, moves, and has its being; and no power of language can exaggerate or paint too forcibly the importance of 
this till-recently 
%[this] 
almost neglected (but 
vast, subtle and universally permeating) element of algebraic thought and expression.''}
%All of analysis will one day be subsumed by the theory of partitions.
%''} 
%\end{center}
%\  \  \  \ \  \  \  \  \ 
 \  \  \  \  \  \  \  \  \  \  \  \  \  \  \  \  \  \  \  \  \  \  \  \  \  \  \  \  \  \  \  \  \  \  \  \  \  \  \  \  \  \   \  \  \  \  \  \  \  \  \  \  \  \   \  \  \  \  \  \  \  \   \  \  \  \   \  \  \  \  \  \  \  \  \  \  \  \  --- J. J. Sylvester\footnote{Thanks to George Andrews and Jim Smoak for providing this quotation, a footnote in \cite{Sylvester_collected}, p. 93.}
\vspace*{\fill}
\vspace*{\fill}\newpage

\vspace*{\fill}
\begin{center}{For Marci and Max} \end{center}
\vspace*{\fill}
\vspace*{\fill}
\newpage

%\begin{center}
\section*{Acknowledgments}
I am grateful to the following people for many inspiring conversations, not to mention lessons and guidance, which influenced this work: my Ph.D. advisor, Ken Ono, and other dissertation committee members David Borthwick and John Duncan; my collaborators and co-authors Larry Rolen, Marie Jameson, Olivia Beckwith, Ian Wagner, Andrew Sills, Amanda Clemm, James Kindt, Lara Patel and Xiaokun Zhang; Professors George Andrews, Krishnaswami Alladi, Andrew Granville, David Leep, Joe Gallian, Penny Dunham, William Dunham, Neil Calkin, Colm Mulcahy, Amanda Folsom, Raman Parimala, Suresh Vennapally, Ron Gould, Bree Ettinger, David Zureick-Brown, Dave Goldsman and Shanshuang Yang; and my colleagues Robert Lemke Oliver, Michael Griffin, Jesse Thorner, Ben Phelan, John Ferguson, Joel Riggs, Maryam Khaqan, Cyrus Hettle, Sarah Trebat-Leder, Lea Beneish, Madeline Locus Dawsey, Victor Manuel Aricheta, Warren Shull, Bill Kay, Anastassia Etropolski, Adele Dewey-Lopez, Alex Rice and Jackson Morrow (whom I also thank for typesetting Chapter 4). I am also grateful to Prof. Vaidy Sunderam, Terry Ingram and Erin Nagle in Emory's Department of Mathematics and Computer Science; and to Emory's Laney Graduate School for electing me for the Woodruff Fellowship and Dean's Teaching Fellowship --- in particular, to Dean Lisa Tedesco and Dr. Jay Hughes, and to Prof. Elizabeth Bounds and Rachelle Green in the Candler School of Theology. %); and to the National Science Foundation. %, for allowing me the opportunity to teach History of Mathematics at Arrendale State Prison.

I am deeply thankful to my wife Marci Schneider and son Maxwell Schneider --- to whom I dedicate this work --- and my parents and siblings, for their confidence in me and support during my graduate school journey; as well as to Marci for compiling and typesetting the bibliography for this dissertation, and to Max (a talented mathematician and programmer) for a lifetime of discussions, and for checking my ideas on computer. \\
\  

\noindent
Robert Schneider\\
\noindent
April 3, 2018
%Emory University%      \end{center}
%
%Much of this dissertation arose from conversations with..
%Ken Ono %Erika?
%David Borthwick, John Duncan
%George Andrews, Krishnaswami Alladi
%Amanda Folsom
%Research collaborators: Andrew Sills, Larry Rolen, Marie Jameson, Olivia Beckwith, Amanda Clemm, Ian Wagner, James Kindt, Xiaokun Zhang, Lara Patel and the Emory Working Group in Number Theory and Molecular Simulation%, and the members of the Elephant 6 collective.
%David Zureick-Brown, Joel Riggs
%David Leep, Andrew Granville, Colm Mulcahy, Joe Gallian
%Marci Schneider, Max Schneider
%Jackson Morrow %Jesse? Xiakun? Lara? Parimala? Vaidy? Bree? Terry?
%Woodruff Fellowship, Dean's Teaching Fellowship, Jay Hughes, Elizabeth Bounds, Rachelle Green, NSF %NSF Grant? %Candler School of Theology? Jay Hughes?
%%Hey thanks guys, you rock.
%
 % %(Optional)
%\input{Sylvester}

% % % % % % % % % % % 
% % Table of Contents % %
% % % % % % % % % % % 
% There should also be a list of figures if you have them.
% this prevents page numbers on table of contents pages
% if there's more than one of them
\addtocontents{toc}{\protect\thispagestyle{empty}}
\tableofcontents  \thispagestyle{empty}
%\listoffigures    \thispagestyle{empty}
%\listoftables     \thispagestyle{empty}
%\listofalgorithms \thispagestyle{empty}

% % % % % % % % % % % % % % % % % % 
% % The bulk of the thesis - Your Chapters % %
% % % % % % % % % % % % % % % % % % 

\frontmatter
\pagestyle{myheadings}
\mainmatter
\chapter{Setting the stage: Introduction, background and summary of results}
\setcounter{page}{1}

%
%%\vspace*{\fill}
%\begin{center}{\it ``All of analysis will one day be subsumed by the theory of partitions.''} -- J. J. Sylvester\footnote{As paraphrased to the author by G. E. Andrews}\end{center}
%%\vspace*{\fill}
%\  
%
%
%

\section{Visions of Euler and Ramanujan}
In antiquity, storytellers began their narratives by invoking the Muses, whose influence would guide the unfolding imagery. It is fitting, then, that we begin this work by praising %the immense curiosity of both 
its main sources of inspiration, Euler and Ramanujan, whose imaginations ranged across much of the %much of the %almost the entire 
landscape of modern 
mathematical thought. %, and who inspired much of this work.

\subsection{Zeta functions, partitions and $q$-series}
%What ecstasy he must have experienced, to be the first to hear the elegant music of combinatorics and complex analysis, to see the convergence of exotic infinite series like fractal trails in a hall of mirrors, to find a wormhole connecting the alien realms of infinite products and sums \cite{Dunham}.%, to feel invisible spaces ripple and respond under his every poke and prod. 

%Only three generations before Euler's time, Descartes had discovered the bridge between Algebra and Geometry. He traced the evanescent landscape of thought in terse heiroglyphs, just the right format for mortal minds: little scribbles pregnant with enormousness. Barrow, Leibniz, Newton, Wallis, and their age cultivated the landscape of Descartes; the Bernoulli family harvested these fruits and began a cottage industry of The Calculus. 

%And so it was that Johann Bernoulli took on a  young pupil, raised closer to the soil of Analysis than any child since the dawn of time. This was Euler, and it was not mere fate that the child --- bright, joyful, pure of heart --- should grow up to encompass the infinite in his daydreams. It was divine intervention in human history: Mathematics reached out to touch our tiny hands.

One marvels at the degree to which our contemporary understanding of $q$-series, integer partitions, and what is now known as the Riemann zeta function %$\zeta (s)$ (where $\operatorname{Re}(s)>1$), 
all emerged nearly fully-formed from Euler's pioneering work \cite{Andrews, dunham1999euler}. 

Euler made spectacular use of product-sum relations, often arrived at by unexpected avenues, thereby inventing one of the principle archetypes of %a large corner of 
modern number theory. Among his many profound identities is the product formula for $\zeta(s)$, % what is now called 
the Riemann zeta function, in which the sum and product converge for $\text{Re}(s)>1$:
\begin{equation}\label{ch1zetaproduct}
\zeta(s):=\sum_{n=1}^{\infty}n^{-s}=\prod_{p\in\mathbb P}(1-p^{-s})^{-1}.
\end{equation}

With this relation, Euler connected the (at the time) cutting-edge theory of infinite series to the timeless set $\mathbb P$ of prime numbers  ---  and founded the modern theory of $\operatorname{L}$-functions. Moreover, in his famed 1744 solution of the ``Basel problem'' posed a century earlier by Pietro Mengoli, which was to find the value of $\sum_{n=1}^{\infty}1/n^2$, Euler showed how to compute even powers of $\pi$  ---  a constant of interest to mathematicians since ancient times ---  using the zeta function, giving explicit formulas of the shape 
\begin{equation}\label{ch1zeta_even}
\zeta(2N)=\pi^{2N}\times\  \mathrm{rational}.
\end{equation}
The evaluation of special functions such as $\zeta(s)$ is another rich thread of number theory. As we will show in this work, there are other classes of zeta functions %(as well as 
(not to mention other formulas for $\pi$) %arising %from %other Eulerian formulas 
arising from the theory of {\it integer partitions}. %in the universe of partition theory. %, as we will show here. %, which (among many connections they make) allow us to calculate odd powers of $\pi$ as well. {\bf RPS: We need at least one more example of this odd powers thing, or delete comment, but it is an interesting angle Ken pointed out.} 

In brief, {partitions} represent different ways to add numbers together to yield other numbers. Let $\mathbb N$ denote the natural numbers $1,2,3,4,5,...$, i.e., the positive integers $\mathbb Z^+$ (we use both notations interchangeably)\footnote{Prof. Paul Eakin at University of Kentucky once said, ``Whenever integers appear, magic happens.''}. We shall now fix some standard notations.% from partition theory.

\begin{definition}
Let $\mathcal P$ %(resp. $\mathcal P^*$) 
denote the set of all integer partitions%(resp. partitions into distinct parts)
. %Let $\mathcal P_X$ (resp. $\mathcal P_X^*$) denote the set of partitions (resp. partitions into distinct parts) into elements of $X\subseteq \mathbb N$. 
For $\lambda_i \in \mathbb N$, let $$\lambda=(\lambda_1,\lambda_2,\dots,\lambda_r),\  \  \  \  \lambda_1\geq\lambda_2\geq\dots\geq\lambda_r\geq 1,$$ denote a generic partition, including the empty partition $\emptyset$. Alternatively, we sometimes write partitions in the form  $\lambda=(1^{m_1}\  2^{m_2}\  3^{m_3}  ...\  k^{m_k}  ...)$ with $m_k=m_k(\lambda)\geq 0$ representing the {\it multiplicity} of $k$ as a part of $\lambda\in \mathcal P$ (we adopt the convention $m_{k}(\emptyset):=0$ for all $k\geq1$). We note that $\lambda$ has only finitely many parts with nonzero multiplicity.% (these are the only parts listed in both notational variants). %ies 
%$m_k$.%, and both notations omit parts with $m_k=0$.
\end{definition}

\begin{definition}
Let $$\ell(\lambda)\  :=\  r=\  m_1+m_2+m_3+...+m_k+...$$ denote the {\it length} of $\lambda$ (the number of parts), and $$|\lambda|\  :=\  \lambda_1+\lambda_2+\lambda_3+\dots+\lambda_r\  =\  m_1+2m_2+3m_3+...+km_k+...$$ denote its {\it size} (the number being partitioned), with the conventions $\ell(\emptyset):=0,\  |\emptyset|:=0$. We write ``$\lambda\vdash n$'' to mean $\lambda$ is a partition of $n$, and ``$\lambda_i\in\lambda$'' to indicate $\lambda_i\in\mathbb N$ is one of the parts of $\lambda$.% including multiplicity.
\end{definition}

For example, we might take $\lambda=(4,3,2,2,1)=(1^1\  2^2\  3^1\  4^1)$, using both notational variants. Then $\ell(\lambda)=1+2+1+1=5$ and $|\lambda|=4+3+2+2+1=12$. It is often useful --- and enlightening --- to write a partition as a Ferrers-Young diagram\footnote{Strictly speaking, the one pictured is a Ferrers diagram; a Young diagram uses unit squares instead of dots.}, such as this visual representation of $(4,3,2,2,1)$, where the first row associates to the largest part $\lambda_1=4$, the second row represents $\lambda_2=3$, and so on:
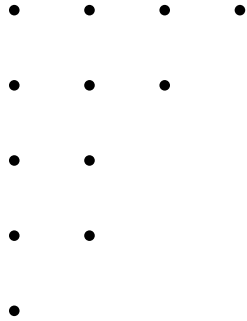
\vspace{.5cm}
\begin{center}
\begin{tikzpicture}[inner sep=0pt,thick,
    dot/.style={fill=black,circle,minimum size=4pt}]
\node[dot] (a) at (0,0) {};
\node[dot] (a) at (1,1) {};
\node[dot] (a) at (0,1) {};
\node[dot] (a) at (0,2) {};
\node[dot] (a) at (1,2) {};
\node[dot] (a) at (0,3) {};
\node[dot] (a) at (1,3) {};
\node[dot] (a) at (0,4) {};
\node[dot] (a) at (1,4) {};
\node[dot] (a) at (2,4) {};
\node[dot] (a) at (3,4) {};
\node[dot] (a) at (2,3) {};
%\draw[-] (-0.1,2.5)--(1.5,2.5);
%\draw[-] (1.5,2.5)--(1.5,4.1);
%\draw[dashed] (-0.1,1.5)--(1.5,1.5);
%\draw[dashed] (1.5,1.5)--(1.5,2.5);
\end{tikzpicture}
\end{center}
We also define the {\it conjugate} $\lambda^*$ of partition $\lambda$ to be the partition given by the transpose of the Ferrers-Young diagram, i.e., the columns of $\lambda$ form the rows of $\lambda^*$. Thus the conjugate of $(4,3,2,2,1)$ is $(5,4,2,1)$ by the diagram above.

Much like the set of positive integers, but perhaps even more richly, the set of integer partitions ripples with striking patterns and beautiful number-theoretic phenomena. In fact, the positive integers $\mathbb N$ are embedded in $\mathcal P$ in a number of ways: obviously, positive integers themselves represent the set of partitions into one part; less trivially, the prime decompositions of integers are in bijective correspondence with the set of prime partitions, i.e., the partitions into prime parts (if we map the number 1, the ``empty prime'' so to speak, to the empty partition $\emptyset$), as Alladi and Erd\H{o}s note \cite{AlladiErdos}. We might also identify the divisors of $n$ with the partitions of $n$ into identical parts, and there are many other interesting ways to associate integers to the set of partitions.

Partitions of $n$ are notoriously challenging to enumerate\footnote{See Appendix A for some elementary approaches to counting partitions.} --- there are just so many of them. Euler found another profound product-sum identity, the generating function for the so-called {\it partition function} $p(n)$ equal to the number of partitions of $n\geq 0$, with the convention $p(0):=1$, viz. %{\bf RPS: removed colon, I think is not correct format for this sentence structure}
\begin{equation}\label{ch1partition_genfctn}
\sum_{n=0}^{\infty}p(n)q^n=%\prod_{n=1}^{\infty}(1-q^n)^{-1}= 
(q;q)_{\infty}^{-1},
\end{equation}
where on the right-hand side we use the usual $q$-Pochhammer symbol notation.

\begin{definition}
For $z, q\in \mathbb C, |q|<1$, the $q$-Pochhammer symbol is defined by $(z;q)_0:=1$ and, for $n \geq	1$, $$(z;q)_n:=\prod_{i=0}^{n-1}(1-zq^i).$$ In the limit as $n\to\infty$, we write $$(z;q)_{\infty}:=\lim_{n\to \infty}(z;q)_n.$$ \end{definition}

With the relation (\ref{ch1partition_genfctn}) and others like it, such as his pentagonal number theorem and $q$-binomial theorem \cite{Berndt}, Euler single-handedly established the theory of integer partitions\footnote{We note that Leibniz appears to have been the first to ask questions about partitions \cite{andrews20008}.}. In particular, much as with the zeta function above, he innovated the use of product-sum generating functions to study partitions, discovering subtle bijections between certain subsets of  $\mathcal P$ and other interesting properties of partitions, often with connections to diverse forms of $q$-hypergeometric series (see \cite{Fine}). %. Many of these identities connect with exotic-looking summations called $q$-hypergeometric series (see \cite{Fine}). %much as with the zeta function above (and the modern theory of $\text{L}$-functions).
%\begin{definition}
%Definition of $q$-hypergeometric series.
%\end{definition} 
%

\subsection{Mock theta functions and quantum modular forms}
Flashing forward almost two centuries from Euler's time, another highly creative explorer ventured into the waters of partitions and $q$-series. When Ramanujan put to sea %left port %cast off 
from India for Cambridge University in 1914, destined to revolutionize number theory, a revolution in physics was already full-sail in Europe. Just one year earlier, the Rutherford--Bohr model of atomic shells heralded the emergence of a paradoxical new {\it quantum theory} of nature that contradicted common sense. %, yet explained anomalous 
%experimental results 
%\cite{Gamow}. 
In 1915, Einstein would describe how space, light, matter, geometry itself, warp and bend in harmonious interplay. %, again explaining anomalous observations 
%\cite{Einstein}. 
The following year, %to Einstein's horror, 
Schwarzschild found Einstein's equations to predict the existence of monstrously inharmonious %unthinkable 
%cosmic % cosmic-scale %anomalies we now call 
{\it black holes}, that %in fact % (since 2016) % as of this writing) 
we can now study directly (just very recently) using interstellar gravitational waves \cite{abbott2016observation}. % as well as black hole analogues in Earth-bound laboratories \cite{Blackhole_lab}. %In 1918, Max Planck won the Nobel Prize in Physics. % Another year later, in 1917, Rutherford discovered a new fundamental particle, the proton, deep within the atomic nucleus. %  in nearby Manchester. 
%How extraordinary, the tide of unconventional ideas that was surging across the world exactly one hundred years ago, in art, literature, music, and society, as well as in the sciences. %All the diversity of physical reality --- and of our own mental experience --- plays out between these enigmatic extremes.%Experimenters began to question the meaning of experiment itself. %The entire body of classical knowledge, even the concept of knowledge itself, were coming into question.

During Ramanujan's five years working with G. H. Hardy, %\cite{Hardy}, %even despite the World War that had erupted, 
news of the paradigm shift in physics must have created a thrill among the mathematicians at Trinity College, Isaac Newton's {\it alma mater}. % (where, at the same time, Bertrand Russell caused a similar shift in the foundation of mathematics \cite{Logicomix}). 
Certainly Hardy would have been aware of the sea change. After all, J. J. Thomson's discovery of the electron, as well as his subsequent ``plum-pudding'' atomic model, had been made at Cambridge's Cavendish Laboratory; Rutherford %(by this time an easy train journey away in Manchester) 
had done his post-doctoral work with Thomson there; % (returning as Thomson's successor in 1919, the year Ramanujan left England for home); 
and Niels Bohr came to Cambridge to work under Thomson in 1911 \cite{gamow1985thirty}. % \cite{???}. 
Moreover, Hardy's %continental 
intellectual colleague David Hilbert was in a public race with Einstein to write down the equations of General Relativity \cite{isaacson2015einstein}. %At Trinity College, Bertrand Russell and his colleagues were actively engaged in rebuilding the foundations of mathematics. 

We don't know how aware Ramanujan was of these happenings in physics, % (of course, by the time he returned to India in 1919, Einstein was a worldwide celebrity), 
yet his flights of imagination and break with academic tradition were expressions of the scientific {\it Zeitgeist} of the age. In Cambridge, he made innovative discoveries in an array of classical topics, from prime numbers to the evaluation of series, products and integrals, to the theory of partitions --- in particular, he discovered startling ``Ramanujan congruences'' relating the partition function $p(n)$ to primes, bridging additive and multiplicative number theory --- all of which would have been accessible to Euler. After returning to India in 1919, as he approached his own tragic event horizon, % \cite{PhelanSchneider}, 
Ramanujan's thoughts ventured into realms that --- like the domains of subatomic particles and gravitational waves --- would require the technology of a future era to navigate \cite{PhelanSchneider}.

In the final letter he sent to Hardy, dated 12 January, 1920 (only a few months before he tragically passed away at age 32), Ramanujan described a new class of mathematical objects he called {\it mock theta functions} \cite{RamanujanCollected}, that mimic certain behaviors of classical {modular forms} (see \cite{Apostol, Ono_web} for details about modular forms). 
%\begin{definition}
%{\bf Definition of modular form.}
%\end{definition}
%\footnoteWe assume some familiarity with {\it modularity} and the behaviors of modular forms having this property \cite{Apostol, Ono_web}. So the unfamiliar reader may proceed, the qualitative ``feel'' of modularity is this: We divide the upper half-plane $\mathbb H$ of $\mathbb C$ into algebraically equivalent (but not geometrically identical) vertical regions of width $1$ called the {\it modular group}, %bounded below by a chain of intersecting half-circles, 
%with an equivalence relation connecting points in different regions. % whose radii have arithmetic significance. 
%(We might think of this construction as ``modular'' in the sense of architecture: we build the upper half-plane from similar-looking, interlocking sub-structures.) The complex values of modular forms at algebraically equivalent points in $\mathbb H$ are connected by simple relations; in particular, for $\tau \in \mathbb H$, the values of $f(\tau +1)$ and $f(-\frac{1}{\tau})$ are related to the value of $f(\tau)$ through M\"{o}bius transformations from complex analyis. When we speak of modularity of a $q$-series, we usually intend that $q:=e^{2\pi i\tau}$, and it is with respect to the variable $\tau\in\mathbb H$ %, as opposed to $q\in\mathbb C$, 
%that the $q$-series has modular properties.}
These interesting $q$-hypergeometric series --- or ``Eulerian'' series, as Ramanujan referred to $q$-series --- turn out to have %He gave an example mock theta function $f(q)$ having 
profoundly curious analytic, combinatorial and algebraic properties. Ramanujan gave a prototypical example $f(q)$ of a mock theta function, defined by the series
\begin{equation}\label{ch1f(q)def}
f(q):=\sum_{n=0}^{\infty}\frac{q^{n^2}}{(-q;q)_n^{2}},
\end{equation}  %, defined by the $q$-series
%\begin{equation}\label{f(q)def}
%f(q):=\sum_{n=0}^{\infty}q^{n^2}(-q;q)_n^{-2},
%\end{equation}
%where $|q|<1$ and $(z;q)_0:=1,\  (z;q)_n:=\prod_{0\leq i \leq n-1}(1-zq^i),\  (z;q)_{\infty}:=\lim_{n\to \infty}(z;q)_n$ is the usual $q$-Pochhammer symbol (see \cite{Berndt}). 
where $|q|<1$. Ramanujan claimed that $f(q)$ is ``almost'' modular in a number of ways. For instance, he provided a pair of %weight $1/2$ 
functions $\pm b(q)$ with
$$b(q) := {(q;q)_{\infty}}{(-q;q)_{\infty}^{-2}}$$
that are modular up to multiplication by $q^{-1/24}$ when $q:=e^{2\pi \text{i} \tau}$, $\tau\in \mathbb H$ (the upper half-plane), to compensate for the singularities arising in the denominator of (\ref{ch1f(q)def}) as $q$ approaches an even-order root of unity $\zeta_{2k}$ (where we define $\zeta_m:=e^{{2 \pi \text{i}}/{m}}$) radially from within the unit circle\footnote{At even-order roots of unity this limiting procedure isn't necessary as there is no pole to reckon with in the denominator and $f(q)$ converges (see Chapter 8).}:
\begin{equation}\label{ch1Watson}
\lim_{q\to \zeta_{2k}}\left(f(q)-(-1)^k b(q) \right)=\mathcal O (1).
\end{equation} 
This type of behavior was first rigorously investigated by Watson in 1936 \cite{Watson_final}, and quantifies to some degree just how ``almost'' modular $f(q)$ is: at least at even-order roots of unity, $f(q)$ looks like a modular form plus a constant. 

Only in the twenty-first century have we begun to grasp the larger meaning of functions such as this, % and to see the connections they make throughout mathematics, 
beginning with Zwegers's innovative Ph.D. thesis \cite{Zwegers} of 2002, and developed in work of other researchers. % (see \cite{ono}). 
We now know Ramanujan's mock theta functions are examples of {\it mock modular forms}, which are the holomorphic parts of even deeper objects called {\it harmonic Maass forms} (see \cite{BFOR} for background). 
%
%\begin{definition}
%{\bf Definition of harmonic Maass form.}
%\end{definition}
%
%\begin{definition}
%{\bf Definition of mock modular form.}
%\end{definition}

%} which has led to the theories of mock modular forms, harmonic Maass forms, and beyond \cite{•}. 
In 2012, Folsom--Ono--Rhoades \cite{FOR} made explicit the limit in (\ref{ch1Watson}), showing that 
\begin{equation}\label{ch1Watson2}
\lim_{q\to \zeta_{2k}}\left(f(q)-(-1)^k b(q) \right)=-4U(-1,\zeta_{2k}),
\end{equation} 
where $U(z,q)$ is the {\it rank generating function for strongly unimodal sequences} in combinatorics (see \cite{BFR}), and is closely related to {\it partial theta functions} and mock modular forms. 
%For instance, in 2012, %Bryson--Ono--Pitman--Rhoades \cite{BOPR} connected $f(q)$ to the important {\it rank generating function for strongly unimodal sequences}, and 
By this connection to $U$, the work of Folsom--Ono--Rhoades along with Bryson--Ono--Pitman--Rhoades \cite{BOPR} reveals that the mock theta function $f(q)$ is strongly connected to the newly-discovered species of {\it quantum modular forms} in the sense of Zagier: functions that are modular on the rational or real numbers (see the definition below) up to the addition of some ``suitably nice'' function, and (in Zagier's words) have ``the `feel' of the objects occurring in perturbative quantum field theory''\cite{Zagier_quantum}.\footnote{See, for instance, \cite{Rejzner_perturbative} about perturbative QFT.}

\begin{definition}\label{ch1qmfdef}
Following Zagier \cite{Zagier_quantum}, we say a function $f\colon \mathbb{P}^1(\mathbb Q) \backslash S \to \mathbb C$, for a discrete subset $S$, is a quantum modular form if $f(x)-f|_k\gamma(x)=h_{\gamma}(x)$ for a ``suitably nice'' function $h_{\gamma}(x)$, with $\gamma \in \Gamma$ a congruence subgroup of $\operatorname{SL}_2(\mathbb Z)$. 
\end{definition}
%
%\begin{remark}
%As noted in Section \ref{Sect0}, we %also 
%anticipate the ``feel'' of quantum theory.%,  counter-intuitive phenomena like quantum tunneling, non-locality, spooky action at a distance, etc.
%\end{remark}
\begin{remark}
In this definition, $|_k$ is the usual Petersson slash operator (see \cite{Ono_web}), and ``suitably nice'' implies some pertinent analyticity condition, e.g. $\mathcal{C}_k,\mathcal{C}_{\infty}$, etc. 
\end{remark}

As a prototype of this new ``quantum'' class of objects, Zagier pointed to a class of ``strange'' functions of $q\in \mathbb C$ that diverge almost everywhere in the complex plane --- except at certain roots of unity, where they are perfectly well-behaved {and turn out to obey modular transformation laws}. One prototypical example of such an object is known in the literature as {\it Kontsevich's ``strange'' function}, an almost nonsensical $q$-hypergeometric series introduced in a 1997 lecture at the Max Planck Institute for Mathematics by Maxim Kontsevich \cite{Zagier_Vassiliev}.

\begin{definition}\label{ch1Fdef}
The ``strange'' function $F(q)$ is defined by the series
\begin{equation}F(q):=\sum_{n=0}^{\infty}(q;q)_n.\end{equation}
\end{definition}

Observing that $(q;q)_{\infty}$ converges inside the unit circle and diverges when $|q|\geq 1$ except at roots of unity, where it vanishes, gives an indication of what we think of as ``strange'' behavior in a function on $\mathbb C$: if we let $q$ scan around the complex plane, $F(q)$ is only non-infinite at isolated points, flickering in and out of comprehensibility along the unit circle.\footnote{Define $\chi_A(z)=1$ if $z\in A\subseteq \mathbb C$ and $=0$ otherwise. Then for any $f(z)$ defined on $B\subseteq \mathbb C$, and $A$ a {discrete} subset of $B$ (with $f(z)\neq 0$ except possibly if $z\in A$), one might think of $f(z)/\chi_A(z)$ as a toy model ``strange'' function --- it is only finite on the points comprising $A$.} % in $\zeta

Modular forms are well known to be connected to partition theory --- the partition generating function $(q;q)_{\infty}^{-1}$ is essentially modular --- as well as to zeta functions and other classical Dirichlet series by the theory of Hecke (see \cite{Apostol}, Ch. 6). But these new-found objects such as mock theta functions and almost-everywhere-divergent ``strange'' functions seem to dwell in a different dimension from classical number theory.

\subsection{Glimpses of an arithmetic of partitions}
In a series of papers in the early 1970s (e.g. \cite{Andrews1, Andrews2}), Andrews introduced the theory of {\it partition ideals}, a deep explanation of generating functions and bijection identities. Using ideas from lattice theory, Andrews provides examples of beautiful algebraic structures within the set $\mathcal P$ of integer partitions, and unifies classical partition identities of Euler, Rogers-Ramanujan (see \cite{SillsRR}), and other authors. Andrews summarized his ideas on partition ideals in his seminal 1976 book \cite{Andrews}. The following year, Alladi and Erd\H{o}s published another innovative study \cite{AlladiErdos} fusing partition theory with analytic number theory to investigate arithmetic functions, and drew a bijection between the set of positive integers $\mathbb Z^+$ and the set of partitions into prime parts (the so-called ``prime partitions''), pointing to deeper arithmetic connections between $\mathbb Z^+$ and $\mathcal P$. 
%
%
%
%
%
%
%In the 1970s, as we noted above, Alladi and Erd\H{o}s \cite{AlladiErdos} used the observation that the prime decompositions of natural numbers are in bijective correspondence with the set of partitions into prime parts (if we associate 1 to the empty partition $\emptyset$) to investigate the behaviors of interesting arithmetic functions. Also in the 1970s, Andrews developed a beautiful theory of {\it partition ideals} \cite{Andrews} using ideas from lattice theory, to unify classical results on partition bijections and generating functions, and discover whole families of new bijections --- teasing the possibility of a universal algebra of partitions. 

In light of these modern, far-reaching ideas, one wonders: to what degree might classical theorems from arithmetic arise as images in $\mathbb N$ (i.e., in prime partitions) % \`{a} la Alladi-Erd\H{o}s) % {\it a la} Alladi-Erd\H{o}s) 
of larger algebraic and set-theoretic structures in $\mathcal P$ such as those discovered by Andrews?% and Alladi-Erd\H{o}s? 

\section{The present work}

The partition generating formula (\ref{ch1partition_genfctn}) doesn't look very much like the zeta function identity \eqref{ch1zetaproduct}, beyond the ``sum = product'' form of both identities. However, generalizing Euler's proofs of these theorems leads to a new class of ``partition zeta functions'', which we define and examine in this work, containing $\zeta(s)$ and classical Dirichlet series as special cases, and intersecting $q$-series generating functions in diverse ways. Further Eulerian methods, together with work of Alladi, Andrews, Fine, Ono, Ramanujan, Zagier and other researchers, give hints of combinatorial structures unifying aspects of multiplicative and additive number theory\footnote{See \cite{MercaSchmidt, Tanay} for recent work at the intersection of additive and multiplicative number theory.}. 

The pursuit of such structures is the central motivation for this work. %As we will justify through a number of examples and applications, 
Through a number of theorems, examples and applications, we propose a philosophical heuristic: \begin{enumerate}
\item Classical multiplicative number theory is a special case (the restriction to prime partitions) of much more general theorems in the universe of partition theory.
\item One expects multiplicative functions and phenomena to have partition counterparts.
\vfill \vfill
\end{enumerate}
\

\subsection{Intersections of additive and multiplicative number theory}
\subsubsection{Chapter 2 preview}
In Chapter 2, we set the stage for this dissertation by proving classical-type connections between the M\"{o}bius function $\mu(n)$ (for $n \in \mathbb N$) and integer partitions. %(adapted from joint work with Marie Jameson \cite{JamesonSchneider}). 
One such result is the following. Let $p_a(n)$ denote the number of partitions of $n$ having length equal to $a$, and define $\widehat{p}_a(n)$ to be the number of partitions of $n$ with length some positive multiple of $a$, i.e.,  $\widehat{p}_a(n) = \sum_{j=1}^\infty p_{aj}(n)$. %some {multiple} of $a$. % for some integer $k\geq 1$. %, and define their generating functions by %, i.e.  \[\widehat{p}_a(n) := \sum_{j=1}^\infty p_{aj}(n).\]  On analogy to the identities for $P(q)$ above, 
Let $P_a(q):=\sum_{k=0}^{\infty}p_a(k)q^k$ and $\widehat{P}_a(q):=\sum_{k=0}^{\infty}\widehat{p}_a(k)q^k$. % denote the $q$-series generating functions of $p_a(n)$ and $\widehat{p}_a(n),$ respectively. 
%We have the following identities for $P_a(q)$ and $p_a(n)$.
\begin{proposition}[Theorem \ref{ch2thm2} in Chapter 2] We have the following pair of identities:
\begin{align*}
P_a(q) &=\sum_{j=1}^\infty{\mu (j)\widehat{P}_{aj}(q)},\\
p_a(n) &= \sum_{j=1}^\infty {\mu (j)\widehat{p}_{aj}(n)}.
\end{align*} 
\end{proposition}

In proving these partition identities, $\mu$ plays a key role, but with respect to the partition lengths $aj$, not the size $n$ as one might anticipate. It is interesting in this theorem and others proved in Chapter 2, to see the interaction of this classical multiplicative function with additive partitions. 

\subsubsection{Chapter 3 preview}
Following up on this multiplicative lead, Chapter 3 is one of the central chapters of this work. We define a partition version of the M\"{o}bius function, also studied privately by Alladi\footnote{K. Alladi, private communication, December 21, 2015}, % \cite{Alladi_private}, 
and use it in various settings in subsequent chapters. 

Furthermore, we present a natural multiplicative theory of integer partitions, % (which are, of course, usually considered in terms of addition), 
and find many theorems of classical number theory and analysis arise as particular cases of extremely general combinatorial structure laws. Let us define a new partition statistic, the {\it norm} of the partition, to complement the length $\ell(\lambda)$ and size $|\lambda|$.

\begin{definition}\label{ch1normdef} We define the {\it norm of $\lambda$}, notated $n_{\lambda}$, by $n_{\emptyset}:=1$ and, for $\lambda$ nonempty, by the product of the parts: \begin{equation*}\label{ch1integer}n_{\lambda}:=\lambda_1 \lambda_2 \cdots \lambda_r.\end{equation*}\end{definition} 
%We adopt the convention $n_{\emptyset}:=1$.
Pushing further in the multiplicative direction, we can define a multiplication operation on the elements of $\mathcal P$, as well as division of partitions.

\begin{definition}\label{ch1productdef}
We define the \textit{product} $\lambda \lambda'$ of two partitions $\lambda,\lambda'\in\mathcal P$ as the multi-set union of their parts listed in weakly decreasing order, e.g., $(5,2,2)(6,5,1)=(6,5,5,2,2,1)$. The empty partition $\emptyset$ serves as the multiplicative identity\footnote{Clearly then, with this multiplication the set $\mathcal P$ is a monoid.}. 
\end{definition}
\begin{definition}\label{ch1divisiondef}
We say a partition $\delta$ \textit{divides}  (or is a ``subpartition'' of) $\lambda$ and write $\delta | \lambda$, if all the parts of $\delta$ are also parts of $\lambda$, including multiplicity, e.g., $(6,5,1)|(6,5,5,2,2,1)$. When $\delta | \lambda$ we define the quotient $\lambda / \delta\in\mathcal P$  formed by deleting the parts of $\delta$ from $\lambda$. We note that $\emptyset$ divides every partition.\end{definition}

Note that in this setting, the partitions $(1), (2), (3),(4),...$, of length one play the role of primes. We can now discuss the partition-theoretic analog of $\mu(n)$ mentioned above.

\begin{definition}\label{ch1moebiusdef}
For $\lambda\in\mathcal P$ we define a partition-theoretic M{\"o}bius function $\mu_{\mathcal P}(\lambda)$ as follows:
$$
\mu_{\mathcal P}(\lambda):= \left\{
        \begin{array}{ll}
            1 & \text{if $\lambda = \emptyset$,}\\
            0 & \text{if $\lambda$ has any part repeated,}\\
            (-1)^{\ell(\lambda)} & \text{otherwise.}
        \end{array}
    \right.
$$
\end{definition}
%\vfill
Note that if $\lambda$ is a prime partition, $\mu_{\mathcal P}(\lambda)$ reduces to $\mu(n_{\lambda})$. Just as in the classical case, we have the following, familiar relations.
%\vfill

%\newpage
\begin{proposition}[Proposition \ref{ch3musum} in Chapter 3]
Summing $\mu_{\mathcal P}(\delta)$ over the subpartitions $\delta$ of $\lambda\in\mathcal P$ gives
$$
\sum_{\delta|\lambda}\mu_{\mathcal P}(\delta)= \left\{
        \begin{array}{ll}
            1 & \text{if $\lambda=\emptyset$,}\\
            0 & \text{otherwise.}
        \end{array}
    \right.
$$
\end{proposition}
%We also prove a partition-theoretic generalization of M{\"o}bius inversion, and many other classical theorems involving $\mu$ generalize to $\mu_{\mathcal P}$. % in \cite{Robert_bracket}. %\begin{proposition}[Proposition XX in \cite{X}]\label{mobinv}
%For functions $f,F\colon \mathcal P \to \C$ we have the equivalence
%$$
%F(\lambda)=\sum_{\delta|\lambda}f(\delta)\  \Longleftrightarrow\  f(\lambda)=\sum_{\delta|\lambda}F(\delta)\mu_{\mathcal P}(\lambda / \delta)
%.$$
%\end{proposition}
We also have a partition-theoretic version of M\"{o}bius inversion. 
\begin{proposition}[Proposition \ref{ch3mobinv} in Chapter 3] 
For $f\colon \mathcal P \to \mathbb C$ define
\begin{equation*}\label{ch1divisorsum}
F(\lambda):=\sum_{\delta | \lambda} f(\delta).\end{equation*}
Then we also have
$$f(\lambda)=\sum_{\delta | \lambda} F(\delta)\mu_{\mathcal P}(\lambda/\delta).$$
\end{proposition}

Now, the classical M{\"o}bius function has a close companion in the Euler phi function $\varphi(n)$, and $\mu_{\mathcal P}$ has a companion as well. 

\begin{definition}\label{ch1phidef}
For $\lambda\in\mathcal P$ we define a partition-theoretic phi function
$$
\varphi_{\mathcal P}(\lambda):=\  n_{\lambda}\prod_{\substack{\lambda_i\in\lambda\\ \text{without}\\  \text{repetition}}}(1-\lambda_i^{-1}).$$
\end{definition}

Clearly $\varphi_{\mathcal P}(\lambda)$ reduces to $\varphi(n_{\lambda})$ if $\lambda$ is a prime partition, and, as with $\mu_{\mathcal P}$, % above, $\varphi_{\mathcal P}$ 
generalizes classical results. %, such as the following familiar-looking identities.
\begin{proposition}[Propositions \ref{ch3phisum} and \ref{ch3phimoeb} in Chapter 3]\label{ch1phisum}
We have that
$$
\sum_{\delta|\lambda}\varphi_{\mathcal P}(\delta)=n_{\lambda},\  \  \  \  \  \  
\varphi_{\mathcal P}(\lambda)=n_{\lambda}\sum_{\delta|\lambda}\frac{\mu_{\mathcal P}(\delta)}{n_{\delta}}
.$$
\end{proposition}  
There are generalizations in partition theory of many other arithmetic objects and theorems, for example, a partition version $\sigma_{\mathcal P}$ of the sum of divisors function $\sigma(n)$, and a partition version of the Cauchy product formula for the product of two infinite series.

\begin{proposition}[Proposition \ref{ch3cauchyprod} in Chapter 3] \label{ch1cauchy}
For $f,g\colon \mathcal P \to \mathbb C$, we have that
$$
\left(\sum_{\lambda\in \mathcal P}f(\lambda)\right) \left(\sum_{\lambda\in \mathcal P}g(\lambda)\right)= \sum_{\lambda\in \mathcal P}\sum_{\delta|\lambda}f(\delta)g(\lambda / \delta),$$
so long as the sums on the left both converge absolutely. \end{proposition}

%We give applications of this fusion of additive and multiplicative ideas, such as %the following 
%explicit formulas for the coefficients of $(q;q)_{\infty}^m$ for any $m\in \mathbb Z$, e.g., Ramanujan's tau function $\tau(n)$ and the number $P_k(n)$ of $k$-color partitions of $n$. %\footnote{In fact, our methods can handle much more complicated products and quotients.}
% that I am currently pursuing. My plan is to continue working toward an algebra of partitions.
%As an application of
%To apply
As our first application of these ideas, we investigate the relatively recently-defined $q$-bracket operator $\left<f\right>_q$ which represents certain expected values in statistical physics, studied by Bloch--Okounkov, Zagier, and others for its quasimodular\footnote{Quasimodular forms are a class containing integer-weight holomorphic modular forms generated by the Eisenstein series $E_2,E_4,E_6$, as opposed to just by $E_4,E_6$ as in the modular case.} properties.

\begin{definition}\label{ch1qbracket} 
We define the {\it $q$-bracket} $\left<f\right>_q$ of a function $f\colon \mathcal P \to \C$ by the expected value %to be the quotient
\[
\left<f\right>_q:=\frac{\sum_{\lambda \in \mathcal P}f(\lambda)q^{|\lambda|}}{\sum_{\lambda \in \mathcal P}q^{|\lambda|}}\in \C[[q]]
.\]
Here, we take the resulting power series to be indexed by partitions, unless otherwise specified.

%where the power series are summed over integer partitions, and there is no dependence of the coefficients on $q$. 
\end{definition}

%
%\begin{remark} 
%%In addition to specifying the $q$-bracket to be primarily a sum over $\mathcal P$, 
%Definition \ref{qbracket} extends the range of the $q$-bracket somewhat; the operator is defined in \cite{BlochOkounkov} and \cite{Zagier} to be a power series in $\Q[[q]]$ instead of $\C[[q]]$, as those works take $f\colon \mathcal P\to\Q$. %We consider the case in which coefficients are functions of $q$ in another study \cite{SchneiderJTP}.
%We may write the $q$-bracket in equivalent forms that will prove useful here:  
%\begin{equation}\label{q-eq}
%\left<f\right>_q=\left(q;q\right)_{\infty}\sum_{\lambda \in \mathcal P}f(\lambda)q^{|\lambda|}=\left(q;q\right)_{\infty}\sum_{n=0}^{\infty}q^n\sum_{\lambda \vdash n}f(\lambda)
%\end{equation}
%
%\end{remark}
%

This $q$-series operator turns out to play a nice role in the theory of partitions, quite apart from questions of modularity. Conversely, in analogy to antiderivatives, we define here an inverse ``$q$-antibracket'' of $f$.\footnote{We will refer to the act of obtaining $\left<f\right>_q$ and $F$ as ``applying the $q$-bracket/antibracket''.}% $\left< f \right>_q^{-1}=\sum_{\lambda \in \mathcal P} {F}(\lambda)q^{|\lambda|}$ such that $\left< F \right>_q=\sum_{\lambda}f(\lambda)q^{|\lambda|}$.\footnote{See Definition \ref{ch3antidef}} %\footnote{This is, in fact, a case of more general phenomena, as we show.} %is a given power series. 
 
\begin{definition}\label{ch1antibracketdef}

We call $F\colon \mathcal P\to \C$ a $q$-{\it antibracket} of $f$ if $\left<F\right>_q= \sum_{\lambda \in \mathcal P}f(\lambda) q^{|\lambda|}$.%, and write $$\left<f\right>_q^{-1}:=F(\lambda).$$% for $f\colon \mathcal P\to \C$, we call $F$ a ``$q$-antibracket of $f$'' (or sometimes just an ``antibracket'').
%Given the power series $\hat{f}(q):=\sum_{\lambda \in \mathcal P}f(\lambda)q^{|\lambda|}$ for $f\colon \mathcal P\to \C$, if we find a function $F\colon \mathcal P\to \C$ whose $q$-bracket is $\hat{f}(q)$ exactly, we call $F$ a ``$q$-antibracket of $f$'' (or sometimes just an ``antibracket'').% and write $\left<f\right>_q^{(-1)}=\sum_{\lambda \in \mathcal P}F(\lambda)q^{|\lambda|}$. 
\end{definition}

As in antidifferentiation, the function $F$ is not unique. Using the partition-theoretic ideas we develop, we can give an explicit formula for coefficients of the $q$-bracket and $q$-antibracket of any function $f$ defined on partitions. 
 
\begin{proposition}[Theorems \ref{ch3thm1} and \ref{ch3thm1.5} in Chapter 3]
The $q$-bracket of $f\colon \mathcal P \to \mathbb C$ is given by
$$\left< f \right>_q=\sum_{\lambda \in \mathcal P}\widetilde{f}(\lambda)q^{|\lambda|},$$ 
where $\widetilde{f}(\lambda)=\sum_{\delta | \lambda}f(\delta)\mu_{\mathcal P}(\lambda / \delta).$ Moreover, let $F(\lambda):=\sum_{\delta | \lambda}f(\delta)$; then a $q$-antibracket of $f$ is given by the coefficients $F$ of %\colon \mathcal P \to \mathbb C$ of
$$\left< f \right>_q^{-1}=\sum_{\lambda \in \mathcal P}{F}(\lambda)q^{|\lambda|}.$$ 
\end{proposition}

%
%$Combining the above ideas, and using partition M\"{o}bius inversion, gives the following useful formulas.
%
%%As I prove in \cite{Robert_bracket}, sums over subpartitions interact in a very surprising and beautiful way with the $q$-Pochhammer symbol $(q;q)_{\infty}:=\prod_{n\geq 1}(1-q^n), |q|<1$. Here is one general result.
%
%
%
%\begin{proposition}[Theorems 4.1 and 4.9 of \cite{Robert_bracket}]\label{bracketgen}
%We have the relations %between sums over all partitions
%$$(q;q)_{\infty}\sum_{\lambda\in\mathcal P}F(\lambda)q^{\lambda}=\sum_{\lambda\in\mathcal P}f(\lambda)q^{\lambda},\  \  \  \  \  \  \  \frac{1}{(q;q)_{\infty}}\sum_{\lambda\in\mathcal P}F(\lambda)q^{\lambda}=\sum_{\lambda\in\mathcal P}\widetilde{F}(\lambda)q^{\lambda}$$
%where %the sums are taken over all partitions $\lambda$, and 
%$\widetilde{F}(\lambda)=\sum_{\delta | \lambda}F(\delta).$ Moreover, we have $f(\lambda)=\sum_{\delta | \lambda}F(\delta)\mu_{\mathcal P}(\lambda / \delta).$
%\end{proposition}
%
%
%
%
%In effect, multiplying by $(q;q)_{\infty}$ ``undoes'' subpartition sums in partition-indexed $q$-series, and dividing by $(q;q)_{\infty}$ sends the coefficients (times $\mu_{\mathcal P}$) into these sums. 

%In \cite{Robert_bracket} I make somewhat clever use of Proposition \ref{bracketgen} to give explicit formulas for coefficients of the {\it $q$-bracket of Bloch-Okounkov} (see  \cite{BlochOkounkov, Zagier}), an operator from statistical physics connected to quasi-modular and $p$-adic modular forms, and to address other $q$-series. 
We apply this $q$-bracket formula to compute coefficients of the reciprocal of the Jacobi triple product (see \cite{Berndt})%\begin{equation}
$$j(z;q):=(z;q)_{\infty}(z^{-1}q;q)_{\infty}(q;q)_{\infty}.$$ %\end{equation}

\begin{proposition}[Corollary \ref{ch3jtpcoeff} of Chapter 3]\label{ch1jtpcoeff}
For $z\neq1$ the reciprocal of the triple product is given by
$$\frac{1}{j(z;q)}=\sum_{n\geq 0}c_n q^{n}\  \  \  \  \text{with}\  \  \  \  c_n=c_n(z)=(1-z)^{-1}\sum_{\lambda \vdash n}\sum_{\delta|\lambda}\sum_{\varepsilon|\delta}z^{\operatorname{crk}(\varepsilon)},
$$
where $\operatorname{crk}(*)$ denotes the crank of a partition as defined by Andrews-Garvan \cite{AndrewsGarvan}.\footnote{See Definition \ref{ch8crk}}
\end{proposition}
%We prove a more complicated formula for $j(z;q)$ itself as well.
We see in Chapter 8 this identity is connected to Ramanujan's mock theta functions.

\subsection{Partition zeta functions}
\subsubsection{Chapter 4 preview}
Arithmetic functions and divisor sums are not the only multiplicative phenomena with connections in partition theory. In Chapter 4 we introduce % \cite{Robert_zeta} and with co-authors Ken Ono and Larry Rolen in \cite{OSR_zeta}, I 
%investigate 
a broad class of {\it partition zeta functions} (and in Chapter 5, partition Dirichlet series) arising from a fusion of Euler's product formulas for both the partition generating function and the Riemann zeta function, which admit interesting structure laws and evaluations as well as classical specializations.   %, which often come with Euler products and nice linear interrelations. %, and interface naturally with partition analogs of arithmetic functions above. %and my colleague Ian Wagner and I have been investigating partition ring theory as well.  
\begin{definition}\label{ch1pzf}
In analogy to the Riemann zeta function $\zeta(s)$, % --- which, of course, was first studied by Euler --- 
for a subset $\mathcal P'$ of $\mathcal P$ and value $s\in \mathbb C$ for which the series converges, %and $f\colon \mathcal P\to \mathbb C$ 
we define a {\it partition zeta function} $\zeta_{\mathcal P'}(s)$  %and {\it partition Dirichlet series} $\mathcal D_{\mathcal P'}(f,s)$, respectively, 
by %by the following sum over partitions $\lambda$ in $\mathcal P'$:
\begin{equation*}\label{ch1zetadef}
\zeta_{\mathcal P'}(s):=\sum_{\lambda\in \mathcal P'}n_{\lambda}^{-s}.%,\  \  \  \  \  \  \mathcal D_{\mathcal P'}(f,s):=\sum_{\lambda\in \mathcal P'}f(\lambda)n_{\lambda}^{-s}.
\end{equation*}
If we let $\mathcal P'$ equal the partitions $\mathcal P_{\mathbb X}$ whose parts all lie in some subset $\mathbb X \subset \mathbb N$, there is also an Euler product
$$\zeta_{\mathcal P_{\mathbb X}}(s)=\prod_{n\in \mathbb X}(1-n^{-s})^{-1}.$$
\end{definition}
%For more on the partition zeta functions, the reader is referred to \ref{Schneider_zeta,ORS,RNT_paper}. 
%Now, if the subset $\mathcal P'$ is finite, such as the partitions of an integer $n$, then of course the sum converges for any $s\in\mathbb C$. On the other hand, i
Of course,  $\zeta(s)$ is the case  $\mathbb X = \mathbb P$; and many classical zeta function identities generalize to partition identities. 
 %Clearly $\zeta_{\mathcal P'}(s)$ diverges for all $s$, if subset $\mathcal P'$ contains infinitely many partitions whose parts are all 1's, % (as infinitely many summands will equal 1), %, and, moreover, having arbitrarily many 1's as parts (as there may be infinitely many copies of each summand). T
%so we must restrict our zeta sums to appropriate proper subsets of $\mathcal P$. %If we further restrict our attention to special subsets of partitions, 
%$$\zeta_{\mathcal P_{\mathbb P}}(s)=\sum_{n\in \mathbb Z^+}n^{-s},$$
%as each prime partition corresponds to the unique factorization of an integer $n$, and each integer's factorization corresponds to some prime partition (with the $n=1$ term coming from the empty partition).
Furthermore, we show how partition zeta sums over other proper subsets of $\mathcal P$ can yield nice closed-form results. % for special values of $s$. %If we use the subset $\mathcal P_{prime}$ of partitions into prime parts (the so-called ``prime partitions'') and take $\operatorname{Re}(s)>1$, then noting that each prime partition corresponds to the prime factorization of a unique integer, the zeta function $\zeta_{\mathcal P_{prime}}(s)$ reduces to the classical zeta function $\zeta(s)$. %; restricting to partitions of length 1 yields the same result. 
To see how subsets influence the evaluations, fix $s=2$ and sum over three unrelated subsets of $\mathcal P$: % as examples of partition zeta identities proved in \cite{Schneider_zeta}, 
partitions $\mathcal P_{even}$ into even parts, partitions $\mathcal P_{prime}$ into prime parts, and partitions $\mathcal P_{dist}$ into distinct parts. %%, and partitions $\mathcal P_{\ell=k}$ of fixed length $k$, 

\begin{proposition}[Corollaries \ref{ch41.4} and \ref{ch41.12} in Chapter 4] 
We have the identities
\begin{flalign*}\label{ch1examples}
\zeta_{\mathcal P_{even}}(2)=\frac{\pi}{2},\  \  \  \  \  \zeta_{\mathcal P_{prime}}(2)=\frac{\pi^2}{6},\  \  \  \  \  \zeta_{\mathcal P_{dist}}(2)=\frac{\operatorname{sinh} \pi}{\pi}%,\  \  \  \  \zeta_{\mathcal P_{\ell = k}}(2)= \frac{2^{2k - 1} - 1}{2^{2k-2}}\zeta(2k)
.\end{flalign*}\end{proposition}
Notice how different choices of partition subsets induce very different partition zeta values for fixed $s$. Interestingly, differing powers of $\pi$ appear in all three examples. % (for further examples, see \cite{OSR_zeta,Robert_zeta}).
Another curious formula involving $\pi$ arises if we take $s=3$ and sum on partitions $\mathcal P_{\geq 2}$ with all parts $\geq 2$ (that is, no parts equal to $1$)\footnote{We call these ``nuclear'' partitions in Appendix A, and see that they encode, in a sense, all of $\mathcal P$.}.
\begin{proposition} [Proposition \ref{ch41.6} in Chapter 4]
We have that 
$$\zeta_{\mathcal P_{\geq 2}}(3)=\frac{3\pi}{\operatorname{cosh}\left(\frac{1}{2}\pi \sqrt{3}\right)}.$$
\end{proposition}

These formulas are appealing, but they look a little too motley to comprise a {\it family} like Euler's values $$\zeta(2k)\in \mathbb Q \pi^{2k}.$$ We produce a class of partition zeta functions that does yield nice evaluations like this. 
\begin{definition} We define
$$\zeta_{\mathcal P}(\{s\}^k):=\sum_{\ell(\lambda)=k}n_{\lambda}^{-s},$$
where the sum is taken over all partitions of fixed length $k\geq 1$ (the $k=1$ case is just $\zeta(s)$).
\end{definition}

\begin{proposition}[Corollary \ref{ch41.7} in Chapter 4]\label{ch11.7}
For $s=2, k\geq 1$, we have the identity 
\begin{equation*}
\zeta_{\mathcal P}(\{2\}^k) = \frac{2^{2k - 1} - 1}{2^{2k-2}}\zeta(2k).
\end{equation*}
For example, we give the following values:
\begin{align*}
\zeta_{\mathcal P}(\{2\}^2)  = \frac{7\pi^4}{360},\  \  \zeta_{\mathcal P}(\{2\}^3) = \frac{31\pi^6}{15120},\  \dots\  ,\  \zeta_{\mathcal P}(\{2\}^{13}) = \frac{22076500342261\pi^{26}}{93067260259985915904000000}%,\  \dots
\end{align*}
\end{proposition}

%There are %In \cite{X} I prove 
We prove increasingly complicated identities for $\zeta_{\mathcal P}(\{2^t\}^k), t\geq 1$, as well. %in \cite{Robert_zeta}. 

\subsubsection{Chapter 5 preview}
In Chapter 5, we more deeply probe certain aspects of partition zeta functions. For example, we are able to prove more than in Proposition \ref{ch11.7} above with respect to rational multiples of powers of $\pi$. %, as follows (which can also be deduced from work of Hoffman).

\begin{proposition}[Corollary \ref{ch5RationalityCor} in Chapter 5]
For $m>0$ even, we have
$$\zeta_{\mathcal P}(\{m\}^k)\in \mathbb Q \pi^{mk}
.$$\end{proposition}

So these zeta sums over partitions of fixed length really do form a family like Euler's zeta values. Inspired by work of Chamberland-Straub \cite{ChamberlandStraub}, in Chapter 5 we also evaluate partition zeta functions over partitions $\mathcal P_{a+m\mathbb N}$ whose parts are all $\equiv a$ modulo $m$. Let $\Gamma$ denote the usual gamma function, and let $e(x):=e^{2 \pi \text{i} x}$.

\begin{proposition}[Proposition \ref{ch5mainthm} in Chapter 5]%[Theorem 1 in \cite{OSR_zeta}]
\label{ch1mainthm}
For $n\geq2$, we have
$$\zeta_{\mathcal P_{a+m\mathbb N}}(n)
=\Gamma(1+a/m)^{-n}\prod_{r=0}^{n-1}\Gamma\left(1+\frac{a-e(r/n)}{m}\right).$$
\end{proposition}

We also address analytic continuation of certain partition zeta functions, which is somewhat rare. In Chapter 4, Corollary \ref{ch41.10}, the analytic continuation of $\zeta_{\mathcal P}(\{s\}^k)$ %,\  \operatorname{Re}(s) > 1,$ {\bf LR: I don't think this range is needed here}
is given for fixed length $k=2$; for $\text{Re}(s)>1$, we can write
\begin{equation}\label{ch1shuffle}
\zeta_{\mathcal P}(\{s\}^2) = \frac{\zeta(2s) + \zeta(s)^2}{2},
\end{equation}
thus $\zeta_{\mathcal P}(\{s\}^2)$ inherits analytic continuation from the Riemann zeta functions on the right. We study analytic continuation more broadly in Chapter 5, and in Corollary \ref{ch5ThirdCor} prove the meromorphic extension of $\zeta_{\mathcal P_{m\mathbb N}}(s)$ to the right half-plane of $\mathbb C$. Moreover, following ideas of Kubota and Leopoldt \cite{KL}, in Theorem \ref{ch5padicInterpThm} we give $p$-adic interpolations for modified versions of $\zeta_{\mathcal P}(\{m\}^k)$ in the $m$-aspect. %We also give applications in the theory of multiple zeta values.

Finally, we give applications in the theory of multiple zeta values, and note examples of {\it partition Dirichlet series} which generalize classical results.
%Clearly $\varphi_{\mathcal P}(\lambda)$ reduces to $\varphi(n_{\lambda})$ if $\lambda$ is a prime partition, and, as with $\mu_{\mathcal P}$, % above, $\varphi_{\mathcal P}$ 
%generalizes classical results. %, such as the following familiar-looking identities.
For instance, for appropriate $s\in \mathbb C,\  \mathbb X \subset \mathbb N$, we get familiar-looking relations like these. % {\it partition Dirichlet series}.% just like the classical cases. % relations like %it is immediate from Euler products that
\begin{proposition} [Proposition \ref{ch5muphidirichlet} in Chapter 5]
Just as in the classical cases, we have the following identities:
\begin{equation*}
  \sum_{\lambda\in\mathcal P_{\mathbb X}}\mu_{\mathcal P}(\lambda) n_{\lambda}^{-s}=  \frac{1}{\zeta_{\mathcal P_{\mathbb X}}(s)},\  \  \  \  \  \  \  \  \sum_{\lambda\in\mathcal P_{\mathbb X}}\varphi_{\mathcal P}(\lambda) n_{\lambda}^{-s}=\frac{\zeta_{\mathcal P_{\mathbb X}}(s-1)}{\zeta_{\mathcal P_{\mathbb X}}(s)}.
\end{equation*}\end{proposition}
%just as in the classical cases. 
%I would like to find further examples of analytic continuation of $\zeta_{\mathcal P}(\{s\}^k)$ % --- perhaps even of $\zeta_{\mathcal P}(\{s\}^x)$ for $x\in\mathbb R$, whatever --- 
%and other zeta forms. Also, I wish to make rigorous a certain heuristic for locating imaginary parts of the zeros of partition zeta functions. 

%This is a sample from my work on partition zeta functions; %

%Considering zeta sums over partitions into distinct parts gives formulas (see \cite{X,X}) analogous to Proposition \ref{1.7} that intersect work of Hoffman on multiple zeta values \cite{X}. 

%I continue to explore this theory of combinatorial zeta functions, seeking new structures; and wish to make rigorous a certain heuristic for locating imaginary parts of the zeros of partition zeta functions. %, that produces sequences similar to the imaginary parts of the zeros of $\zeta(s)$.

\subsection{Partition formulas for arithmetic densities}
%Here is a different connection between partitions and zeta functions. 
\subsubsection{Chapter 6 preview}
In Chapter 6 we explore a different connection between partitions and zeta functions. Alladi proves in \cite{A} a surprising {\it duality principle} connecting arithmetic functions to sums over smallest or largest prime factors of divisors, and applies this principle to prove for $\gcd(r, t) = 1$ that
\begin{equation}\label{ch1AlladiSum}
-\sum_{\substack{ n \geq 2 \\ p_{\rm{min}}(n) \equiv r (\text{mod}\  t)}} \mu(n)n^{-1}= \frac{1}{\varphi(t)},
\end{equation}
where $p_{\rm{min}}(n)$ denotes the smallest prime factor of $n$, and ${1}/{\varphi(t)}$ represents the proportion of primes in a fixed arithmetic progression modulo $t$. Using analogous dualities from partition generating functions (smallest/largest parts instead), and replacing $\mu$ with $\mu_{\mathcal P}$, in Chapter 6 we extend Alladi's ideas to compute arithmetic densities of other subsets of $\mathbb N$ using partition-theoretic $q$-series.

\begin{proposition}[Theorems \ref{ch6rt}--\ref{ch6free} of Chapter 6]
For suitable subsets $\mathbb S$ of $\mathbb N$ with arithmetic density $d_{\mathbb S}$,
\[- \lim_{q \to 1} \sum_{\substack{ \lambda \in \mathcal{P} \\ \rm{sm}(\lambda) \in \mathbb S}} \mu_{\mathcal{P}} (\lambda)q^{\vert \lambda \vert} =  d_{\mathbb S},\]
where $\rm{sm}(\lambda)$ denotes the smallest part of $\lambda$, and $q\to 1$ from within the unit circle.   
\end{proposition}
In particular, if we denote $k$th-power-free integers by $\mathbb S^{(k)}_{\rm{fr}}$, we prove a partition formula to compute $1/\zeta(k)$ as the limiting value of a partition-theoretic $q$-series as $q\to 1$.

\begin{proposition}[Corollary \ref{ch6free2} of Chapter 6] \label{ch1free2}
If $k \geq 2$, then 
\[-\lim_{q \to 1} \sum_{\substack{ \lambda \in \mathcal{P} \\ \rm{sm}(\lambda) \in {\mathbb S^{(k)}_{\rm{fr}}}}} \mu_{\mathcal{P}}(\lambda) q^{\vert \lambda \vert}= \frac{1}{\zeta(k)}.\]
\end{proposition}

We discuss further consequences, such as an interesting bijection between subsets of partitions.
%Alladi has fully generalized his duality in partition theory; he and I are discussing further applications. %have discussed 
%pursuing further applications of his ideas. %, for instance studying Andrews's $\text{spt}$-function and an analogous $\text{lpt}(n)$ (largest parts instead). 

%
%\subsection{``Strange'' functions, quantum modular forms and mock theta functions}
%
%\subsubsection{Chapter 7 preview}
%
%Inspired by Zagier's work \cite{Zagier_quantum} with Kontsevich's ``strange'' function $F(q)$, in Chapter 7 we construct a vector-valued quantum modular form $\phi(x):= $ $\begin{pmatrix}\theta_1^S(e^{2 \pi i x})& \theta_2^S(e^{2 \pi i x})& \theta_3^S(e^{2 \pi i x})\end{pmatrix}^T$ whose components $\theta_i^S\colon \mathbb Q \to \mathbb C$ are similarly ``strange''. (I study other ``strange'' series in \cite{strange}.)
%\begin{proposition} [Theorem 1.1 of \cite{RolenSchneider}] \label{mainthm} 
%We have that $\phi(x)$ is a weight $3/2$ vector-valued quantum modular form. In particular, we have that
%\[
%\phi(z+1)-\begin{pmatrix}1&0&0\\0&0&\zeta_{12}\\0&\zeta_{24}&0 \end{pmatrix}\phi(z)=0
%,
%\]
%and we also have
%\[
%\left(\frac{z}{-i}\right)^{-3/2}\phi(-1/z)+\begin{pmatrix}0&\sqrt{2}&0\\1/\sqrt{2}&0&0\\0&0&1\end{pmatrix}\phi(z)=\begin{pmatrix}0&\sqrt{2}&0\\1/\sqrt{2}&0&0\\0&0&1\end{pmatrix}g(z)
%,
%\]
%where $g(z)$ is a $3$-dimensional vector of smooth functions defined in \cite{RolenSchneider} as period integrals.
%\end{proposition}
\subsection{``Strange'' functions, quantum modularity, mock theta functions and unimodal sequences}% and quantum behavior}

\subsubsection{Chapter 7 summary}

In Chapter 7 we turn our attention to quantum modular forms, which figure into Chapter 8 as well. These are $q$-series that, in addition to being ``almost'' modular, generically ``blow up'' as $q$ approaches the unit circle from within, %or even wildly scan through complex values (think Casorati-–Weierstrass theorem \cite{Knopp}), 
but are finite %(or at least manageable) 
when $q$ radially approaches certain roots of unity or other isolated points --- in which case the limiting values have been related to special values of $\text{L}$-functions \cite{BFOR} --- and might even extend %(possibly with a discontinuous jump) to %pass intact to 
to the complex plane beyond the unit circle in the variable $q^{-1}$, % as in (\ref{ch1umtfquantum}), 
a phenomenon called {\it renormalization} (see \cite{Rhoades}).

Inspired by Zagier's work \cite{Zagier_quantum} with Kontsevich's ``strange'' function $F(q)$ defined above, as well as work by Andrews, Jim\'{e}nez-Urroz and Ono \cite{A-J-O}, we construct a vector-valued quantum modular form $\phi(x):= $ $\begin{pmatrix}\theta_1^S(e^{2 \pi i x})& \theta_2^S(e^{2 \pi i x})& \theta_3^S(e^{2 \pi i x})\end{pmatrix}^T$ whose components $\theta_i^S\colon \mathbb Q \to \mathbb C$ are similarly ``strange''. %(I study other ``strange'' series in \cite{strange}.)% like $F(q)$% in much the same way as $F(q)$. We use $\phi(x)$ to find generating functions for special zeta values.
\begin{proposition} [Theorem \ref{ch7mainthm} of Chapter 7] %\label{ch1mainthm} 
We have that $\phi(x)$ is a weight $3/2$ vector-valued quantum modular form. In particular, we have that
\[
\phi(z+1)-\begin{pmatrix}1&0&0\\0&0&\zeta_{12}\\0&\zeta_{24}&0 \end{pmatrix}\phi(z)=0
,
\]
and we also have
\[
\left(\frac{z}{-i}\right)^{-3/2}\phi(-1/z)+\begin{pmatrix}0&\sqrt{2}&0\\1/\sqrt{2}&0&0\\0&0&1\end{pmatrix}\phi(z)=\begin{pmatrix}0&\sqrt{2}&0\\1/\sqrt{2}&0&0\\0&0&1\end{pmatrix}g(z)
,
\]
where $g(z)$ is a $3$-dimensional vector of smooth functions defined as period integrals.
\end{proposition}

Moreover, finite evaluations of $\theta_i^S$ at odd-order roots of unity lead to closed-form evaluations of complicated-looking period integrals, via relations between certain $\text{L}$-functions and ``strange'' series.%, giving further ways to compute $\pi$ with $q$-hypergeometric series.

\subsubsection{Chapter 8 summary}

%In Chapter 7, we apply ideas from Chapter 4 to relate $j(z;q)$ and the $q$-bracket to the odd-order universal mock theta function $g_3(z,q)$ of Gordon-McIntosh (see \cite{GordonMcIntosh}) andDefine  the rank generating function $\widetilde{U_k}(q)$ (resp. $U_k(q)$) for unimodal (resp. strongly unimodal) sequences with $k$-fold peak, proving formulas such as these.  %Loosely speaking, the $q$-bracket of the reciprocal of $j(z;q)$ is proportional to $g(z^{-1},q^{-1})$ (which converges) plus $\widetilde{U_1}(q)$, leading to nice identities 

%such as the 
%following pair  % and a finite formula for Ramanujan's %prototype mock theta function 
%prototype mock theta function, $f(q)$. 
%(evaluations at roots of unity $\zeta_*$ represent radial limits as $q\to \zeta_*$). %, such as this pair.

Some of the interesting properties of quantum modular forms, such as finiteness at roots of unity and renormalization phenomena, extend to other $q$-hypergeometric series such as the ``universal'' mock theta function $g_3$. In Chapter 8, we apply partition-theoretic results from Chapter 3 as well as ideas from statistical physics, to show that $g_3$ arises naturally from the reciprocal of the classical Jacobi triple product $j(z;q)$ --- and is intimately tied to rank generating functions for unimodal sequences --- under the action of the $q$-bracket. %We show that $g_3$ and other $q$-series also display ``quantum''-like behaviors.

Let $j_z\colon \mathcal P \to \mathbb C$ denote the partition-indexed coefficients of $j(z,q)^{-1}=\sum_{\lambda \in \mathcal P}j_z(\lambda)q^{|\lambda|}$. It turns out the odd-order universal mock theta function $g_3$ (in an ``inverted'' form) and the rank generating function $\widetilde{U}(z,q)$ for unimodal sequences arise together as components of $\left<j_z\right>_q$.

\begin{proposition}[Theorem \ref{ch8theorem1} in Chapter 8]\label{ch1theorem1} 
For $0<|q|<1, z\neq 0,z\neq 1$, the following statements are true:
\begin{enumerate}[(i)]
            \item \label{ch1thm1.1}
            We have the $q$-bracket formula% relating the universal mock theta function $g_3$ and the unimodal rank generating function $U$:
            \begin{equation*}\left<j_z\right>_q=1\  +\  \left[z(1-q)+z^{-1}q\right] g_3(z^{-1},q^{-1})\  +\  \frac{zq^2}{1-z}\widetilde{U}(z,q).\end{equation*}
 
            \item The ``inverted'' mock theta function component in part (\ref{ch1thm1.1}) converges, and can be written in the form
\begin{equation*}\label{ch1umtfsum}
g_3(z^{-1},q^{-1})=\sum_{n=1}^{\infty}\frac{q^n}{(z;q)_n (z^{-1}q;q)_n}.
\end{equation*}
        \end{enumerate}

\end{proposition}

Let $\zeta_m:=e^{2\pi \text{i}/m}$ be a primitive $m$th root of unity. Define the rank generating function $\widetilde{U_k}(z,q)$ (resp. $U_k(z,q)$) for unimodal (resp. strongly unimodal) sequences with $k$-fold peak. Then we have interesting relations between $g_3$ and $\widetilde{U}_k, U_k$.

\begin{proposition}[Corollaries \ref{ch8another_cor} and \ref{ch8cor2} of Chapter 8]\label{ch1another_cor}
For $|q|<1<|z|$, we have
\begin{equation*}
g_3(z^{-1},q^{-1})=\frac{z}{1-z}\sum_{k=1}^{\infty}\widetilde{U}_{k}(z,q) z^{-k}q^k.
\end{equation*}
For $|z|<1$, the radial limit as $q\to \zeta_m$ an $m$th order root of unity is given by
\begin{equation*}
g_3(z,\zeta_m)=\frac{z-1}{z}\sum_{k=1}^{\infty} {U}_{k}(-z,\zeta_m)z^{k}\zeta_m^{-k}.
\end{equation*}
\end{proposition}

We then find $g_3(z,q)$ to extend in $q$ to the entire complex plane minus the unit circle, and give a finite formula for $g_3$ (as well as other $q$-series) at roots of unity, that is simple by comparison to other such formulas in the literature. For instance, we prove the following, simple formula for the mock theta function $f(q)$.

\begin{proposition}[Example \ref{ch8example2} in Chapter 8]\label{ch1example2}
For $\zeta_m$ an odd-order root of unity we have
\begin{equation*}\label{ch1f_finite4}
f(\zeta_m)=\frac{4}{3}\sum_{n=1}^{m}(-1)^n (-\zeta_m^{-1};\zeta_m^{-1})_n.
\end{equation*}
\end{proposition}

%Then we show in Chapter 4 %and 5 %\ref{ch5Partition_zeta} and \ref{ch6OnoRolenSchneider} 
%that equations \eqref{zetaproduct} and \eqref{partition_genfctn} both arise as specializations of a single ``\"{u}ber'' product-sum formula, and find in Chapters 4 and 5 many interesting evaluations, relations and properties pertaining to partition zeta functions as well as ``partition Dirichlet series''. As another application of our partition-theoretic ideas, in Chapter 6 we explore a different connection between partitions and zeta functions: we use generalizations of ideas of Alladi \cite{Alladi_arithmetic_density} to give partition-theoretic formulas for arithmetic densities of certain subsets of $\mathbb N$. In addressing the densities of square-free (and more generally, for $k\geq 2$, $k$th power-free) integers, we arrive at interesting partition formulas to compute the reciprocal of $\zeta(k)$ as the limiting case of a certain $q$-series.  
%
%We then turn our attention to mock theta functions, quantum modular forms and ``strange'' functions. We use ``strange'' functions similar to Kontsevich's function to find a new, vector-valued quantum modular form whose components are all ``strange'' in Chapter 7, and examine other classes of ``strange'' series in Chapter 8. 
%THEOREM
%\subsubsection{Chapter 9 preview}

%THEOREM

We indicate similar formulas for other $q$-hypergeometric series and $q$-continued fractions, % (even ``strange'' series and fractions as in Chapter 7), 
and look at interesting ``quantum''-type behaviors of mock theta functions and other $q$-series inside, outside, and on the unit circle. Finally, we speculate about the nature of connections between partition theory, $q$-series and physical reality.
\\

\begin{remark}
In the Appendices, we give follow-up points and observations related to work in various chapters.\end{remark}
			\clearpage %
\chapter{Combinatorial applications of M\"{o}bius inversion% (joint work with Marie Jameson)%\footnote{This chapter is a joint work with Marie Jameson.}
}%\begin{abstract}
{\bf Adapted from \cite{JamesonSchneider}, a joint work with Marie Jameson}
%In important work on the parity of the partition function, Ono \cite{Ono2010} related values of the partition function to coefficients of a certain mock theta function modulo 2. In this paper, we use M\"obius inversion to give analogous results which relate several combinatorial functions via identities rather than congruences.
%\end{abstract}

%\maketitle

\section{Introduction and Statement of Results}

In this chapter we glimpse connections between additive number theory and the multiplicative branch of the theory, which we will follow up on in subsequent chapters. As we noted in the previous section, product-sum identities are ubiquitous in number theory and the theory of $q$-series. For example, recall Euler's identity
\[\prod_{n=1}^\infty(1-q^n) = \sum_{k=-\infty}^\infty (-1)^kq^{k(3k-1)/2},\] and Jacobi's identity \[\prod_{n=1}^\infty(1-q^n)^3 = \sum_{n=0}^\infty (-1)^n(2n+1)q^{n(n+1)/2}.\]

More recently, Borcherds defined ``infinite product modular forms'' \[F(z) = q^h\prod_{n=1}^\infty(1-q^n)^{a(n)},\] where $q:=e^{2\pi iz}$ and the $a(n)$'s are coefficients of certain weight 1/2 modular forms (see Chapter 4 of \cite{Ono2010}). This was generalized by Bruinier and Ono \cite{BruinierOno} to the setting where the exponents $a(n)$
are coefficients of harmonic Maass forms.

At first glance, this does not look like the stuff of combinatorics. However, one might consider the partition function $p(n)$ and ask whether the product
\begin{equation} \label{ch2thequestion}
\prod_{n=1}^\infty (1-q^n)^{p(n)}
\end{equation}
has any special properties.  In this direction, recent work of Ono \cite{Ono2010}  studies the parity of $p(n).$ For $1<D\equiv 23 \pmod{24},$ Ono defined
\[\Psi_D(q):= \prod_{m=1}^\infty\prod_{0\leq b\leq D-1}\left(1-\zeta_D^{-b}q^m\right)^{\leg{-D}{b}C(\overline{m};Dm^2)},\]
where $\overline{m}$ is the reduction of $m \pmod{12},$ $\zeta_D:=e^{2\pi i/D},$ and $C(\overline{m};Dm^2)$ is the coefficient of a mock theta function. It turns out that
\[C(\overline{m};n) \equiv \begin{cases} p\left(\frac{n+1}{24}\right) \pmod{2} & \text{if } \overline{m}\equiv 1,5,7,11 \pmod{12},\\0 & \text{otherwise.} \end{cases}\]
Ono considers the logarithmic derivative
\begin{equation} \label{ch2logderiv}
\sum_{n=1}^\infty B_D(n)q^n := \frac{1}{\sqrt{-D}}\cdot \frac{q\frac{d}{dq}\Psi_D(q)}{\Psi_D(q)} = \sum_{m=1}^\infty mC(\overline{m};Dm^2)\sum_{n=1}^\infty\leg{-D}{n}q^{mn}
\end{equation}
and notes that reducing mod 2 gives
\begin{equation} \label{ch2reductionmod2}
\frac{1}{\sqrt{-D}}\cdot \frac{q\frac{d}{dq}\Psi_D(q)}{\Psi_D(q)} \equiv \sum_{\begin{subarray}{c} m\geq1\\ \gcd(m,6)=1 \end{subarray}}p\left(\frac{Dm^2+1}{24}\right)\sum_{\begin{subarray}{c} n\geq1\\ \gcd(n,D)=1 \end{subarray}}q^{mn}\pmod{2}.
\end{equation}

This observation was instrumental in proving results regarding the parity of the partition function \cite{Ono2010}. However, if one desires %presently % in this work 
%we desire 
to establish identities rather than congruences, %so 
it seems pertinent to again consider products of the form \eqref{ch2thequestion}, but now at the level of $q$-series identities.

From this perspective, we wish to explore the logarithmic derivative of
\begin{equation} \label{ch2infiniteproduct}
\prod_{n=1}^\infty(1-q^n)^{a(n)}
\end{equation}
 for other, more general combinatorial functions $a(n).$ Then for a nonnegative integer $n$, define
\begin{align*}
Q(n) &:= \text{number of partitions of $n$ into distinct parts}\\
\widehat{Q}(n) &:= \text{number of partitions of $n$ whose parts occur with the same multiplicity}
\end{align*} 
and
\begin{align*}
F_Q(q) &:=\sum_{n=1}^\infty Q(n)q^n\\
F_{\widehat{Q}}(q) &:= \sum_{n=1}^\infty \widehat{Q}(n)q^n\\
\Psi(Q;q) &:=\prod_{n=1}^\infty (1-q^n)^{Q(n)/n}.
\end{align*} 

\begin{theorem} \label{ch2thm1}
We have that \[\frac{q\frac{d}{dq}\Psi(Q;q)}{\Psi(Q;q)} = -F_{\widehat{Q}}(q).\] Moreover, for all $n\geq 1$ we have \[Q(n) = \sum_{d|n}\mu(d)\widehat{Q}(n/d),\] where $\mu$ denotes the M\"obius function.
\end{theorem}

For example, one can compute that 
\begin{flalign*}
\Psi(Q;q) &=1-q-\frac{1}{2}q^2-\frac{1}{6}q^3+\frac{1}{24}q^4+\frac{43}{120}q^5-\frac{233}{720}q^6+\cdots\\
\frac{q\frac{d}{dq}\Psi(Q;q)}{\Psi(Q;q)}  &= -q-2q^2-3q^3-4q^4-4q^5-8q^6-\cdots\\
F_{\widehat{Q}}(q) &= q+2q^2+3q^3+4q^4+4q^5+8q^6+\cdots = -\frac{q\frac{d}{dq}\Psi(Q;q)}{\Psi(Q;q)}.
\end{flalign*}

In fact, while it is not obvious from a combinatorial perspective, this theorem is simple; it follows from the straightforward observation that \[\widehat{Q}(n) = \sum_{d|n}Q(d).\] Now we present two results in a slightly different direction that are perhaps more surprising. Looking again to the work of Ono \cite{Ono2010}, we can apply M\"obius inversion to \eqref{ch2logderiv} to find
\begin{equation}\label{ch2mobiusinv}
C(\overline{n};Dn^2) = \frac{1}{n}\sum_{d|n} \mu(d)\leg{-D}{d}B_D(n/d).
\end{equation}
It is natural to ask whether there are analogs of this statement for related $q$-series, even if the series do not arise as logarithmic derivatives of Borcherds products.

We begin our search of interesting combinatorial functions by noting that the generating function for the partition function $p(n)$ obeys the identity of Euler
\[P(q):=\sum_{n=0}^\infty p(n)q^n = \sum_{n=0}^\infty \frac{q^{n^2}}{(q)_n^2}\]
where $(q)_n$ is the $q$-Pochhammer symbol, defined by $(q)_0=1$ and $(q)_n=\prod_{k=1}^n (1-q^k)$ for $n\geq1$. We wish to investigate other functions of a similar form, such as those presented in the following theorems, which are formally analogous to \eqref{ch2mobiusinv} but involving other combinatorial functions. 

Let $p_a(n)$ denote the number of partitions of $n$ into $a$ parts, and define $\widehat{p}_a(n)$ to be the number of partitions of $n$ into $ak$ parts for some integer $k\geq 1$, i.e.,  \[\widehat{p}_a(n) := \sum_{j=1}^\infty p_{aj}(n).\]  On analogy to the identities for $P(q)$ above, we let $P_a(q)$ and $\widehat{P}_a(q)$ denote the generating functions of $p_a(n)$ and $\widehat{p}_a(n),$ respectively. Then we have the following identities for $P_a(q)$ and $\widehat{P}_a(q)$.
\begin{theorem}\label{ch2thm2} We have that
\begin{align*}
P_a(q) &=\sum_{n=1}^\infty{\mu (n)\widehat{P}_{an}(q)}\\
p_a(n) &= \sum_{j=1}^\infty {\mu (j)\widehat{p}_{an}(n)}.
\end{align*} 
\end{theorem}

Observe that for $a=1,$ we have that $p_1(n)=1$ for all integers $n$, and also that
\[\widehat{p}_1(n)= \sum_{j=1}^\infty {p_j(n)}=p(n).\] 
In this case, the generating functions are given by
\[P_1(q) = \sum_{n=1}^\infty p_1(n)q^n = \sum_{n=1}^\infty q^n = \frac{q}{1-q}\] 
and
\[\widehat{P}_1(q)=\sum_{n=1}^\infty{\widehat{p}_1(n)q^n}=\sum_{n=1}^\infty{p(n)q^n}.\] 
Thus by Theorem \ref{ch2thm2}, we have the explicit identities
\[P_1(q)=\sum_{n=1}^\infty \mu(n)\widehat{P}_n(q) =\frac{q}{1-q}\] 
and, perhaps more interestingly,
\[\ p_1(n)= \sum_{j=1}^\infty {\mu (j)\widehat{p}_j(n)}=1.\]

Looking again for identities similar to those given above for $P(q),$ for a positive integer $a$ set
\begin{align*}
B_a(q) &:= \sum^{\infty }_{n=1}{\frac{q^{n^2+an}}{(q)^2_n}}=:\sum^{\infty }_{N=1}{b_a(N)q^N}\\
\widehat{B}_a(q) &:= \sum^{\infty }_{n=1}{\frac{q^{n^2+an}}{(q)^2_n\left(1-q^{an}\right)}}=:\sum^{\infty }_{N=1}{\widehat{b}_a(N)q^N}.
\end{align*}

Generalizations of $q$-series such as $B_a(q)$ and $\widehat{B}_a(q)$ have been studied by Andrews \cite{Andrews}.  One can give a combinatorial interpretation for the coefficients $b_a(N)$ and $\widehat{b}_a(N)$ as follows.

Consider the Ferrers diagram of a given partition of an integer $N$ with an $n\times n$ Durfee square, and having a rectangle of base $n$ and height $m$ adjoined immediately below the $n\times n$ Durfee square. For example, the partition of $N=12$ shown below has a $2\times 2$ Durfee square (marked by a solid line), and either a $2\times 2$ or $2\times 1$ rectangle below it (the $2\times 1$ rectangle is marked by a dashed line).  
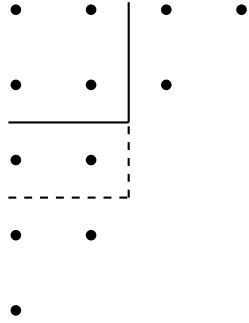
\vspace{0.5cm}
\begin{center}
\begin{tikzpicture}[inner sep=0pt,thick,
    dot/.style={fill=black,circle,minimum size=4pt}]
\node[dot] (a) at (0,0) {};
\node[dot] (a) at (1,1) {};
\node[dot] (a) at (0,1) {};
\node[dot] (a) at (0,2) {};
\node[dot] (a) at (1,2) {};
\node[dot] (a) at (0,3) {};
\node[dot] (a) at (1,3) {};
\node[dot] (a) at (0,4) {};
\node[dot] (a) at (1,4) {};
\node[dot] (a) at (2,4) {};
\node[dot] (a) at (3,4) {};
\node[dot] (a) at (2,3) {};
\draw[-] (-0.1,2.5)--(1.5,2.5);
\draw[-] (1.5,2.5)--(1.5,4.1);
\draw[dashed] (-0.1,1.5)--(1.5,1.5);
\draw[dashed] (1.5,1.5)--(1.5,2.5);
\end{tikzpicture}
\end{center}
We refer to this rectangular region of the diagram as an $n\times m$ ``Durfee rectangle,'' and note that a given Ferrers diagram may have nested Durfee rectangles of sizes $n\times 1, n\times 2, \ldots, n\times M$, where $M$ is the height of the largest such rectangle (assuming that at least one Durfee rectangle is present in the diagram).  

We then have that
\begin{align*}
b_a(N)=& \# \text{ of partitions of } N \text{ having an } n\times n \text{ Durfee square and at least an } n\times a \text{ Durfee}\\ & \text{rectangle}\\
\widehat{b}_a(N) =&\# \text{ of partitions of } N \text{ having an } n\times n \text{ Durfee square and at least an } n \times a \text{ Durfee}\\
& \text{rectangle (counted with multiplicity as an } n \times a \text{ rectangle may be nested within }\\
& \text{taller Durfee rectangles of size } n\times ak, \text{ for } k\geq1).
\end{align*}
Assuming these notations, we have the following result.

\begin{theorem} \label{ch2thm3} We have that
\[\widehat{b}_a(n)= \sum_{j=1}^\infty {b_{aj}(n)}.\]
Moreover, we have
\begin{align*}
B_a(q)&=\sum_{n=1}^\infty{\mu (n)\widehat{B}_{an}(q)}\\
b_a(n)&= \sum_{j=1}^\infty {\mu (j)\widehat{b}_{aj}(n)}.
\end{align*}
\end{theorem}

%%%%%%%%%%%%%%%%%%%%%%%%%%%%%%%%%%%%%%%%%%%%%%%%%%%%%%%%%%%%%%%%%%%%%%%%%

\section{Proof of Theorem \ref{ch2thm1}}
First we prove a lemma regarding logarithmic derivatives.

\begin{lemma}\label{ch2lem}
For any sequence $\{a(n)\},$ we have that
\[\frac{q\frac{d}{dq}\left(\prod_{n=1}^\infty(1-q^n)^{a(n)}\right)}{\prod_{n=1}^\infty(1-q^n)^{a(n)}} = - \sum_{n=1}^\infty \sum_{d|n}a(d)dq^n.\]
\end{lemma}
\begin{proof}
Since $\log(1-x) = -\sum_{m=1}^\infty \frac{x^m}{m},$ we have that
\begin{align*}
\frac{q\frac{d}{dq}\left(\prod_{n=1}^\infty(1-q^n)^{a(n)}\right)}{\prod_{n=1}^\infty(1-q^n)^{a(n)}} &= q\frac{d}{dq}\left(\log\left(\prod_{n=1}^\infty(1-q^n)^{a(n)}\right)\right) = q\frac{d}{dq}\left(\sum_{n=1}^\infty a(n)\log\left(1-q^n\right)\right)\\
&= -q\frac{d}{dq}\left(\sum_{n=1}^\infty a(n)\sum_{m=1}^\infty \frac{q^{mn}}{m}\right) = -\left(\sum_{n=1}^\infty a(n)\sum_{m=1}^\infty nq^{mn}\right)\\
&= -\sum_{n=1}^\infty\sum_{d|n}a(d)dq^n
\end{align*}
as desired.
\end{proof}

\begin{proof}[Proof of Theorem \ref{ch2thm1}]
First note that for all $n\geq 1$ we have \[\widehat{Q}(n) =\sum_{d|n}Q(d),\] so $Q(n) = \sum_{d|n}\mu(d)\widehat{Q}(n/d)$ by M\"obius inversion. By Lemma \ref{ch2lem}, we have that \[\frac{q\frac{d}{dq}\Psi(Q;q)}{\Psi(Q;q)} = - \sum_{n=1}^\infty \sum_{d|n}Q(d)q^n = - \sum_{n=1}^\infty \widehat{Q}(n)q^n\] as desired.
\end{proof}

%%%%%%%%%%%%%%%%%%%%%%%%%%%%%%%%%%%%%%%%%%%%%%%%%%%%%%%%%%%%%%%%%%%%%%%%%
\section{Proof of Theorems \ref{ch2thm2} and \ref{ch2thm3}}

Suppose that for each positive integer $a$, we have two arithmetic functions $f(a;n)$ and $\widehat{f}(a;n)$ such that \[\widehat{f}(a;n) = \sum_{j=1}^\infty f(aj;n),\] where the above sum converges absolutely. We will define their generating functions as follows:
\begin{align*}
F(a;q) &:= \sum_{n=1}^\infty f(a;n)q^n\\
\widehat{F}(a;q) &:= \sum_{n=1}^\infty\widehat{f}(a;n)q^n.
\end{align*}
We then have the following result.
\begin{lemma} \label{ch2lemma2} We have that
\[F(a;q)=\sum_{n=1}^\infty\mu (n)\widehat{F}(an;q)\]
and
\[f(a;n)= \sum_{j=1}^\infty \mu(j)\widehat{f}(aj;n).\]
\end{lemma}
 
\begin{proof}
Recall that \[\sum_{d|n}\mu(n) = \begin{cases} 1 & \text{if } n=1\\ 0 & \text{otherwise.}\end{cases}\] It follows that
\begin{align*}
F\left(a;q\right) &=\sum_{n=1}^\infty\left(\sum_{k\ge 1} f(an;k)q^k\right) \sum_{d|n} \mu(d)\\
&=\sum_{n=1}^\infty\mu(n) \sum_{k\ge 1} \left( \sum_{j=1}^\infty  f(anj;k)\right)q^k\\
&=\sum_{n=1}^\infty\mu(n) \sum_{k\ge 1}\widehat{f}(an;n)q^k\\
&=\sum_{n=1}^\infty\mu(n) \widehat{F}(an;q).
\end{align*}
Then by comparing coefficients, one finds that $f(a;n) = \sum_{j=1}^\infty \widehat{f}(aj;n),$ as desired.
\end{proof}

This lemma can be used to prove both Theorem \ref{ch2thm2} and Theorem \ref{ch2thm3}. We note that Lemma \ref{ch2lemma2} can be applied in extremely general settings, and one has great freedom in creatively choosing the constant $a$ to be varied. For instance, taking $a=1$ gives rise to any number of identities, as $1$ can be inserted as a factor practically anywhere in a given expression.   

\begin{proof}[Proof of Theorem \ref{ch2thm2}] The theorem follows by a direct application of Lemma \ref{ch2lemma2}.
\end{proof}

\begin{proof}[Proof of Theorem \ref{ch2thm3}] First note that \[\widehat{b}_a(N) = \sum_{j=1}^\infty b_{aj}(N),\] since

\begin{align*}
\widehat{B}_a(q) &= \sum^{\infty }_{n=1}\frac{q^{n^2+an}}{(q)^2_n\left(1-q^{an}\right)} = \sum_{n=1}^\infty \frac{q^{n^2+an}}{(q)^2_n} \sum_{j=0}^\infty q^{ajn}\\
&= \sum_{j=1}^\infty \sum_{n=1}^\infty \frac{q^{n^2+ajn}}{(q)^2_n} = \sum_{j=1}^\infty B_{aj}(q).
\end{align*}
The rest follows by applying Lemma \ref{ch2lemma2}.
\end{proof}

 	\clearpage%
\chapter{Multiplicative arithmetic of partitions and the \texorpdfstring{$q$}{Lg}-bracket}{\bf Adapted from \cite{Robert_bracket}}% of Bloch-Okounkov}

%\begin{document}
%\begin{abstract}
%We present a natural multiplicative theory of integer partitions (which are usually considered in terms of addition), and find many theorems of classical number theory arise as particular cases of extremely general combinatorial structure laws. We then see that the relatively recently-defined $q$-bracket operator $\left<f\right>_q$, studied by Bloch--Okounkov, Zagier, and others for its quasimodular properties, plays a deep role in the theory of partitions, quite apart from questions of modularity. Moreover, we give an explicit formula for the coefficients of  $\left<f\right>_q$ for any function $f$ defined on partitions, and, conversely, give a partition-theoretic function whose $q$-bracket is a given power series.  %Finally, extending the definition of the $q$-bracket somewhat, we connect the Jacobi triple product formula to the rank and crank generating functions in partition theory and the enigmatic mock theta functions of Ramanujan. 
%\end{abstract}
%
%\bibliographystyle{amsplain}
%\maketitle

\section{Introduction: the \texorpdfstring{$q$}{Lg}-bracket operator}\label{ch3section1}%, a joint work with Ken Ono and Larry Rolen}

In the previous chapter, we fused techniques from the theory of %parallels between 
partitions and $q$-series %  ---  connected since the time of Euler  ---  
with applications of the M\"{o}bius function, which is central to multiplicative number theory. Here we develop further ideas at the intersection of the additive and multiplicative branches of number theory, with applications to a $q$-series operator from statistical physics.

In a groundbreaking paper of 2000 \cite{BlochOkounkov}, Bloch and Okounkov introduced the $q$-bracket operator $\left<f\right>_q$ of a function $f$ defined on the set of integer partitions, and showed that the $q$-bracket can be used to produce the complete graded ring of quasimodular forms. %quasimodular forms. 
We note that Definition \ref{ch1qbracket} in Section 1 extends the range of the $q$-bracket somewhat; the operator is defined in \cite{BlochOkounkov} to be a power series in $\Q[[q]]$ instead of $\C[[q]]$, as Bloch--Okounkov take $f\colon \mathcal P\to\Q$. %We consider the case in which coefficients are functions of $q$ in another study \cite{SchneiderJTP}.
We may write the $q$-bracket in equivalent forms that will prove useful here:  
\begin{equation}\label{ch3q-eq}
\left<f\right>_q=\left(q;q\right)_{\infty}\sum_{\lambda \in \mathcal P}f(\lambda)q^{|\lambda|}=\left(q;q\right)_{\infty}\sum_{n=0}^{\infty}q^n\sum_{\lambda \vdash n}f(\lambda).
\end{equation}

A recent paper \cite{Zagier} by Zagier examines the $q$-bracket operator from a number of enlightening perspectives, and finds broader classes of quasimodular forms arising from its application. This %study 
chapter is inspired by Zagier's treatment, as well as by ideas of Andrews \cite{Andrews} and Alladi--Erd\H{o}s \cite{AlladiErdos}.

While computationally, the $q$-bracket operator boils down to multiplying a power series by $(q;q)_{\infty}$ as in (\ref{ch3q-eq}), conceptually the $q$-bracket represents a sort of weighted average of the function $f$ over all partitions. Zagier gives an interpretation of the $q$-bracket as the ``expectation value of an observable $f$ in a statistical system whose
states are labelled by partitions'' \cite{Zagier}. Such sums over partitions are ubiquitous in statistical mechanics and quantum physics. We will keep in the backs of our minds the poetic feeling that the partition-theoretic structures we encounter are, somehow, part of the fabric of physical reality. %%(whatever that means).

We begin this chapter's study by considering the $q$-bracket of a prominent statistic in partition theory, the {\it rank} function $\operatorname{rk}(\lambda)$ introduced by Freeman Dyson \cite{Dyson} to give combinatorial explanations for the Ramanujan congruences\footnote{See \cite{Andrews}} $p(5n+4)\equiv 0\  (\operatorname{mod}\  5)$ and $p(7n+5)\equiv 0\  (\operatorname{mod}\  7)$, which we will define by $\operatorname{rk}(\emptyset):=1$\footnote{This definition of $\text{rk}(\emptyset)$ is nonstandardized, but fits conveniently with the ideas of this chapter.} and, for nonempty $\lambda$, by
$$\operatorname{rk}(\lambda):=\operatorname{lg}(\lambda)-\ell(\lambda)$$
where we let $\operatorname{lg}(\lambda)$ denote the {\it largest part} of the partition (similarly, we write $\operatorname{sm}(\lambda)$ for the {\it smallest part}). Noting that $\sum_{\lambda\vdash n}\operatorname{rk}(\lambda)=1$ if $n=0$ (i.e., if $\lambda=\emptyset$) and is equal to $0$ otherwise, as conjugate partitions cancel in the sum and self-conjugate partitions have rank zero, then 
$$\sum_{\lambda\in\mathcal P}\operatorname{rk}(\lambda)q^{|\lambda|}=\sum_{n=0}^{\infty}q^n\sum_{\lambda\vdash n}\operatorname{rk}(\lambda)=1.$$
Therefore we have that
\begin{equation}\label{ch3rank}
\left<\operatorname{rk}\right>_q=(q;q)_{\infty}.
\end{equation}
We see by comparison with the Dedekind eta function $\eta(\tau):=q^{\frac{1}{24}}(q;q)_{\infty}$ that $\left<\operatorname{rk}\right>_q$ is very nearly a weight-$1/2$ modular form. 

Now, recall the weight-$2k$ Eisenstein series central to the theory of modular forms \cite{Ono_web}
\begin{equation}\label{ch3Eisenstein}
E_{2k}(\tau)=1-\frac{4k}{B_{2k}}\sum_{n=1}^{\infty}\sigma_{2k-1}(n)q^n,
\end{equation}
where $k\geq 1$, $B_j$ denotes the $j$th Bernoulli number, $\sigma_{*}$ is the classical sum-of-divisors function, and $q = e^{2\pi i\tau}$ as above. It is not hard to see the $q$-bracket of the ``size'' function
\[
\left<|\cdot|\right>_q=-q\frac{\frac{d}{dq}{(q;q)_{\infty}}}{(q;q)_{\infty}}=\frac{1-E_2(\tau)}{24}
\]
is essentially quasimodular; the series $E_2(\tau)$ is the prototype of a quasimodular form.  

The near-modularity of the $q$-bracket of basic partition-theoretic functions is among the operator's most fascinating features. Bloch--Okounkov give a recipe for constructing quasimodular forms using $q$-brackets of shifted symmetric polynomials \cite{BlochOkounkov}. Zagier expands on their work to find infinite families of quasimodular $q$-brackets, including families that lie outside Bloch and Okounkov's methods \cite{Zagier}. Griffin--Jameson--Trebat-Leder build on these methods to find $p$-adic modular and quasimodular forms as well \cite{GJTL}. While it appears at first glance to be little more than convenient shorthand, the $q$-bracket notation identifies --- induces, even --- intriguing classes of partition-theoretic  phenomena. 

In this study, we give an exact formula for the coefficients of $\left<f\right>_q$ for any function $f\colon \mathcal P\to\C$. We also answer the converse problem, viz. for an arbitrary power series $\widehat{f}(q)$ we give a function $F$ defined on $\mathcal P$ such that $\left<F\right>_q=\widehat{f}(q)$ exactly. The main theorems appear in Section \ref{ch33}. %, via the inverse operator we call the ``$q$-antibracket''.%, all the while possessing explicit finite combinatorial formulas for the coefficients of the resulting power series, in a correspondence not unlike that between differentiation and integration. %and explore some related questions. 

Along the way, we outline a simple, general multiplicative theory of integer partitions, which specializes to many fundamental results in classical number theory. In hopes of presenting a continuous story arc and preserving the flow of ideas, and because most of the proofs are closely analogous to classical cases, we suppress explicit proofs in this chapter, giving gestures and pertinent steps within the exposition, as needed. 

We present an idealistic perspective:  Multiplicative number theory in $\Z$ is a special case of vastly general combinatorial laws, one out of an infinity of parallel number theories in a partition-theoretic multiverse. It turns out the $q$-bracket operator plays a surprisingly natural role in this multiverse.

\  
\  
\  
\

\section{Multiplicative arithmetic of partitions}\label{ch3Multiplicative theory of partitions}

In Definition \ref{ch1normdef} we introduced a complementary statistic to the length $\ell(\lambda)$ and size $|\lambda|$ of $\lambda=(\lambda_1,\lambda_2,...,\lambda_r)$, that we call the \textit{norm} of the partition, viz. % The author and his collaborator Andrew Sills have more recently preferred using the updated terminology.} 
$$n_{\lambda}:=\lambda_1\lambda_2\cdots \lambda_r$$ 
with the convention $n_{\emptyset}:=1$ (it is an empty product). The norm may not seem to be a very natural statistic --- after all, partitions are defined additively with no straightforward connection to multiplication --- but this product of the parts shows up in partition-theoretic formulas scattered throughout the literature \cite{Andrews, Fine}, and will prove to be important to the theory indicated here as well\footnote{This statistic was first introduced in \cite{Schneider_zeta, Robert_bracket} as the ``integer'' of $\lambda$.}. % (the theory presented in \cite{Robert} sits inside the present theory too). 

Recall from Definition \ref{ch1productdef} the product $\lambda \gamma$ of two partitions $\lambda,\gamma$ (combine the parts and reorder into canonical weakly decreasing form). 
%  
%%Pushing further in the multiplicative direction, we define a simple, intuitive multiplication operation on the elements of $\mathcal P$.
%
%\begin{definition}\label{ch3productdef}
%We define the \textit{product} $\lambda \lambda'$ of two partitions $\lambda,\lambda'\in\mathcal P$ as the multi-set union of their parts listed in weakly decreasing order, e.g. $(5,3,2,2)(4,2,1,1)=(5,4,3,2,2,2,1,1)$. The empty partition $\emptyset$ serves as the multiplicative identity. 
%\end{definition}
%
Then it makes sense to write $\lambda^2:=\lambda \lambda, \lambda^3:=\lambda \lambda \lambda$, so on. It is easy to see that we have the following relations:
\begin{align*}\label{ch3log relations}
&n_{\lambda\lambda'}=n_{\lambda}n_{\lambda'}, &&n_{\lambda^a}=n_{\lambda}^a,\\
&\ell(\lambda\lambda')=\ell(\lambda)+\ell(\lambda'), &&\ell(\lambda^a)=a\cdot\ell(\lambda),\\
&|\lambda\lambda'|=|\lambda|+|\lambda'|, &&|\lambda^a|=a |\lambda|.
\end{align*}
Note that length and size both resemble logarithms.

In Definition \ref{ch1divisiondef} we also define division $\lambda/\delta$ of partitions $\lambda, \delta$ if $\delta$ is a subpartition of $\gamma$ (delete the parts of $\delta$ from $\lambda$). 
%\begin{definition}\label{ch3divisiondef}
%We will say a partition $\delta$ \textit{divides} $\lambda$, or is a ``subpartition'' of $\lambda$, and will write $\delta | \lambda$, if all of the parts of $\delta$ are also parts of $\lambda$, including multiplicity, e.g. we write $(4,2,1,1)|(5,4,3,2,2,2,1,1)$. When $\delta | \lambda$ we might also discuss the quotient $\lambda / \delta\in\mathcal P$ %(or $\frac{\lambda}{\delta}$) 
%formed by deleting the parts of $\delta$ from $\lambda$.
%\end{definition}
%
%\begin{remark}
%That is, $\delta$ is a ``sub-partition'' of $\lambda$ in the sense of Alladi \cite{Alladi2015}, Chow--Fan--Goemans--Vondrak \cite{CFGV}, et al. 
Note that both the empty partition $\emptyset$ and $\lambda$ itself are divisors of every partition $\lambda$. 
%\end{remark}
 Then we also have the following relations:
$$n_{\lambda / \lambda'}=\frac{n_{\lambda}}{n_{\lambda'}},\  \  \ell(\lambda / \lambda')=\ell(\lambda)-\ell(\lambda'),\  \  |\lambda / \lambda'|=|\lambda|-|\lambda'|.
$$

%Closure of partition multiplication follows immediately from the definition of partition ideals of order 1 \cite{Andrews}, and we have closure of division as order-1 ideals are closed under taking subpartitions. 

%We postpone until a later time any question of division between elements of $\mathcal P$ with differing parts (i.e. the question of ``rational'' partitions). In short, there is an interesting theory in that direction, which contains the theory being proposed here.

On analogy to the prime numbers in classical arithmetic, the partitions into one part (e.g. $(1),(3),(4)$) are both prime and irreducible under this simple multiplication. The analog of the Fundamental Theorem of Arithmetic is trivial: of course, every partition may be uniquely decomposed into its parts. Thus we might rewrite a partition $\lambda$ in terms of its ``prime'' factorization
$\lambda=(a_1)^{m_1} (a_2)^{m_2}...(a_t)^{m_t}$, where $a_1 > a_2 > ... > a_t \geq 1$ are the distinct numbers appearing in $\lambda$ such that $a_1=\operatorname{lg}(\lambda)$ (the largest part of $\lambda$), $a_t=\operatorname{sm}(\lambda)$ (the smallest part), and $m_i$ denotes the multiplicity of $a_i$ as a part of $\lambda$. Clearly, then, we have 
\begin{equation}
n_{\lambda}=a_1^{m_1} a_2^{m_2}\cdots a_t^{m_t}.
\end{equation} 

\begin{remark}We note in passing that we also have a dual formula for the norm $n_{\lambda^*}$ of the conjugate $\lambda^*$ of $\lambda$, written in terms of $\lambda$, viz.
\begin{equation}
n_{\lambda^*}=M_1^{a_1-a_2}M_2^{a_2-a_3}\cdots M_{t-1}^{a_{t-1}-a_t}M_t^{a_t},
\end{equation}
where $M_k:=\sum_{i=1}^{k}m_i$ (thus $M_t=\ell(\lambda)$), which is clear from the Ferrers-Young diagrams.% of $\lambda$ and $\lambda^*$.
\end{remark}

Fundamental classical concepts such as coprimality, greatest common divisor, least common multiple, etc., apply with exactly the same meanings in the partition-theoretic setting, if one replaces ``prime factors of a number'' with ``parts of a partition'' in the classical definitions. 

\begin{remark}
If $\mathcal P' \subseteq \mathcal P$ is an infinite subset of $\mathcal P$ closed under partition multiplication and division, then the multiplicative theory presented in this study still holds when the relations are restricted to $\mathcal P'$. %In particular, a partition ideal of order 1 in the sense of Andrews \cite{Andrews} --- such as the set of partitions into prime parts --- is closed under partition multiplication and division, and conforms to the structures noted here.
\end{remark}

\  
\  
\  
\

\section{Partition-theoretic analogs of classical functions}\label{ch3muphisigma}

A number of important functions from classical number theory have partition-theoretic analogs, giving rise to nice summation identities that generalize their classical counterparts. %We see examples of this phenomenon in the author's study of partition zeta functions \cite{Robert}.%, and the theory outlined there extends to a more general class of partition-theoretic analogs of classical Dirichlet series
One of the most fundamental classical arithmetic functions, related to factorization of integers, is the M{\"o}bius function. As in Definition \ref{ch1moebiusdef}, we can define a natural partition-theoretic analog of $\mu$ as well:
%
%\begin{definition}\label{ch3moebiusdef}
%For $\lambda\in\mathcal P$ we define a partition-theoretic M{\"o}bius function $\mu_{\mathcal P}(\lambda)$ as follows:
$$
\mu_{\mathcal P}(\lambda):= \left\{
        \begin{array}{ll}
            1 & \text{if $\lambda=\emptyset$,}\\
            0 & \text{if $\lambda$ has any part repeated,}\\
            (-1)^{\ell(\lambda)} & \text{otherwise.}
        \end{array}
    \right.
$$
%\end{definition}

Just as in the classical case, we have by inclusion-exclusion the following, familiar relation.

\begin{proposition}\label{ch3musum}
Summing $\mu_{\mathcal P}(\delta)$ over the divisors $\delta$ of $\lambda\in\mathcal P$, we have
$$
\sum_{\delta|\lambda}\mu_{\mathcal P}(\delta)= \left\{
        \begin{array}{ll}
            1 & \text{if $\lambda=\emptyset$,}\\
            0 & \text{otherwise.}
        \end{array}
    \right.
$$
\end{proposition}

Furthermore, we have a partition-theoretic generalization of the M{\"o}bius inversion formula, which is proved along the lines of proofs of the classical formula.
\begin{proposition}\label{ch3mobinv}
For a function $f\colon \mathcal P \to \C$ we have the equivalence
$$
F(\lambda)=\sum_{\delta|\lambda}f(\delta)\  \Longleftrightarrow\  f(\lambda)=\sum_{\delta|\lambda}F(\delta)\mu_{\mathcal P}(\lambda / \delta)
.$$

\end{proposition}

\begin{remark}
Alladi has also considered the above partition M{\"o}bius function identities, in unpublished work\footnote{K. Alladi, private communication, Dec. 21, 2015}.% \cite{Alladi2015}.
\end{remark}

In classical number theory, M{\"o}bius inversion is often used in conjunction with order-of-summation swapping principles for double summations. These have an obvious partition-theoretic generalization as well, reflected in the following identity.

\begin{proposition}\label{ch3sumswap}
Consider a double sum involving functions $f,g\colon \mathcal P \to \C$. Then we have the formula
$$\sum_{\lambda\in\mathcal P} f(\lambda) \sum_{\delta | \lambda} g(\delta) = \sum_{\lambda\in\mathcal P} g(\lambda) \sum_{\gamma\in\mathcal P} f(\lambda \gamma).
$$
\end{proposition}

%\begin{remark}
%To be very cautious, we might assume the subset $\mathcal P'$ to be a partition ideal of order 1 in the sense of Andrews \cite{Andrews}, as these subsets of $\mathcal P$ are closed under multiplication and division of partitions. %All of the subsets of $\mathcal P$ considered in this paper are partition ideals of order 1.
%\end{remark}

The preceding propositions will prove useful in the next section, to evaluate the coefficients of the $q$-bracket operator. %Continuing the theme of connections between arithmetic functions and partition theory, 

Now, the classical M{\"o}bius function has a close companion in the Euler phi function $\varphi(n)$, also known as the totient function, which counts the number of natural numbers less than $n$ that are coprime to $n$. This sort of statistic does not seem meaningful in the partition-theoretic frame of reference, as there is not generally a well-defined greater- or less-than ordering of partitions. However, if we sidestep this business of ordering and counting for the time being, we find it is possible to define a partition analog of $\varphi$ which is naturally compatible with the identities above, as well as with classical identities involving the Euler phi function. 

Recall that $n_{\lambda}$ denotes the norm of $\lambda$, i.e., the product of its parts.

\begin{definition}\label{ch3phidef}
For $\lambda\in\mathcal P$ we define a partition-theoretic phi function $\varphi_{\mathcal P}(\lambda)$ by
$$
\varphi_{\mathcal P}(\lambda):=\  n_{\lambda}\prod_{\substack{\lambda_i\in\lambda\\ \text{without}\\  \text{repetition}}}(1-\lambda_i^{-1})
,$$
where the product is taken over only the distinct numbers composing $\lambda$, that is, the parts of $\lambda$ without repetition.
\end{definition}

Clearly if $1\in\lambda$ then $\varphi_{\mathcal P}(\lambda)=0$, which is a bit startling by comparison with the classical phi function that never vanishes. This phi function filters out partitions containing 1's. %(This is a convenient property: partitions containing 1's can create convergence issues in partition zeta functions and other partition-indexed series.)

As with the M{\"o}bius function above, the partition-theoretic $\varphi_{\mathcal P}(\lambda)$ yields generalizations of many classical expressions. For instance, there is a familiar-looking divisor sum, which is proved along classical lines.

\begin{proposition}\label{ch3phisum}
We have that
$$
\sum_{\delta|\lambda}\varphi_{\mathcal P}(\delta)=n_{\lambda}
.$$
\end{proposition}  

%\begin{remark}
%We note that it is not difficult to prove a variant of this formula related to sums over subpartitions of $\lambda$ containing $t$ as a part, viz.
%\begin{equation}
%\sum_{ \substack {\delta | \lambda \\ t \in \delta} } \varphi_{\mathcal P}(\delta) = \frac{n_{\lambda}\  (t^{m_t+1}-1)}{t^{m_t}(t+1)} 
%,\end{equation}
%where $m_t$ is the multiplicity of $t$ in $\lambda$. 
%\end{remark}

We also find a partition analog of the well-known relation connecting the $\mu$ and $\varphi$.

\begin{proposition}\label{ch3phimoeb}
We have the identity
$$
\varphi_{\mathcal P}(\lambda)=n_{\lambda}\sum_{\delta|\lambda}\frac{\mu_{\mathcal P}(\delta)}{n_{\delta}}
.$$
\end{proposition}  

Combining the above relations, we arrive at a nicely balanced identity. 

\begin{proposition}\label{ch3phimoeb2}
For $f\colon \mathcal P\to\C$ let $F(\lambda):=\sum_{\delta|\lambda}f(\delta)$. Then we have
$$
\sum_{\lambda\in\mathcal P}\frac{\mu_{\mathcal P}(\lambda)F(\lambda)}{n_{\lambda}}
=\sum_{\lambda\in\mathcal P}\frac{\varphi_{\mathcal P}(\lambda)f(\lambda)}{n_{\lambda}}
.$$
\end{proposition}

\begin{remark}
Replacing $\mathcal P$ with the set $\mathcal P_{\P}$ of partitions into prime parts (the so-called ``prime partitions''), then the divisor sum above takes the form $F(n)=\sum_{d|n}f(d)$ (with $n=n_{\lambda}$) and Proposition \ref{ch3phimoeb2} specializes to the following identity, which is surely known classically:
$$
\sum_{n=1}^{\infty}\frac{\mu(n)F(n)}{n}
=\sum_{n=1}^{\infty}\frac{\varphi(n)f(n)}{n}.
$$
\end{remark}

%We see above that, in a sense, the weights of $\mu$ (which is ``small'' by comparison to $\varphi$) and $F$ (which is ``large'' by comparison to $f$) are balanced perfectly with respect to the weights of $\varphi$ and $f$ (which also form a product roughly of the form ``$\text{large}\times \text{small}$'').  

A number of other important arithmetic functions have partition-theoretic analogs, too\footnote{Of course, $(-1)^{\ell(\lambda)}$ is the analog of Liouville's function $(-1)^{\text{\#\{prime factors of $n$\}}}$ (usually denoted by $\lambda(n)$ in the literature) 
in this setting, and we have the classical-like identity $\sum_{\delta|\lambda}(-1)^{\ell(\delta)}=1$ if all multiplicities $m_i(\lambda)$ are even and $=0$ otherwise.}
, such as the sum-of-divisors function $\sigma_a$. \begin{definition}\label{ch3sigma}
For $\lambda\in\mathcal P, a\in\Z_{\geq 0}$, we define the function 
$$
\sigma_{\mathcal P,a}(\lambda):=\sum_{\delta|\lambda}n_{\delta}^a,
$$
with the convention $\sigma_{\mathcal P}(\lambda):=\sigma_{\mathcal P,1}(\lambda)$.
\end{definition}
One might wonder about ``perfect partitions'' or other analogous phenomena related to $\sigma_a$ classically. This partition sum-of-divisors function will come into play in the next section. We note that the functions $\mu_{\mathcal P}, \varphi_{\mathcal P}$ and $\sigma_{\mathcal P, a}$ are, just as in the classical cases, multiplicative in a partition sense.

\begin{definition} \label{ch3multfctn}
We say a function $f\colon \mathcal P\to \mathbb C$ is multiplicative (resp. completely multiplicative) if for $\lambda, \gamma \in \mathcal P$ with $\operatorname{gcd}(\lambda, \gamma)=\emptyset$ (resp. with no condition on the $\operatorname{gcd}$),  $$f(\lambda \gamma)=f(\lambda)f(\gamma)$$ \end{definition}

Another classical principle central to analysis is the Cauchy product formula, which gives the product of two infinite series in terms of the convolution of their summands. In the partition-theoretic setting, the Cauchy product takes the following form, in which the summands effectively give a partition version of Dirichlet convolution from multiplicative number theory (see \cite{Apostol}).

\begin{proposition} \label{ch3cauchyprod}
Consider the product of two absolutely convergent sums over partitions, whose summands involve the functions $f,g\colon \mathcal P \to \C$. Then we have the formula
$$\left(\sum_{\lambda\in\mathcal P}f(\lambda) \right) \left(\sum_{\lambda\in\mathcal P}g(\lambda) \right) =\sum_{\lambda\in\mathcal P}\sum_{\delta | \lambda}f(\delta)g(\lambda/\delta).$$
\end{proposition}

The proof of this partition Cauchy product proceeds exactly as in the classical case: we expand the left-hand side and compare the resulting terms\footnote{In Appendix B we apply this partition Cauchy product formula to give coefficients of Ramanujan's tau function and the counting function for $k$-color partitions.}. Then it is immediate that the product of two partition-indexed power series for $|q|<1$ is
\begin{equation}
\left(\sum_{\lambda\in\mathcal P}f(\lambda)q^{|\lambda|} \right) \left(\sum_{\lambda\in\mathcal P}g(\lambda)q^{|\lambda|}  \right) =\sum_{\lambda\in\mathcal P}q^{|\lambda|} \sum_{\delta | \lambda}f(\delta)g(\lambda/\delta).
\end{equation}

%We close this section by noting that the theory of partition zeta functions outlined in the author's recent report \cite{Robert} extends to a more general class of partition-theoretic analogs of classical Dirichlet series.
%
%\begin{definition}\label{ch3Dirichlet}
%Given an arbitrary subset $\mathcal P'\subseteq \mathcal P$ and function $f \colon \mathcal P' \rightarrow \C$, we define partition-theoretic Dirichlet series of the form
%\[
%D_{\mathcal P'}(f,s):=\sum_{\lambda\in\mathcal P'}f(\lambda)n_{\lambda}^{-s},
%\]
%where the sum is taken over partitions in $\mathcal P'$ and we assume $\operatorname{Re}(s)>1$. 
%\end{definition}
%
%These partition Dirichlet series transform much as in the classical case when combined with the ideas presented above. The product-sum identities from \cite{Robert} also generalize, under the right conditions.
%
%\begin{proposition}
%\label{ch3DirichletProduct2}
%For an arbitrary subset $X\subseteq\Z^+$, let $\mathcal P_{X}$ denote the set of partitions into elements of $X$. Then if $f\colon \mathcal P\to\C$ is completely multiplicative and $\operatorname{Re}(s)>1$, we have the identity
%$$D_{\mathcal \mathcal P_{X}}(f,s)=\prod_{n\in X}\left(1-\frac{f(n)}{n^s}\right)^{-1}.$$
%\end{proposition}

We reiterate, these familiar-looking identities not only mimic classical theorems, they fully generalize the classical cases. The definitions and propositions above all specialize to their classical counterparts when we restrict our attention to the set $\mathcal P_{\P}$ of prime partitions; then, as a rule-of-thumb, we just replace partitions with their ``norms'' in the formulas (other parameters may need to be adjusted too). This is due to the bijective correspondence between natural numbers and $\mathcal P_{\P}$ noted by Alladi and Erd\H{o}s \cite{AlladiErdos}: the set of ``norms'' of prime partitions (including $n_{\emptyset}$) is precisely the set of positive integers $\Z^+$, by the Fundamental Theorem of Arithmetic. Yet prime partitions form a narrow slice, so to speak, of the set $\mathcal P$ over which these general relations hold sway. 

Many well-known laws of classical number theory arise as special cases of underlying partition-theoretic structures.

\  
\  
\  
\

\section{Role of the \texorpdfstring{$q$}{Lg}-bracket} \label{ch33}% (and ``\texorpdfstring{$q$}{Lg}-antibracket'')}

We return now to the $q$-bracket operator of Bloch--Okounkov, which we recall from Definition \ref{ch1qbracket}. The $q$-bracket arises naturally in the multiplicative theory outlined above. To see this, take $F(\lambda)=\sum_{\delta|\lambda}f(\delta)$ for $f\colon \mathcal P \to\C$. It follows from Proposition \ref{ch3sumswap} that
\begin{align*}
\sum_{\lambda\in\mathcal P} F(\lambda)q^{|\lambda|} & =\sum_{\lambda\in\mathcal P} q^{|\lambda|} \sum_{\delta | \lambda} f(\delta) = \sum_{\lambda\in\mathcal P} f(\lambda) \sum_{\gamma\in\mathcal P} q^{|\lambda \gamma|}
\\
&=\sum_{\lambda\in\mathcal P} f(\lambda) \sum_{\gamma\in\mathcal P} q^{|\lambda|+|\gamma|}
=\left(\sum_{\lambda\in\mathcal P} f(\lambda)q^{|\lambda|}\right) \left(\sum_{\gamma\in\mathcal P} q^{|\gamma|}\right)
.
\end{align*}

%%%This started out much like standard proofs involving Lambert series summed over natural numbers, but went more smoothly due to the fact $|\lambda \gamma|=|\lambda|+|\gamma|$, which is a special feature of partition-theoretic summations.

Observing that the rightmost sum above is equal to $(q;q)_{\infty}^{-1}$, then by comparison with Definition \ref{ch1qbracket} of the $q$-bracket operator, we arrive at the two central theorems of this study. Together they give a type of $q$-bracket inversion, converting divisor sums into power series, and vice versa. 

\begin{theorem}\label{ch3thm1}
For an arbitrary function $f\colon \mathcal P \to\C$, if
$$
F(\lambda)=\sum_{\delta|\lambda}f(\delta)\
$$
then
$$
\left<F\right>_q=\sum_{\lambda\in\mathcal P}f(\lambda)q^{|\lambda|}.
$$
\end{theorem}
 
In the converse direction, we can also write down a simple function whose $q$-bracket is a given partition-indexed power series.

\begin{theorem}\label{ch3thm1.5}
Consider an arbitrary power series  of the form
$$\sum_{\lambda\in\mathcal P}f(\lambda)q^{|\lambda|}.
$$
Then we have the function $F\colon \mathcal P\to\C$ given by
$$F(\lambda)=\sum_{\delta|\lambda}f(\delta),
$$  
such that $\left<F\right>_q=\sum f(\lambda)q^{|\lambda|}.$
\end{theorem}

%\begin{corollary}
%Any power series of the form $\sum_{\lambda} a_{\lambda}q^{|\lambda|}$ is the $q$-bracket $\left<A\right>_q$ of the partition divisor sum $A(\lambda)=\sum_{\delta|\lambda}a_{\delta}$. Conversely, any function $B(\lambda)$ defined on partitions can be written as a divisor sum $\sum_{\delta|\lambda}b_{\delta}$ involving coefficients $b_{*}$ of its $q$-bracket, as in Theorem \ref{ch3thm1}. 
%\end{corollary}

%\begin{remark}
%It seems possible there are other functions whose $q$-brackets are also $\sum f(\lambda) q^{|\lambda|}$.
%\end{remark}

These theorems are consequences of Theorem \ref{ch3cauchyprod}. We wish to apply Theorems \ref{ch3thm1} and \ref{ch3thm1.5} to examine the $q$-brackets of partition-theoretic analogs of classical functions introduced in Section \ref{ch3muphisigma}. 

Recall Definition \ref{ch3sigma} of the sum of divisors function $\sigma_{\mathcal P,a}(\lambda)$. Then $\sigma_{\mathcal P,0}(\lambda)=\sum_{\delta | \lambda}1$ counts the number of partition divisors (i.e., sub-partitions) of $\lambda \in\mathcal P$, much as in the classical case. It is immediate from Theorem \ref{ch3thm1} that
\begin{equation}\label{ch3sigmabracket}
\left<\sigma_{\mathcal P,0}\right>_q=(q;q)_{\infty}^{-1}.
\end{equation}
If we note that $(q;q)_{\infty}$ is also a factor of the $q$-bracket on the left-hand side, we can see as well
\begin{equation}\label{ch3sigmabracket2}
\sum_{\lambda\in\mathcal P}\sigma_{\mathcal P,0}(\lambda)q^{|\lambda|}=(q;q)_{\infty}^{-2}
.
\end{equation}
Remembering also from Equation \ref{ch3rank} the identity $\left<\operatorname{rk}\right>_q=(q;q)_{\infty}$, we have seen a few instances of interesting power series connected to powers of $(q;q)_{\infty}$ via the $q$-bracket operator. %We will see more.

%The first identity is immediate from the theorem; the second follows by noting that $(q;q)_{\infty}$ is also a factor of the $q$-bracket on the left-hand side of the first equation.

Now let us recall the handful of partition divisor sum identities from Section \ref{ch3muphisigma} involving the partition-theoretic functions $\varphi_{\mathcal P},\  \sigma_{\mathcal P,a}$, and the ``norm of a partition'' function $n_{*}$. Theorem \ref{ch3thm1} reveals that these three functions form a close-knit family, related through (double) application of the $q$-bracket.

\begin{corollary}\label{ch3qbracketsystem}
We have the pair of identities 
\begin{align*}\label{ch3system}
\left<\sigma_{\mathcal P}\right>_q &=\sum_{\lambda \in \mathcal P}n_{\lambda}q^{|\lambda|}
\\
\left<n_{*}\right>_q &=\sum_{\lambda \in \mathcal P}\varphi_{\mathcal P}(\lambda)q^{|\lambda|}
.\end{align*}
\end{corollary}

%\begin{remark}
%A more general statement holds for a more general version of the $q$-bracket where the sums are taken over $\mathcal P_{X}$ (partitions into elements of $X\subseteq\Z^+$) in the bracket definition.
%\end{remark}

%Of course, this same method of proof, applied to integer sums, gives us relations involving Eisenstein series and $\sigma_k(n)$ which lead to generators for modular and quasimodular forms. Then it is not so surprising that quasimodular phenomena might also arise from the $q$-bracket operator.

The coefficients of $\left<\sigma_{\mathcal P}\right>_q$ are of the form $n_{*}$; applying the $q$-bracket a second time to the function $n_{*}$ gives us the rightmost summation, whose coefficients are the values of $\varphi_{\mathcal P}$.

In fact, it is evident that this operation of applying the $q$-bracket more than once can be continued indefinitely; thus we feel the need to introduce a new notation, on analogy to differentiation.

\begin{definition}\label{ch3diffndefantidef}
If we apply the $q$-bracket repeatedly, say $n\geq 0$ times, to the function $f$, we denote this operator by $\left<f\right>_q^{(n)}$. We define $\left<f\right>_q^{(n)}$ by the equation
$$\left<f\right>_q^{(n)}:=(q;q)_{\infty}^n\sum_{\lambda \in \mathcal P}f(\lambda)q^{|\lambda|}\in\C[[q]].$$
\end{definition}

\begin{remark}
It follows from the definition above that $\left<f\right>_q^{(0)}=\sum_{\lambda \in \mathcal P}f(\lambda)q^{|\lambda|},\  \left<f\right>_q^{(1)}=\left<f\right>_q$. %, and $\left<f\right>_q^{(a)}\left<f\right>_q^{(b)}=\left<f\right>_q^{(a+b)}$ for $a,b\geq 0$. 
\end{remark}

Theorem \ref{ch3thm1.5} gives us a converse construction as well, allowing us to write down a function $F(\lambda)$ whose $q$-bracket is a given power series, i.e., a {\it $q$-antibracket} from Definition \ref{ch1antibracketdef}. % $$\left< f \right>_q^{-1}:=\sum_{\lambda \in \mathcal P} {F}(\lambda)q^{|\lambda|}\  \  \text{such that}\  \  \left< F \right>_q=\sum_{\lambda}f(\lambda)q^{|\lambda|}.$$ %Then we define an inverse ``$q$-antibracket'', analogous to the antiderivative in calculus.
%\begin{definition}
%
%We call $F\colon \mathcal P\to \C$ a $q$-{\it antibracket} of $f$ if $\left<F\right>_q= \sum_{\lambda \in \mathcal P}f(\lambda) q^{|\lambda|}$.% for $f\colon \mathcal P\to \C$, we call $F$ a ``$q$-antibracket of $f$'' (or sometimes just an ``antibracket'').
%%Given the power series $\hat{f}(q):=\sum_{\lambda \in \mathcal P}f(\lambda)q^{|\lambda|}$ for $f\colon \mathcal P\to \C$, if we find a function $F\colon \mathcal P\to \C$ whose $q$-bracket is $\hat{f}(q)$ exactly, we call $F$ a ``$q$-antibracket of $f$'' (or sometimes just an ``antibracket'').% and write $\left<f\right>_q^{(-1)}=\sum_{\lambda \in \mathcal P}F(\lambda)q^{|\lambda|}$. 
%\end{definition}
%
%\begin{remark}
%Much like antiderivatives, a $q$-antibracket of $f$ is generally not unique.
%\end{remark} 
The act of taking the antibracket might be carried out repeatedly as well. %However, there may be different functions $F$ having $\left<F\right>_q=\sum f(\lambda)q^{|\lambda|}$. 
We define a canonical class of $q$-antibrackets related to $f$ by extending Definition \ref{ch3diffndefantidef} to allow for negative values of $n$.

\begin{definition}\label{ch3antidef}
If we repeatedly divide the power series $\sum_{\lambda\in\mathcal P} f(\lambda)q^{|\lambda|}$ by $(q;q)_{\infty}$, say $n > 0$ times, we notate this operator as 
$$\left<f\right>_q^{(-n)}:=(q;q)_{\infty}^{-n}\sum_{\lambda \in \mathcal P}f(\lambda)q^{|\lambda|}\in\C[[q]].$$
We take the resulting power series to be indexed by partitions, unless otherwise specified. We call the function on $\mathcal P$ defined by the coefficients of $\left<f\right>_q^{(-1)}$ the ``canonical $q$-antibracket'' of $f$ (or sometimes just ``the antibracket'').
\end{definition}

Taken together, Definitions \ref{ch3diffndefantidef} and \ref{ch3antidef} describe an infinite family of $q$-brackets and antibrackets. The following identities give an example of such a family (and of the use of these new bracket notations).

\begin{corollary}\label{ch3compact}
%Thus $\left<f\right>_q^{(a)}\left<f\right>_q^{(b)}=\left<f\right>_q^{(a+b)}$ holds for negative values of $a$ or $b$ as well. 

Corollary \ref{ch3qbracketsystem} can be written more compactly as 
$$
\left<\sigma_{\mathcal P}\right>_q^{(2)}=\left<n_{*}\right>_q^{(1)}=\left<\varphi_{\mathcal P}\right>_q^{(0)}
.$$ 
We can also condense Corollary \ref{ch3qbracketsystem} by writing
$$
\left<\sigma_{\mathcal P}\right>_q^{(0)}=\left<n_{*}\right>_q^{(-1)}=\left<\varphi_{\mathcal P}\right>_q^{(-2)}
.$$
\end{corollary}

Both of the compact forms above preserve the essential message of Corollary \ref{ch3qbracketsystem}, that these three partition-theoretic functions are directly connected through the $q$-bracket operator, or more concretely (and perhaps more astonishingly), simply through multiplication or division by powers of $(q;q)_{\infty}$. 

Along similar lines, we can encode Equations \eqref{ch3rank}, \eqref{ch3sigmabracket}, and \eqref{ch3sigmabracket2} in a single statement, noting an infinite family of power series that contains $\left<\operatorname{rk}\right>_q$ and $\left<\sigma_{\mathcal P,0}\right>_q$.

\begin{corollary}\label{ch3ranksigma}
For $n\in\Z$, we have the family of $q$-brackets
\[
\left<\operatorname{rk}\right>_q^{(n)}
=\left<\sigma_{\mathcal P,0}\right>_q^{(n+2)}\\
=(q;q)_{\infty}^n
.\]
\end{corollary}

\begin{remark}
Here we see the $q$-bracket connecting with modularity properties. For instance, another member of this family is $\left<\operatorname{rk}\right>_q^{(24)}=q^{-1}\Delta (\tau)$, where $\Delta$ is the important modular discriminant function having Ramanujan's tau function as its coefficients \cite{Ono_web}.
\end{remark}

The identities above worked out easily because we knew in advance what the coefficients of the $q$-brackets should be, due to the divisor sum identities from Section \ref{ch3muphisigma}. Theorems \ref{ch3thm1} and \ref{ch3thm1.5} provide a recipe for turning partition divisor sums $F$ into coefficients $f$ of power series, and vice versa. 

However, generally a function $F\colon \mathcal P\to\C$ is not given as a sum over partition divisors. If we wished to write it in this form, what function $f\colon \mathcal P\to\C$ would make up the summands? In classical number theory this question is answered by the M{\"o}bius inversion formula; indeed, we have the partition-theoretic analog of this formula in Equation \eqref{ch3mobinv}. 

Recall the ``divided by'' notation $\lambda / \delta$ from Definition \ref{ch1divisiondef}. Then we may write the function $f$ (and thus the coefficients of $\left<F\right>_q$) explicitly using partition M{\"o}bius inversion. 

\begin{theorem}\label{ch3thm2}
The $q$-bracket of the function $F\colon \mathcal P\to\C$ is given explicitly by
$$
\left<F\right>_q=\sum_{\lambda \in \mathcal P}f(\lambda)q^{|\lambda|}
,
$$
where the coefficients can be written in terms of $F$ itself:
$$
f(\lambda)=\sum_{\delta|\lambda}F(\delta)\mu_{\mathcal P}(\lambda / \delta).
$$
\end{theorem}

%{\bf To do now: formulas for %1. case that $F(\lambda)$ is not obviously a divisor sum (partition Moebius inversion), 
%2. case of power series over integers with coefficients as divisor sums, 3. case of power series with arbitrary coefficients.... Then give exact formula for coefficients of reciprocal of JTP as example application.}

We already know from Theorem \ref{ch3thm1.5} that the coefficients of the canonical antibracket of $f$ are written as divisor sums over values of $f$. Thus, much like $\operatorname{rk}(\lambda)$ in Corollary \ref{ch3ranksigma}, every function $f$ defined on partitions can be viewed as the generator, so to speak, of the (possibly infinite) family of power series $\left<f\right>_q^{(n)}$ for $n\in\Z$, whose coefficients can be written in terms of $f$ as $n$-tuple sums of the shape $\sum_{\delta_1|\lambda}\sum_{\delta_2|\delta_1}...\sum_{\delta_n|\delta_{n-1}}$ constructed by repeated application of the above theorems.
%
%\begin{corollary}\label{ch3family}
%For $n\in\Z,f\colon \mathcal P\to\C,$ we have the family of $q$-brackets
%\[
%\left<f\right>_q^{(n)}=\sum_{\lambda \in \mathcal P}f_{(n)}(\lambda)q^{|\lambda|}
%,\]
%where the coefficients $f_{(n)}\colon \mathcal P\to\C$ can be written in terms of $f$:
%
%\[
%f_{(n)}(\lambda)= \left\{
%        \begin{array}{ll}
%            \sum_{\delta_1|\lambda}\sum_{\delta_2|\delta_1}...\sum_{\delta_n|\delta_{n-1}}f(\delta_n) & \text{if $n<0$,}\\
%            \  \\
%            f(\lambda) & \text{if $n=0$,}\\
%            \  \\
%            \sum_{\delta_1|\lambda}\sum_{\delta_2|\delta_1}...\sum_{\delta_n|\delta_{n-1}}f(\delta_n)\mu_{\mathcal P}(\delta_{n-1}/\delta_n) & \text{if $n>0$}
%        \end{array}
%    \right.
%\]
%
%\end{corollary}
%
%
%\  
%\  

This suggests the following useful fact. 

\begin{corollary}\label{ch3two-way}
If two power series are members of the family $\left<f\right>_q^{(n)}\  \left(n\in\Z\right),$ then the coefficients of each series can be written explicitly in terms of the coefficients of the other. 
\end{corollary}

\section{The \texorpdfstring{$q$}{Lg}-antibracket and coefficients of power series over \texorpdfstring{$\Z_{\geq 0}$}{Lg}}\label{ch3Applications of the $q$-antibracket}

Theorems \ref{ch3thm1}, \ref{ch3thm1.5}, and \ref{ch3thm2} together provide a two-way map between the coefficients of families of power series indexed by partitions. In this section, we address the question of computing the antibracket (loosely speaking) of coefficients indexed not by partitions, but by natural numbers as usual. We remark immediately that a function defined on $Z_{\geq 0}$ may be expressed in terms of partitions in a number of ways, which are generally not equivalent. Thus there is more than one function $F\colon \mathcal P\to\C$ such that $\left<F\right>_q=\sum_{n=0}^{\infty} c_n q^n$ for a given sequence $c_n$ of coefficients. Here we treat only the canonical antibracket found using Theorem \ref{ch3thm1.5}.
 
There are three classes of power series of the form $\sum_{n=0}^{\infty} c_n q^n$ that we examine: (1) the coefficients $c_n$ are sums $\sum_{\lambda \vdash n}$ over partitions of $n$; (2) the coefficients $c_n$ are sums $\sum_{d | n}$ over divisors of $n$; and (3) the coefficients $c_n$ are an arbitrary sequence of complex numbers. 

The class (1) above is already given by Theorem \ref{ch3thm1}; to keep this section relatively self-contained, we rephrase the result here. 

\begin{corollary}\label{ch3class1}
For $c_n=\sum_{\lambda \vdash n}f(\lambda)$ we can write
\begin{equation*}
\sum_{n=0}^{\infty}c_n q^n=\sum_{\lambda\in\mathcal P}f(\lambda)q^{|\lambda|}.
\end{equation*}
Then we have a function $F(\lambda)=\sum_{\delta|\lambda}f(\delta)$ such that 
$
\left<F\right>_q=\sum_{n=0}^{\infty}c_n q^n
$.
\end{corollary}

Thus the power series of class (1) are already in a form subject to the $q$-bracket machinery detailed in the previous section. The class (2) with coefficients of the form $\sum_{d|n}$ is a little more subtle. We introduce a special subset $\mathcal P_{=}$ which bridges sums over partitions and sums over the divisors of natural numbers. 

\begin{definition}\label{ch3=def}
We define the subset $\mathcal P_{=}\subseteq\mathcal P$ to be the set of partitions into equal parts, that is, whose parts are all the same positive number, e.g. $(1),(1,1),(4,4,4)$. We make the assumption $\emptyset\notin\mathcal P_{=}$, as the empty partition has no positive parts.
\end{definition}

The divisors of $n$ correspond exactly (in two different ways) to the set of partitions of $n$ into equal parts, i.e., partitions of $n$ in $\mathcal P_{=}$. For example, compare the divisors of $6$
$$
1,2,3,6
$$
with the partitions of $6$ into equal parts 
$$
(6),(3,3),(2,2,2),(1,1,1,1,1,1)
.$$ 
Note that for each of the above partitions $(a,a,...,a)\vdash 6$, we have that $a\cdot\ell\left((a,a,...,a)\right)$ $=6$. We see from this example that for any $n\in\Z^+$ we can uniquely associate each divisor $d|n$ to a partition $\lambda\vdash n,\lambda\in\mathcal P_{=},$ by taking $d$ to be the length of $\lambda$. (Alternatively, we could identify the divisor $d$ with  $\operatorname{lg}(\lambda)$ or $\operatorname{sm}(\lambda)$, as defined above, which of course are the same in this case. We choose here to associate divisors to $\ell(\lambda)$ as length is a universal characteristic of partitions, regardless of the structure of the parts.) %The complement divisor $n/d$ is then equal to either the largest/smallest part or to the length of $\lambda$, respectively, that is, to the opposite statistic to the one associated with $d$.

By the above considerations, it is clear that
\begin{equation}\label{ch3=prop}
\sum_{d|n}f(d)=\sum_{\substack{\lambda\vdash n\\  \lambda\in\mathcal P_{=}}}f\left(\ell(\lambda)\right).
\end{equation}
This leads us to a formula for the coefficients of a power series of the class (2) discussed above.  

\begin{corollary}\label{ch3class2}
For $c_n=\sum_{d|n}f(d)$ we can write
\begin{equation*}\label{ch34.1}
\sum_{n=0}^{\infty}c_n q^n=\sum_{\lambda\in\mathcal P_{=}}f\left(\ell(\lambda)\right)q^{|\lambda|}.
\end{equation*}
Then we have a function 
$$F(\lambda)=\sum_{\substack{\delta|\lambda \\  \delta\in\mathcal P_{=}}}f\left(\ell(\delta)\right)$$ such that $
\left<F\right>_q=\sum_{n=0}^{\infty}c_n q^n
$.
\end{corollary}

The completely general class (3) of power series with arbitrary coefficients $c_n\in\C$ follows right away from Corollary \ref{ch3class2} by classical M\"{o}bius inversion, as $f(n):=\sum_{d|n}c_d\  \mu_{\mathcal P}(n/d)\Rightarrow c_n=\sum_{d|n}f(d)$.

\begin{corollary}\label{ch3class3}
For $c_n\in\C$ we can write
\begin{equation*}
\sum_{n=0}^{\infty}c_n q^n=\sum_{\lambda\in\mathcal P_{=}}q^{|\lambda|}\sum_{d|\ell(\lambda)}c_d\  \mu\left(\frac{\ell(\lambda)}{d}\right).
\end{equation*}
Then we have a function 
$$F(\lambda)=\sum_{\substack{\delta|\lambda \\  \delta\in\mathcal P_{=}}}\sum_{d|\ell(\delta)}c_d\  \mu\left(\frac{\ell(\delta)}{d}\right)$$ such that $
\left<F\right>_q=\sum_{n=0}^{\infty}c_n q^n
$.
\end{corollary}

\begin{remark} 
We point out an alternative expression for sums of the shape of $F(\lambda)$ here, that can be useful for computation. If we write out the factorization of a partition $\lambda=$ $(a_1)^{m_1}(a_2)^{m_2}...$ $(a_t)^{m_t}$ as in Section \ref{ch3Multiplicative theory of partitions}, a divisor of $\lambda$ lying in $\mathcal P_{=}$ must be of the form $(a_i)^m$ for some $1\leq i \leq t$ and $1 \leq m \leq m_i$. Then for any function $\phi$ defined on $\Z^+$ we see 
\begin{equation}
\sum_{\substack{\delta|\lambda \\  \delta\in\mathcal P_{=}}}\phi\left(\ell(\delta)\right)=\sum_{i=1}^{t}\sum_{j=1}^{m_i}\phi\left(\ell((a_i)^j)\right)=\sum_{i=1}^{t}\sum_{j=1}^{m_i}\phi(j).
\end{equation}
\end{remark}

Given the ideas developed above, we can now pass between $q$-brackets and arbitrary power series, summed over either natural numbers or partitions.

\  
\  
\  
\

\section{Applications of the \texorpdfstring{$q$}{Lg}-bracket and \texorpdfstring{$q$}{Lg}-antibracket}\label{ch3applications}
We close this report by briefly illustrating some of the methods of the previous sections through two examples.

\subsection{Sum of divisors function}
In classical number theory, for $a\geq 0$ the divisor sum $\sigma_{\mathcal P,a}(n):=\sum_{d|n}d^a$ is particularly important to the theory of modular forms; as seen in Equation \ref{ch3Eisenstein}, for odd values of $a$, power series  of the form
$$
\sum_{n=0}^{\infty}\sigma_{\mathcal P,a}(n)q^n
$$
comprise the Fourier expansions of Eisenstein series \cite{Ono_web}, which are the building blocks of modular and quasimodular forms. As a straightforward application of Corollary \ref{ch3class2} following directly from the definition of $\sigma_{\mathcal P,a}(n)$, we give a function $\mathcal S_a$ defined on partitions whose $q$-bracket is the power series above.

\begin{corollary}\label{ch3sigmasum}
We have the partition-theoretic function 
$$\mathcal S_a(\lambda):=\sum_{\substack{\delta|\lambda \\  \delta\in\mathcal P_{=}}}\ell(\delta)^a
$$
such that 
$$\left<\mathcal S_a\right>_q=\sum_{n=0}^{\infty}\sigma_{\mathcal P,a}(n) q^n
.$$
\end{corollary}

\begin{remark}
We note that Zagier gives a different function $S_{2k-1}(\lambda)=\sum_{\lambda_i\in\lambda}\lambda_i^{2k-1}$ (the moment function) that also has the $q$-bracket $\sum \sigma_{2k-1}(n) q^n$ \cite{Zagier}. This is an example of the non-uniqueness of antibrackets of functions defined on natural numbers noted previously.
\end{remark}

Thus we see the $q$-bracket operator brushing up against modularity, once again.

%In fact, it is clear from the comments preceding Equation \ref{ch3=prop} that we could replace $f(\ell(\lambda))$ with $f(\text{``largest part of }\lambda\text{''})$ or $f(\text{``smallest part of }\lambda\text{''})$ in \ref{ch3=prop} (and all the subsequent expressions depending on \ref{ch3=prop}) without changing the left-hand side of the equation at all.

\subsection{Reciprocal of the Jacobi triple product}

We turn our attention now to another fundamental object in the subject of modular forms. Let $j(z;q)$ denote the classical \textit{Jacobi triple product} \cite{Berndt}
\begin{equation}\label{ch3jtp}
j(z;q):=(z;q)_{\infty}(z^{-1}q;q)_{\infty}(q;q)_{\infty}.
\end{equation}

The reciprocal of the triple product 
$$j(z;q)^{-1}=\sum_{\lambda\in\mathcal P}j_z(\lambda)q^{|\lambda|}$$
is interesting in its own right. For instance, $j(z;q)^{-1}$ plays a role not unlike the role played by $(q;q)_{\infty}$ in the $q$-bracket operator, for the Appell--Lerch sum $m(x,q,z)$ important to the study of mock modular forms (see \cite{BFOR, HickersonMortenson}).

Our goal will be to derive a formula for the coefficients $j_z(\lambda)$ above. If we multiply $j(z;q)^{-1}$ by $(1-z)$ to cancel the pole at $z=1$, it behaves nicely under the action of the $q$-bracket. Let us write
\begin{equation}\label{ch3normalized}
(1-z)j(z;q)^{-1}=\frac{1}{(zq;q)_{\infty}(z^{-1}q;q)_{\infty}(q;q)_{\infty}}=\sum_{\lambda\in\mathcal P}J_z(\lambda)q^{|\lambda|}.
\end{equation}

Let $\operatorname{crk}(\lambda)$ denote the {\it crank} of a partition, an important partition-theoretic statistic whose existence was conjectured by Dyson \cite{Dyson} to explain the Ramanujan congruence $p(11n+7)\equiv 0\  (\operatorname{mod}\  11)$, and which was written down almost half a century later by Andrews and Garvan \cite{AndrewsGarvan}. Crank is not unlike Dyson's rank, but is a bit more complicated. % (thus we do not give it explicitly). 

\begin{definition}\label{ch8crk}
The {\it crank} $\operatorname{crk}(\lambda)$ of a partition $\lambda$ is equal to its largest part %$\operatorname{lg}(\lambda)$ of partition $\lambda$ 
if the multiplicity $m_1(\lambda)$ of 1 as a part of $\lambda$ is $=0$ (that is, there are no $1$'s), and if $m_1(\lambda)> 0$ then $\operatorname{crk}(\lambda)=\#\{\text{parts of $\lambda$ that are larger than $m_1(\lambda)$}\}-m_1(\lambda).$
\end{definition}

We define $M(n,m)$ to be the number of partitions of $n$ having crank equal to $m\in\Z$; then the Andrews--Garvan \textit{crank generating function} $C(z;q)$ is given by
\begin{equation}\label{ch3crankgen}
C(z;q):=\frac{(q;q)_{\infty}}{(zq;q)_{\infty}(z^{-1}q;q)_{\infty}}=\sum_{n=0}^{\infty}M_z(n)q^n,
\end{equation}
where we set
\begin{equation}\label{ch3M}
M_z(n):=\sum_{\lambda\vdash n}z^{\operatorname{crk}(\lambda)}=\sum_{m=-\infty}^{\infty}M(n,m)z^m.
\end{equation}

The function $C(z;q)$ has deep connections. When $z=1$, Equation \ref{ch3crankgen} reduces to Euler's partition generating function formula \cite{Berndt}. For $\zeta\neq 1$ a root of unity, $C(\zeta;q)$ is a modular form, and Folsom--Ono--Rhoades show the crank generating function to be related to the theory of quantum modular forms \cite{FOR}.  

Comparing Equations \ref{ch3normalized} and \ref{ch3crankgen}, we have the following  relation:
\begin{equation}
\left<J_z\right>_q^{(2)}=C(z;q).
\end{equation}
We see $J_z(\lambda)$ and $z^{\operatorname{crk}(\lambda)}$ are related through a family of $q$-brackets; then using Corollaries \ref{ch3class1}, \ref{ch3class2}, and \ref{ch3class3}, we can write $J_z(\lambda)$ explicitly. Noting that $J_z(\lambda)=(1-z)j_z(\lambda)$, we arrive at the formula we seek.
  
\begin{corollary}\label{ch3jtpcoeff}
The partition-indexed coefficients of $j(z;q)^{-1}$ are
\[
j_z(\lambda)=(1-z)^{-1}\sum_{\delta|\lambda}\sum_{\varepsilon|\delta}z^{\operatorname{crk}(\varepsilon)}
\]
for $z\neq1$. In terms of the coefficients $M_z(*)$ given by Equation \ref{ch3M}:
\[
j_z(\lambda)=(1-z)^{-1}\sum_{\delta|\lambda}\sum_{\substack{\varepsilon|\delta \\  \varepsilon\in\mathcal P_{=}}}\sum_{d|\ell(\varepsilon)}M_z(d)\mu\left(\frac{\ell(\varepsilon)}{d}\right).
\]
\end{corollary}

\begin{remark}
By Corollary \ref{ch3two-way}, we can also write $M_z(n)$ in terms of the coefficients of $j(z;q)^{-1}$.
\end{remark}

%%%%%%FOR THESIS: GIVE FORMULA FOR THE TRIPLE PRODUCT ITSELF AND MOST GENERAL VERSIONS OF Q-BRACKET THEOREMS

We apply the $q$-bracket operator to the function $j(z;q)^{-1}$ from a somewhat different perspective in Chapter 8.%\footnote{See Appendix \ref{ch3app:A} for further notes on Chapter 3.} %the next chapter. %another study \cite{SchneiderJTP}.
\\

\begin{remark}
{See Appendix \ref{app:B} for further notes on Chapter 3.} 
\end{remark}

 	\clearpage%
\chapter{Partition-theoretic zeta functions}{\bf Adapted from \cite{Schneider_zeta}}%, a joint work with Ken Ono and Larry Rolen}

\section{Introduction, notations and central theorem}
The additive-multiplicative connections highlighted in the previous chapter extend to other realms of multiplicative number theory as well, viz., the study of zeta functions and Dirichlet series in analytic number theory\footnote{See \cite{apostolmodular}, Ch. 6, for connections between Dirichlet series and $q$-series via the theory of modular forms.}. 

We need to introduce one more notation, in order to state the central theorem of this chapter. Define $\varphi_n(f;q)$ by $\varphi_0(f;q) := 1$ and 
\begin{equation*}
\varphi_n(f;q) := \prod_{k=1}^n (1-f(k)q^k)
\end{equation*}
where $n\geq 1$, for an arbitrary function $f \colon \mathbb N \rra \mathbb C$. When the infinite product converges, let $\varphi_{\infty}(f;q):= \lim_{n\rra \infty}\varphi_n(f;q)$. We think of $\varphi$ as a generalization of the $q$-Pochhammer symbol. Note that if we set $f$ equal to a constant $z$, then $\varphi$ does specialize to the $q$-Pochhammer symbol, as $\varphi_n(z;q) = (zq;q)_n $ and $\varphi_{\infty}(z;q) = (zq;q)_{\infty}$. 

As with Euler product formula and partition generating function formula \eqref{ch1zetaproduct} and \eqref{ch1partition_genfctn}, respectively, it is the reciprocal $1/\varphi_{\infty}(f;q)$ that interests us. With the above notations, we have the following system of identities.

\begin{theorem}\label{ch41.1}
If the product converges, then $1/\varphi_{\infty}(f;q) = \prod_{n=1}^{\infty}(1-f(n)q^n)^{-1}$ may be expressed in a number of equivalent forms, viz.
\begin{align}
\frac{1}{\varphi_{\infty}(f;q)} & =  \sum_{\lambda \in \mathcal P} q^{|\lambda|} \prod_{\lambda_i \in \lambda} f(\lambda_i) \label{ch44}\\ %\displaybreak
& =  1+ \sum_{n=1}^{\infty} q^n \frac{f(n)}{\varphi_n(f;q)} \label{ch45}\\ 
& =  1+ \frac{1}{\varphi_{\infty}(f;q)} \sum_{n=1}^{\infty} q^n f(n)\varphi_{n-1}(f;q) \label{ch46} \\ 
& =  1 + \sum_{n=1}^{\infty} \frac{(-1)^n (q\inv)^{\frac{n(n-1)}{2}}}{\varphi_n\pwr{\frac{1}{f};q\inv} \prod_{k=1}^{n-1}f(k)} \label{ch47}  \\
& = 1+\cfrac{\sum_{(6)}}{1-\cfrac{\sum_{(5)}}{1+\cfrac{\sum_{(6)}}{1-\cfrac{\sum_{(5)}}{1+ \cdots}}}} \label{ch48}
\end{align}
where $\sum_{(5)},\sum_{(6)}$ in \eqref{ch48} denote the summations appearing in \eqref{ch45} and \eqref{ch46}, respectively.
\end{theorem}
The product on the right-hand side of identity \eqref{ch44} above is taken over the parts $\lambda_i$ of $\lambda$. Note that the summation in \eqref{ch47} converges for $q\inv$ outside the unit circle (it may converge inside the circle as well). Note also that, by L`Hospital's rule, any power series $\sum_{n=1}^{\infty}f(n)q^n$ with constant term zero can be written as the limit
\begin{equation*}
\sum_{n=1}^{\infty}f(n)q^n = \lim_{z\rra 0}z\inv \pwr{\frac{1}{\varphi_{\infty}(zf ; q)} - 1}.
\end{equation*}
It is obvious that if $f$ is completely multiplicative, then $\prod_{\lambda_i \in \lambda} f(\lambda_i)=f(n_{\lambda})$, where $n_{\lambda}$ is the {\it norm} of $\lambda$ defined above. We record one more, obvious consequence of Theorem \ref{ch41.1}, as we assume it throughout this paper. As before, let $\mathbb {\mathbb X} \subseteq \mathbb Z^+$, and take $\mathcal P_{\mathbb {\mathbb X}} \subseteq \mathcal P$ to be the set of partitions into elements of ${\mathbb {\mathbb X}}$. Then clearly by setting $f(n)=0$ if $n\notin {\mathbb {\mathbb X}}$ in Theorem \ref{ch41.1}, we see
\begin{equation*}
\frac {1}{\prod_{n\in {\mathbb {\mathbb X}}} (1-f(n)q^n)} =  \sum_{\lambda \in \mathcal P_{\mathbb {\mathbb X}}} q^{|\lambda|} \prod_{\lambda_i \in \lambda} f(\lambda_i).
\end{equation*}
The remaining summations in the theorem (aside from \eqref{ch47}, which may not converge) are taken over $n\in {\mathbb {\mathbb X}}$.

We see from Theorem \ref{ch41.1} that we may pass freely between the shapes \eqref{ch44} -- \eqref{ch48}, which specialize to a number of classical expressions. For example, taking $f\equiv 1$ in the theorem gives the following fact.

\begin{coro}\label{ch41.2}
The partition generating function formula \eqref{ch1partition_genfctn} is true.
\end{coro}

Assuming $\Real (s)>1$, if we take $q = 1$, $f(n) = 1/n^s$ if $n$ is prime and $=0$ otherwise, then Theorem \ref{ch41.1} yields another classical fact, plus a formula giving the zeta function as a sum over primes.

\begin{coro}\label{ch41.3}
The Euler product formula \eqref{ch1zetaproduct} for the zeta function is true. We also have the identity 
\begin{equation*}
\zeta(s) = 1 + \sum_{p \in \mathbb P}\frac{1}{p^s \prod_{r\in \mathbb P,\,r\leq p}\pwr{1 - \frac{1}{r^s}}}.
\end{equation*}
\end{coro}

Finally, tying this section in with the previous chapter, we give generating functions for the partition-theoretic phi function $\varphi_{\mathcal P}$ and sum of divisors function $\sigma_{\mathcal P}$. Setting $f(n)=n$ in Theorem \ref{ch41.1}, it is clear we have $$\prod_{n=1}^{\infty}(1-nq^n)^{-1}=\sum_{\lambda\in \mathcal P} n_{\lambda}q^{|\lambda|},$$ the generating function for the norm $n_{\lambda}$. Then Proposition \ref{ch3qbracketsystem} yields the following.

\begin{coro}\label{ch41.4} We have the identities
$$\prod_{n=1}^{\infty}\frac{1-q^n}{1-nq^n}=\sum_{\lambda \in \mathcal P}\varphi_{\mathcal P}(\lambda)q^{|\lambda|},\  \  \  \  \  \  \  \  \prod_{n=1}^{\infty}\frac{1}{(1-q^n)(1-nq^n)}=\sum_{\lambda \in \mathcal P}\sigma_{\mathcal P}(\lambda)q^{|\lambda|}.$$
\end{coro}
\vspace{2em}

\section{Partition-theoretic zeta functions}
A multitude of nice specializations of Theorem \ref{ch41.1} may be obtained. We would like to focus on an interesting class of partition sums arising from Euler's product formula for the sine function
\begin{equation}\label{ch43}
x\prod_{n=1}^{\infty} \pwr{1 - \frac{x^2}{\pi^2n^2}} = \sin x,
\end{equation}
combined with Theorem \ref{ch41.1}. Taking $q = 1$ (as we have done in Corollary \ref{ch41.3}), we begin by noting an easy partition-theoretic formula that may be used to compute the value of $\pi$.

Let $\mathcal P_{m\mathbb Z} \subseteq \mathcal P$ denote the set of partitions into multiples of $m$. Recall from above that the {\it norm} $n_{\lambda}$ of a partition $\lambda$ is the product $\lambda_1\lambda_2\cdots \lambda_r$ of its parts.

\begin{coro}\label{ch41.4}
Summing over partitions into even parts, we have the formula
\begin{equation*}
\frac{\pi}{2} = \sum_{\lambda \in \mathcal P_{2\mathbb Z}} \frac{1}{n_{\lambda}^2}.
\end{equation*}
\end{coro}
We notice that the form of the sum of the right-hand side resembles $\zeta(2)$. Based on this similarity, we wonder if there exists a nice partition-theoretic analog of $\zeta(s)$ possessing something of a familiar zeta function structure --- perhaps Corollary \ref{ch41.4} gives an example of such a function? However, in this case it is not so: the above identity arises from different types of phenomena from those associated with $\zeta(s)$. We have an infinite family of formulas of the following shapes.
\begin{coro}\label{ch41.5}
Summing over partitions into multiples of any whole number $m > 1$, we have
\begin{eqnarray}
\sum_{\lambda \in \mathcal P_{m\mathbb Z}} \frac{1}{n_{\lambda}^2} &=& \frac{\pi}{m \, \sin \pwr{\frac{\pi}{m}}} \\
\sum_{\lambda \in \mathcal P_{m\mathbb Z}} \frac{1}{n_{\lambda}^4} &=& \frac{\pi^2}{m^2 \, \sin \pwr{\frac{\pi}{m}} \sinh \pwr{\frac{\pi}{m}}},
\end{eqnarray}
and increasingly complicated formulas can be computed for $\sum_{\lambda \in \mathcal P_{m \mathbb Z}} 1/n_{\lambda}^{2^t}$, $t \in \mathbb Z^+$.
\end{coro}

Examples like these are appealing, but their right-hand sides are not entirely reminiscent of the Riemann zeta function, aside from the presence of $\pi$. Certainly they are not as tidy as expressions of the form $\zeta(2k) = ``\pi^{2k} \times \text{rational}"$. Based on the previous corollaries, a reasonable first guess at a partition-theoretic analog of $\zeta(s)$ might be to define
\begin{equation*}
\zeta_{\mathcal P}(s) := \sum_{\lambda \in \mathcal P} \frac{1}{n_{\lambda}^s} = \frac{1}{\prod_{n=1}^{\infty} \pwr{1 - \frac{1}{n^s}}}, \Real (s) > 1.
\end{equation*}
Of course, neither side of this identity converges, but using Definition \ref{ch1zetadef} of a partition zeta function $\zeta_{\mathcal P'}(s)$ from Chapter 1, viz.
 %and $f\colon \mathcal P\to \mathbb C$ 
%we define a {\it partition zeta function} $\zeta_{\mathcal P'}(s)$  %and {\it partition Dirichlet series} $\mathcal D_{\mathcal P'}(f,s)$, respectively, 
%by the following sum over partitions $\lambda$ in $\mathcal P'$:
\begin{equation*}%\label{ch1zetadef}
\zeta_{\mathcal P'}(s):=\sum_{\lambda\in \mathcal P'}n_{\lambda}^{-s} %,\  \  \  \  \  \  \mathcal D_{\mathcal P'}(f,s):=\sum_{\lambda\in \mathcal P'}f(\lambda)n_{\lambda}^{-s}.
\end{equation*}
for a subset $\mathcal P'$ of $\mathcal P$ and $s\in \mathbb C$ for which the series converges, we obtain convergent expressions if we omit the first term and perhaps subsequent terms of the product to yield $\zeta_{\mathcal P_{\geq a}}(s) := \sum_{\lambda \in \mathcal P_{\geq a}} 1/n_{\lambda}^s = \prod_{n=a}^{\infty} (1-1/n^s)\inv\,\,\,(a\geq 2)$, where $\mathcal P_{\geq a} \subset \mathcal P$ denotes the set of partitions into parts greater than or equal to $a$. For instance, we have the following formula.
\begin{coro}\label{ch41.6}
Summing over partitions into parts greater than or equal to $2$, we have 
\begin{equation*}
\zeta_{\mathcal P_{\geq 2}}(3) = \sum_{\lambda \in \mathcal P_{\geq 2}} \frac{1}{n_{\lambda}^3} = \frac{3\pi }{\cosh \pwr{\frac{1}{2} \pi \sqrt{3}}}.
\end{equation*}
\end{coro}

While it is an interesting expression, stemming from an identity of Ramanujan \cite{RamanujanCollected}, once again this formula does not seem to evoke the sort of structure we anticipate from a zeta function --- of course, the value of $\zeta(3)$ is not even known. We need to find the ``right'' subset of $\mathcal P$ to sum over, if we hope to find a nice partition-theoretic zeta function. As it turns out, there are subsets of $\mathcal P$ that naturally produce analogs of $\zeta(s)$ for certain arguments $s$.

\begin{definition}\label{ch4zetafixed}
We define a partition-theoretic generalization $\zeta_{\mathcal P}(\brk{s}^k)$ of the Riemann zeta function by the following sum over all partitions $\lambda$ of fixed length $\ell(\lambda) = k \in \mathbb Z_{\geq 0}$ at argument $s \in \mathbb C$, $\Real (s) > 1$:
\begin{equation}\label{ch411}
\zeta_{\mathcal P}(\brk{s}^k) := \sum_{\ell(\lambda) = k} \frac{1}{n_{\lambda}^s}. 
\end{equation}
\end{definition}

\begin{remark}
This is a fairly natural formation, being similar in shape (and notation) to the weight $k$ multiple zeta function $\zeta(\brk{s}^k)$, which is instead summed over length-$k$ partitions into distinct parts; Hoffman gives interesting formulas relating $\zeta_{\mathcal P}(\brk{s}^k)$ (in different notation) to combinations of multiple zeta functions \cite{Hoffman1}, which exhibit rich structure.
\end{remark}

We have immediately that $\zeta_{\mathcal P}(\brk{s}^0) = 1/n_{\emptyset}^s = 1$ and  $\zeta_{\mathcal P}(\brk{s}^1) = \zeta_{\mathcal P}(\brk{s}) = \zeta(s)$. Using Theorem \ref{ch41.1} and proceeding (see Section \ref{ch4Proofs of theorems and corollaries}) much as Euler did to find the value of $\zeta(2k)$ \cite{dunham1999euler}, we are able to find explicit values for $\zeta_{\mathcal P}(\brk{2}^k)$ at every positive integer $k >0$. Somewhat surprisingly, we find that in these cases $\zeta_{\mathcal P}(\brk{2}^k)$ is a rational multiple of $\zeta(2k)$.

\begin{coro}\label{ch41.7}
For $k>0$, we have the identity 
\begin{equation*}
\zeta_{\mathcal P}(\brk{2}^k) = \sum_{\ell(\lambda) = k} \frac{1}{n_{\lambda}^2} = \frac{2^{2k - 1} - 1}{2^{2k-2}}\zeta(2k).
\end{equation*}
For example, we have the following values:
\begin{align*}
\zeta_{\mathcal P}(\brk{2}) &= \zeta(2) = \frac{\pi^2}{6}, \\
\zeta_{\mathcal P}(\brk{2}^2) &= \frac{7}{4}\zeta(4) = \frac{7\pi^4}{360}, \\
\zeta_{\mathcal P}(\brk{2}^3) &= \frac{31}{16}\zeta(6)= \frac{31\pi^6}{15120},\dots , \\
\zeta_{\mathcal P}(\brk{2}^{13}) &= \frac{33554431}{16777216}\zeta(26) = \frac{22076500342261\pi^{26}}{93067260259985915904000000},\dots
\end{align*}
\end{coro}

Corollary \ref{ch41.7} reveals that $\zeta_{\mathcal P}(\brk{2}^k)$ is indeed of the form $``\pi^{2k} \times \text{rational}"$ for all positive $k$, like the zeta values $\zeta(2k)$ given by Euler (we note that $\zeta(26)$ is the highest zeta value Euler published) \cite{dunham1999euler}. We have more: we can find $\zeta_{\mathcal P}(\brk{2^t}^k)$ explicitly for all $t \in \mathbb Z^+$. These values are finite combinations of well-known zeta values, and are also of the form $``\pi^{2^tk} \times \text{rational}"$.

\begin{coro}\label{ch41.8}
For $k>0$ we have the identity
\begin{align*}
\zeta_{\mathcal P}(\brk{4}^k) & = \sum_{\ell(\lambda) = k}\frac{1}{n_{\lambda}^4}\\ 
& = \frac{1}{16^{k-1}} \pwr{\sum_{n=0}^{2k} (-1)^n(2^{2n-1}-1)(2^{4k - 2n - 1} - 1)\zeta(2n)\zeta(4k - 2n)},
\end{align*}
and increasingly complicated formulas can be computed for $\zeta_{\mathcal P}(\brk{2^t}^k), t \in \mathbb Z^+$.
\end{coro}
\begin{remark}
The summation on the far right above may be shortened by noting the symmetry of the summands around the $n=k$ term.
\end{remark}

It would be desirable to understand the value of $\zeta_{\mathcal P}(\brk{s}^k)$ at other arguments $s$; the proof we give below (see Section \ref{ch4Proofs of theorems and corollaries}) does not shed much light on this question, being based very closely on Euler's formula \eqref{ch41.3}, which forces $s$ be a power of $2$. Also, if we solve Corollary \ref{ch41.6} for $\zeta(0)$, we conclude that $\zeta(0) = \frac{2^{-2}}{2\inv - 1}\zeta_{\mathcal P}(\brk{2}^0) = -1/2$, which is the value of $\zeta(0)$ under analytic continuation. Can $\zeta_{\mathcal P}(\brk{s}^k)$ be extended via analytic continuation for values of $k > 1$? In a larger sense we wonder: do nice zeta function analogs exist if we sum over other interesting subsets of $\mathcal P$? In Chapter 5 we will follow up on these questions.

We do have a few general properties shared by convergent series $\sum 1/n_{\lambda}^s$ summed over large subclasses of $\mathcal P$. First we need to refine some of our previous notations.

\begin{definition}
Take any subset of partitions $\mathcal P' \subseteq \mathcal P$. Then for $\Real (s) > 1$, on analogy to classical zeta function theory, when these expressions converge we define
\begin{align}
\zeta_{\mathcal P'} &:= \sum_{\lambda \in \mathcal P'} \frac{1}{n_{\lambda}^s}, &\eta_{\mathcal P'}(s)& := \sum_{\lambda \in \mathcal P'} \frac{(-1)^{\ell(\lambda)}}{n_{\lambda}^s}, &\zeta_{\mathcal P'}(\brk{s}^k) &:= \sum_{\begin{tiny}
\begin{array}{c}
\lambda \in \mathcal P' \\
\ell(\lambda) = k
\end{array}\end{tiny}} \frac{1}{n_{\lambda}^s}.
\end{align}
\end{definition}

\begin{remark} 
As important special cases, we have $\zeta_{\mathcal P_{\mathbb P}}(s) = \zeta(s)$ and $\zeta_{\mathcal P_{\mathbb Z^+}}(\brk{s}^k) = \zeta_{\mathcal P}(\brk{s}^k)$. It is also easy to see that $\zeta_{\mathcal P'}(s)=\sum_{k=0}^{\infty} \zeta_{\mathcal P'}(\brk{s}^k)$ and $\eta_{\mathcal P'}(s)=\sum_{k=0}^{\infty} (-1)^{k} \zeta_{\mathcal P'}(\brk{s}^k)$ if we assume absolute convergence. Moreover, given absolute convergence, we may write $\zeta_{\mathcal P'}(s),\,\zeta_{\mathcal P'}(\brk{s}^k)$ as classical Dirichlet series related to multiplicative partitions: we have $\zeta_{\mathcal P'}(s)=\sum_{j=1}^{\infty}\# \{\lambda \in \mathcal P'\,|\,\, n_{\lambda}=j\}\, j^{-s}$ and $\zeta_{\mathcal P'}(\brk{s}^k)(s)=\sum_{j=1}^{\infty}\# \{\lambda \in \mathcal P'\,|\,\, \ell(\lambda)=k,n_{\lambda}=j\}\,j^{-s}$ (see \cite{ChamberlandStraub} for more about multiplicative partitions).

\end{remark} %%%THERE IS NEW LANGUAGE ADDED TO THIS PARAGRAPH SINCE SUBMISSION%%%

As previously, take ${\mathbb {\mathbb X}} \subseteq \mathbb Z^+$ and take $\mathcal P_{{\mathbb {\mathbb X}}} \subseteq \mathcal P$ to denote partitions into elements of ${\mathbb {\mathbb X}}$ (thus $\mathcal P_{\mathbb Z^+} = \mathcal P$). Note that $\zeta_{\mathcal P_{\mathbb {\mathbb X}}}(s) = \prod_{n \in {\mathbb {\mathbb X}}} \pwr{1 - \frac{1}{n^s}}\inv$ is divergent if $1 \in {\mathbb {\mathbb X}}$ and, when ${\mathbb {\mathbb X}}$ is finite (thus there is no restriction on the value of $\Real(s)$), if $s=i\pi j/\log n$ for any $n\in {\mathbb {\mathbb X}}$ and even integer $j$. Similarly, when ${\mathbb {\mathbb X}}$ is finite, $\eta_{\mathcal P_{\mathbb {\mathbb X}}}(s) = \prod_{n \in {\mathbb {\mathbb X}}} \pwr{1 + \frac{1}{n^s}}\inv$ is divergent if $s=i\pi k/\log n$ for any $n\in {\mathbb {\mathbb X}}$ and odd integer $k$. Clearly if $\mathbb Y \subseteq \mathbb Z^+$, then from the product representations we also have $\zeta_{\mathcal P_{\mathbb {\mathbb X}}}(s)\zeta_{\mathcal P_{\mathbb Y}}(s)=\zeta_{\mathcal P_{{\mathbb X} \cup Y}}(s)\zeta_{\mathcal P_{{\mathbb X} \cap \mathbb Y}}(s)$ and $\eta_{\mathcal P_{\mathbb X}}(s)\eta_{\mathcal P_{\mathbb Y}}(s)=\eta_{\mathcal P_{{\mathbb {\mathbb X}} \cup \mathbb Y}}(s)\eta_{\mathcal P_{{\mathbb X} \cap \mathbb Y}}(s)$. 

Many interesting subsets of partitions have the form $\mathcal P_{\mathbb X}$, in particular those to which Theorem \ref{ch41.1} most readily applies. Note that such subsets $\mathcal P_{\mathbb X}$ are partition ideals of order 1, in the sense of Andrews \cite{Andrews}. With the above notations, we have the following useful ``doubling" formulas.

\begin{coro}\label{ch41.9}
If $\zeta_{\mathcal P_{\mathbb X}}(s)$ converges over $\mathcal P_{\mathbb X} \subseteq \mathcal P$, then 
\begin{equation}\label{ch413}
\zeta_{\mathcal P_{\mathbb X}}(2s) = \zeta_{\mathcal P_{\mathbb X}}(s)\eta_{\mathcal P_{\mathbb X}}(s).
\end{equation}
Furthermore, for $n \in \mathbb Z_{\geq 0}$ we have the identity 
\begin{equation}\label{ch414}
\zeta_{\mathcal P_{\mathbb X}}\left(\{2^{n+1}s\}^k\right) = \sum_{j=0}^{2^nk}(-1)^j\zeta_{\mathcal P_{\mathbb X}}\left(\brk{2^ns}^j\right)\zeta_{\mathcal P_{\mathbb X}}\left(\brk{2^ns}^{2^nk - j}\right).
\end{equation}
\end{coro}
\begin{remark}
As in Corollary \ref{ch41.8}, the summation on the right-hand side of \eqref{ch414} may be shortened by symmetry.
\end{remark}

If we take ${\mathbb X} = \mathbb P$, then \eqref{ch413} reduces to the classical identity $\zeta(2s) = \zeta(s)\sum_{n=1}^{\infty} \lambda(n)/n^s$, where $\lambda(n)$ is Liouville's function. Another specialization of Corollary \ref{ch41.9} leads to new information about $\zeta_{\mathcal P}(\brk{s}^k)$: we may extend the domain of $\zeta_{\mathcal P}(\brk{s}^k)$ to $\Real (s) > 1$ if we take ${\mathbb X} = \mathbb Z^+$, $n = 0$, $k = 2$. We find $\zeta_{\mathcal P}(\brk{s}^2)$ inherits analytic continuation from the sum on the right-hand side below.

\begin{coro}\label{ch41.10}
For $\Real (s) > 1$, we have 
\begin{equation*}
\zeta_{\mathcal P}(\brk{s}^2) = \frac{\zeta(2s) + \zeta(s)^2}{2}.
\end{equation*}
\end{coro}

\begin{remark}
This resembles the series shuffle product formula for multiple zeta values \cite{besser416double}.
\end{remark}

Another interesting subset of $\mathcal P$ is the set of partitions $\mathcal P^*$ into \textit{distinct} parts; also of interest is the set of partitions $\mathcal P_{\mathbb X}^*$ into distinct elements of ${\mathbb X} \subseteq \mathbb Z^+$ (thus $\mathcal P_{\mathbb Z^+}^* = \mathcal P^*$). However, partitions into distinct parts are not immediately compatible with the identities in Theorem \ref{ch41.1}. Happily, we have a dual theorem that leads us to zeta functions summed over $\mathcal P_{\mathbb X}^*$ for any ${\mathbb X}\subseteq \mathbb Z^+$. Let us recall the infinite product $\varphi_{\infty}(f;q)$ from Theorem \ref{ch41.1}.

\begin{theorem}\label{ch41.11}
If the product converges, then $\varphi_{\infty}(f;q) = \prod_{n=1}^{\infty}(1 - f(n)q^n)$ may be expressed in a number of equivalent forms, viz.
\begin{flalign}
\varphi_{\infty}(f;q) &= \sum_{\lambda \in \mathcal P^*}(-1)^{\ell(\lambda)} q^{|\lambda|} \prod_{\lambda_i \in \lambda}f(\lambda_i) \label{ch415}\\ 
&= \begin{array}{c}
1 - \sum_{(6)}
\end{array} \label{ch416}\\ 
& = \begin{array}{c}
1 - \varphi_{\infty}(f;q) \sum_{(5)} 
\end{array} \label{ch417}\\%\end{flalign}
&= 1-\cfrac{\sum_{(5)}}{1+\cfrac{\sum_{(6)}}{1-\cfrac{\sum_{(5)}}{1+\cfrac{\sum_{(6)}}{1- \cdots}}}} \label{ch418}
\end{flalign} 
%\displaybreak
where $\sum_{(5)},\sum_{(6)}$ are exactly as in Theorem \ref{ch41.1}, and the sum in \eqref{ch415} is taken over the partitions into distinct parts.
\end{theorem}

\begin{remark}
Note that there is not a nice ``inverted" sum of the form \eqref{ch47} here.
\end{remark}

Just as with Theorem \ref{ch41.1}, we may write arbitrary power series as limiting cases, and we have the obvious identity
\begin{equation}\label{ch4obvious}
\prod_{n\in {\mathbb X}} (1-f(n)q^n) =  \sum_{\lambda \in \mathcal P_{\mathbb X}^*} (-1)^{\ell(\lambda)}q^{|\lambda|} \prod_{\lambda_i \in \lambda} f(\lambda_i),
\end{equation}
with the remaining summations in Theorem \ref{ch41.11} being taken over elements of ${\mathbb X}$. 

\begin{remark}
Clearly, the summation on the right-hand side of \eqref{ch4obvious}, as well as the $\mathbb X = \mathbb Z^+$ case \eqref{ch415}, can be rewritten in the form $$\sum_{\lambda \in \mathcal P_{\mathbb X}} \mu_{\mathcal P}(\lambda)q^{|\lambda|} \prod_{\lambda_i \in \lambda} f(\lambda_i).$$ However, to keep our notations absolutely general in this chapter from a set-theoretic perspective, we will for the most part label subsets of partitions into distinct parts with the ``$*$''  superscript, as opposed to filtering out terms with repeated parts using $\mu_{\mathcal P}$.
\end{remark}

For completeness, we record another obvious but useful consequence of Theorems \ref{ch41.1} and \ref{ch41.11}. The following statement might be viewed as a generalized eta quotient formula, with coefficients given explicitly by finite combinatorial sums\footnote{In Appendix C we discuss an interesting class of ``sequentially congruent'' partitions suggested by this formula.}.

\begin{coro}\label{ch4etaquotient}
For $f_j$ defined on ${\mathbb X}_j\subseteq \mathbb Z^+$, consider the double product
\[
\prod_{j=1}^{n}\prod_{k_j\in {\mathbb X}_j}\left(1\pm f_j(k_j)q^{k_j}\right)^{\pm 1}
=\sum_{k=0}^{\infty}c_k q^k
,\]
where the $\pm$ sign is fixed for fixed $j$, but may vary as $j$ varies. Then the coefficients $c_k$ are given by the $(n-1)$-tuple sum
\begin{multline*}
c_k= \sum_{k_2=0}^{k}\sum_{k_3=0}^{k_2}\dots\sum_{k_{n}=0}^{k_{n-1}}\left(\sum_{\substack{\lambda\vdash k_{n}\\ \lambda\in\mathcal P_{{\mathbb X}_{n}}^{\pm}}}\prod_{\lambda_i\in\lambda}f_{n}(\lambda_i)\right)\left(\sum_{\substack{\lambda\vdash (k_{n-1}-k_{n})\\ \lambda\in\mathcal P_{{\mathbb X}_{n-1}}^{\pm}}}\prod_{\lambda_i\in\lambda}f_{n-1}(\lambda_i)\right)\\
\times\left(\sum_{\substack{\lambda\vdash (k_{n-2}-k_{n-1})\\ \lambda\in\mathcal P_{{\mathbb X}_{n-2}}^{\pm}}}\prod_{\lambda_i\in\lambda}f_{n-2}(\lambda_i)\right)\dots\left(\sum_{\substack{\lambda\vdash (k-k_2)\\ \lambda\in\mathcal P_{{\mathbb X}_{1}}^{\pm}}}\prod_{\lambda_i\in\lambda}f_1(\lambda_i)\right)
\end{multline*}
in which we have set $\mathcal P_{{\mathbb X}_j}^-:=\mathcal P_{{\mathbb X}_j}$ and $\mathcal P_{{\mathbb X}_j}^+:=\mathcal P_{{\mathbb X}_j}^*$ with the $\pm$ sign as associated to each $j$ above.
\end{coro}

\begin{remark} The $+$ or $-$ signs in the formula for $c_k$ indicate partitions arising from the numerator or denominator, respectively, of the double product. One may replace $f_j$ with $-f_j$ to effect further sign changes.
\end{remark}

Analogous corollaries to those following Theorem \ref{ch41.1} are available, but we wish right away to apply this theorem to the problem at hand, the investigation of partition zeta functions. We have 
\begin{equation}\zeta_{\mathcal P_{\mathbb X}^*}(s) = \prod_{n \in {\mathbb X}} (1 + \frac{1}{n^s}),\  \  \  \  \  \  \  \  \  \  \  \  \eta_{\mathcal P_{\mathbb X}^*}(s) = \prod_{n \in {\mathbb X}} (1 - \frac{1}{n^s}),\end{equation} 
again noting that in fact $$\eta_{\mathcal P_{{\mathbb X}}^*}(s)=\sum_{\lambda \in \mathcal P}\mu_{\mathcal P}(\lambda)n_{\lambda}^{-s}.$$ It is immediate then from \eqref{ch415} that for $\Real (s) > 1$ we also have the following relations, where the sum on the left-hand side of each equation is taken over the partitions into distinct elements of ${\mathbb X}$:
\begin{equation}
\zeta_{\mathcal P_{\mathbb X}^*}(s)  = \frac{1}{\eta_{\mathcal P_{\mathbb X}}(s)},\  \  \  \  \  \  \  \  \  \  \  \  \eta_{\mathcal P_{\mathbb X}^*}(s) = \frac{1}{\zeta_{\mathcal P_{\mathbb X}}(s)}. \label{ch419}
\end{equation}
Note that $\zeta_{\mathcal P_{\mathbb X}^*}(s)$ and $\eta_{\mathcal P_{\mathbb X}^*}(s)$ are finite sums (and entire functions of $s$) if ${\mathbb X}$ is a finite set, unlike $\zeta_{\mathcal P_{\mathbb X}}(s)$ and $\eta_{\mathcal P_{\mathbb X}}(s)$. Note also that $\eta_{\mathcal P_{\mathbb X}^*}(s) = 0$ identically if $1 \in {\mathbb X}$, with zeros when ${\mathbb X}$ is finite at the values $s=i\pi j/\log n$ for any $n\in {\mathbb X}$ and $j$ even. Unlike $\zeta_{\mathcal P}(s)$, we can see from \eqref{ch419} that $\zeta_{\mathcal P^*}(s)$ is well-defined on $\Real (s) > 1$ (thus both $\zeta_{\mathcal P_{\mathbb X}^*}$ and $\eta_{\mathcal P_{\mathbb X}^*}$ are well-defined over all subsets $\mathcal P_{\mathbb X}^*$ of $\mathcal P^*$); when ${\mathbb X}$ is finite, $\zeta_{\mathcal P^*}(s)$ has zeros at $s=i\pi k/\log n$ for $n\in {\mathbb X}$ and $k$ odd. Morever, we have $\zeta_{\mathcal P_{\mathbb X}^*}(s)\zeta_{\mathcal P_{\mathbb Y}^*}(s)=\zeta_{\mathcal P_{{\mathbb X} \cup \mathbb Y}^*}(s)\zeta_{\mathcal P_{{\mathbb X} \cap \mathbb Y}^*}(s)$ and $\eta_{\mathcal P_{\mathbb X}^*}(s)\eta_{\mathcal P_{\mathbb Y}^*}(s)=\eta_{\mathcal P_{{\mathbb X} \cup \mathbb Y}^*}(s)\eta_{\mathcal P_{{\mathbb X} \cap \mathbb Y}^*}(s)$. Here is an example of a zeta sum of this form. %%%NEW LANGUAGE TO THIS PARAGRAPH SINCE SUBMISSION%%%

\begin{coro}\label{ch41.12}
Summing over partitions into distinct parts, we have that 
\begin{equation*}
\zeta_{\mathcal P^*}(2) = \sum_{\lambda \in \mathcal P^*} \frac{1}{n_{\lambda}^2} = \frac{\sinh \pi }{\pi}.
\end{equation*}
\end{coro}

Zeta sums over partitions into distinct parts admit an important special case: as we remarked beneath definition \eqref{ch411}, the multiple zeta function $\zeta(\brk{s}^k)$ can be written 
\begin{equation}\label{ch420}
\zeta(\brk{s}^k) := \sum_{\lambda_1 > \lambda_2 > \cdots > \lambda_k \geq 1} \frac{1}{\lambda_1^s\lambda_2^s\cdots \lambda_k^s} = \sum_{\begin{tiny}
\begin{array}{c}
\lambda \in \mathcal P^* \\
\ell(\lambda) = k
\end{array}\end{tiny}} \frac{1}{n_{\lambda}^s} = \zeta_{\mathcal P^*}(\brk{s}^k).
\end{equation}
Using this notation, we can derive even simpler formulas for the multiple zeta values $\zeta(\brk{2^t}^k)$ than those found for $\zeta_{\mathcal P}(\brk{2^t}^k)$ in Corollaries \ref{ch41.7} and \ref{ch41.8}, such as these.% evaluations.%For instance, we have the following values.

\begin{coro}\label{ch41.13}
For $k>0$ we have the identities
\begin{flalign*}
\zeta(\brk{2}^k) &= \frac{\pi^{2k}}{(2k+1)!}, \\
\zeta(\brk{4}^k) &= \pi^{4k} \sum_{n=0}^{2k} \frac{(-1)^n}{(2n+1)!(4k-2n+1)!}, \\
\zeta(\brk{8}^k) &= \pi^{8k} \sum_{n=0}^{4k}(-1)^n \pwr{\sum_{i=0}^n \frac{(-1)^i}{(2i+1)!(2n-2i +1)!}}\\
&\  \  \  \  \  \  \  \  \  \  \  \  \  \  \  \  \  \  \times\pwr{\sum_{i=0}^{4k-n} \frac{(-1)^i}{(2i+1)!(8k-2n-2i+1)!}},
\end{flalign*}
and increasingly complicated formulas of the shape $``\pi^{2^tk} \times \text{finite sum of fractions}"$ can be computed for multiple zeta values of the form $\zeta(\brk{2^t}^k),$ $t \in \mathbb Z^+$.
\end{coro}

\begin{remark}
The first identity above is proved in \cite{Hoffman1} by a different approach from that taken here (see Section \ref{ch4Proofs of theorems and corollaries}); it is possible the other identities in the corollary are also known.
\end{remark}

The summations in Corollary \ref{ch41.13} arise from quite general properties: we have these ``doubling'' formulas comparable to Corollary \ref{ch41.9}.

\begin{coro}\label{ch41.14}
If $\zeta_{\mathcal P_{\mathbb X}^*}(s)$ converges over $\mathcal P_{\mathbb X}^* \subseteq \mathcal P$, then 
\begin{equation}\label{ch421}
\eta_{\mathcal P_{\mathbb X}^*}(2s) = \eta_{\mathcal P_{\mathbb X}^*}(s)\zeta_{\mathcal P_{\mathbb X}^*}(s).
\end{equation}
Furthermore, for $n \in \mathbb Z_{\geq 0}$ we have
\begin{equation}\label{ch422}
\zeta_{\mathcal P_{\mathbb X}^*}\left(\{2^{n+1}s\}^k\right) = \sum_{j=0}^{2^nk}(-1)^j \zeta_{\mathcal P_{\mathbb X}^*}\left(\brk{2^ns}^j\right)\zeta_{\mathcal P_{\mathbb X}^*}\left(\brk{2^ns}^{2^nk-j}\right).
\end{equation}
\end{coro}

\begin{remark}
Once again, the summation on the right-hand side of \eqref{ch422} may be be shortened by symmetry. Equation \eqref{ch422} yields a family of multiple zeta function identities when we let ${\mathbb X} = \mathbb Z^+$. 
\end{remark}

We note that by recursive arguments, from \eqref{ch413} and \eqref{ch421} together with \eqref{ch48}, we have these curious product formulas connecting sums over partitions into distinct parts to their counterparts involving unrestricted partitions:
\begin{align*}
\zeta_{\mathcal P_{\mathbb X}^*}(s)\zeta_{\mathcal P_{\mathbb X}^*}(2s)\zeta_{\mathcal P_{\mathbb X}^*}(4s)\zeta_{\mathcal P_{\mathbb X}^*}(8s)\cdots &= \zeta_{\mathcal P_{\mathbb X}}(s), \\
\eta_{\mathcal P_{\mathbb X}}(s)\eta_{\mathcal P_{\mathbb X}}(2s)\eta_{\mathcal P_{\mathbb X}}(4s)\eta_{\mathcal P_{\mathbb X}}(8s)\cdots &= \eta_{\mathcal P_{\mathbb X}^*}(s).
\end{align*}

Now, if we take ${\mathbb X} = \mathbb P$ then \eqref{ch421} becomes the well-known classical identity $\zeta(2s)\inv =\zeta(s)\inv \sum_{n=1}^{\infty} |\mu(n)|/n^s$, where $\mu(n)$ is the M\"obius function. %We might view the simple quantity $(-1)^{\ell(\lambda)}$ as a partition-theoretic generalization of $\mu$; it 
As we have noted, the quantity $(-1)^{\ell(\lambda)}$ is exactly $\mu_{\mathcal P}(\lambda)$ when $\lambda$ is in $\mathcal P^*$, and otherwise is a partition version of Liouville's function which specializes to the classical %specializes to the M\"obius function (when considering partitions into distinct prime parts), and also to 
Liouville's function when we consider unrestricted prime partitions. %), as we saw above. K. Alladi has observed this correspondence as well (personal communication, December 22, 2015). 

%It is fascinating --- and rather mysterious --- that partitions (which are defined additively, with no connection to multiplication) into parts that are prime numbers (defined multiplicatively) should have significant number-theoretic connections. 

The literature abounds with product formulas which, when fed through the machinery of the identities noted here, produce nice partition zeta sum variants; the interested reader is referred to \cite{ChamberlandStraub} as a starting point for further study.

\vspace{2em}

\section{Proofs of theorems and corollaries}\label{ch4Proofs of theorems and corollaries}

\begin{proof}[Proof of Theorem \ref{ch41.1}]
Identity \eqref{ch44} appears in a different form as \cite[Eq.~22.16]{Fine}. The proof proceeds formally, much like the standard proof of \eqref{ch41.1}; we expand $1/ \varphi_{\infty}(f;q)$ as a product of geometric series
\begin{multline*}
\frac{1}{\varphi_{\infty}(f;q)} = (1+ f(1)q + f(1)^2q^2 + f(1)^3q^3 + \dots)\\
\times(1 + f(2)q^2 + f(2)^2q^4 + f(2)^3q^6 + \dots)\times\cdots
\end{multline*}
and multiply out all the terms (without collecting coefficients in the usual way). The result is the partition sum in \eqref{ch44}. 

Identities \eqref{ch45} and \eqref{ch46} are proved using telescoping sums. Consider that 
\begin{align*}
\frac{1}{\varphi_{\infty}(f;q)} &= \frac{1}{\varphi_0(f;q)} + \sum_{n=1}^{\infty} \pwr{\frac{1}{\varphi_n(f;q)} - \frac{1}{\varphi_{n-1}(f;q)}} \\
&= 1 + \sum_{n=1}^{\infty} \frac{1}{\varphi_{n-1}(f;q)} \pwr{\frac{1}{1- f(n)q^n} - 1} \\
&= 1 + \sum_{n=1}^{\infty}q^n \frac{f(n)}{\varphi_n(f;q)} = \begin{array}{c}
1 + \sum_{(5)}
\end{array},
\end{align*}
recalling the notation $\sum_{(5)}$ (as well as $\sum_{(6)}$) from the theorem, which is \eqref{ch45}. Similarly, we can show
\begin{align*}
\varphi_{\infty}(f;q) &= \varphi_0(f;q) + \sum_{n=1}^{\infty} (\varphi_n(f;q) - \varphi_{n-1}(f;q)) \\
&= 1 - \sum_{n=1}^{\infty}q^nf(n) \varphi_{n-1}(f;q) = \begin{array}{c}
1- \sum_{(6)}
\end{array}.
\end{align*}
Thus we have 
$$\begin{array}{c}
\sum_{(5)}
\end{array} = \frac{1}{\varphi_{\infty}(f; q)} - 1 = \frac{1 - \varphi_{\infty}(f;q)}{\varphi_{\infty}(f;q)} = \frac{\sum_{(6)}}{\varphi_{\infty}(f;q)},$$
which leads to \eqref{ch46}.

%The proof of \eqref{ch47} is similar to the proof we gave of \cite[Theorem~1.1(1)]{rolen2013strange}. 
To prove \eqref{ch47}, substitute the identity
$$\varphi_n(f;q) = \prod_{k=1}^{n} (1-f(k)q^k) = (-1)^nq^{n(n+1)/2}\varphi_n(1/f ; q\inv)\prod_{k=1}^n f(k)$$
term-by-term into the sum \eqref{ch45} and simplify to find the desired expression.

The proof of \eqref{ch48} is inspired by the standard proof of the continued fraction representation of the golden ratio. It follows from the proof above of \eqref{ch45} and \eqref{ch46} that 
\begin{align*}
\frac{1}{\varphi_{\infty}(f;q)} &= 1 + \frac{\sum_{(6)}}{\varphi_{\infty}(f;q)} \\
&= 1 + \frac{\sum_{(6)}}{1-\varphi_{\infty}(f;q)\sum_{(5)}} \\
&= 1+\cfrac{\sum_{(6)}}{1-\cfrac{\sum_{(5)}}{1/\varphi_{\infty}(f;q)}}.
\end{align*}
We notice that the expression on the left-hand side is also present on the far right in the denominator. We replace this term $1/\varphi_{\infty}(f;q)$ in the denominator with the entire right-hand side of the equation; reiterating this process indefinitely gives \eqref{ch48}.
\end{proof}
\begin{remark}
The series$\begin{array}{c}
\sum_{(5)},\sum_{(6)}
\end{array}$enjoy other nice, golden ratio-like relationships. For instance, because $$\begin{array}{c}
(1 + \sum_{(5)})(1-\sum_{(6)}) = 1 
\end{array},$$ it is easy to see that 
$$\begin{array}{c}
\sum_{(5)} - \sum_{(6)} = \sum_{(5)}\sum_{(6)},
\end{array}$$
which resembles the formula $\varphi - 1/\varphi = \varphi \cdot 1/\varphi$ involving the golden ratio $\varphi$ and its reciprocal.
\end{remark}

\begin{proof}[Proof of Corollary \ref{ch41.2}]
This is immediate upon letting $f\equiv 1$ in \eqref{ch44}, as
$$\sum_{\lambda \in \mathcal P}q^{|\lambda|}\prod_{\lambda_i \in \lambda}f(\lambda_i) = 1 + \sum_{n=1}^{\infty}q^n\sum_{\lambda \vdash n}\prod_{\lambda_i \in \lambda} f(\lambda_i).$$
\end{proof}

\begin{proof}[Proof of Corollary \ref{ch41.3}]
As noted above, we assume $\Real (s)>1$. Let $q = 1$, $f(n) = 1/n^s$ if $n$ is prime and $=0$ otherwise; then by \eqref{ch44}
$$\frac{1}{\prod_{p \in \mathbb P}\pwr{1-\frac{1}{p^s}}} = \sum_{\lambda \in \mathcal P_{\mathbb P}} \frac{1}{n_{\lambda}^s}.$$
Consider the prime decomposition of a positive integer $n = p_1^{a_1}p_2^{a_2}\cdots p_r^{a_r}$, $p_1>p_2>\cdots > p_r$. We will associate this decomposition to the unique partition into prime parts $\lambda = (p_1,\dots , p_1,p_2,\dots,$ $p_2,\dots ,p_r,\dots ,p_r) \in \mathcal P$, where $p_k \in \mathbb P$ is repeated $a_k$ times (thus $n$ is equal to $n_{\lambda}$). As he have discussed previously, every positive integer $n\geq 1$ is associated to exactly one partition into prime parts (with $n = 1$ associated to $\emptyset \in \mathcal P_{\mathbb P}$), and conversely: there is a bijective correspondence between $\mathbb Z^+$ and $\mathcal P_{\mathbb P}$. % (Alladi-Erd\H{o}s use this fact in \cite{AlladiErdos}). 
Therefore we see by absolute convergence that 
$$\sum_{n\geq 1}\frac{1}{n^s} = \sum_{\lambda \in \mathcal P_{\mathbb P}} \frac{1}{n_{\lambda}^s}.$$
Equating the left-hand sides of the above two identities gives Euler's product formula \eqref{ch1zetaproduct}. The series given for $\zeta(s)$ follows immediately from Theorem \eqref{ch45} with the above definition of $f$.
\end{proof}

\begin{proof}[Proof of Corollary \ref{ch41.4}]
This is actually a special case of the subsequent Corollary \ref{ch41.5}, setting $m=2$ in the first equation (see below).
\end{proof}

\begin{proof}[Proof of Corollary \ref{ch41.5}]
We begin with an identity equivalent to \eqref{ch43} and its ``$+$" companion:
\begin{align*}
\frac{\pi z}{\sin (\pi z)} &= \frac{1}{\prod_{n=1}^{\infty} \pwr{1- \frac{z^2}{n^2}}}, &\frac{\pi z}{\sinh (\pi z)}& = \frac{1}{\prod_{n=1}^{\infty} \pwr{1+ \frac{z^2}{n^2}}}.
\end{align*}
If $\omega_k := e^{2\pi i /k}$, then $\omega_{2k}^2 = \omega_k$ and we have, by multiplying the above two identities, the pair
$$
\frac{\pi^2z^2}{\sin (\pi z)\sinh (\pi z)} = \frac{1}{\prod_{n=1}^{\infty} \pwr{1-\frac{z^4}{n^4}}},\  \  \frac{\omega_4\pi^2z^2}{\sin (\omega_8\pi z)\sinh (\omega_8\pi z)} = \frac{1}{\prod_{n=1}^{\infty} \pwr{1+\frac{z^4}{n^4}}}.
$$
Multiplying these two equations, and repeating this procedure indefinitely, we find identities like
\begin{align*}
\frac{\omega_4\pi^4z^4}{\sin (\pi z)\sinh(\pi z)\sin(\omega_8\pi z)\sinh(\omega_8 \pi z)} &= \frac{1}{\prod_{n=1}^{\infty} \pwr{1 - \frac{z^8}{n^8}}}, 
\end{align*}
\begin{multline*}
\frac{\omega_4^2\pi^8z^8}{\sin(\pi z)\sinh(\pi z) \sin (\omega_8 \pi z)\sinh (\omega_8 \pi z) }\\
\times\frac{1}{\sin(\omega_{16}\pi z) \sinh (\omega_{16}\pi z)\sin (\omega_8\omega_{16}\pi z) \sinh (\omega_8\omega_{16} \pi z)}
\end{multline*}
$$
\quad \quad \quad = \frac{1}{\prod_{n=1}^{\infty} \pwr{1 - \frac{z^{16}}{n^{16}}}},
$$
as well as their ``$+$" companions, and so on. On the other hand, it follows from \eqref{ch44} that 
$$\frac{1}{\prod_{n=1}^{\infty} \pwr{1 - \frac{zq^n}{n^s}}} = \sum_{\lambda \in \mathcal P}q^{|\lambda|} \prod_{\lambda_i \in \lambda} \frac{z}{\lambda_i^s} = \sum_{\lambda \in \mathcal P}\frac{z^{\ell(\lambda)}q^{|\lambda|}}{n_{\lambda}^s}.$$
Replacing $z$ with $\pm z^{2^t}$ and taking $q = 1$ in the above expression, it is easy to see that we have 
\begin{align*}
\frac{1}{\prod_{n=1}^{\infty} \pwr{1 - \frac{z^{2^t}}{n^{2^t}}}} &= \sum_{\lambda \in \mathcal P}\frac{z^{2^t \ell(\lambda)}}{n_{\lambda}^{2^t}}, &\frac{1}{\prod_{n=1}^{\infty} \pwr{1 + \frac{z^{2^t}}{n^{2^t}}}} &= \sum_{\lambda \in \mathcal P}\frac{(-1)^{\ell(\lambda)}z^{2^t \ell(\lambda)}}{n_{\lambda}^{2^t}}.
\end{align*}
These series have closed forms given by complicated trigonometric and hyperbolic expressions such as the ones above. Setting $z = 1/m$ in such expressions yields the explicit values advertised in the corollary for
\begin{align*}
\frac{1}{\prod_{n=1}^{\infty} \pwr{1 - \frac{1}{m^{2^t}n^{2^t}}}} & = \frac{1}{\prod_{n=1}^{\infty} \pwr{1 - \frac{1}{(mn)^{2^t}}}}\\
& = \frac{1}{\prod_{n\equiv 0  \, (\text{mod }m)} \pwr{1 - \frac{1}{n^{2^t}}}} = \sum_{\lambda \in \mathcal P_{m\mathbb Z}} \frac{1}{n_{\lambda}^{2^t}}.
\end{align*}\end{proof}

\begin{remark}
More generally, let $\mathcal P_{a(m)}$ denote the set of partitions into parts $\equiv a \pmod m$ (so $\mathcal P_{m\mathbb Z}$ is $\mathcal P_{0(m)}$ in this notation). It is clear that if $\lambda \in \mathcal P_{a(m)}$ then $n_{\lambda}^s \equiv a^s \pmod m$, thus we find 
$$\frac{1}{\prod_{n\equiv a \, (\text{mod }m)} (1-n^sq^n)}  = \sum_{\lambda \in \mathcal P_{a(m)}} n_{\lambda}^sq^{|\lambda|} \equiv \frac{1}{(a^sq^a ; q^m)_{\infty}} \pmod m .$$
Of course, these expressions diverge as $q \rra 1$ so $\zeta_{\mathcal P_{a(m)}}(-s)$ does not make sense, but we wonder: do there exist similarly nice relations that involve $\zeta_{\mathcal P_{a(m)}}(s)$ or a related form?
\end{remark}

\begin{proof}[Proof of Corollary \ref{ch41.6}]
We apply \eqref{ch44} to the following formula submitted by Ramanujan as a problem to the \textit{Journal of the Indian Mathematical Society}, reprinted as \cite[Question~261]{RamanujanCollected}:
$$\prod_{n=2}^{\infty} \pwr{1 - \frac{1}{n^3}} = \frac{\cosh \pwr{\frac{1}{2} \pi \sqrt{3}}}{3\pi}.$$
Take $q=1$, $f(n) = 1/n^3$ if $n>1$ and $=0$ otherwise in \eqref{ch44}. Comparing the result with the above formula gives the corollary.
\end{proof}

\begin{remark}
Ramanujan gives a companion formula $\prod_{n=1}^{\infty} \pwr{1 + \frac{1}{n^3}}=\cosh \pwr{\frac{1}{2}\pi \sqrt{3}}/\pi$ in the same problem \cite{RamanujanCollected}. Multiplying this infinite product by the one above and using \eqref{ch44} yields a closed form for $\sum_{\lambda \in \mathcal P_{\geq 2}} 1/n_{\lambda}^6$ as well.
\end{remark}

\begin{proof}[Proof of Corollary \ref{ch41.7}]
Consider the sequence $\beta_{2k}$ of coefficients of the expansion
\begin{equation}\label{ch423}
\frac{z}{\sinh z} =  \frac{1}{\prod_{n=1}^{\infty} \pwr{1 + \frac{z^2}{\pi^2n^2}}} = \sum_{k=0}^{\infty} \beta_{2k}z^{2k}.
\end{equation}
From the Maclaurin series for the hyperbolic cosecant and Euler's work relating the zeta function to the Bernoulli numbers, it follows that 
\begin{equation}\label{ch424}
\beta_{2k} = \frac{4(-1)^k(2^{2k-1} - 1)\zeta(2k)}{(2\pi)^{2k}}.
\end{equation}
On the other hand, from \eqref{ch44} we have
$$\frac{1}{\prod_{n=1}^{\infty} \pwr{1 + \frac{z^2}{\pi^2n^2}}} = \sum_{\lambda \in \mathcal P} \frac{(-1)^{\ell(\lambda)} z^{2\ell(\lambda)}}{\pi^{2\ell(\lambda)}n_{\lambda}^2} = \sum_{k=0}^{\infty} \frac{(-1)^kz^{2k}}{\pi^{2k}}\sum_{\ell(\lambda) = k} \frac{1}{n_{\lambda}^2},$$
thus 
$$\beta_{2k} = \frac{(-1)^k}{\pi^{2k}}\zeta_{\mathcal P}(\brk{2}^k).$$
The corollary is immediate by comparing the two expressions for $\beta_{2k}$ above. 
\end{proof}

\begin{proof}[Proof of Corollary \ref{ch41.8}]
Much as in the proof of Corollary \ref{ch41.7} above, we have from \eqref{ch43} that 
$$\frac{z}{\sin z} = \sum_{k=0}^{\infty} \frac{z^{2k}}{\pi^{2k}}\sum_{\ell(\lambda) = k} \frac{1}{n_{\lambda}^2} = \sum_{k=0}^{\infty} \alpha_{2k}z^{2k}$$
with 
\begin{equation}\label{ch425}
\alpha_{2k} = \frac{4(2^{2k-1} - 1)\zeta(2k)}{(2\pi)^{2k}} = (-1)^k\beta_{2k}.
\end{equation}
Using the Cauchy product 
\begin{equation}\label{ch426}
\pwr{\sum_{k=0}^{\infty} a_kz^k}\pwr{\sum_{k=0}^{\infty} b_kz^k} = \sum_{k=0}^{\infty}z^k \sum_{n=0}^k a_nb_{k-n},
\end{equation}
we see after some arithmetic 
$$\frac{z^2}{\sin z \sinh z} = \pwr{\sum_{k=0}^{\infty} \alpha_{2k}z^{2k}}\pwr{\sum_{k=0}^{\infty} \beta_{2k}z^{2k}} = \sum_{k=0}^{\infty} \gamma_{4k}z^{4k},$$
where 
$$\gamma_{4k} = \sum_{n=0}^{2k}\alpha_{2n}\beta_{4k - 2n},$$
with $\alpha_*,\beta_*$ as in \eqref{ch425},\eqref{ch426} respectively. On the other hand, the proof of Corollary \ref{ch41.5} implies 
$$\frac{z^2}{\sin z \sinh z} = \frac{1}{\prod_{n=1}^{\infty} \pwr{1 - \frac{z^4}{\pi^4n^4}}} = \sum_{k=0}^{\infty} \frac{z^{4k}}{\pi^{4k}}\sum_{\ell(\lambda) = k} \frac{1}{n_{\lambda}^4},$$
thus 
$$\gamma_{4k} = \frac{1}{\pi^{4k}} \zeta_{\mathcal P}(\brk{4}^k).$$
Comparing the two expressions for $\gamma_{4k}$ above, the theorem follows, just as in the previous proof.

We can carry this approach further to find $\zeta_{\mathcal P}(\brk{2^t}^k)$ for $t>2$, much as in the proof of Corollary \ref{ch41.5}. For instance, to find $\zeta_{\mathcal P}(\brk{8}^k)$ we begin by noting 
\begin{small}
\begin{align*}
\pwr{\sum_{k=0}^{\infty} \frac{z^{4k}}{\pi^{4k}} \zeta_{\mathcal P}(\brk{4}^k)}\pwr{\sum_{k=0}^{\infty} \frac{(-1)^kz^{4k}}{\pi^{4k}} \zeta_{\mathcal P}(\brk{4}^k)} & = \frac{1}{\prod_{n=1}^{\infty} \pwr{1 - \frac{z^4}{\pi^4n^4}}\pwr{1+\frac{z^4}{\pi^4n^4}}}\\
& = \sum_{k=0}^{\infty} \frac{z^{8k}}{\pi^{8k}} \zeta_{\mathcal P}(\brk{8}^k).
\end{align*}
\end{small}
We compare the coefficients on the left-and right-hand sides, using \eqref{ch426} to compute the coefficients on the left. Likewise, for $\zeta_{\mathcal P}(\brk{16}^k)$ we compare the coefficients on both sides of the equation
$$\pwr{\sum_{k=0}^{\infty} \frac{z^{8k}}{\pi^{8k}} \zeta_{\mathcal P}(\brk{8}^k)}\pwr{\sum_{k=0}^{\infty} \frac{(-1)^kz^{8k}}{\pi^{8k}} \zeta_{\mathcal P}(\brk{8}^k)}  = \sum_{k=0}^{\infty} \frac{z^{16k}}{\pi^{16k}} \zeta_{\mathcal P}(\brk{16}^k),$$
and so on, recursively, to find $\zeta_{\mathcal P}(\brk{2^t}^k)$ as $t$ increases. It is clear from induction that $\zeta_{\mathcal P}(\brk{2^t}^k)$ is of the form ``$\pi^{2^t} \times \text{rational}$'' for all $t \in \mathbb Z^+$.
\end{proof}

\begin{proof}[Proof of Corollary \ref{ch41.9}]
We have already seen these principles at work in the proofs of Corollaries \ref{ch41.5} and \ref{ch41.8}. We have 
$$\pwr{\sum_{\lambda\in \mathcal P_{\mathbb X}} \frac{z^{\ell(\lambda)}}{n_{\lambda}^s}} \pwr{\sum_{\lambda \in \mathcal P_{\mathbb X}} \frac{(-1)^{\ell(\lambda)z^{\ell(\lambda)}}}{n_{\lambda}^s}} = \frac{1}{\prod_{n\in {\mathbb X}} \pwr{1 - \frac{z}{n^s}} \pwr{1 + \frac{z}{n^s}}} = \sum_{\lambda \in \mathcal P_{\mathbb X}} \frac{z^{2\ell(\lambda)}}{n_{\lambda}^{2s}}.$$
Letting $z = 1$ gives \eqref{ch413}. If we replace $z$ with $z^s$ we may rewrite the above equation in the form
$$\pwr{\sum_{k=0}^{\infty} z^{sk}\zeta_{\mathcal P_{\mathbb X}}(\brk{s}^k)} \pwr{\sum_{k=0}^{\infty}(-1)^k z^{sk}\zeta_{\mathcal P_{\mathbb X}}(\brk{s}^k)} = \sum_{k=0}^{\infty} z^{2sk} \zeta_{\mathcal P_{\mathbb X}}(\brk{2s}^k).$$
Using \eqref{ch426} on the left and comparing coefficients on both sides gives the $n = 0$ case of \eqref{ch414}; the general formula follows from the $n=0$ case by induction.
\end{proof}

\begin{proof}[Proof of comments following Corollary \ref{ch41.9}]
Taking ${\mathbb X} = \mathbb P$ we see $(-1)^{\ell(\lambda)}$ specializes to Liouville's function $\lambda (n_{\lambda}) = (-1)^{\Omega (n_{\lambda})}$ (here we are using $``\lambda"$ in two different ways), where $\Omega(N)$ is the number of prime factors of $N$ with multiplicity. That \eqref{ch413} therefore becomes $\zeta(s)\sum_{n=1}^{\infty} \lambda(n)/n^s = \zeta(2s)$ follows from arguments similar to the proof of Corollary \ref{ch41.3}.
\end{proof}

\begin{proof}[Proof of Corollary \ref{ch41.10}]
This identity follows immediately by taking $\mathcal P_{\mathbb X} = \mathcal P$, $n = 0$, $k=2$ in \eqref{ch414} and simplifying.
\end{proof}

\begin{proof}[Proof of Theorem \ref{ch41.11}]
The proof of \eqref{ch415} is similar to Euler's proof that the number of partitions of $n$ into distinct parts is equal to the number of partitions into odd parts \cite{Berndt}. We expand the product
$$\varphi_{\infty}(f;q) = (1-f(1)q)(1-f(2)q^2)(1-f(3)q^3)\cdots ,$$
which results in \eqref{ch415}.

Identities \eqref{ch416} and \eqref{ch417} follow directly from the proof of \eqref{ch45},\eqref{ch46} above. Moreover, the proof of \eqref{ch418} is much like the proof of \eqref{ch48}. We note that
\begin{align*}
\frac{1}{\varphi_{\infty}(f;q)} &= 1 - \varphi_{\infty}(f;q)\begin{array}{c}
\sum_{(5)}
\end{array} \\
&= 1 - \frac{\sum_{(5)}}{1/\varphi_{\infty}(f;q)},
\end{align*}
and replace the term $1/\varphi_{\infty}(f;q)$ in the denominator on the right with the continued fraction in \eqref{ch48}.
\end{proof}

\begin{proof}[Proof of Corollary \ref{ch4etaquotient}]
The formula follows easily from the leading identities in Theorems \ref{ch41.1} and \ref{ch41.11}. We note that 
\begin{align*}
\prod_{j=1}^{n}\prod_{k_j\in {\mathbb X}_j}\left(1\pm f_j(k_j)q^{k_j}\right)^{\pm 1}
& =\prod_{j=1}^{n}\left(\sum_{\lambda\in\mathcal P_{{\mathbb X}_j}^{\pm}}q^{|\lambda|}\prod_{\lambda_i\in\lambda}f_j(\lambda_i)\right)
\\
& =\prod_{j=1}^{n}\left(\sum_{k_j=0}^{\infty}q^{k_j}\sum_{\substack{\lambda\vdash k_j\\ \lambda\in\mathcal P_{{\mathbb X}_j}^{\pm}}}\prod_{\lambda_i\in\lambda}f_j(\lambda_i)\right)
\end{align*}
and repeatedly apply Equation \eqref{ch426} on the right. 
\end{proof}

\begin{proof}[Proof of Corollary \ref{ch41.12}]
The identity is immediate from Theorem \ref{ch41.11} by letting $z=1$ in 
$$\frac{\sinh (\pi z)}{\pi z} = \prod_{n=1}^{\infty} \pwr{1 + \frac{z^2}{n^2}} = \sum_{\lambda \in \mathcal P^*} \frac{z^{\ell(\lambda)}}{n_{\lambda}^2}.$$
\end{proof}

\begin{proof}[Proof of Corollary \ref{ch41.13}]
This proof proceeds much like the proofs of Corollaries \ref{ch41.5}, \ref{ch41.7}, \ref{ch41.8} above, only more easily. We have from \eqref{ch43} and Theorem \ref{ch41.11}, together with the Maclaurin expansion of the sine function, that 
$$\frac{\sin z}{z} = \sum_{k=0}^{\infty} \frac{(-1)^kz^{2k}}{\pi^{2k}}\zeta(\brk{2}^k) = \sum_{k=0}^{\infty} \frac{(-1)^kz^{2k}}{(2k+1)!}.$$
Comparing the coefficients of the two summations above gives $\zeta(\brk{2}^k)$. We carry this approach further to find $\zeta(\brk{2^t}^k)$ for $t > 1$. We proceed inductively from the case above. Take the identity 
$$\pwr{\sum_{k=0}^{\infty} \frac{z^{2^{t-1}k}}{\pi^{2^{t-1}k}}\zeta\left(\brk{2^{t-1}}^k\right)} \pwr{\sum_{k=0}^{\infty} \frac{(-1)^kz^{2^{t-1}k}}{\pi^{2^{t-1}k}}\zeta\left(\brk{2^{t-1}}^k\right)} = \sum_{k=0}^{\infty}\frac{z^{2^tk}}{\pi^{2^tk}} \zeta(\brk{2^t}^k)$$
and compare coefficients on the left- and right-hand sides, using \eqref{ch426} to compute the coefficients on the left; expressions such as the remaining ones in the statement of the corollary result. It is clear from induction that $\zeta(\brk{2^t}^k)$ always has the form $``\pi^{2^tk} \times \text{finite sum of fractions}"$.
\end{proof}

\begin{proof}[Proof of Corollary \ref{ch41.14}]
This proof is nearly identical to the proof of Corollary \ref{ch41.9}. From the associated product representations it is clear that 
$$\pwr{\sum_{\lambda \in \mathcal P_{\mathbb X}^*} \frac{z^{\ell(\lambda)}}{n_{\lambda}^s}}\pwr{\sum_{\lambda \in \mathcal P_{\mathbb X}^*} \frac{(-1)^{\ell(\lambda)}z^{\ell(\lambda)}}{n_{\lambda}^s}} = \sum_{\lambda \in \mathcal P_{\mathbb X}^*} \frac{(-1)^{\ell(\lambda)} z^{2\ell(\lambda)}}{n_{\lambda}^{2s}}.$$
Letting $z = 1$ gives \eqref{ch424}. If we replace $z$ with $z^s$ we may rewrite the above equation as 
$$\pwr{\sum_{k=0}^{\infty} z^{sk}\zeta_{\mathcal P_{\mathbb X}^*}(\brk{s}^k)} \pwr{\sum_{k=0}^{\infty} (-1)^kz^{sk}\zeta_{\mathcal P_{\mathbb X}^*}(\brk{s}^k)} = \sum_{k=0}^{\infty} (-1)^kz^{2sk} \zeta_{\mathcal P_{\mathbb X}^*}(\brk{2s}^k).$$
Again using \eqref{ch426} on the left and comparing coefficients on both sides gives the $n=0$ case of \eqref{ch422}; the general formula follows by induction.
\end{proof}

\begin{proof}[Proof of comments following Corollary \ref{ch41.14}]
Taking ${\mathbb X} = \mathbb P$ in Theorem \ref{ch41.11} and noting that $\lambda \in \mathcal P_{\mathbb P}^*$ implies $n_{\lambda}$ is squarefree, we see $(-1)^{(\lambda)} = \mu(n_{\lambda})$, where $\mu$ denotes the classical M\"obius function; therefore, we have the identity
$$\sum_{n=1}^{\infty}\frac{\mu(n)}{n^s} = \sum_{\lambda \in \mathcal P_{\mathbb P}^*} \frac{\mu(n_{\lambda})}{n_{\lambda}^s} = \eta_{\mathcal P_{\mathbb P}^*}(s) = \frac{1}{\zeta_{\mathcal P_{\mathbb P}(s)}} = \frac{1}{\zeta(s)}.$$
On the other hand, we have $\zeta_{\mathcal P_{\mathbb P}^*}(s) = \sum_{n \text{ squarefree}} 1/n^s = \sum_{n=1}^{\infty} |\mu(n)|/n^s$.
\end{proof}
\  
\

\begin{remark}
{See Appendix \ref{app:C} for further notes on Chapter 4.} 
\end{remark}

 	\clearpage%
\chapter{Partition zeta functions: further explorations}{\bf Adapted from \cite{ORS}, a joint work with Ken Ono and Larry Rolen}

%
%\abstract{}
%We introduce and survey results on two families of zeta functions connected to the multiplicative and additive theories of integer partitions.  In the case of the multiplicative theory, we provide specialization formulas and results on the analytic continuations of these ``partition zeta functions'', find unusual formulas for the Riemann zeta function, prove identities for multiple zeta values, and see that some of the formulas allow for $p$-adic interpolation. The second family we study was anticipated by Manin and makes use of modular forms, functions which are intimately related to integer partitions by universal polynomial recurrence relations. We survey recent work on these zeta polynomials, including the proof of their Riemann Hypothesis. 
%

%%%%%%%%%%%%%%%%%%%%%%%%%%%%%%%%%%%%%%%%%%%%%%%%%%%%%%%%%%%%%%%%%%%%%%%%%%%%%%%%%%%%%%%%%%%
%%%%%%%%%%%%%%%%%%%%%%%%%%%%%%%%%%%%%%%%%%%%%%%%%%%%%%%%%%%%%%%%%%%%%%%%%%%%%%%%%%%%%%%%%%%

\section{Following up on the previous chapter}

In Chapter 4, we see the Riemann zeta function as the prototype for a new class of combinatorial objects arising from Eulerian methods. In this chapter we record a number of further identities relating certain zeta functions arising from the theory of partitions to various objects in number theory such as Riemann zeta values, multiple zeta values, and infinite product formulas. Some of these formulas are related to results in the literature; they are presented here as examples of this new class of partition-theoretic zeta functions. We also give several formulas for the Riemann zeta function, and results on the analytic continuation (or non-existence thereof) of zeta-type series formed in this way. Furthermore, we discuss the $p$-adic interpolation of these zeta functions in analogy with the classical work of Kubota and Leopoldt on $p$-adic continuation of the Riemann zeta function \cite{KL}.

%The proofs are collected at the end of each section{\bf RPS: to keep the prose and structure uncluttered and elegant}, and are suppressed where they might be easily provided by the reader. 

%

\section{Evaluations}
We saw in the previous chapter a variety of simple closed forms for partition zeta functions, depending on the natures of the subsets of partitions being summed over. Different subsets induce different zeta phenomena. %; we are interested to evaluate $\zeta_{\mathcal P'}(s)$ on different subsets $\mathcal P'\subsetneq\mathcal P$. 
In what follows, we consider the evaluations of a small sampling of possible partition zeta functions having particularly pleasing formulas.

\subsection{Zeta functions for partitions with parts restricted by congruence conditions}

%%%%2

Our first line of study will concern sets $\mathbb M\subset\mathbb N$ that are defined by congruence conditions. Note by considering Euler products as in Definition \ref{ch1zetadef} %by \eqref{ch5DirichletProduct} 
that for disjoint $\mathbb M_1, \mathbb M_2\subset \mathbb Z^+$,
\[
\zeta_{\mathcal P_{\mathbb M_1\cup\mathbb M_2}}(s)
=
\zeta_{\mathcal P_{\mathbb M_1}}(s)\zeta_{\mathcal P_{\mathbb M_2}}(s)
.
\]
Hence, to study any set of partitions determined by congruence conditions on the parts, it suffices to consider series of the form
\[
\zeta_{\mathcal P_{a+m\mathbb N}}(s)
,
\]
where $a\in\Z_{\geq0}$, $m\in\N$  (excluding the case $a=0$, $m=1$, where the zeta function clearly diverges), and $\mathcal P_{a+m\mathbb N}$ is partitions into parts congruent to $a$ modulo $m$. We see examples of the case $\zeta_{\mathcal P_{m\mathbb N}}(2^N)=\zeta_{\mathcal P_{0+m\mathbb N}}(2^N)$ in Corollaries \ref{ch41.4} and \ref{ch41.4}; we are interested in the most general case, with $s=n\in\mathbb N$.   

%{\bf LR: Now we cite a formula from the previous section here}
%{\bf ***RPS: Insert description of easy cases from previous zeta paper as an equation.***}. {\bf LR: I think this isn't needed. Just put a few simple cases of this formula in the last section, and then cite them here, saying which specializations they correspond to for us, and then segue to say we have the more general formula}

Our first main result is then the following, where $\Gamma$ is the usual gamma function of Euler and $e(x):=e^{2\pi ix}$. The proof will use an elegant and useful formula highlighted by Chamberland and Straub in \cite{ChamberlandStraub}, which we note was also inspired by previous work on multiplicative partitions in \cite{ChamberlandJohnsonNadeauWu}. In fact, the following result is a generalization of Theorem 8 of \cite{ChamberlandJohnsonNadeauWu} which in our notation corresponds to $a=m=1$. 

\begin{theorem}\label{ch5mainthm}
For $n\geq2$, we have
\[
\zeta_{\mathcal P_{a+m\mathbb N}}(n)
=\Gamma(1+a/m)^{-n}\prod_{r=0}^{n-1}\Gamma\left(1+\frac{a-e(r/n)}{m}\right)
.
\]
\end{theorem}
Theorem \ref{ch5mainthm} has several applications. 
By considering the expansion of the logarithm of the gamma function, we easily obtain the following result, in which $\gamma$ is the Euler-Mascheroni constant and the principal branch of the logarithm is taken.

\begin{corollary}\label{ch5FirstCor}
For any $m,n\geq2$, we have that
\begin{equation*}
\begin{aligned}
\log\left(\zeta_{\mathcal P_{a+m\mathbb N}}(n)\right)
&
=
n\log(1+a/m)+\frac{a(n+1)}m(1-\gamma)
-
\sum_{r=0}^{n-1}
\log\left(
1+\frac{a-e(r/n)}m
\right)
\\
&
+
\sum_{r=0}^{n-1}
\sum_{k\geq2}\frac{(-1)^k(\zeta(k)-1)\left(a^k+\left(a-e(r/n)\right)^k\right)}{km^k}
.
\end{aligned}
\end{equation*}
\end{corollary}

When $a=0$ and $m\geq2$, we obtain the following strikingly simple formula, which is similar to Theorem 7 of \cite{ChamberlandJohnsonNadeauWu} that in our notation corresponds to the case $a=m=1$.
\begin{corollary}\label{ch5SecondCor}
For any $m,n\geq2$, we have that
\[
\log\left(\zeta_{\mathcal P_{m\mathbb N}}(n)\right)
=
n
\sum_{\substack{k\geq2\\ n|k}}\frac{\zeta(k)}{km^{k}}
.
\]
\end{corollary}

%{\bf RPS: Here there needs some transitional wording, re-ordered text follows, which we still want to evaluate for inclusion.} 

%%%%4

\subsection{Connections to ordinary Riemann zeta values}
%{\bf LR: Renamed subsection}
%{\bf LR: Added a few sentences}

In addition to providing interesting formulas for values of more exotic partition-theoretic zeta functions, the above results also lead to curious formulas for the classical Riemann zeta function. In fact, $\zeta(s)$ 
is itself a partition zeta function, summed over prime partitions, so it is perhaps not too surprising to find that we can learn something about it from a partition-theoretic perspective. Then we continue the theme of evaluations by recording a few results expressing the value of $\zeta$ at integer argument $n>1$ in terms of gamma factors. %, independently of multiplicative partitions. 

In the first, curious identity, let $\mu$ denote the classical M\"obius function. We point out that this is essentially a generalization of a formula for the case $a=m=1$ given in Equation 11 of \cite{ChamberlandJohnsonNadeauWu}. 

\begin{corollary}\label{ch5FourthCor}
For all $m,n\geq2$, we have
\[
\zeta(n)=m^n\sum_{\substack{k\geq 1}}\frac{\mu(k)}k\sum_{r=0}^{nk-1}\log\left(\Gamma\left(1-\frac{e\left(\frac r{nk}\right)}m\right)\right)
.
\]
\end{corollary}

%{\bf RPS: I think we questioned whether to keep the following three formulas or not. Let's leave them in for now, see what Ken thinks. LR: agreeed.}

The next identity gives $\zeta(n)$ in terms of the $n$th derivative of a product of gamma functions.  The authors were not able to find this formula in the literature; however, given the well-known connections between $\Gamma$ and $\zeta$, as well as the known example below the following theorem, it is possible that the identity is known.
\begin{theorem}\label{ch5RobertCor}
For integers $n>1$, we have
\[
\zeta(n)=\frac{1}{n!} \lim_{z\to 0^+} \frac{\mathrm d^n}{\mathrm d z^n}\prod_{j=0}^{n-1}\Gamma\left(1-ze(j/n)\right)
.
\]
\end{theorem}

\begin{example}As an example of implementing the above identity, take $n=2$; then using Euler's well-known product formula for the sine function, it is easy to check that
\[
\zeta(2)=\frac{1}{2!} \lim_{z\to 0^+} \frac{\mathrm d^2}{\mathrm d z^2}\Gamma\left(1+z\right)\Gamma\left(1-z\right)
=\frac{1}{2!} \lim_{z\to 0^+} \frac{\mathrm d^2}{\mathrm d z^2}\frac{\pi z}{\sin(\pi z)}=\frac{\pi^2}{6}
.
\]
\end{example} 

This last formula for $\zeta(n)$, following from a formula in Chapter 4 together with the preceding theorem, is analogous to some extent to the classical identity $\sin(n)=\frac{e^{in}-e^{-in}}{2i}$.
\begin{corollary}\label{ch5RobertCor2}
For integers $n>1$, we have
\[
\zeta(n)=\lim_{z\to 0^+}\frac{\prod_{j=0}^{n-1}\Gamma\left(1-ze(j/n)\right)-\prod_{j=0}^{n-1}\Gamma\left(1-ze(j/n)\right)^{-1}}{2z^n}
.
\]
\end{corollary}
%%%%%%%%%%%

%%%%5
\subsection{Zeta functions for partitions of fixed length}

We now consider zeta sums of the shape $\zeta_{\mathcal P}(\{s\}^k)$ as in Definition \ref{ch4zetafixed}. Our first aim will be to extend Corollary \ref{ch41.7} from the previous chapter. %above (which is Corollary 2.4 of \cite{Robert}).  %; in that work, the second author defined the partition zeta function {\bf RPS: Is this even necessary? Maybe just use the case P'=P below? LR: I don't understand the question. We already defined some of this notation above, so if new notation isn't needed, make sure not to redefine, or to say that we are recalling the definition.}
%\[
%\zeta_{\mathcal P'}(\{s\}^k):=\sum_{\substack{\lambda\in\mathcal P'\\ \ell(\lambda)=k}}n_{\lambda}^{-s}, \mathrm{Re}(s)>1
%,
%\]
%taken over partitions in some subset $\mathcal P'$ of $\mathcal P$, where for a general partition $\lambda$, we will denote by $\ell(\lambda)$ the {\it length} of $\lambda$, i.e., the number of parts in $\lambda$.
%
%Here we will consider such sums taken over all partitions
%\[
%\zeta_{\mathcal P}(\{s\}^k)=\sum_{\substack{\lambda\in\mathcal P\\ \ell(\lambda)=k}}n_{\lambda}^{-s}.
%\]
%{\bf LR: Think already defined}
%As we will see, zeta values of this form intersect classes of series considered by Hoffman in \cite{Hoffman1}.

%{\bf ***RPS: Need to give a couple of nice examples from zeta paper here.*** LR: Cite from section 2.}

Let $[z^n]f$ represent the coefficient of $z^n$ in a power series $f$. Using this notation, we show the following, which in particular gives an algorithmic way to compute each $\zeta_{\mathcal P}(\{m\}^k)$ in terms of Riemann zeta values for $m\in\mathbb N_{\geq 2}$.
\begin{theorem}\label{ch5SecondTheorem}
For all $m\geq2$, $k\in\N$, we have
\begin{align*}
\zeta_{\mathcal P}(\{m\}^k)
&=
\pi^{mk}[z^{mk}]\prod_{r=0}^{m-1}\Gamma\left(1-\frac{z}{\pi}e(r/m)\right)\  \\
&=
\pi^{mk}[z^{mk}]\operatorname{exp}\left(\sum_{j\geq1}\frac{\zeta(mj)}{j}\left(\frac{z}{\pi}\right)^{mj}\right)
.
\end{align*}
\end{theorem}
%%%%%%%%%%%%%

Generalizing the comments just below Corollary \ref{ch41.8}, the next corollary follows directly from Theorem \ref{ch5SecondTheorem} (using the fact that $\zeta(k)\in\mathbb Q\pi^k$ for even integers $k$).

\begin{corollary}\label{ch5RationalityCor}
For $m\in 2\N$ even, we have that
\[
\zeta_{\mathcal P}(\{m\}^k)\in\mathbb Q\pi^{mk}
.
\]
\end{corollary}
\begin{remark}
This can also be deduced from Theorem 2.1 of \cite{Hoffman1}.
\end{remark}

%%%%7

To conclude this section, we note one explicit method for computing the values $\zeta_{\mathcal P}(\{m\}^k)$ at integral $k,m$ (especially if $m$ is even, in which case the zeta values below are completely elementary).

\begin{corollary}\label{ch5DeterminantCor}
For $m\geq2, k\in\N$, and $j\geq i$, 
set 
\[
\alpha_{i,j}
:=
\zeta(m(j-i+1))\frac{(k-i)!}{\pi^{m(j-i+1)}(k-j)!}
.\]
Then we have 
\[
\zeta_{\mathcal P}(\{m\}^k)
=
\frac{\pi^{mk}}{k!} \det
\begin{pmatrix} 
\alpha_{1,1}&\alpha_{1,2}&\alpha_{1,3}&\ldots&\alpha_{1,k}
\\
-1 & \alpha_{2,2} &\alpha_{2,3} &\ldots & \alpha_{2,k}
\\
0&-1& \alpha_{3,3}&\ldots & \alpha_{3,k}
\\
\vdots &\vdots&\vdots &\ddots &\vdots
\\
0&0&\ldots&-1&\alpha_{k,k} 
\end{pmatrix}
.
\]
\end{corollary}

\begin{remark}
There are results resembling these in Knopfmacher and Mays \cite{KM}.
\end{remark}

\section{Analytic continuation and $p$-adic continuity}

%%%%3

%\begin{remark}
%Corollary \ref{ch5FirstCor} and Corollary \ref{ch5SecondCor} extend results for $(a,m)=(0,1)$ given in \cite{ChamberlandJohnsonNadeauWu}.
%end{remark}

%{\bf LR: All is new. Robert, please check carefully RPS: Will do, Larry, but deleting this comment for clean-ness.}

If we jump forward about 100 years from the pathbreaking work of Euler concerning special values of the Riemann zeta function at %positive 
even %(and even negative) 
integers, we arrive at the famous work of Riemann in connection with prime number theory (see \cite{Edwards}). Namely, in 1859, Riemann brilliantly described the most significant properties of $\zeta(s)$ following that of an Euler product: the analytic continuation and functional equation for $\zeta(s)$. It is for this reason, of course, that the zeta function is named after Riemann, and not Euler, who had studied this function in some detail, and even conjectured a related functional equation. In particular, this analytic continuation allowed Riemann to bring the zeta function, and indeed the relatively new field of complex analysis, to the forefront of number theory by connecting its roots to the distribution of prime numbers.

It is natural therefore, whenever one is faced with new zeta functions, to ask about their prospect for analytic continuation. Here, we offer a brief study of some of these properties, in particular showing that the situation for our zeta functions is much more singular. Partition-theoretic zeta functions in fact naturally give rise to functions with essential singularities. Here, we use
Corollary \ref{ch5SecondCor}
to study the continuation properties of partition zeta functions over partitions $\mathcal P_{m\mathbb N}$ into multiples of $m>1$. In order to state the result we first define, for any $\varepsilon >0$, the right half-plane $\mathbb H_{\varepsilon}:=\{z\in\C : \operatorname{Re}(z)>\varepsilon\}$, and we denote by $\frac1{\N}$ the set $\{1/n : n\in\N\}$.
\begin{corollary}\label{ch5ThirdCor}
For any $\varepsilon>0$ and $m>1$, $\zeta_{\mathcal P_{m\mathbb N}}(s)$ has a meromorphic extension to $\mathbb H_{\varepsilon}$ with poles exactly at $\mathbb H_{\varepsilon}\cap\frac1{\N}$. In particular, there is no analytic continuation beyond the right half-plane $\operatorname{Re}(s)>0$, as there would be an essential singularity at $s=0$.
\end{corollary}
\begin{remark}For the function $\zeta_{\mathcal P_{\N}}(s)$, a related discussion of poles and analytic continuation was made by the user mohammad-83 in a MathOverflow.net question.\end{remark}
%%%%%%%

Finally, we follow Kubota and Leopoldt \cite{KL}, who showed $\zeta$ could be modified slightly to obtain modified zeta functions for any prime $p$ which extend $\zeta$ to the space of $p$-adic integers $\mathbb Z_p$, to yield further examples of $p$-adic zeta functions of this sort. These continuations are based on the original observations of Kubota and Leopoldt, and, in a rather pleasant manner, on the evaluation formulas discussed above.

In particular, we will use Corollary \ref{ch5DeterminantCor} to $p$-adically interpolate modified versions of $\zeta_{\mathcal P}(\{m\}^k)$ in the $m$-aspect. Given the connection discussed in Section \ref{ch5MZVSection} to multiple zeta values, these results should be compared with the literature on $p$-adic multiple zeta values (e.g., see \cite{Furusho}), although we note that our $p$-adic interpolation procedure seems to be more direct in the special case we consider.  

The continuation in the $m$-aspect of this function is also quite natural, as the case $k=1$ is just that of the Riemann zeta function. Thus, it is natural to search for a suitable $p$-adic zeta function that specializes to the function of Kubota and Leopoldt when $k=1$. It is also desirable to find a $p$-adic interpolation result which makes the partition-theoretic perspective clear.

Here, we provide such an interpretation. Let us first denote the set of partitions with parts not divisible by $p$ as $\mathcal P_p$; then we consider the length-$k$ partition zeta values $\zeta_{\mathcal P_p}(\{s\}^k)$. Note that for $k=1$, $\zeta_{\mathcal P_p}(\{s\}^1)$ is just the Riemann zeta function with the Euler factor at $p$ removed (as considered by Kubota and Leopoldt). We then offer the following $p$-adic interpolation result.
\begin{theorem}\label{ch5padicInterpThm}
Let $k\geq1$ be fixed, and let $p\geq k+3$ be a prime. Then $\zeta_{\mathcal P_p}(\{s\}^{k})$ can be extended to a continuous function for $s\in\Z_p$ which agrees with $\zeta_{\mathcal P_p}(\{s\}^{k})$ on a positive proportion of integers.
\end{theorem}

\section{Connections to multiple zeta values}\label{ch5MZVSection}

%%%%8

%\subsection{Connections to multiple zeta values}
Our final application of the circle of ideas related to partition zeta functions and infinite products will be in the theory of multiple zeta values. 

\begin{definition}\label{ch5mzvdef}
We define for natural numbers $m_1,m_2,\ldots ,m_k$ with $m_k>2$ the {\it multiple zeta value} (commonly written ``MZV'')
\begin{equation*}
\zeta(m_1,m_2,\ldots,m_k):=\sum_{n_1>n_2>\ldots>n_k\geq 1}\frac{1}{n_1^{m_1}\ldots n_k^{m_k}}
.
\end{equation*}
We call $k$ the {\it length} of the MZV. Furthermore, if $m_1=m_2=\ldots =m_k$ are all equal to some $m\in\mathbb N$, we use the common notation
\begin{equation}\label{ch5mzvpzv}
\zeta(\{m\}^k):=\sum_{n_1>n_2>\ldots>n_k\geq 1}\frac{1}{\left(n_1 n_2\ldots n_k\right)^{m}}.
\end{equation}
\end{definition}

Multiple zeta values have a rich history and enjoy widespread connections; interested readers are referred to Zagier's short note \cite{ZagierMZV}, %Hoffman's webpage \cite{Hoffman2}, 
and for a more detailed treatment, the excellent lecture notes of Borwein and Zudilin \cite{BorweinZudilin}. There are many nice closed-form identities in the literature; for example, one can show (see \cite{Hoffman1}) on analogy to Corollary \ref{ch41.7} that   
\begin{equation}
\zeta(\{2\}^k) = \frac{\pi^{2k}}{(2k+1)!},
\end{equation}
which we prove, along with similar (but more complicated) expressions for $\zeta(\{2^t \}^k)$, in the previous chapter. % for $\zeta(\{2^t\}^k)$ for all $t\in\mathbb N$ are given in Chapter 4, parallel to those mentioned in Section \ref{ch5Partition-theoretic zeta functions} for $\zeta_{\mathcal P}(\{2^t \}^k)$.  

Observe from its definition that the partition zeta function $\zeta_{\mathcal P}(\{m\}^k)$ %(see Definition \ref{ch5pzv}) 
can be rewritten in a similar-looking form to \eqref{ch5mzvpzv} above:
\begin{equation}\label{ch5mzvpzv2}
\zeta_{\mathcal P}(\{m\}^k)=\sum_{n_1\geq n_2\geq \ldots\geq n_k\geq 1}\frac{1}{\left(n_1 n_2\ldots n_k\right)^{m}}.
\end{equation}
In fact, if we take $\mathcal P^*$ to denote partitions into distinct parts, then \eqref{ch5mzvpzv} reveals $\zeta(\{m\}^k)$ is equal to the partition zeta function $\zeta_{\mathcal P^*}(\{m\}^k)$ summed over length-$k$ partitions into distinct parts, as pointed out in the preceding chapter. Series such as those in \eqref{ch5mzvpzv2} have been considered and studied extensively by Hoffman (for instance, see \cite{Hoffman1}). 

%
%{\bf ***RPS: This section should include one or two easy MZV examples from previous paper.*** LR: I think now just cite equaiton or two from section 2.}

By reorganizing sums of the shape \eqref{ch5mzvpzv2}, we arrive at interesting relations between $\zeta_{\mathcal P}(\{m\}^k)$ and families of MZVs. In order to describe these relations, we first recall that a {\it composition} is simply a finite tuple of natural numbers, and we call the sum of these integers the {\it size} of the composition. Denote the set of all compositions by $\mathcal C$ and write $|\lambda|=k$ for $\lambda=(a_1,a_2,\ldots,a_j)\in\mathcal C$ if $k=a_1+a_2+\ldots+a_j$. Then we obtain the following. 

%{\bf RPS: Switched $n$ to $m$ in the arguments of the MZV on the RHS below... was n , should be m as a factor of each argument, right?}

\begin{proposition}\label{ch5DecouplingCorollary}
Assuming the notation above, we have that
\[
\zeta_{\mathcal P}(\{m\}^k)
=
\sum_{\substack{\lambda=(a_1,\ldots, a_j)\in\mathcal C\\ |\lambda|=k}}\zeta(a_1m,a_2m,\ldots,a_jm)
.
\]
\end{proposition}
\begin{remark}
Proposition \ref{ch5DecouplingCorollary} is analogous to results of Hoffman; the reader is referred to Theorem 2.1 of \cite{Hoffman1}.
\end{remark}

In particular, for any $n>1$ we can find the following reduction of $\zeta(\{n\}^k)$ to MZVs of smaller length. We note that in Theorem 2.1 of \cite{Hoffman1}, Hoffman also shows directly how to write these values in terms of products (as opposed to simply linear combinations) of ordinary Riemann zeta values: hints, perhaps, of further connections.
We remark in passing that this can be thought of as a sort of ``parity result'' (cf. \cite{IKZ,Tsumura}).
\begin{corollary}\label{ch5ParallelMZVLowerOrders}
For any $n,k>1$, the MZV $\zeta(\{n\}^k)$ of length $k$ can be written as an explicit linear combination of MZVs of lengths less than $k$.
\end{corollary}

As our final result, we give a simple formula for $\zeta(\{n\}^k)$. This formula is probably already known; if $k=2$ it follows from a well-known result of Euler (see the discussion of $H(n)$ on page 3 of \cite{Zagier}); % and is closely related to (11) and (32) of \cite{borwein1997evaluations}. 
the idea of the proof is also similar to what has appeared in, for example, \cite{Zagier}. However, the authors have decided to include it due to connections with the ideas used throughout this paper, and the simple deduction of the formula from expressions necessary for the proofs of the results described above. 

\begin{proposition}\label{ch5ParallelValuesExp}
The MZV $\zeta(\{n\}^k)$ of length $k$ can be expressed as a linear combination of products of ordinary $\zeta$ values. In particular, we have
\[
\zeta(\{n\}^k)
=
(-1)^k
\left[z^{nk}\right]
\operatorname{exp}
\left(-
\sum_{j\geq1}
\frac{\zeta(nj)}jz^{nj}
\right)
.
\]
\end{proposition}
%{\bf LR: Minus sign fixed here LR: Can this go? It looks like there should be another set of parentheses.}
\begin{remark}
This formula is equivalent to a special case of Theorem 2.1 of \cite{Hoffman1}. However, since the approach is very simple and ties in with the other ideas in this paper, we give a proof for the reader's convenience.
\end{remark} 

The proof of Corollary \ref{ch5DeterminantCor} yields a similar determinant formula here.
\begin{corollary}\label{ch5DeterminantCor2}
For $n\geq2, k\in\N$, and $j\geq i$, 
set 
\[
\beta_{i,j}
:=
-
\zeta(n(j-i+1))\frac{(k-i)!}{(k-j)!}
.\]
Then we have 
\[
\zeta(\{n\}^k)
=
\frac{(-1)^k}{k!} \det
\begin{pmatrix} 
\beta_{1,1}&\beta_{1,2}&\beta_{1,3}&\ldots&\beta_{1,k}
\\
-1 & \beta_{2,2} &\beta_{2,3} &\ldots & \beta_{2,k}
\\
0&-1& \beta_{3,3}&\ldots & \beta_{3,k}
\\
\vdots &\vdots&\vdots &\ddots &\vdots
\\
0&0&\ldots&-1&\beta_{k,k} 
\end{pmatrix}
.
\]
\end{corollary}
%{\bf LR: Minus signs fixed here too LR: Can this go now?}
\begin{remark} We can see from the above corollary that $\zeta(\{n\}^k)$ is a linear combination of products of zeta values, which is closely related to formulas of Hoffman \cite{Hoffman1}.
\end{remark}

\section{Proofs}
\subsection{Machinery}

\subsubsection{Useful formulas}
In this section, we collect several formulas that will be key to the proofs of the theorems above.  We begin with the following beautiful formula given by Chamberland and Straub in Theorem 1.1 of \cite{ChamberlandStraub}). In fact, this formula has a long history, going back at least to Section 12.13 of \cite{WhitttakerWatson}, and we note that Ding, Feng, and Liu independently discovered this same result in Lemma 7 of \cite{DingFengLiu}.
\begin{theorem}\label{ch5Armin11}
If $n\in\N$ and $\alpha_1,\ldots,\alpha_n$ and $\beta_1,\ldots,\beta_n$ are complex numbers, none of which are non-positive integers, with $\sum_{j=1}^n\alpha_j=\sum_{j=1}^n\beta_j$, then we have
\[
\prod_{k\geq0}
\prod_{j=1}^n\frac{(k+\alpha_j)}{(k+\beta_j)}=\prod_{j=1}^n\frac{\Gamma(\beta_j)}{\Gamma(\alpha_j)}
.
\]
\end{theorem}

We will also require two Taylor series expansions for $\log\Gamma$, both of which follow easily from Euler's product definition of the gamma function \cite{Edwards}. The first expansion, known as Legendre's series, is valid for $|z|<1$ (see (17) of \cite{Wrench}):

\begin{equation}\label{ch5LogGammaSecondFormula}
\log\Gamma(1+z)
=
-\gamma z+\sum_{k\geq2}\frac{\zeta(k)}k(-z)^k
.
\end{equation}
We also have the following expansion valid for $|z|<2$\footnote{See for instance (5.7.3) of NIST Digital Library of Mathematical Functions, \url{http://dlmf.nist.gov/}, Release 1.0.6 of 2013-05-06.}:
\begin{equation}\label{ch5573Nist}
\log\Gamma(1+z)=-\log(1+z)+z(1-\gamma)+\sum_{k\geq2}(-1)^k(\zeta(k)-1)\frac{z^k}k
.
\end{equation}

Furthermore, we need a couple of facts about Bell polynomials (see Chapter 12.3 of \cite{Andrews}). The $n$th \emph{complete Bell polynomial} is the sum
\[
B_n(x_1,\dots,x_n):=\sum_{i=1}^n B_{n,i}(x_1,x_2,\dots,x_{n-i+1})
.
\]
The $i$th term here is the polynomial
\begin{align*}
B_{n,i}&(x_1,x_2,\dots,x_{n-i+1})\\
&:=\sum \frac{n!}{j_1!j_2!\cdots j_{n-i+1}!}
\left(\frac{x_1}{1!}\right)^{j_1}\left(\frac{x_2}{2!}\right)^{j_2}\cdots\left(\frac{x_{n-i+1}}{(n-i+1)!}\right)^{j_{n-i+1}},
\end{align*}
where we sum over all sequences $j_1, j_2,..., j_{n-i+1}$ of nonnegative integers such that $j_1+j_2+\cdots+j_{n-i+1}=i$ and $j_1+2j_2+3j_3+\cdots+(n-i+1)j_{n-i+1}=n$.

With these notations, we use a specialization of the classical Fa\`{a} di Bruno formula \cite{diBruno}, which allows us to write the exponential of a formal power series as a power series with coefficients related to complete Bell polynomials\footnote{We prove Fa\`{a} di Bruno's formula and give other partition-theoretic applications in Appendix D.}: 
\begin{equation}\label{ch5FdB}
\operatorname{exp}\left(\sum_{j=1}^\infty \frac{a_j}{j!} x^j \right)
= \sum_{k=0}^\infty \frac{B_k(a_1,\dots,a_k)}{k!} x^k.
\end{equation}

Fa\`{a} di Bruno also gives an identity \cite{diBruno} that specializes to the following formula for the $k$th complete Bell polynomial in the series above as the determinant of a certain $k\times k$ matrix: 

\begin{equation}\label{ch5determinant}
B_k(a_1,\dots,a_k)
=
\det
\left(
\begin{smallmatrix} 
a_1&\binom{k-1}{1}a_2&\binom{k-1}{2}a_3&\binom{k-1}{3}a_4&\ldots&\ldots&a_k
\\\\
-1 & a_1 &\binom{k-2}{1}a_2&\binom{k-2}{2}a_3&\ldots&\ldots&a_{k-1}
\\\\
0&-1& a_1&\binom{k-3}{1}a_2&\ldots&\ldots& a_{k-2}
\\\\
0&0&-1& a_1&\ldots&\ldots& a_{k-3}
\\\\
0&0&0&-1&\ldots&\ldots& a_{k-4}
\\\\
\vdots &\vdots&\vdots &\vdots &\ddots &\ddots &\vdots
\\\\
0&0&0&0&\ldots&-1&a_1 
\end{smallmatrix}
\right).
\end{equation}
\pagebreak
\subsection{Proofs of Theorems \ref{ch5mainthm} and \ref{ch5RobertCor}, and their corollaries}

We begin with the proof of our first main formula.
\begin{proof}[Proof of Theorem \ref{ch5mainthm}]
By Euler products, as in the previous chapter, we find that %\eqref{ch5DirichletProduct}, we find that 
\[
\zeta_{\mathcal P_{a+m\mathbb N}}(n)
=
\prod_{k\in a+m\N}\frac{k^n}{k^n-1}=\prod_{j\geq1}\frac{(a+mj)^n}{(a+mj)^n-1}
=
\prod_{j\geq0}\prod_{r=0}^{n-1}\frac{(j+1+a/m)^n}{\left(j+1+\frac{a-e(r/n)}{m}\right)}
.
\]
Using Theorem \ref{ch5Armin11}  and the well-known fact that
\begin{equation}\label{ch5ExpSumZero}
\sum_{j=0}^{n-1}e(j/n)=0
\end{equation} directly gives the desired result.
\end{proof}
\begin{proof}[Proof of Corollary \ref{ch5FirstCor}]
For this, we apply \eqref{ch5573Nist}
and use \eqref{ch5ExpSumZero},
 the obvious fact that 
 \[|(a-e(j/n))/m|<2
 ,
 \]
and the easily-checked fact that \[1+(a-e(j/n))\] is never a negative real number for $j=0,\ldots,n-1$. 
\end{proof}
\begin{proof}[Proof of Corollary \ref{ch5SecondCor}]
Here, we simply use \eqref{ch5LogGammaSecondFormula}.
Again, the corollary is proved following a short, elementary computation, using the classical fact that 
\[
\sum_{r=0}^{n-1}e(rk/n)=\begin{cases}n&\text{ if }n|k,\\ 0&\text{ else.}\end{cases}
\]
\end{proof}
\begin{proof}[Proof of Corollary \ref{ch5ThirdCor}]
By Corollary \ref{ch5SecondCor}, we find for $n\geq2$ that
\[
\log\left(\zeta_{\mathcal P_{m\mathbb N}}(n)\right)
=
\sum_{k\geq1}\frac{\zeta(nk)}{km^{kn}}
.
\]
Suppose that $\operatorname{Re}(s)>0$ and $s\not\in\frac1{\N}$.  Then letting 
\[
K
:=
\max
\{
\lceil1/\operatorname{Re}(s)\rceil+1,\operatorname{Re}(s)
\},
\]
it clearly suffices to show that 
\[\sum_{k\geq K}\frac{\zeta(sk)}{km^{ks}}\]
converges. But in this range on $k$, by choice we have $\operatorname{Re}(sk)>1$, so that using the assumption $m\geq2$, we find for $\operatorname{Re}(s)>0$ the upper bound
\begin{align*}
\sum_{k\geq K}\frac{\zeta(sk)}{km^{ks}}
&\leq \zeta(Ks)\sum_{k\geq K}\frac1{k2^{k\operatorname{Re}(s)}}
\leq \zeta(Ks)\sum_{k\geq 1}\frac1{k2^{k\operatorname{Re}(s)}}\\
&=
-\zeta(Ks)\log\left(2^{-\operatorname{Re}(s)}\left(2^{\operatorname{Re}(s)}-1\right)\right)
,
\end{align*}
and note that in the argument of the logarithm in the last step, by assumption we have $2^{\operatorname{Re}(s)}-1>0$. 

Conversely, if $s\in\frac1{\N}$, then it is clear that this representation shows there is a pole of the extended partition zeta function, as one of the terms gives a multiple of $\zeta(1)$.
\end{proof}
\begin{proof}[Proof of Corollary \ref{ch5FourthCor}]
We utilize a variant of M\"obius inversion, reversing the order of summation in the double sum $\sum_{k\geq1}\sum_{d|k}\mu(d)f(nk)k^{-s}$; if
\[
g(n)=\sum_{k\geq1}\frac{f(kn)}{k^s}
,
\]
then 
\[
f(n)=\sum_{k\geq1}\frac{\mu(k)g(kn)}{k^s}
.
\]
Applying this inversion procedure to Corollary \ref{ch5SecondCor}, so that $g(n)=\log\zeta_{\mathcal P_{m\mathbb N}}(n)$ (taking $s=1$), and $f(n)=\zeta(n)/m^n$, we directly find that 
\[
\zeta(n)=m^n\sum_{\substack{k\geq1}}\frac{\mu(k)}k\log\left(\zeta_{\mathcal P_{m\mathbb N}}(nk)\right)
.
\]
Applying Theorem \ref{ch5mainthm} then gives the result.
\end{proof}

\begin{proof}[Proof of Theorem \ref{ch5RobertCor}]
By the comments following Theorem \ref{ch41.1}, for $\mathbb M\in\mathbb N$ we have 
\begin{equation}\label{ch5ProofEq}
\prod_{k\in\mathbb M}\left(1-\frac{z^s}{k^s}\right)^{-1}=1+z^s\sum_{k\in\mathbb M}\frac{1}{k^{s}\prod_{\substack{j\in\mathbb M\\ j\leq k}}\left(1-\frac{z^s}{j^s}\right)}
;
\end{equation}
thus
\[
\sum_{k\in\mathbb M}k^{-s}=\lim_{z\to 0^+}\frac{\prod_{k\in\mathbb M}\left(1-\frac{z^s}{k^s}\right)^{-1}-1}{z^s}
.\]
Taking $\mathbb M=\mathbb N, s=n\in\Z_{\geq 2}$, we apply L'Hospital's rule $n$ times to evaluate the limit on the right-hand side. The theorem then follows by noting, from Theorem \ref{ch5Armin11}, that in fact
\[
\prod_{k\in\mathbb N}\left(1-\frac{z^n}{k^n}\right)^{-1}=\prod_{j=0}^{n-1}\Gamma\left(1-ze(j/n)\right)
.\] 
\end{proof}

\begin{proof}[Proof of Corollary \ref{ch5RobertCor2}]
Picking up from the proof of Theorem \ref{ch5RobertCor} above, it follows also from Theorem \ref{ch41.1} that
\[
\prod_{k\in\mathbb M}\left(1-\frac{z^s}{k^s}\right)=1-z^s\sum_{k\in\mathbb M}\frac{\prod_{\substack{j\in\mathbb M\\ j<k}}\left(1-\frac{z^s}{j^s}\right)}{k^{s}}
.
\]
Subtracting this equation from \eqref{ch5ProofEq}, making the substitutions $\mathbb M=\mathbb N$, $s=n\geq 2$ as in the proof above, and using Theorem \ref{ch5Armin11}, gives the corollary.
\end{proof}

\subsection{Proof of Theorem \ref{ch5SecondTheorem} and its corollaries}
\begin{proof}[Proof of Theorem \ref{ch5SecondTheorem}]
Using a similar method as in Chapter 4 and a similar rewriting to that used in the proof of Theorem \ref{ch5mainthm}, we note that a short elementary computation shows
\[
\sum_{k\geq0}\frac{z^{mk}}{\pi^{mk}}\zeta_{\mathcal P}(\{m\}^k)=\prod_{k\geq1}\frac{1}{1-\frac{z^m}{\pi^mk^m}}=\prod_{k\geq0}\prod_{r=0}^{m-1}\frac{(k+1)^m}{\left(k+1-\frac{z}{\pi}e(r/m)\right)}
.
\]
Much as in the proof of Theorem \ref{ch5RobertCor}, using Theorem \ref{ch5Armin11}, we directly find that this is equal to 
\[
\prod_{r=0}^{m-1}\Gamma\left(1-\frac{z}{\pi}e(r/m)\right)
,
\]
which gives the first equality in the theorem. 
Applying Equation \eqref{ch5LogGammaSecondFormula} (formally we require $|z|<\pi$, but we are only interested in formal power series here anyway), we find immediately, using a very similar calculation to that in the proof of Corollary \ref{ch5SecondCor}, that
\begin{equation}\label{ch5PowSer1}
\begin{aligned}
\sum_{k\geq0}\left(\frac{z}{\pi}\right)^{mk}\zeta_{\mathcal P}(\{m\}^k)
&
=
\operatorname{exp}\left(\sum_{r=0}^{m-1}\sum_{j\geq2}\frac{\zeta(j)}j\left(\frac{z}{\pi}\right)^{mj}e(rj/m)\right)
\\
&
=
\operatorname{exp}\left(m\sum_{\substack{j\geq2\\ m|j}}\frac{\zeta(j)}j\left(\frac{z}{\pi}\right)^{j}\right)
,
\end{aligned}
\end{equation}
which is equivalent to the second equality in the theorem.
\end{proof}

\begin{proof}[Proof of Corollary \ref{ch5DeterminantCor}]
Replace $x$ with $z^m$ in Equation \ref{ch5FdB}, and set
\[
a_j=\frac{(j-1)!\zeta(mj)}{\pi^{mj}} 
\]
on the left-hand side (which becomes the right-hand side of \eqref{ch5PowSer1}). Then comparing the right side of \ref{ch5FdB} to the left side of \ref{ch5PowSer1}, we deduce that
\[
\zeta_{\mathcal P}(\{m\}^k)=\frac{\pi^{mk}}{k!}B_k(a_1,\dots,a_k)
.\]
To complete the proof, we substitute the determinant in \ref{ch5determinant} for $B_k(a_1,\dots,a_k)$ and rewrite the terms in the upper half of the resulting matrix as $\alpha_{i,j}$, as defined in the statement of the corollary.  
\end{proof}
\begin{proof}[Proof of Theorem \ref{ch5padicInterpThm}]
In analogy with the calculation of Theorem \ref{ch5SecondTheorem}, we find that
\begin{align*}
\sum_{k\geq0}z^{mk}\zeta_{\mathcal P_p}(\{m\}^k)&=\prod_{\substack{k\geq1\\ p\nmid k}}\frac{1}{1-\frac{z^m}{k^m}}=\frac{\prod_{k\geq0}\prod_{r=0}^{m-1}\frac{(k+1)^m}{\left(k+1-ze(r/m)\right)}}{\prod_{k\geq0}\prod_{r=0}^{m-1}\frac{(k+1)^m}{\left(k+1-\frac{z}{p}e(r/m)\right)}
}\\
&=\prod_{r=0}^{m-1}\frac{\Gamma\left(1-ze(r/m)\right)}{\Gamma\left(1-\frac{z}{p}e(r/m)\right)}
.
\end{align*}
As in the calculation of \eqref{ch5PowSer1}, this is equal to 
\[
\operatorname{exp}\left(\sum_{j\geq1}\frac{\zeta(mj)}{j}\left(z\right)^{mj}\left(1-1/p^{mj}\right)\right)
,
\]
so if we set
\[
\alpha_{i,j}^{(p)}(m)
:=
\zeta^*(m(j-i+1))\frac{(k-i)!}{(k-j)!}\  \  \  \text{where}\  \  \  
\zeta^*(s):=(1-p^{-s})\zeta(s),
\]
then we have 

\[
\zeta_{\mathcal P_p}(\{m\}^k)
=
\frac{1}{k!} \det
\begin{pmatrix} 
\alpha^{(p)}_{1,1}&\alpha^{(p)}_{1,2}&\alpha^{(p)}_{1,3}&\ldots&\alpha^{(p)}_{1,k}
\\
-1 & \alpha^{(p)}_{2,2} &\alpha^{(p)}_{2,3} &\ldots & \alpha^{(p)}_{2,k}
\\
0&-1& \alpha^{(p)}_{3,3}&\ldots & \alpha^{(p)}_{3,k}
\\
\vdots &\vdots&\vdots &\ddots &\vdots
\\
0&0&\ldots&-1&\alpha^{(p)}_{k,k} 
\end{pmatrix}
.
\]
We further define $\zeta_{\mathcal P_p}(\{m\}^k)$ for more general values in $\mathbb C$, such as $m\in-\mathbb N$ using the analytic continuation of $\zeta$ in each of the factors $\alpha_{i,j}^{(p)}(m)$.
Next we recall the Kummer congruences, which state that if $k_1,k_2$ are positive even integers not divisible by $(p-1)$ and $k_1\equiv k_2\pmod{p^{a+1}-p^a}$ for $a\in\mathbb N$ where $p>2$ is prime, then 
\[
\left(1-p^{k_1-1}\right)\frac{B_{k_1}}{k_1}\equiv\left(1-p^{k_2-1}\right)\frac{B_{k_2}}{k_2}\pmod{p^{a+1}}
.
\]
Let us take $S_{s_0}$ to be the set of natural numbers congruent to $s_0$ modulo $p-1$.
The Kummer congruences then imply that for any $s_0\not\equiv0\pmod{p-1}$, and for any $k_1,k_2\in S_{s_0}$ with $k_1\equiv k_2\pmod{p^a}$ and $k_1,k_2>1$, that
\[
\zeta^*(1-k_1)\equiv\zeta^*(1-k_2)\pmod{p^{a+1}}
.
\]
If we choose $m_1,m_2\in S_{s_0}$ with $m_1\equiv m_2\pmod{p^a}$, then the values $1-(1-m_1)(j-i+1)$, $1-(1-m_2)(j-i+1)$ are in $S_{1+(s_0-1)(j-i-1)}$ and are congruent modulo $p^a$, and as $p>k$ the additional factorial terms (inside and outside the determinant) are $p$-integral. Now in our determinant, $j-i+1$ ranges through $\{1,2,\ldots, k\}$, and we want to find an $s_0$ such that $1+(s_0-1)r\not\equiv0\pmod{p-1}$ for $r\in\{1,2,\ldots k\}$.
 If we take $s_0=2$, then the largest value of $1+(s_0-1)r$ is $k+1$, which is by assumption less than $p-1$, and hence not divisible by it. Hence, in our case, $s_0=2$ suffices.
Thus, if $m_1,m_2\in S_{2}$ with $m_1\equiv m_2\pmod{p^a}$, then 
\[
\zeta_{\mathcal P_p}(\{1-m_1\}^k)\equiv\zeta_{\mathcal P_p}(\{1-m_2\}^k)\pmod{p^{a+1}}
.
\]
This shows that our zeta function is uniformly continuous on $S_2$ in the $p$-adic topology. As this set is dense in $\Z_p$, we have shown the function extends in the $m$-aspect to $\Z_p$.

\end{proof}

\subsection{Proofs of results concerning multiple zeta values}
\begin{proof}[Proof of Proposition \ref{ch5DecouplingCorollary}]
Recall from \eqref{ch5mzvpzv2} 
that we need to study the sum
\[
\sum_{n_1\geq n_2 \geq \ldots \geq n_k\geq1}\frac1{(n_1n_2\ldots n_k)^m}
.
\]
The proof is essentially combinatorial accounting, keeping track of the number of ways to split up a sum
\[
\sum_{n_1\geq n_2 \geq \ldots \geq n_k\geq1}
\]
over all all $k$-tuples of natural numbers into a chain of equalities and strict inequalities. Suppose that we have 
\[
n_{1}\geq n_{2}\geq\ldots\geq n_{k}\geq1
.
\]
Then if any of these inequalities is an equality, say $n_{j}=n_{j+1}$, in the contribution to the sum 
\[
\sum_{n_1\geq n_2\geq \ldots\geq n_k\geq 1}(n_1\ldots n_k)^{-m}
,
\]
the terms $n_{j}$ and $n_{j+1}$ ``double up''. That is, we can delete the $n_{j+1}$ and replace the $n_{j}^{-m}$ in the sum with a $n_{j}^{-2m}$. Thus, the reader will find that our goal is to keep track of different orderings of $>$ and $=$, taking symmetries into account. The possible chains of $=$ and $>$ are encoded by the set of compositions of size $k$, by associating to the composition $(a_1,\ldots,a_j)$ the chain of inequalities
\[
n_{1}=\ldots=n_{a_1}>n_{a_1+1}=\ldots=n_{a_1+a_2}>n_{a_2+1}>\ldots >n_{k}
.
\]
That is, the number $a_1$ determines the number of initial terms on the right which are equal before the first inequality, $a_2$ counts the number of equalities in the next block of inequalities, and so on. It is clear that the sum corresponding to the each composition then contributes the desired amount to the partition zeta value in the corollary.
\end{proof}

\begin{proof}[Proof of Corollary \ref{ch5ParallelMZVLowerOrders}]
In Proposition \ref{ch5DecouplingCorollary}, comparison with Corollary \ref{ch5DeterminantCor} shows that we have a linear relation among MZVs and products of zeta values. Observe that in $\zeta_{\mathcal P}(\{m\}^k)$, the only composition of length $k$ is $(1,1,\ldots,1)$, which contributes $k! \zeta(\{m\}^k)$ to the right-hand side of Proposition \ref{ch5DecouplingCorollary}, and that the rest of the compositions are of lower length, hence giving MZVs of smaller length; the corollary follows immediately.
\end{proof}

\begin{proof}[Proof of Proposition \ref{ch5ParallelValuesExp}]
Consider the multiple zeta value $\zeta(\{n\}^k)$ of length $k$. Then we directly compute
\[
\sum_{k\geq0}(-1)^k\zeta(\{n\}^k)z^{nk}
=
\prod_{m\geq1}
\left(
1-\left(\frac zm\right)^n
\right)
=
\prod_{m\geq0}\prod_{r=0}^{n-1}\frac{(m+1-ze(r/n))}{(m+1)^n}
.
\]
By Theorem \ref{ch5Armin11}, this equals 
\[
\prod_{r=0}^{n-1}\Gamma(1-ze(r/n))^{-1}
.
\] 
Using precisely the same computation as was made in the proof of Theorem \ref{ch5SecondTheorem}, we find that this is equal to
\[
\operatorname{exp}\left(-n\sum_{\substack{j\geq2\\ n|j}}\frac{\zeta(j)}jz^{j}\right)
.
\]
Hence, we have that
\[
\zeta(\{n\}^k)
=(-1)^k
\left[z^{nk}\right]
\operatorname{exp}
\left(-
\sum_{j\geq1}
\frac{\zeta(nj)}jz^{nj}
\right)
.
\]
\end{proof}
\begin{proof}[Proof of Corollary \ref{ch5DeterminantCor2}]
Here we proceed exactly as in the proof of Corollary \ref{ch5DeterminantCor}, except we make the simpler substitution
\[
a_k=(k-1)!\zeta(nk) 
\] 
into Equation \ref{ch5FdB}, and compare with Proposition \ref{ch5ParallelValuesExp}. In the final step, we replace the terms in the upper half of the matrix with $\beta_{i,j}$ as defined in the statement of the corollary.  
\end{proof}

\section{Partition Dirichlet series}

We have presented samples of a few varieties of flora one finds at the fertile intersection of combinatorics and analysis. What unifies all of these is the perspective that they represent instances of partition zeta functions, with proofs that fit naturally into the Eulerian theory we propound in this work. 

We close this chapter by noting a general class of partition-theoretic analogs of classical Dirichlet series having the form
\[
\mathcal D_{\mathcal P'}(f,s):=\sum_{\lambda\in\mathcal P'}f(\lambda)n_{\lambda}^{-s},
\]
where $\mathcal P'$ is a proper subset of $\mathcal P$ and $f \colon \mathcal P' \rightarrow \C$. Of course, partition zeta functions arise from the specialization $f\equiv 1$, just as in the classical case. 

Taking $\mathcal P'=\mathcal P_{\mathbb M}$ as defined previously, then if $f:=f(n_{\lambda})$ is completely multiplicative with appropriate growth conditions, it follows from Theorem \ref{ch41.1} that $\mathcal D_{\mathcal \mathcal P_{\mathbb M}}(f,s)$ has the Euler product %\eqref{ch5DirichletProduct}
\begin{equation}\label{ch5DirichletProduct2}
\mathcal D_{\mathcal \mathcal P_{\mathbb M}}(f,s)=\prod_{j\in \mathbb M}\left(1-\frac{f(j)}{j^s}\right)^{-1}\  \left(\mathrm{Re}(s)>1\right),
\end{equation}
and nearly the entire theory of partition zeta functions developed in the previous chapter % noted here (and developed in \cite{Robert}) 
extends to these series as well. Moreover, incorporating partition-arithmetic functions from Chapter 3, by very much the same steps as proofs of the classical cases, we have familiar-looking formulas such as these. We take $\text{Re}(s)$ so the series converge absolutely.

\begin{theorem}\label{ch5muphidirichlet} Generalizing the classical cases, we have the following identities:
\begin{equation*}
  \sum_{\lambda\in\mathcal P_{\mathbb X}}\mu_{\mathcal P}(\lambda) n_{\lambda}^{-s}=  \frac{1}{\zeta_{\mathcal P_{\mathbb X}}(s)},\  \  \  \  \  \  \sum_{\lambda\in\mathcal P_{\mathbb X}}\varphi_{\mathcal P}(\lambda) n_{\lambda}^{-s}=\frac{\zeta_{\mathcal P_{\mathbb X}}(s-1)}{\zeta_{\mathcal P_{\mathbb X}}(s)}.
\end{equation*}\end{theorem}

For $f,g\colon \mathcal P \to \mathbb C$, let us define a partition analog of Dirichlet convolution\footnote{An analogy which was suggested to the author by Olivia Beckwith}, viz. 
\begin{equation}
(f*g)(\lambda):=\sum_{\delta|\lambda}f(\delta)g(\lambda/\delta).\end{equation}
Then the partition Cauchy product in Proposition \ref{ch3cauchyprod} yields another familiar relation.
\begin{theorem} We have
\begin{equation}
\left(\sum_{\lambda\in\mathcal P}f(\lambda)n_{\lambda}^{-s} \right) \left(\sum_{\lambda\in\mathcal P}g(\lambda)n_{\lambda}^{-s}   \right) =\sum_{\lambda\in\mathcal P}(f*g)(\lambda)n_{\lambda}^{-s}  .
\end{equation}
\end{theorem}
\  
\

\begin{remark}
{See Appendix \ref{app:D} for further notes on Chapter 5.} 
\end{remark}

%
%\section*{Acknowledgements}
%The authors thank Armin Straub for useful discussion on the history of Theorem \ref{ch5Armin11}. 
%\section{Zeta polynomials}
%
%We now turn our attention toward a different layer of connections between partitions and zeta functions, via the theory of modular forms. Although Euler's generating function for $p(n)$ is essentially modular, and Euler also anticipated the study of $L$-functions that are intimately tied to modular forms, the true depth of such observations did not come into view until further work on complex analysis was carried out in the nineteenth century. 
%
%It turns out that all modular forms are related to partitions in a very direct way. Here we recall the case of this connection for modular forms for the full modular group $\textrm{SL}_2(\Z)$. We then use these modular forms to define the second class of functions that are the topic of this paper, the {\it zeta polynomials} associated to modular forms.
 	\clearpage%
\chapter{Partition-theoretic formulas for arithmetic densities}{\bf Adapted from \cite{OSW}, a joint work with Ken Ono and Ian Wagner}
%
%\begin{abstract}
%If $\gcd(r,t)=1$, then  a theorem of Alladi offers the M\"obius sum identity
%$$-\sum_{\substack{ n \geq 2 \\ p_{\rm{min}}(n) \equiv r \pmod{t}}} \mu(n)n^{-1}= \frac{1}{\varphi(t)}.
%$$ 
%Here $p_{\rm{min}}(n)$ is the smallest prime divisor of $n$.  The right-hand side represents the proportion of primes in a fixed arithmetic progression modulo $t$.   Locus generalized this to Chebotarev densities for Galois extensions.  Answering a question of Alladi, we obtain analogs of these results to arithmetic densities of subsets of positive integers using $q$-series and integer partitions.  For suitable subsets $\mathbb S$ of the positive integers with density $d_{\mathbb S}$, we prove that 
%\[- \lim_{q \to 1} \sum_{\substack{ \lambda \in \mathcal{P} \\ \rm{sm}(\lambda) \in \mathbb S}} \mu_{\mathcal{P}} (\lambda)q^{\vert \lambda \vert} =  d_{\mathbb S},\]
%where the sum is taken over integer partitions $\lambda$, $\mu_{\mathcal{P}}(\lambda)$ is a partition-theoretic M\"obius function,   
% $\vert \lambda \vert$ is the size of partition $\lambda$, and $\rm{sm}(\lambda)$ is the smallest part of $\lambda$.   In particular, we obtain partition-theoretic formulas for even powers  of $\pi$ when considering power-free integers.
%\end{abstract}
%
%
%
%\maketitle
%
%

\section{Introduction and statement of results}

Consider again the classical M\"{o}bius function $\mu(n)$, and let us rewrite the well-known fact %\textit{Merten's theorem}.  We have
$\sum_{n=1}^{\infty} \mu(n)/n= 0$
in the form
\begin{equation}\label{ch6classicalfact}
-\sum_{n=2}^{\infty} \frac{\mu(n)}{n} = 1.
\end{equation}
For notational convenience define $\mu^{*}(n) := - \mu(n)$. Now, (\ref{ch6classicalfact}) above can be interpreted as the statement that one-hundred percent 
%$100 \%$ 
of integers $n \geq 2
$ are divisible by at least one prime. This idea was used by Alladi  \cite{A} to prove that if $\gcd(r, t) = 1$, then
\begin{equation}\label{ch6AlladiSum}
\sum_{\substack{ n \geq 2 \\ p_{\rm{min}}(n) \equiv r \pmod{t}}} \frac{\mu^{*}(n)}{n} = \frac{1}{\varphi(t)}. 
\end{equation}
Here $\varphi(t)$ is Euler's phi function, and $p_{\rm{min}}(n)$ is the smallest prime factor of $n$.  

Alladi has asked\footnote{K. Alladi, ``A duality between the largest and smallest prime factors via the Moebius function and arithmetical consequences'', Emory University Number Theory Seminar, February 28, 2017.} for a partition-theoretic generalization of this result.  We answer his question by obtaining an analog of a generalization that was recently obtained by Locus \cite{L}.  Locus began by interpreting Alladi's theorem as a device for computing densities of primes in arithmetic progressions.  She generalized this idea, and proved analogous formulas for the Chebotarev densities of Frobenius elements in unions of conjugacy classes of Galois extensions. 

We recall Locus's result.  Let $S$ be a subset of primes with Dirichlet density, and define
\begin{equation}\label{ch6maddie}
\mathfrak{F}_{S}(s) := \sum_{\substack{n \geq 2 \\ p_{\rm{min}}(n) \in S}} \frac{\mu^{*}(n)}{n^{s}}.  
\end{equation}
Suppose $K$ is a finite Galois extension of $\Q$ and $p$ is an unramified prime in $K$.  Define
\[\left[ \frac{K/\Q}{p} \right] := \left\{ \left[ \frac{K/\Q}{\mathfrak{p}} \right] \colon \mathfrak{p} \subseteq \mathcal{O}_{K} \ \rm{is \  a \  prime \  ideal \  above} \ \it{p} \right\},\]
where $\left[ \frac{K/\Q}{\mathfrak{p}} \right]$ is the Artin symbol (for example, see Chapter 8 of \cite{Artin_ref_needed}), and $\mathcal{O}_{K}$ is the ring of integers of $K$.  It is well known that $\left[ \frac{K/\Q}{p} \right]$ is a conjugacy class $C$ in $G= \rm{Gal}( \emph{K}/\Q)$.  If we let
\begin{equation}\label{ch6OnePointThree}
S_{C} := \left\{ p \ \rm{prime}: \left[ \frac{\emph{K}/\Q}{\emph{p}} \right] = C \right\},
\end{equation}
then Locus proved (see Theorem $1$ of \cite{L}) that 
$$
\mathfrak{F}_{S_{C}}(1) = \frac{ \# C}{\# G}. 
$$

\begin{remark}
Alladi's formula (\ref{ch6AlladiSum}) is the cyclotomic case of Locus's Theorem.\end{remark}

We now turn to Alladi's question concerning a partition-theoretic analog.   %A \textit{partition} is a finite non-increasing sequence of positive integers, say $\lambda = (\lambda_{1}, \lambda_{2},..., \lambda_{\ell(\lambda)})$, where $\ell(\lambda)$ denotes the length of $\lambda$, i.e. the number of parts.  The size of $\lambda$ is $\vert \lambda \vert := \lambda_{1} + \lambda_{2} + \cdot \cdot \cdot + \lambda_{\ell(\lambda)}$, i.e., the number being partitioned.  Furthermore, w
Let $\rm{sm}(\lambda) := \lambda_{\ell(\lambda)}$ denote the smallest part of $\lambda$ (resp. $\rm{lg}(\lambda) := \lambda_{1}$ the largest part of $\lambda$). Also, recall the partition-theoretic M\"{o}bius function $\mu_{\mathcal P}$ from previous chapters.  Notice that $\mu_{\mathcal{P}}(\lambda) = 0$ if $\lambda$ has any repeated parts, which is analogous to the vanishing of $\mu(n)$ for integers $n$ which are not square-free.  In particular, the parts in partition $\lambda$ play the role of prime divisors of $n$ in this analogy, as in Chapter 3.  We define $\mu_{\mathcal{P}}^{*}(\lambda):= - \mu_{\mathcal{P}}(\lambda)$ as in Locus's theorem, for aesthetic reasons.  

The table below offers a description of the objects which are related with respect to this analogy.  However, it is worthwhile to first compare the generating functions for $\mu(n)$ and $\mu_{\mathcal{P}}(\lambda)$.  Using the Euler product for the Riemann zeta function, it is well known that the Dirichlet generating function for $\mu(n)$ is
\begin{equation}\label{ch6OnePointFive}
\frac{1}{\zeta(s)} = \prod_{p \ \rm{prime}} \left( 1 - \frac{1}{p^s} \right) = \sum_{m=1}^{\infty} \mu(m)m^{-s}. 
\end{equation}
As we noted in Chapter 3, the generating function for $\mu_{\mathcal{P}}(\lambda)$ is  
$$
(q;q)_{\infty} = \prod_{n=1}^{\infty} (1 - q^n) = \sum_{\lambda} \mu_{\mathcal{P}}(\lambda) q^{\vert \lambda \vert}. 
$$
By comparing the generating functions for $\mu(n)$ and $\mu_{\cal{P}}(\lambda)$,  we see that prime factors and integer parts
of partitions are natural analogs of each other. The following table offers the identifications that make up this analogy. 

\vspace{0.5cm}

\begin{center}
\begin{tabular}{ | c | c | }
\hline
Natural number $m$ & Partition $\lambda$\\ \hline
Prime factors of $m$ & Parts of $\lambda$ \\ \hline
Square-free integers & Partitions into distinct parts \\ \hline
$\mu(m)$ & $\mu_{\mathcal{P}}(\lambda)$ \\ \hline
$p_{\rm{min}}(m)$ & $\rm{sm}(\lambda)$ \\ \hline
$p_{\rm{max}}(m)$ & $\rm{lg}(\lambda)$ \\ \hline
$m^{-s}$ & $q^{\vert \lambda \vert}$ \\ \hline
${\zeta(s)^{-1}}$ & ${(q;q)_{\infty}}$ \\  \hline
$s=1$ & $q\rightarrow 1$\\
\hline
\end{tabular}
\end{center}
\vspace{0.5cm}

%
%\begin{remark}There are further  analogies between multiplicative number theory and the theory of integer partitions. %Of course, natural numbers themselves can be viewed as partitions of length one, and also as partitions into all $1$'s (a key insight of  elementary arithmetic). Less trivially, 
%For instance, in  \cite{AlladiErdos} Alladi and Erd\H{o}s exploited a bijection between prime factorizations of integers and partitions into prime parts, to study an interesting arithmetic function; recently, the first two authors have shown that many theorems in multiplicative number theory are special cases of much more general partition-theoretic phenomena \cite{OnoRolenS,Schneider_arithmetic}.
%\end{remark}
%

Suppose that
$\mathbb S$ is a subset of the positive integers with arithmetic density
\[\lim_{X \to \infty} \frac{\# \{ n \in \mathbb S : n \leq X \}}{X} = d_{\mathbb S}.\]
The partition-theoretic counterpart to (\ref{ch6maddie}) is 
\begin{equation}\label{ch6OnePointNine}
F_{\mathbb{S}}(q) := \sum_{\substack{ \lambda \in \mathcal{P} \\ \rm{sm}(\lambda) \in \mathbb{S}}} \mu_{\mathcal{P}}^{*}(\lambda) q^{\vert \lambda \vert}. 
\end{equation}
To state our results, we define 
\begin{equation}\label{ch6OnePointTen}
\mathbb S_{r,t} := \{n \in \Z^{+} : n \equiv r \pmod{t} \}. 
\end{equation}
These sets are simply the positive integers in an arithmetic progression $r$ modulo $t$. % \pmod t$.

Our first result concerns the case where $t=2$. Obviously, the arithmetic densities of $\mathbb S_{1,2}$ and $\mathbb S_{2,2}$ are both ${1}/{2}$. The theorem below offers a formula illustrating these densities and also offers curious lacunary $q$-series
identities.
%
%\begin{comment}
%Our first result is really a special case of our general theorem.  We emphasize it, however, because of a nice identity that stands out as being reminiscent of Euler's Pentagonal Number Theorem (see \cite{Berndt_spirit}) which says that
%\[\prod_{n=1}^{\infty}(1-q^n) = \sum_{n= -\infty}^{\infty} (-1)^{n} q^{\frac{n(3n-1)}{2}}.\]
%\end{comment}
%
%\begin{comment}
%\begin{theorem} \label{ch612}
%Assuming the notation above, the following are true:
%\begin{enumerate}
%\item We have the pair of $q$-identities 
%\[\sum_{\substack{ \lambda \in \mathcal{P} \\ \rm{sm}( \lambda) \ \rm{odd}}} \mu^{*}(\lambda)q^{\vert \lambda \vert} = -\sum_{n=1}^{\infty} (-1)^{n} q^{n^2},\]
%\[\sum_{\substack{ \lambda \in \mathcal{P} \\ \rm{sm}( \lambda) \ \rm{even}}} \mu^{*}(\lambda)q^{\vert \lambda \vert} = 1+ \sum_{n=1}^{\infty} (-1)^{n} q^{n^2} - \sum_{m= -\infty}^{\infty} (-1)^{m}q^{\frac{m(3m-1)}{2}}.\]
%\item We have that
%\[\lim_{q \to 0^{+}} \sum_{\substack{ \lambda \in \mathcal{P} \\ \rm{sm}( \lambda) \ \rm{odd}}} \mu^{*}(\lambda)q^{\vert \lambda \vert}= \lim_{u \to 0^{+}} \sum_{\substack{ \lambda \in \mathcal{P} \\ \rm{sm}( \lambda) \ \rm{even}}} \mu^{*}(\lambda)q^{\vert \lambda \vert} = \frac{1}{2}.\]
%\end{enumerate}
%\end{theorem}
%\end{comment}

\begin{theorem} \label{ch612}
Assume the notation above.
\begin{enumerate}[(1)]
\item The following $q$-series identities are true: 
\[F_{\mathbb S_{1,2}}(q) = \sum_{n=1}^{\infty} (-1)^{n+1} q^{n^2},\]
\[F_{\mathbb S_{2,2}}(q) = 1+ \sum_{n=1}^{\infty} (-1)^{n} q^{n^2} - \sum_{m= -\infty}^{\infty} (-1)^{m}q^{\frac{m(3m-1)}{2}}.\]
\item We have that
\[\lim_{q \to 1} F_{\mathbb S_{1,2}}(q) = \lim_{q \to 1} F_{\mathbb S_{2,2}}(q) = \frac{1}{2}.\]
\end{enumerate}
\end{theorem}

\begin{remark}
The limits in Theorem \ref{ch612} are understood as $q$ tends to $1$ from within the unit disk.
\end{remark}

\begin{example}

For complex $z$ in the upper-half of the complex plane, let $q(z) := \rm{exp} \left( - \frac{2 \pi i}{\emph{z}} \right)$.  Therefore, if $z \to 1$ in the upper-half plane, then $q(z) \to 1$ in the unit disk.  The table below displays a set of such $z$ beginning to approach $1$ and the corresponding values of $F_{\mathbb S_{1,2}}(q(z))$.

\begin{center}
\begin{tabular}{ | c | c | }
\hline
$z$ & $F_{\mathbb S_{1,2}}(q(z))$\\ \hline
$1+.10i$ & $0.458233...$ \\ \hline
$1+.09i$ & $0.471737...$ \\ \hline
$1 + .08i$ & $0.482784...$ \\ \hline
$1+.07i$ & $0.491003...$ \\ \hline
$1+.06i$ & $0.496296...$ \\ \hline
$1+.05i$ & $0.498998...$ \\ \hline
$1+.04i$ & $0.499919...$ \\ \hline
$1+.03i$ & $0.500048...$ \\ \hline
$1+.02i$ & $0.500024...$ \\ \hline
$1+.01i$ & $0.500006...$ \\ 
\hline
\end{tabular}
\end{center}
\vspace{0.5cm}
\end{example}

Theorem \ref{ch612} (1) offers an immediate combinatorial interpretation.  Let $D^{+}_{\rm{even}}(n)$ denote the number of partitions of $n$ into an even number of distinct parts with smallest part even, and let $D^{+}_{\rm{odd}}(n)$ denote the number of partitions of $n$ into an even number of distinct parts with smallest part odd.  Similarly, let $D^{-}_{\rm{even}}(n)$ denote the number of partitions of $n$ into an odd number of distinct parts with smallest part even, and let $D^{-}_{\rm{odd}}(n)$ denote the number of partitions of $n$ into an odd number of distinct parts with smallest part odd. 
To make this precise, for integers $k$ let $\omega(k):=\frac{k(3k-1)}{2}$ 
be the index $k$ {\it pentagonal number}.

\begin{corollary} \label{ch6cor}
Assume the notation above. We have the following bijections:
\begin{enumerate}[(1)]
\item For partitions into distinct parts whose smallest part is odd, we have
\[D_{\rm{odd}}^{+}(n) - D_{\rm{odd}}^{-}(n) = \begin{cases} 0 & {\rm{if}} \  n \ \rm{is \ not \ a \ square} \\
1 & {\rm{if}} \ n \ \rm{is \ an \ even \ square} \\
-1 & {\rm{if}} \ n \ \rm{is \ an \ odd \ square}. \end{cases}\]
\item For partitions into distinct parts whose smallest part is even, we have
\[D_{\rm{even}}^{+}(n) - D_{\rm{even}}^{-}(n) = \begin{cases}
-1 & {\rm{if}} \ n \ \rm{is \ an \ even \ square \ and \ not \ a \ pentagonal \ number} \\
1 & {\rm{if}} \ n \ \rm{is \ an \ odd \ square \ and \ not \ a \ pentagonal \ number} \\
1 & {\rm{if}} \ n \ \rm{is \ an \ even\ index \ pentagonal \ number \ and \ not \ a \ square} \\
-1 & {\rm{if}} \ n \ \rm{is \ an \ odd \ index \ pentagonal \ number \ and \ not \ a \ square} \\
0 & {\rm{otherwise.}} \\
 \end{cases}\]
\end{enumerate}
\end{corollary}

\begin{question}
It would be interesting to obtain a combinatorial proof of Corollary \ref{ch6cor} by refining Franklin's proof of Euler's Pentagonal Number Theorem (see pages 10-11 of \cite{Andrews}).
\end{question}

Our proof of Theorem \ref{ch612} makes use of the $q$-Binomial Theorem and some well-known $q$-series identities.  It is natural to ask whether such a relation holds for general sets $\mathbb S_{r,t}$.  The following theorem shows that Theorem \ref{ch612} is indeed a special case of a more general phenomenon.  

\begin{theorem} \label{ch6rt}
If  $0\leq r<t$ are integers and $\gcd(m,t)=1$, then we have that
$$\lim_{q \to \zeta} F_{\mathbb S_{r,t}}(q) = \frac{1}{t},
$$
where $\zeta$ is a primitive $m$th root of unity.
\end{theorem}

\begin{remark}
The limits in Theorem \ref{ch6rt} are understood as $q$ tends to $\zeta$ from within the unit disk.
\end{remark}

Obviously, these results hold for any set $\mathbb S$ of positive integers that is a finite union of arithmetic progressions.  It turns out that this theorem can also be used to compute arithmetic densities of more complicated sets arising systematically from a careful study of arithmetic progressions.  We focus on the sets of positive integers $\mathbb S^{(k)}_{\rm{fr}}$ which are $k$th power-free.  In particular, we have that \[\mathbb S^{(2)}_{\rm{fr}} = \{ 1, 2, 3, 5, 6, 7, 10, 11, 13, \dots \}.\]
It is well known that the arithmetic densities of these sets are given by
$$
\lim_{X\rightarrow +\infty}\frac{\# \left \{1\leq n\leq X \ : \ n\in \mathbb S_{\rm{fr}}^{(k)}\right \}}{X}=
\prod_{p \ {\text {\rm prime}}} \left (1-\frac{1}{p^k}\right)=\frac{1}{\zeta(k)}.
$$
To obtain partition-theoretic formulas for these densities, we first compute a partition-theoretic formula for the density of 
\begin{equation}
\mathbb S^{(k)}_{\rm{fr}}(N):= \{n \geq 1 \ : \ p^{k} \nmid n \ \rm{for \ every} \ \emph{p} \leq \emph{N} \}. 
\end{equation}

\begin{theorem} \label{ch6free}
If $k, N \geq 2$ are integers, then we have that
\[ \lim_{q \to 1} F_{\mathbb S^{(k)}_{\rm{fr}}(N)}(q) = \prod_{p \leq N \ \rm{prime}} \left(1 - \frac{1}{p^k} \right).\]
\end{theorem}

The constants in Theorem \ref{ch6free} are the arithmetic densities of positive integers that are not divisible by the $k$th power of any prime $p \leq N$, namely $\mathbb S^{(k)}_{\rm{fr}}(N)$.
Theorem \ref{ch6free} can be used to calculate the arithmetic density of $\mathbb S^{(k)}_{\rm{fr}}$ by letting $N\rightarrow +\infty$.

\begin{corollary} \label{ch6free2}
If $k \geq 2$, then 
\[\lim_{q \to 1} F_{\mathbb S^{(k)}_{\rm{fr}}}(q) = \frac{1}{\zeta(k)}.\]
Furthermore, if $k \geq 2$ is even, then 
\[\lim_{q \to 1} F_{\mathbb S^{(k)}_{\rm{fr}}}(q) =(-1)^{\frac{k}{2}+1} \frac{k!}{B_{k}\cdot 2^{k-1}}\cdot \frac{1}{\pi^k},\]
where $B_{k}$ is the $k$th Bernoulli number.
\end{corollary}

This chapter is organized as follows.  In Section $6.2.1$ we discuss the $q$-Binomial Theorem, which will be an essential tool for our proofs, as well as a duality principle for partitions related to ideas of Alladi.  In Section $6.2.2$ we will use the $q$-Binomial Theorem to prove results related to Theorem \ref{ch6rt}.  Section 6.3 will contain the proofs of all of the theorems, and Section 6.4 will contain some nice examples. 

\section{The $q$-Binomial Theorem and its consequences}

In this section we recall elementary $q$-series identities, and we offer convenient reformulations for
the functions $F_{\mathbb{S}}(q)$.

\subsection{Nuts and bolts}
Let us recall the classical $q$-Binomial Theorem (see \cite{Andrews} for proof). 

\begin{lemma}\label{ch6lemma1}
For $a,z\in \mathbb C, |q|<1$ we have the identity
\begin{equation*}
\frac{(az;q)_{\infty}}{(z;q)_{\infty}}=\sum_{n=0}^{\infty}\frac{(a;q)_n}{(q;q)_n}z^n.
\end{equation*}
\end{lemma}

We recall the following well-known $q$-product identity (for proof, see page 6 of \cite{Fine}).

\begin{lemma}\label{ch6lemma2}
Using the above notations, we have that
$$
\frac{(q;q)^2_{\infty}}{(q^2;q^2)_{\infty}}= 1 + 2 \sum_{n=1}^{\infty} (-1)^{n} q^{n^2}.
$$
\end{lemma}

The following elementary lemma plays a crucial role in this paper.

\begin{lemma} \label{ch6lemma3}
If $\mathbb S$ is a subset of the positive integers, then the following are true:
\begin{displaymath}
F_{\mathbb S}(q)= \sum_{n \in \mathbb S} q^{n} \prod_{m=1}^{\infty} (1-q^{m+n})\\%%%%%NEED TO UPDATE \mathbb S TO SOMETHING CORRECT!!!!
= (q;q)_{\infty}\cdot  \sum_{\substack{ \lambda \in \mathcal{P} \\ \rm{lg}(\lambda) \in \mathbb S}} q^{\vert \lambda \vert}.
\end{displaymath}
\end{lemma}
\begin{rmk}
Lemma \ref{ch6lemma3} may be viewed as a partition-theoretic case of Alladi's duality principle, which was originally stated in \cite{A} as a relation between functions on smallest versus largest prime divisors of integers, and was given in full partition-theoretic generality by Alladi in a lecture at Emory University\footnote{See previous footnote in this chapter}, although we don't use that formula here. 
\end{rmk}

\begin{proof}
By inspection, we see that
$$
F_{\mathbb{S}}(q)=\sum_{\substack{ \lambda\in \mathcal P\\
{\text {\rm sm}}(\lambda)\in \mathbb{S}}} \mu^*_{\mathcal{P}}(\lambda) q^{|\lambda|}=
\sum_{n\in \mathbb{S}}q^n\prod_{m=1}^{\infty}(1-q^{m+n}).
$$
By factoring out $(q;q)_{\infty}$ from each summand, we find that
\begin{align*}
F_{\mathbb S}(q) &= \sum_{n \in \mathbb S} q^{n} \prod_{m=1}^{\infty} (1-q^{m+n}) 
= (q;q)_{\infty} \cdot \sum_{n \in \mathbb S} \frac{q^n}{(q;q)_{n}} \\
&= (q;q)_{\infty} \cdot \sum_{\substack{ \lambda \in \mathcal{P} \\ \rm{lg}(\lambda) \in \mathbb S}} q^{\vert \lambda \vert}.
\end{align*}
\end{proof}

\subsection{Case of $F_{\mathbb S_{r,t}}(q)$}

Here we specialize Lemma~\ref{ch6lemma3} to the sets  $\mathbb{S}_{r,t}$. The next lemma describes the $q$-series
$F_{\mathbb{S}_{r,t}}(q)$ in terms of a finite sum of quotients of infinite products. To prove this lemma we make
use of the $q$-Binomial Theorem.

\begin{lemma} \label{ch6finite}
If $t$ is a positive integer and $\zeta_t:=e^{2\pi i/t}$, then
\[F_{\mathbb S_{r,t}}(q) = (q;q)_{\infty} \cdot \frac{1}{t} \left[ \sum_{m=1}^{t} \frac{\zeta_{t}^{-mr}}{(\zeta_{t}^{m} q;q)_{\infty}} \right].\]
\end{lemma}

\begin{proof}
From Lemma \ref{ch6lemma3}, we have that
\[F_{\mathbb S_{r,t}}(q) = (q;q)_{\infty}\cdot \sum_{n=0}^{\infty} \frac{q^{tn+r}}{(q;q)_{tn+r}}.\]
By applying the $q$-Binomial Theorem (see Lemma \ref{ch6lemma1}) with $a=0$ and $z=\zeta_t^m q$,  we find that
\begin{align*}
\frac{1}{t} \left[ \sum_{m=1}^{t} \frac{\zeta_{t}^{-mr}}{(\zeta_{t}^{m} q;q)_{\infty}} \right] = \frac{1}{t} \left[ \sum_{m=1}^{t} \sum_{n =0}^{\infty} \frac{ \zeta_{t}^{m(n-r)}q^{n}}{(q;q)_{n}} \right]. \end{align*}
Due to the orthogonality of roots of unity we have 
$$
\sum_{m=1}^{t} \zeta_{t}^{m(n-r)} = \begin{cases} t & \text{if} \ n \equiv r \pmod{t} \\
0 & \text{otherwise}. \end{cases}
$$
Hence, this sum allows us to sieve on the sum in $n$ leaving only those summands with $n\equiv r\pmod t$, namely
the series
\[\sum_{n =0}^{\infty} \frac{q^{tn+r}}{(q;q)_{tn+r}}.\]
Therefore, it follows that
\[F_{\mathbb S_{r,t}}(q) = (q;q)_{\infty} \cdot \frac{1}{t} \left[ \sum_{m=1}^{t} \sum_{n =0}^{\infty} \frac{ \zeta_{t}^{m(n-r)}q^{n}}{(q;q)_{n}} \right].\]
\end{proof}

\begin{lemma} \label{ch6limit}
If $a$ and $m$ are positive integers and $\zeta$   is a primitive $m$th root of unity, then
\[\lim_{q \to 1} \frac{(q;q)_{\infty}}{(\zeta^{a}q;q)_{\infty}} = \begin{cases} 1 & {\rm{if}} \ m \mid a \\
0 & \rm{otherwise}. \end{cases}\] 
\end{lemma}

\begin{proof}
%{\bf CLEAN UP...}
Since $(aq;q)_{\infty}^{\pm 1}$ is an analytic function of $q$ inside the unit disk (i.e., of $q:=e^{2\pi i z}$ with $z$ in the upper half-plane) when $|a|\leq 1$, the quotient on the left-hand side of Lemma \ref{ch6limit} is well-defined as a function of $q$ (resp. of $z$), and we can take limits from inside the unit disk.  When $m \mid a$, the $q$-Pochhammer symbols cancel and the quotient is identically $1$.  When $m \nmid a$, then $(q;q)_{\infty}$ clearly vanishes as $q\to 1$ while $(\zeta^{a}q;q)_{\infty}$ is non-zero; thus the quotient is zero.
\end{proof} 

\section{Proofs of these results}

\subsection{Proof of Theorem \ref{ch612}}
Here we prove Theorem \ref{ch612} (1); we defer the proof of the second
part until the next section because it is a special case of Theorem \ref{ch6rt}.
\begin{proof} [Proof of Theorem~\ref{ch612} (1)]
By Lemma \ref{ch6finite} we have 
\begin{align*}
F_{\mathbb S_{1,2}}(q) &= (q;q)_{\infty} \cdot \frac{1}{2} \left[ \frac{1}{(q;q)_{\infty}} - \frac{1}{(-q;q)_{\infty}} \right] \\
&= \frac{1}{2} \left[ 1 - \frac{(q;q)_{\infty}}{(-q;q)_{\infty}} \right]\\
&=\frac{1}{2}\left [1-\frac{(q;q)_{\infty}^2}{(q^2;q^2)_{\infty}}\right ].
\end{align*}
Lemma \ref{ch6lemma2} now implies that
\[F_{\mathbb S_{1,2}}(q) = \sum_{n=1}^{\infty} (-1)^{n+1}q^{n^2}.\]
To prove the $F_{\mathbb S_{2,2}}(q)$ identity, first recall that $\sum_{\lambda \in \mathcal{P}}\mu_{\mathcal{P}}^{*}(\lambda) q^{\vert \lambda \vert} = -(q;q)_{\infty}$.  From this we know $F_{\mathbb S_{1,2}}(q) + F_{\mathbb S_{2,2}}(q) = 1 - (q;q)_{\infty}$.  Using the identity for $F_{\mathbb S_{1,2}}(q)$ and Euler's Pentagonal Number Theorem completes the proof.
\end{proof}

\begin{proof}[Proof of Corollary \ref{ch6cor}]

Case (1). This corollary follows immediately from Theorem \ref{ch612} (1). The reader should recall
that $F_{\mathbb S_{1,2}}(q)$ is the generating function for $\mu_{\mathcal{P}}^*(\lambda)=
-\mu_{\mathcal{P}}(\lambda)$.

Case (2). This corollary is not as immediate as case (1). Of course, we must classify the integer pairs $m$ and $n$ for which $n^2=m(3m-1)/2$. After simple manipulation, we find that this holds if and only if
$$
(6m-1)^2-6(2n)^2=1.
$$
In other words, we require that $(x,y)=(6m-1,2n)$ be a solution to the Pell equation
$$
x^2-6y^2=1.
$$
It is well known that all of the positive solutions to Pell's equation are of the form
$(x_k, y_k)$, where
$$
x_k +\sqrt{6}\cdot y_k = (5+2\sqrt{6})^k.
$$
Using this description and the elementary congruence properties of $(x_k,y_k)$, one easily obtains
Corollary \ref{ch6cor} (2).

\end{proof}

\subsection{Proof of Theorem \ref{ch6rt}}
Here we prove the general limit formulas for the arithmetic densities of $\mathbb{S}_{r,t}$.
\begin{proof}[Proof of Theorem~\ref{ch6rt}]
From Lemma \ref{ch6finite} we have 
\[F_{\mathbb S_{r,t}}(q) = (q;q)_{\infty} \cdot \frac{1}{t} \left[ \sum_{m=1}^{t} \frac{\zeta_{t}^{-mr}}{(\zeta_{t}^{m} q;q)_{\infty}} \right].\]
We stress that we can take a limit here because we have a finite sum of functions which are analytic inside the unit disk.  Using Lemma \ref{ch6limit} we see that 
\[\lim_{q \to 1} \frac{(q;q)_{\infty}}{(\zeta_{t}^{m} q;q)_{\infty}} = \begin{cases} 1 & {\rm{if}} \  m=t \\
0 & \rm{otherwise}. \end{cases}\]
From this we have 
\[\lim_{q\to 1} F_{\mathbb S_{r,t}}(q) = \frac{1}{t}.\]
The proof for $q \to \zeta$ where $\zeta$ is a primitive $m$th root of unity with $\gcd(m, t)=1$ follows the exact same steps.

\end{proof}

\subsection{Proofs of Theorem \ref{ch6free} and Corollary \ref{ch6free2}}
Here we will prove Theorem \ref{ch6free} and Corollary \ref{ch6free2} by building up $k$th power-free sets using arithmetic progressions.  We prove Theorem \ref{ch6free} first.
\begin{proof}[Proof of Theorem \ref{ch6free}]
The set of numbers not divisible by $p^{k}$ for any prime $p \leq N$ can be built as a union of sets of arithmetic progressions.  Therefore, for a given fixed $N$ we only need to understand divisibility by $p^{k}$ for all primes $p \leq N$.  Because the divisibility condition for each prime is independent from the other primes, we have
$$
F_{\mathbb S_{\rm{fr}}^{(k)}(N)}(q) = \sum_{\substack{0\leq r < M\\ p^k\nmid \ r}}
F_{\mathbb{S}_{r,M}}(q),
$$
where $M:=\prod_{\substack{p\leq N \ \\ {\text {\rm prime}}}} p^k$.
We have a finite number of summands, and the result now follows immediately from
Theorem~\ref{ch6rt}.
\end{proof}

Next, we will prove Corollary \ref{ch6free2}.
\begin{proof}[Proof of Corollary \ref{ch6free2}]
For fixed $N$ define $\zeta_{N}(k) := \prod_{p \leq N \rm{prime}} \left(\frac{1}{1-p^k} \right)$, so $\lim_{q \to 1} F_{\mathbb S_{\rm{fr}}^{(k)}(N)}(q)$ $= \frac{1}{\zeta_{N}(k)}$.  It is clear $\lim_{N \to \infty} \zeta_{N}(k)= \zeta(k)$.  It is in this sense that we say $\lim_{q \to 1} F_{\mathbb S_{\rm{fr}}^{(k)}}(q) = \frac{1}{\zeta(k)}$.
\end{proof}

\section{Examples}

\begin{example}
In the case of $\mathbb S_{1,3}$, which has arithmetic density 1/3, Theorem \ref{ch6rt} holds for any $m$th root of unity where $3 \nmid m$.  The two tables below illustrate this as $q$ approaches $\zeta_{1} = 1$ and $\zeta_{4} = i$, respectively, from within the unit disk.

\begin{center}
\begin{tabular}{ | c | c | }
\hline
$q$ & $F_{\mathbb S_{1,3}}(q)$\\ \hline
$0.70$ & $0.340411885...$ \\ \hline
$0.75$ & $0.335336994...$ \\ \hline
$0.80$ & $0.333552814...$ \\ \hline
$0.85$ & $0.333331545...$ \\ \hline
$0.90$ & $0.333333329...$ \\ \hline
$0.95$ & $0.333333333...$ \\   
\hline
\end{tabular}
\end{center}
\vspace{0.5cm}

%Here is a table indicating the value 1/3 for 1 mod 3 but with  $q \to i$.

\begin{center}
\begin{tabular}{ | c | c | }
\hline
$q$ & $F_{\mathbb S_{1,3}}(q)  $\\ \hline
$0.70i$ & $\approx 0.034621+0.793781i$ \\ \hline
$0.75i$ & $\approx 0.057890+0.802405i$ \\ \hline
$0.80i$ & $\approx 0.097030+0.771774i$ \\ \hline
$0.85i$ & $\approx 0.167321+0.674712i$ \\ \hline
$0.90i$ & $\approx 0.294214+0.454400i$ \\ \hline
$0.95i$ & $\approx 0.424978+0.067775i$ \\ \hline 
$0.97i$ & $\approx 0.376778-0.016187i$ \\ \hline
$0.98i$ & $\approx 0.340170+0.005772i$ \\ \hline
$0.99i$ & $\approx 0.332849+0.000477i$ \\  
\hline
\end{tabular}
\end{center}
\vspace{0.1cm}
\end{example}

\begin{example}
The table below illustrates Theorem \ref{ch6free} for the set $\mathbb S^{(2)}_{\rm{fr}}(5)$, which has arithmetic density $16/25 = 0.64$.  These are the positive integers which are not divisible by 4, 9 and 25.
Here we give numerics for the case of $F_{\mathbb S^{(2)}_{\rm{fr}}(5)} (q)$ as $q\to 1$ along the real axis.
 
\begin{center}
\begin{tabular}{ | c | c | }
\hline
$q$ & $F_{\mathbb S^{(2)}_{\rm{fr}}(5)} (q)$\\ \hline
$0.90$ & $0.615367...$ \\ \hline
$0.91$ & $0.619346...$ \\ \hline
$0.92$ & $0.625991...$ \\ \hline
$0.93$ & $0.631607...$ \\ \hline
$0.94$ & $0.631748...$ \\ \hline
$0.95$ & $0.631029...$ \\ \hline
$0.96$ & $0.638291...$ \\ \hline
$0.97$ & $0.639893...$ \\  
\hline
\end{tabular}
\end{center}
%\vspace{0.1cm}

\end{example}
%\begin{example}
%If $k=2$, then $\lim_{q \to 1} F_{\mathbb S^{(2)}_{\rm{fr}}}(q) = \frac{6}{\pi^{2}} \approx 0.60792...$.  This is interesting to compare to the case when $N=5$
%\end{example} 

\begin{example}
Here we approximate the density of $\mathbb S_{\text{fr}}^{(4)}$, the fourth power-free positive integers. Since $\zeta(4)=\pi^4/90$, it follows that the arithmetic density of $\mathbb S_{\text{fr}}^{(4)}$ is
${90}/{\pi^4} \approx 0.923938...$. Here we choose $N=5$ and compute the arithmetic density of $\mathbb S_{\text{fr}}^{(4)}(5)$, the positive integers which are not divisible by $2^4, 3^4$, and $5^4$. The density
of this set is $208/225 \approx 0.924444...$. This density is fairly close to the density of fourth power-free integers because the convergence in the $N$ aspect is significantly faster for fourth power-free integers than for square-free integers, as discussed above.

%\vspace{0.1cm}
 
\begin{center}
\begin{tabular}{ | c | c | }
\hline
$q$ & $F_{\mathbb S_{\text{fr}}^{(4)}(5)}(q)$\\ \hline
$0.90$ & $0.934926...$ \\ \hline
$0.91$ & $0.936419...$ \\ \hline
$0.92$ & $0.936718...$ \\ \hline
$0.93$ & $0.935027...$ \\ \hline
$0.94$ & $0.931517...$ \\ \hline
$0.95$ & $0.925619...$ \\ \hline
$0.96$ & $0.921062...$ \\ \hline
$0.97$ & $0.925998...$ \\ \hline
$0.98$ & $0.924967...$ \\  
\hline
\end{tabular}
\end{center}
\vspace{0.2cm}
\end{example}
\begin{remark}
{See Appendix \ref{app:E} for further notes on Chapter 6.} 
\end{remark}

 	\clearpage%
\chapter{``Strange'' functions and a vector-valued quantum modular form}{\bf Adapted from \cite{RolenSchneider}, a joint work with Larry Rolen}
%\begin{abstract}
%In important work on the parity of the partition function, Ono \cite{Ono} related values of the partition function to coefficients of a certain mock theta function modulo 2. In this paper, we use M\"obius inversion to give analogous results which relate several combinatorial functions via identities rather than congruences.
%\end{abstract}

%\maketitle

%\begin{abstract}
%Since their definition in 2010 by Zagier, quantum modular forms have been connected to numerous different topics such as strongly unimodal sequences, ranks, cranks, and  asymptotics for mock theta functions near roots of unity. These are functions that are not necessarily defined on the upper half plane but \emph{a priori} are defined only on a subset of $\Q$, and whose obstruction to modularity is some analytically ``nice'' function. Motivated by Zagier's example of the quantum modularity of Kontsevich's ``strange'' function $F(q)$, we revisit work of Andrews, Jim{\'e}nez-Urroz, and Ono to construct a natural vector-valued quantum modular form whose components are similarly ``strange''. 
% \end{abstract}

\section{Introduction and Statement of Results}
In this chapter and the next, we pivot away from partition theory (at least explicitly) to focus on certain interesting classes of $q$-series, which we will then tie back to the ideas of the previous sections in the final chapter. 

In a seminal 2010 Clay lecture, Zagier defined a new class of function with certain automorphic properties called a ``quantum modular form'' \cite{Zagier_quantum}, as in Definition \ref{ch1qmfdef}. Roughly speaking, this is a complex function on the rational numbers which has modular transformations modulo ``nice'' functions. Although the definition is intentionally vague, Zagier gave a handful of motivating examples to serve as prototypes of quantum behavior. For example, he defined quantum modular forms related to Dedekind sums, sums over quadratic polynomials, Eichler integrals and other interesting objects. %, and to Kontsevich's ``strange'' function. %More concisely, as given in Chapter 1, Zagier proposed the following.
%\begin{definition} We say that a function $f$ from a subset of $\mathbb{P}^1(\Q)$ to $\C$ is a \emph{quantum modular form} if $f(x)-f|_k\gamma(x)=h_{\gamma}(x)$ for a ``suitably nice'' function $h_{\gamma}(x)$.
%\end{definition}
%Here $|_k$ is the usual Petersson slash operator, and ``suitably nice'' implies some pertinent analyticity condition, e.g. $\mathcal{C}_k,\mathcal{C}_{\infty}$, etc. 
One of the most striking examples of quantum modularity is described in Zagier's paper on Vassiliev invariants \cite{Zagier_Vassiliev}, in which he studies the Kontsevich ``strange'' function introduced in Definition \ref{ch1Fdef}, viz. 
\begin{equation}F(q):=\sum_{n=0}^{\infty}(q;q)_n,\end{equation}
where we take $q:=e^{2\pi i z}$ with $z\in\C$.\\ \indent This function is strange indeed, as it does not converge on any open subset of $\C$, but converges (as a finite sum) for $q$ any root of unity. In 2012, Bryson, Pitman, Ono, and Rhoades showed \cite{BOPR} that $F(q^{-1})$ agrees to infinite order at roots of unity with a function $U(-1,q)$ which is also well-defined on the upper-half plane $\mathbb{H}$, obtaining a quantum modular form that is a ``reflection'' of $F(q)$ and that naturally extends to $\mathbb{H}$. Moreover, $U(-1,q)$ counts unimodal sequences having a certain combinatorial statistic. 

Zagier's study of $F(q)$ depends on the formal $q$-series identity
\begin{equation}
\displaystyle\sum_{n=0}^{\infty}\left(\eta(24z)-q(1-q^{24})(1-q^{48})\cdots(1-q^{24n})\right)=\eta(24z)D(q)+E(q),
\end{equation}
\noindent
where $\eta(z):=q^{1/24}(q;q)_{\infty}$, $D(q)$ is an Eisenstein-type series, and $E(q)$ is a ``half-derivative'' of $\eta(24z)$ (such formal half-derivatives will be discussed in Section 7.2). We refer to such an identity as a ``sum of tails'' identity. 
%Here 
In this chapter we revisit Zagier's construction using work of Andrews, Jim{\'e}nez-Urroz, and Ono on more general sums of tails formulas \cite{A-J-O} (see also \cite{Andrews-Advances}). We construct a natural three-dimensional vector-valued quantum modular form associated to tails of infinite products. Moreover, the components are analogous ``strange'' functions; they do not converge on any open subset of $\C$ but make sense for an infinite subset of $\Q$. We define:
\begin{equation}H(q)=\begin{pmatrix}\theta_1\\ \theta_2 \\ \theta_3\end{pmatrix}:=\begin{pmatrix}\eta(z)^2/\eta(2z)\\ \eta(z)^2/\eta(z/2)\\ \eta(z)^2/\eta(\frac{z}{2}+\frac12) \end{pmatrix}. \end{equation}
We also note that $\theta_3=\zeta_{48}^{-1}\cdot\frac{\eta(z/2)\eta(2z)}{\eta(z)}$ by the following identity which is easily seen by expanding the product definition of $\eta(z)$:
\begin{equation}\label{ch7eta_identity1}
\eta(z+1/2)=\zeta_{48}\cdot\frac{\eta(2z)^3}{\eta(z)\cdot\eta(4z)},
\end{equation}
where $\zeta_k:=e^{2\pi i/k}$. From this it follows that if we let 
\begin{equation*}
\quad\quad F_9(z):=\eta(z)^2/\eta(2z),\quad\quad\quad\quad\quad\quad\quad\quad
F_{10}(z):=\eta(16z)^2/\eta(8z)
\end{equation*}
then 
\begin{equation}\label{ch7F9F10}
H(q)=\begin{pmatrix}F_9(q)&F_{10}(q^{1/16})
&\zeta_{12}^{-1}F_{10}(\zeta_{16}\cdot q^{1/16})\end{pmatrix}^T\end{equation}
(the notations $F_9$ and $F_{10}$ come from \cite{A-J-O}).
For convenience, we recall the classical theta-series identities for $F_9$ and $F_{10}$:
\begin{equation}
F_9(q)= 1+2\displaystyle\sum_{n=1}^{\infty}(-1)^nq^{n^2},     \quad\quad\quad\quad\quad\quad\quad\quad F_{10}(q)=\displaystyle\sum_{n=0}^{\infty}q^{(2n+1)^2}
.\end{equation}
It is simple to check that $H(z)$ is a $3$-dimensional vector-valued modular form using basic properties of $\eta(z)$, as we describe in Section 7.4. To each component $\theta_i$ we associate for all $n\geq0$ a finite product $\theta_{i,n}$:
\begin{equation}
\theta_{1,n}:=\frac{(q;q)_n}{(-q;q)_n},\quad\quad\quad \theta_{2,n}:=q^{\frac{1}{16}}\cdot\frac{(q;q)_n}{(q^{\frac12};q)_{n+1}},\quad\quad\quad\theta_{3,n}:=\frac{\zeta_{16}}{\zeta_{12}}\cdot q^{\frac{1}{16}}\cdot\frac{(q;q)_n}{(-q^{\frac12};q)_{n+1}},
\end{equation}
such that $\theta_{i,n}\rightarrow\theta_i$ as $n\rightarrow\infty$.
Next, we construct corresponding ``strange'' functions
$
\theta_i^S:=\sum_{n=0}^{\infty}\theta_{i,n}.
$
Note that these functions do not make sense on any open subset of $\C$, but that each $\theta_i^S$ is defined for an infinite set of roots of unity and, in particular, $\theta_2^S$ is defined for all roots of unity.
Our primary object of study will then be the vector of ``strange'' series $H_Q(z):=\begin{pmatrix}\theta_1^S(z)& \theta_2^S(z)& \theta_3^S(z)\end{pmatrix}^T.$ In order to obtain a quantum modular form, we first define $\phi_i(x):=\theta_i^S(e^{2 \pi i x})$ from a subset of $\Q$ to $\C$, and let $\phi(x):=\begin{pmatrix}\phi_1(x)& \phi_2(x)& \phi_3(x)\end{pmatrix}^T.$ We then show the following result.
\begin{theorem}\label{ch7mainthm} Assume the notation above. Then the following are true:
\begin{enumerate}[(1)] %%%This statement needs work... also there is a funny parenthesis in G(q)%%%
\item There exist $q$-series $G_i$ (see Section 7.4) which are well-defined for $|q|<1$, such that $\theta_i^S(q^{-1})=G_i(q)$ for any root of unity for which $\theta_i^S$ converges.
\item We have that $\phi(x)$ is a weight $3/2$ vector-valued quantum modular form. In particular, we have that 
\[
\phi(z+1)-\begin{pmatrix}1&0&0\\0&0&\zeta_{12}\\0&\zeta_{24}&0 \end{pmatrix}\phi(z)=0
,
\]
and we also have that
\[
\left(\frac{z}{-i}\right)^{-3/2}\phi(-1/z)+\begin{pmatrix}0&\sqrt{2}&0\\1/\sqrt{2}&0&0\\0&0&1\end{pmatrix}\phi(z)=\begin{pmatrix}0&\sqrt{2}&0\\1/\sqrt{2}&0&0\\0&0&1\end{pmatrix}g(z)
,
\]
where $g(z)$ is a $3$-dimensional vector of smooth functions defined as period integrals in Section 7.3.
\end{enumerate}
\end{theorem}
In addition, we deduce the following corollary regarding generating functions of special values of zeta functions from the sums of tails identities. Let
\begin{equation}
H_9(t,\zeta):=-\frac14\displaystyle\sum_{n=0}^{\infty}\frac{(1-\zeta e^{-t})(1-\zeta^2e^{-2t})\cdots(1-\zeta^n e^{-nt})}{(1+\zeta e^{-t})(1+\zeta^2e^{-2t})\cdots(1+\zeta^ne^{-nt})},
\end{equation}
\begin{equation}
H_{10}(t,\zeta):= -2(\zeta e^{-t})^{1/8}\displaystyle\sum_{n=0}^{\infty}\frac{(1-\zeta e^{-2t})(1-\zeta^2e^{-4t})\cdots(1-\zeta^n e^{-2nt})}{(1-\zeta e^{-t})(1-\zeta^2e^{-3t})\cdots(1-\zeta^ne^{-(2n+1)t})}.
\end{equation}
\begin{rmk}
Note that there are no rational numbers for which all three components of $\phi$ make sense simultaneously. To be specific, $\phi_1(z)$ makes sense for rational numbers which correspond to primitive odd order roots of unity, $\phi_2(z)$ makes sense for all rational numbers, and $\phi_3(z)$ converges at even order roots of unity. Hence, by (2) of Theorem \ref{ch7mainthm}, we understand that each of the six equations of the vector-valued transformation laws is true where the corresponding component in the equation is well-defined; as there are no equations in which $\phi_1$ and $\phi_3$ both appear, then for all the equations there is an infinite subset of rationals on which this is possible.
\end{rmk}

For a root of unity $\zeta$, we define the following two $\text{L}$-functions

\[L_1(s,\zeta):=\sum_{n=1}^{\infty}\frac{(-\zeta)^{n^2}}{n^s},\]

\[L_2(s,\zeta):=\sum_{n=1}^{\infty}\left(\frac2n\right)^2\cdot\frac{\zeta^{\frac{n^2}{8}}}{n^s}.\]
Then we have the following.

\begin{corollary}\label{ch71.1}
Let $\zeta=e^{2\pi i \alpha}$ be a primitive $k^{\operatorname{th}}$ root of unity, $k$ odd for $H_9$ and $k$ even for $H_{10}$. Then as $t\searrow0$, we have as power series in $t$

\begin{equation}\label{ch7Hurwitz1}  H_{9}(t,\zeta)=\sum_{n=0}^{\infty}\frac{L_1(-2n-1,\zeta)(-t)^n}{n!},
\end{equation}
\begin{equation}\label{ch7Hurwitz2} H_{10}(t,\zeta)=\sum_{n=0}^{\infty}\frac{L_2(-2n-1,\zeta)(-t)^n}{8^nn!}.
\end{equation}
\end{corollary}

To illustrate our results by way of an application, we provide a  numerical example which gives finite evaluations of seemingly complicated period integrals. First define 
\begin{equation*}
\Omega(x):=\int_{x}^{i\infty}\frac{\theta_1(z)}{(z-x)^{3/2}}\, dz
\end{equation*} 
for $x\in\Q$, and consider $\theta_1^S(\zeta_k)$ for $k$ odd, which is a finite sum of $k^{\text{th}}$ roots of unity. Then the proof of Theorem \ref{ch7mainthm} will imply that $\Omega(1/k)=\pi i(1+i)\theta_1^S(\zeta_k)$ by showing that the period integral $\Omega(x)$ is a ``half-derivative'' which is related to $\theta_1^S$ at roots of unity by a sum of tails formula. The following table gives finite evaluations of $\theta_1^S(\zeta_k)$ and numerical approximations to the integrals $\Omega(1/k)$.\\ 
\begin{center} 
\begin{tabular}{c c c c} % centered columns (4 columns)
\hline\hline                        %inserts double horizontal lines\
$k$&$\pi i(i+1)\theta_1^S(\zeta_k)$&$\int_{1/k+10^{-9}}^{10^9i}\frac{\theta_1(z)}{(z-1/k)^{\frac32}}\, dz$\\ [0.7ex] % inserts table
%heading
\hline                    % inserts single horizontal line
\\
$3$&$\pi i(i+1)(-2 \zeta_{3} + 3)\sim-7.1250+18.0078i$&$-7.1249+18.0078i$
\\
$5$&$\pi i(i+1)(-2 \zeta_{5}^{3} - 2 \zeta_{5}^{2} - 8 \zeta_{5} + 3)\sim12.078+35.7274i$&$12.078+35.7273i$
\\
$7$&$\pi i(i+1)(6 \zeta_{7}^{4} - 2 \zeta_{7}^{2} - 10 \zeta_{7} + 7)\sim52.0472+25.685i$&$52.0474+25.685i$
\\
$9$&$\pi i(i+1)(8 \zeta_{9}^{4} - 16 \zeta_{9} + 3)\sim76.4120-28.9837i$&$76.4116-28.9836i$
\\
\hline     %inserts single line
\end{tabular}
\end{center}
\vspace{0.2in}
%%%the following section is no longer accurate as the paper has been restructured%%%
The %paper 
chapter is organized as follows. In Section 7.2 we recall the identities of \cite{A-J-O}, and in Section 7.3 we describe the modularity properties of Eichler integrals of half-integral weight modular forms. In Section 7.4 we complete the proof of Theorem \ref{ch7mainthm}. We finish with the proof of Corollary \ref{ch71.1} in Section 7.5.
%%end section that needs correcting%%
\section{Preliminaries}
In this section, we describe some of the machinery needed to prove Theorem \ref{ch7mainthm}.
\subsection{Sums of Tails Identities}
Here we recall the work of Andrews, Jim{\'e}nez-Urroz, and Ono on sums of tails identities. To state their results for $F_9$ and $F_{10}$ and connect $\theta_i^S$ to quantum modular objects, we formally define a ``half-derivative operator'' by
\begin{equation}
\sqrt{\theta}\left(\displaystyle\sum_{n=0}^{\infty}a(n)q^n\right):=\displaystyle\sum_{n=1}^{\infty}\sqrt{n}a(n)q^n.
\end{equation}
If we have a generic $q$-series $f(q)$, we will also denote $\sqrt{\theta}f(q):=\widetilde{f}(q)$.
Then Andrews, Jim{\'e}nez-Urroz, and Ono show \cite{A-J-O} that for finite versions $F_{9,i}, F_{10,i}$ associated to $F_{9},F_{10}$ the following holds true:
\begin{theorem}[Andrews-Jim{\'e}nez-Urroz-Ono]As formal power series, we have that
\begin{equation}
\displaystyle\sum_{n=0}^{\infty}\left(F_{9}(z)-F_{9,n}(z)\right)=2F_{9}(z)E_1(z)+2\sqrt{\theta}(F_{9}(z))
,\end{equation}

\begin{equation}
\displaystyle\sum_{n=0}^{\infty}\left(F_{10}(z)-F_{10,n}(z)\right)=F_{10}(z)E_2(z)+\frac12\sqrt{\theta}(F_{10}(z)),
\end{equation}
where the $E_i(z)$ are holomorphic Eisenstein-type series. \end{theorem} %%%Need to check subscripts of summations in Andrews/Ono/etc.%%%
\noindent
In particular, as $F_9,F_{10}$ vanish to infinite order while $E_1,E_2$ are holomorphic at all cusps where the ``strange'' functions are well-defined, we have for $q$ an appropriate root of unity that the ``strange'' function associated to $F_i$ equals $\widetilde{F_i}$ to infinite order. As the series $\theta_2,\theta_3$ do not have integral coefficients, we make the definitions $\widetilde{\theta}_2(z):=\widetilde{F}_{10}(z/16)$ and $\widetilde{\theta}_3(z):=\widetilde{F}_{10}(z/16+1/16)$. By the definition of the strange series, we obtain the following. 
\begin{corollary}\label{ch72.1}
At appropriate roots of unity where each ``strange'' series is defined, we have that
\begin{equation}
\theta_1^S(q)=2\widetilde{\theta_1}(q),\quad\quad\quad\theta_2^S(q)=\frac12\widetilde{\theta_2}(q),\quad\quad\quad\theta_3^S(q)=\frac12\widetilde{\theta_3}(q).
\end{equation}
\end{corollary}
\section{Properties of Eichler Integrals}
In the previous section we have seen that at a rational point $x$, each component of $\phi(x)$ agrees up to a constant with a ``half-derivative'' of the corresponding theta function at $q=e^{2\pi i x}$. Thus, we can reduce part (2) of Theorem \ref{ch7mainthm} to a study of modularity of such half-derivatives. We do so following the outline given in \cite{Zagier_Vassiliev}, which is further explained in the weight $3/2$ case in \cite{L-Zagier}. Recall that in the classical setting of weight $2k$ cusp forms, $1\leq k\in\Z$, we define the \emph{Eichler integral} of $f(z)=\sum_{n=1}^{\infty} a(n)q^n$ as a formal $(k-1)^{\text{st}}$ antiderivative
$
\widetilde{f}(z):=\sum_{n=1}^{\infty} n^{1-k}a(n)q^n.
$
Then $\widetilde{f}$ is nearly modular of weight $2-k$, as the differentiation operator $\frac{\, d}{\, dq}$ does not preserve modularity but preserves near-modularity. More specifically, $\widetilde{f}(z+1)=\widetilde{f}(z)$ and $z^{k-2}\widetilde{f}(-1/z)-\widetilde{f}(z)=g(z)$ where $g(z)$ is the \emph{period polynomial}. This polynomial encodes deep analytic information about $f$ and can also be written as $g(x)=c_k\int_0^{i\infty}f(z)(z-x)^{k-2}\, dz$ for a constant $c_k$ depending on $k$. Suppose we now begin with a weight $1/2$ vector-valued modular form $f$ with $n$ components $f_i$ such that and $f(-1/z)=M_Sf(z)$, for $M_S$ both $n\times n$ matrices (the transformation under translation is routine).  %%Question: should dz be italicized in general, throughout? Also, nXn contains a typo wherein the second "n" is not italicized. Also, check that I my rewording of final sentence off this paragraph makes correct usage of M_T and M_S.%%
\\
\indent
In this case, of course, it does not make sense to speak of a half-integral degree polynomial, and the integral above does not even converge. However, we may remedy the situation so that the analysis becomes similar to the classical case. We formally define $\widetilde{f}$ by taking a formal antiderivative (in the classical sense) on each component. As $1-k=1/2$, we have in fact $\widetilde{f_i}=\sqrt{\theta}f_i$. We would like to determine an alternative way to write the Eichler integral as an actual integral, so that we may use substitution and derive modularity properties of $\widetilde{f}$ from $f$. However, the integral $g(z)=c_{1/2}\int_0^{i\infty}f(z)(z-x)^{-3/2}\, dz$ no longer makes sense. To remedy this in the weight $3/2$ case, Lawrence and Zagier define another integral $f^*(x):=c_k\int_{\bar{x}}^{\infty}\frac{f(z)}{(z-x)^{\frac12}}\, dz,$ which is meaningful for $x$ in the lower half plane $\mathbb{H}^-$.
\\ %%%In the above paragraph and other places throughout, I don't love the usage "makes sense"... can we go with other more specific phrases in each place, such as "is not valid" or "is not exactly true"?%%%
\indent
Here we sketch their argument in the weight $1/2$ case for completeness, and as the analysis involved in our own work differs slightly. Returning to our vector-valued form $f$, recall that the definition of the Eichler integral of $f$ corresponds with $\sqrt{\theta}f$. For $x\in\mathbb{H}^-$, we define
\begin{equation}
 f^*(x)=\left(\frac{-i}{\pi(1+i)}\right)\cdot\int_{\bar{x}}^{i\infty}\frac{f(z)}{(z-x)^{\frac32}}\, dz.
 \end{equation}
To evaluate this integral, use absolute convergence to exchange the integral and the sum, and note that for $q_z=e^{2\pi i z},$
\begin{equation}
\int_{\bar{x}}^{i\infty}\frac{q_z^n}{(z-x)^{\frac32}}\, dz=\left((2+2i)\pi\sqrt{n}q_z^n\operatorname{erfi}\left((1+i)\sqrt{\pi n(z-x)}\right)-\frac{2q_x^n}{(z-x)^{\frac12}}\right)\bigg|_{z=\bar{x}}^{i\infty}
\end{equation}
where $\operatorname{erfi}(x)$ is the \emph{imaginary error function}. As in \cite{L-Zagier}, we have that $\widetilde{f}(x+i y)=f^*(x-iy)$ as full asymptotic expansions for $x\in\Q$, $0<y\in\R$. To see this, note that at the lower limit, the antiderivative vanishes as $y\rightarrow 0$ as $\operatorname{erfi}(0)=0$ and although the square root in the denominator goes to zero, for each rational at which we are evaluating our ``strange'' series, the corresponding theta functions vanish to infinite order, which makes this term converge. For the upper limit, the square root term immediately vanishes, and we use the fact that $\lim_{x\rightarrow \infty}\operatorname{erfi}(1+i)\sqrt{i x+y}=i$ for $x,y\in\R$. 
\\
\indent
%%The following paragraph is incomplete: "More formal etc..."%%
Thus, as in \cite{L-Zagier}, we have that $\widetilde{f}(x)=f^*(x)$ to infinite order at rational points. In the case of $\theta_1$, we have that $\widetilde{\theta}_1(x)=\theta^*(x)$, but for $\theta_2$ and $\theta_3$ we have to divide by $4=\sqrt{16}$ due to the non-integrality of the powers of $q$ in order to agree with the definition of $\widetilde{\theta}_i$.  Using this together with Corollary \ref{ch72.1}, in all cases we find that $\theta_i^S(q)=\theta_i^*(q)$  at roots of unity where both sides are defined. Now, the modularity properties for the integral follow \emph{mutatis mutandis} from \cite{L-Zagier} using the modularity of $f$  and a standard $u$-substitution. More precisely, suppose $f(-1/z)(z)^{-\frac{1}{2}}=M_Sf(z)$. Then we have shown that the following modularity properties hold for $f^*(z)$ when $z\in\mathbb{H}^-$, and hence also hold for $\widetilde{f}(z)$ for each component at appropriate roots of unity where each ``strange'' function is defined. By this, we mean that the modularity conditions in the following proposition can be expressed as six equations, and each of these equations is true precisely where the corresponding ``strange'' series make sense. 
\begin{proposition}
If $g(x):=\left(\frac{-i}{\pi(1+i)}\right)\cdot\int_{0}^{i\infty}\frac{f(z)}{(z-x)^{\frac32}}\, dz$, then
\[
\left(\frac{x}{-i}\right)^{-\frac{3}{2}}f(-1/x)+M_Sf(x)=M_Sg(x).
\]
\end{proposition} 
It is also explained in \cite{L-Zagier} why $g_{\alpha}(z)$ is a smooth function for $\alpha\in\R$. Although $g(x)$ is \emph{a priori} only defined in $\mathbb{H}^-$, we may take any path $L$ connecting $0$ to $i\infty$. Then we can holomorphically continue $g(x)$ to all of $\C-L$. Thus, we obtain a continuation of $g$ which is smooth on $\R$ and analytic on $\R-\{0\}$.
\section{Proof of Theorem \ref{ch7mainthm}}
Here we complete the proofs of parts (1) and (2) of Theorem \ref{ch7mainthm}.
\subsection{Proof of Theorem \ref{ch7mainthm} (1)}
We show that at appropriate roots of unity, our ``strange'' functions $\theta_i^S$ are reflections of $q$-series which are convergent on $\mathbb{H}$. 
Using (\ref{ch7F9F10}), it suffices to show for $\theta_1^S$ that $\sum_{n=0}^{\infty}\frac{(q^{-1};q^{-1})_n}{(-q^{-1};q^{-1})_n}$ agrees at odd roots of unity with a $q$-series convergent when $|q|<1$. To factor out inverse powers of $q$, we observe that 
\begin{equation}\label{ch7inversion}
(a^{-1};q^{-\alpha})_n=(-1)^na^nq^{\frac{\alpha n(n-1)}{2}}(a;q^{\alpha})_n.
\end{equation}
Applying this identity to the numerator and denominator term-by-term, we have at odd order roots of unity
\begin{equation}\theta_1^S(q^{-1})=\sum_{n=0}^{\infty}(-1)^n\frac{(q;q)_n}{(-q;q)_n}=2\displaystyle\sum_{n=0}^{\infty}\frac{q^{2n+1}(q;q)_{2n}}{(1+q^{2n+1})(-q;q)_{2n}}.
\end{equation}
The series on the right-hand side is clearly convergent for $|q|<1$, and results from pairing consecutive terms of the left-hand series as follows:
\begin{equation*}\frac{(q;q)_{2n}}{(-q;q)_{2n}}-\frac{(q;q)_{2n+1}}{(-q;q)_{2n+1}}=\frac{(q;q)_{2n}}{(-q;q)_{2n}}\left(1-\frac{1-q^{2n+1}}{1+q^{2n+1}}\right)=\frac{2q^{2n+1}(q;q)_{2n}}{(1+q^{2n+1})(-q;q)_{2n}}.
\end{equation*}
\begin{rmk}Alternatively, one can show the convergence of $\theta_1^S(q^{-1})$ by letting $a=1,b=-1,t=-1$ in Fine's identity \cite{Fine}
\begin{equation}\sum_{n=0}^{\infty}\frac{(aq;q)_n}{(bq;q)_n}(t)^n=\frac{1-b}{1-t}+\frac{b-atq}{1-t}\sum_{n=0}^{\infty}\frac{(aq;q)_n}{(bq;q)_n}(tq)^n,
\end{equation}\end{rmk}
giving \begin{equation}\label{ch7inversefine}\theta_1^S(q^{-1})=1+\frac{q-1}{2}\sum_{n=0}^{\infty}\frac{(q;q)_n}{(-q;q)_n}(-q)^n\end{equation} which also converges for $|q|<1$. 
\\
\indent
Similarly, we use (\ref{ch7inversion}) to study $\theta_2^S,\theta_3^S$. Note that it suffices by (\ref{ch7F9F10}) to study $\sum_{n=0}^{\infty}\frac{(q^{-2};q^{-2})_n}{(q^{-3};q^{-2})_{n}}$. Factorizing as above, we find that 
\begin{equation}
\sum_{n=0}^{\infty}\frac{(q^{-2};q^{-2})_n}{(q^{-3};q^{-2})_n}=\sum_{n=0}^{\infty}\frac{q^n(q^{2};q^{2})_n}{(q^{3};q^{2})_n}
,
\end{equation}
the right-hand side of which is clearly convergent on $\mathbb{H}$. We note that in general, similar inversion formulas result from applying (\ref{ch7inversion}) to diverse $q$-series and other expressions involving eta functions, $q$-Pochhammer symbols and the like.
\subsection{Proof of Theorem \ref{ch7mainthm} (2)}
\begin{proof}Here we complete the proof of Theorem \ref{ch7mainthm} using the results of Sections 7.2 and 7.3.  
Note that by the Corollary (\ref{ch72.1}) to the sums of tails formulas of Andrews, Jim{\'e}nez-Urroz, and Ono in \cite{A-J-O}, each component of $H(q)$ agrees to infinite order at rational numbers with a multiple of the corresponding Eichler integral. By the discussion of Eichler integrals in Section 7.3, the value of each $\widetilde{\theta_i}$ agrees at rationals with the value of the corresponding $\theta_i^*$.  Therefore, by the discussion of the modularity properties of $\theta_i^*$, we need only to describe the modularity of $H(q)$. This is simple to check using the usual transformation laws
\begin{equation}
\eta(z+1)=\zeta_{24}\eta(z),
\end{equation}
\begin{equation}
\eta(-1/z)=\left(\frac{z}{i}\right)^{\frac12}\eta(z),
\end{equation}
and (\ref{ch7eta_identity1}). Hence we see that 
\begin{equation}
H(z+1)=\begin{pmatrix}1&0&0\\0&0&\zeta_{12}\\0&\zeta_{24}&0 \end{pmatrix}H(z)
,
\end{equation}
\begin{equation}
H(-1/z)=\left(\frac{z}{i}\right)^{\frac12}\begin{pmatrix}0&\sqrt{2}&0\\1/\sqrt{2}&0&0\\0&0&1\end{pmatrix}H(z)
,
\end{equation}
and the corresponding transformations of $\theta_i^*$ follow.
\end{proof}
\section{Proof of Corollary \ref{ch71.1}}
\begin{proof}
The proof of Corollary 1.1 is a generalization of and proceeds similarly to the proofs of Theorems 4 and 5 of \cite{A-J-O}. As the sums of tails identities in Theorem 2.1 show that the ``strange'' functions $F_9$ and $F_{10}$ agree \emph{to infinite order} with the half derivatives of $F_9$ and $F_{10}$ at the roots of unity under consideration, the coefficients in the asymptotic expansion of $H_i(t,\zeta)$ for $i=9,10$ agree up to a constant factor with the coefficients of the asymptotic expansion of $\sqrt{\theta}F_i(\zeta e^{-t})$. Recalling the classical theta series expansions for $F_i$ in (1.6), the first part of Corollary 1.1 follows immediately from the following well-known fact:
 
 \begin{lemma}[Proposition 5 of \cite{Hikami}]
 Let $\chi(n)$ be a periodic function with mean value zero and $L(s,\chi):=\sum_{n=0}^{\infty}\chi(n)n^{-s}$. As $t\searrow0$, we have
 
\[\sum_{n=0}^{\infty}n\chi(n)e^{-n^2t}\sim\sum_{n=0}^{\infty}L(-2n-1,\chi)\frac{(-t)^n}{n!}.\]
 
 \end{lemma}
The proof follows from taking a Mellin transform, making a change of variables, and picking up residues at negative integers. The assumption on the coefficients $\chi(n)$ assures that $L(s,\chi)$ can be analytically continued to $\C$. The mean value zero condition is easily checked in our case; for example for $F_9$ one needs to verify that $\{(-\zeta)^{n^2}\}_{n\geq0}$ is mean value zero for $\zeta$ a primitive order $2k+1$ root of unity, and for $F_{10}$ one must check that $\{\zeta^{\frac{(2n+1)^2}{8}}\}_{n\geq0}$ is mean value zero for an even order root of unity $\zeta$. These may both be checked using well-known results for the generalized quadratic Gauss sum
\begin{equation}
G(a,b,c):=\displaystyle\sum_{n=0}^{c-1}e\left(\frac{an^2+bn}{c}\right).
\end{equation}
In particular, for $F_9$, for an odd order root of unity $\zeta$, $-\zeta$ is primitive of order $k$ where $k\equiv2\pmod4$, so we need that $G(a,0,k)=0$ when $k\equiv2\pmod4$, which fact is well known. For $F_{10}$, we may use the standard fact that $G(a,b,c)=0$ whenever $4|c$, $(a,c)=1$, and $0<b\in2\Z+1$ to obtain our result. This Gauss sum calculation follows, for instance, by using the multiplicative property of Gauss sums together with an application of Hensel's lemma. 

In the case of $F_{10}$, note that the formula for $H_{10}(t,\alpha)$ is obtained by substituting $q=\zeta e^{-t}$ into the ``strange'' function for $F_{10}$ after letting $q\rightarrow q^{\frac18}.$ A simple change of variables in the Mellin transform in the foregoing proof of the present Lemma adjusts for the $1/8$ powers by giving an extra factor of $8^s$ before taking residues.
\end{proof}
\  
\

\begin{remark}
{See Appendix \ref{app:F} for further notes on Chapter 7.} 
\end{remark}

 	\clearpage%
\chapter{Jacobi's triple product, mock theta functions, unimodal sequences and the $q$-bracket}{\bf Adapted from \cite{SchneiderJTP}}%, a joint work with Ken Ono and Larry Rolen}

\section{Introduction}%Mock theta functions from a classical (and physics) perspective}%Preamble}%Introduction}
We do not know what sparked Ramanujan to discover mock theta functions,  % --- probably not imagery from quantum theory --- 
but we see in this %paper 
chapter that they are indeed natural functions to study from a classical perspective. It turns out in Section \ref{ch8Sect1} all of the mock theta functions Ramanujan wrote to Hardy about --- to be precise, the odd-order {\it universal mock theta function} of Gordon--McIntosh that essentially specializes to the odd-order mock theta functions Ramanujan wrote down \cite{GordonMcIntosh} --- arise from the {\it Jacobi triple product}, a fundamental object in number theory and combinatorics \cite{Berndt}, and are generally ``entangled'' with rank generating functions for unimodal sequences, under the action of the {\it $q$-bracket} operator from statistical physics and partition theory that we studied in Chapter 3, which boils down to multiplication by $(q;q)_{\infty}$. % \cite{Zagier}. %Moreover, mock theta functions that also is related to work in Ramanujan's ``lost'' notebook \cite{AndrewsBerndt}.%, of which we saw a suggestive example in the preceding paragraph. 
In Section \ref{ch8Sect2} we find finite formulas for the odd-order universal mock theta function %up to changes of variables and multiplication by powers of $z$ and $q$, 
and indicate similar formulas for other $q$-hypergeometric series. %In some cases,  % --- which, while they may not necessarily be quantum modular, still 
%we do encounter the ``feel'' of quantum theory.  
%\\
   
\section{Connecting the triple product to mock theta functions via partitions and unimodal sequences}\label{ch8Sect1}

%Recall Ramanujan's mock theta functions, discussed in Chapter 1. 
At the wildest boundaries of nature, %where science breaks apart, 
we see tantalizing hints of $q$-series. In the previous chapter we investigated a class of almost-modular forms having the ``feel'' of quantum phenomena \cite{Zagier_quantum}. In a different ``quantum'' connection, Borcherds proposed a proof of the Jacobi triple product identity
\begin{equation}\label{ch8jtp}
j(z;q):=(z;q)_{\infty}(z^{-1}q;q)_{\infty}(q;q)_{\infty}=\sum_{n=-\infty}^{\infty}(-1)^n z^{n}q^{n(n+1)/2}
\end{equation}
where $q,z\in \mathbb C, |q|<1, z\neq 0$\footnote{We note the simple zero at $z=1$ from the $(1-z)$ factor in $(z;q)_{\infty}$. }, based on the {\it Dirac sea} model of the quantum vacuum, plus ideas from partition theory (see \cite{Cameron}): the quantum states of fermions, which obey the Pauli exclusion principle, are conceptually analogous to partitions into distinct parts; quantum states of bosons, which are unrestricted in the number that can occupy any state, correspond to partitions with unrestricted multiplicities of parts\footnote{In Appendix B we draw further analogies to particle physics by introducing ``antipartitions'' that annihilate partitions, yielding a multiplicative group structure on the partitions.}. The triple product  is implicit in countless famous classical identities (see \cite{Berndt}). Up to multiplication by rational powers of $q$, $j(z;q)$ specializes to the Jacobi theta function (that Ramanujan constructed ``mock'' versions of), a weight $1/2$ modular form which is also important in physics as a solution to the heat equation. 
In Borcherds's proof, it is as if this beautiful, versatile identity \eqref{ch8jtp} emerges from properties of empty space. % at a scale so small the meaning of space-time measurement is not even well-defined. 

Also from the universe of $q$-hypergeometric series, mock theta functions and their generalization mock modular forms \cite{BFOR} are connected conjecturally 
to deep mysteries in physics, like %: monstrous moonshine, quantum gravity, and 
%such as %the 
mind-bending phenomena at the edges of black holes \cite{DMZ,DGO}. %where measurement again becomes meaningless. 
All the diversity of physical reality --- and of our own mental experience --- plays out quite organically between these enigmatic extremes. %, %that cohabitate the spectrum of nature 
%even though our understanding of their physics is disjoint. %, whether or not we understand their connection. %conditions.
Perhaps not unrelatedly, in this chapter we see there is an organic %behind-the-scenes 
connection between the Jacobi triple product and mock theta functions, under the action of the $q$-bracket of Bloch--Okounkov studied in Chapter 3.\footnote{Indeed, there are many interesting connections between partition theory, $q$-series and statistical physics; for instance, see Ch. 8 of \cite{Andrews_monograph}, Ch. 22 of \cite{Zwiebach}, and work of the author and his collaborators through the Emory Working Group in Number Theory and Molecular Simulation \cite{Kindt}.} %\ref{ch8ch4qbracket}. %, which could be at play in these extreme conditions. % under the action of the $q$-bracket of Bloch and Okounkov \cite{BlochOkounkov, Zagier}, a $q$-series operator that arises out of statistical mechanics and partition theory \cite{Zagier}. 
%We take the somewhat poetic stance %, based on  ``evidence'' in the literature, % to topics in physics, 
%we suggest 
%that $q$-hypergeometric series such as these --- along with their wonderful transformations, interrelations and structure%, and connections% to combinatorics, modular forms, representation theory
% --- are woven into the fabric of nature. %physical 
%reality.

%For brevity, we refer the reader to the author's related paper \cite{Robert_bracket} for standard notations and terminology, and move directly to introducing the principal actors. 

%Let us fix notations and concepts. % before moving on to our main results. 
%We take $q,z\in \mathbb C, |q|<1, z\neq 0$ throughout, unless stated otherwise. % (as below in Section \ref{ch8Sect2}). 
%The triple product is defined by % (using the notation of \cite{HickersonMortenson}) by% the product of $q$-Pochhammer symbols
%\begin{equation}\label{ch8jtp}
%j(z;q):=(z;q)_{\infty}(z^{-1}q;q)_{\infty}(q;q)_{\infty}.
%\end{equation}

The odd-order {\it universal mock theta function} $g_3(z,q)$ of Gordon and McIntosh \cite{GordonMcIntosh}, which specializes to Ramanujan's original list of mock theta functions up to changes of variables and multiplication by rational powers of $q$ and $z$ (with $z$ a rational power of $q$ times a root of unity), is defined as
\begin{equation}\label{ch8umtf}
g_3(z,q):=\sum_{n=1}^{\infty} \frac{q^{n(n-1)}}{(z;q)_{n}(z^{-1}q;q)_{n}},
\end{equation}
and, like the triple product, is subject to all sorts of wonderful transformations.\footnote{We note the simple pole at $z=1$.} %(Note the similarities in the $q$-Pochhammer symbols between the above equations.) 

Let us recall that a {\it unimodal sequence} of integers is of the type %(or {\it weakly unimodal}) if it is of the type
\begin{equation*}
0\leq a_1\leq a_2\leq...\leq a_r \leq \overline{c}\geq b_1\geq b_2\geq...\geq b_s\geq 0.
\end{equation*}
The term $\overline{c}$ is called the {\it peak} of the sequence; generalizing this concept, if $\overline{c}$ occurs with multiplicity $\geq k$, we might consider the unimodal sequence with a {\it $k$-fold peak}  
\begin{equation*}
0\leq a_1\leq a_2\leq...\leq a_r \leq \overline{c}\  \overline{c}\  ...\  \overline{c} \geq b_1\geq b_2\geq...\geq b_s\geq 0,
\end{equation*}
where ``$\overline{c}\  \overline{c}\  ...\  \overline{c}$'' denotes $k$ repetitions of $\overline{c}$. %, as J. Lovejoy discussed with the author \cite{Lovejoy_private}. 
When all the inequalities above %in (\ref{ch8unimodaldef}) 
are strictly ``$<$'' or ``$>$'' the sequence is {\it strongly unimodal}. 

If $r$ is the number of $a_i$ to the left and $s$ is the number of $b_j$ to the right of a unimodal sequence, the difference %statistic 
$s-r$ is called the {\it rank} of the sequence; and the sum of all the terms including the peak is the {\it weight} of the sequence. Another series that plays a role here is the {\it rank generating function $\widetilde{U}(z,q)$ for %weakly 
unimodal sequences}, % (or ``unimodal rank generating function''), 
given by %{\bf RPS: Do I need $q^{n+1}$ instead in the numerator?}% If so, update all subsequent related results.} 
\begin{equation}\label{ch8U}
\widetilde{U}(z,q):=\sum_{n=0}^{\infty}\frac{q^n}{(zq;q)_n (z^{-1}q;q)_n}= \sum_{n=0}^{\infty}\sum_{m=-\infty}^{\infty}\widetilde{u}(m,n)z^m q^n,
\end{equation}
where $\widetilde{u}(m,n)$ is the number of unimodal sequences of rank $m$ and weight $n$. Each summand of the first infinite series is the generating function for unimodal sequences with peak term $n$: the factor $(z^{-1}q;q)_n^{-1} $ generates $a_i\leq n$, $(zq;q)_n^{-1}$ generates $b_j \leq n$ and the $q^n$ factor inserts $n$ as the peak term $\overline{c}$ (following \cite{BOPR, KimLovejoy}). If we replace $z$ with $-z$, the right-most series is actually the very first expression Andrews revealed from Ramanujan's ``lost'' notebook (\cite{AndrewsLost}, Eq. 1.1) shortly after unearthing the papers at Trinity College \cite{SchneiderLost}. %; it is related to %(among a diversity of shapes) 
This form, which is related to partial theta functions \cite{KimLovejoy}, was swimming alongside mock theta functions in the Indian mathematician's imagination during his final year. %; %; % the pages containing the work from Ramanujan's final year in India. 
%{\bf RPS: Fix this sentence...}  (noting there is not a pole at $z=1$ in this case), and 
%in fact, the closely related rank generating function $U(z,q)$ for {\it strongly} unimodal sequences has been recently connected to the mock theta function $f(q)$ \cite{BOPR, FolsomOnoRhoades}. % through the theory of quantum modularity \cite{FolsomOnoRhoades}.% {\bf RPS: Check this reference for accuracy}; it is also known that $U(i;q)$ is a mixed mock modular form and $U(-1,q)$ is quantum modular \cite{BOPR}.) % and the ``strange'' function (\ref{ch8strange}) of Kontsevich \cite{BOPR}. 
%
%Let $\mathcal P$ denote the set of integer partitions $\lambda = (\lambda_1, \lambda_2,...,\lambda_r)$ for $\lambda_i \in \mathbb Z^+$, $\lambda_1 \geq \lambda_2\geq ...\geq \lambda_r\geq 1$. We let $\ell (\lambda):=r$ be the {\it length} of $\lambda$, let $|\lambda|$ denote the {\it size} (i.e., the number being partitioned), and take ``$\lambda \vdash n$'' to mean $\lambda$ is a partition of $n\geq 1$.  
Finally, following Bloch--Okounkov \cite{BlochOkounkov} as well as Zagier \cite{Zagier}, we define the {\it $q$-bracket} $\left<f\right>_q$ of a function $f\colon \mathcal P \to \mathbb C$ %, with $\mathcal P$ the set of integer partitions, 
to be given by
\begin{equation}\label{ch8qbracket}
\left<f\right>_q:\  =\  \frac{\sum_{\lambda\in \mathcal P}f(\lambda)q^{|\lambda|}}{\sum_{\lambda\in \mathcal P} q^{|\lambda|}}\  =\  (q;q)_{\infty}\sum_{\lambda\in \mathcal P}f(\lambda)q^{|\lambda|},
\end{equation}
where the sums are taken over all partitions. This operator represents the expected value in statistical physics of a measurement over a grand ensemble whose states are indexed by partitions with weights $f$, for a canonical choice of $q$; this is the content of the quotient in the middle of (\ref{ch8qbracket}).

%It is marvelous that the $q$-bracket provides a physical interpretation for multiplying power series by $(q;q)_{\infty}$. 
However, we proceed formally here using the right-most expression, %viewing the $q$-bracket as a fascinating object in the theory of $q$-series, 
without drawing too much physical interpretation (while always keeping the mysterious feeling that our formulas resonate in physical reality). %ripple in deep space and the quantum ocean). 
Simply multiplying by $(q;q)_{\infty}$ induces quite interesting $q$-series phenomena:
Bloch--Okounkov \cite{BlochOkounkov}, Zagier \cite{Zagier}, and Griffin--Jameson--Trebat-Leder \cite{GJTL} show that the $q$-bracket can produce families of modular, quasimodular and $p$-adic modular forms; and the present author finds the $q$-bracket to play a natural role in partition theory as well \cite{Robert_bracket, Tanay}, modularity aside. (We highly recommend Zagier's paper \cite{Zagier} for more about the $q$-bracket.) 
 %Unlike previous usages of $\left<f\right>_{q}$, we will allow that $f(\lambda)=f(\lambda, q)$ may also depend on $q$ (having dealt fully with standard coefficients in \cite{Robert_bracket}). 

We will see here that the reciprocal of the Jacobi triple product 
%We will apply the $q$-bracket to the triple product; actually, it is the reciprocal 
\begin{equation*}
j(z;q)^{-1}=:\sum_{\lambda \in \mathcal P} j_z(\lambda) q^{|\lambda|}
\end{equation*}
has a very rich and interesting interpretation in terms of the $q$-bracket operator, which (multiplying $j(z;q)^{-1}$ by $(q;q)_{\infty}$) has the shape  %.  an interesting object in its own right, and the associated $q$-bracket
\begin{equation*}
\left<j_z\right>_q=\frac{1}{(z;q)_{\infty}(z^{-1}q;q)_{\infty}}.
\end{equation*} 
%that we address. 
Note that this $q$-bracket also has a simple pole at $z=1$. We abuse notations somewhat in writing the coefficients $j_z$ in this way, as if $z\in \mathbb C$ were a constant. %, in order to put them in a convenient form for writing the $q$-bracket. 
In fact, 
$j_z$ is a map from the partitions to $\mathbb Z[z]$, which we found in Chapter 3 %the author finds in \cite{Robert_bracket} 
to be given explicitly by
\begin{equation}
j_z(\lambda)=(1-z)^{-1}\sum_{\delta|\lambda}\sum_{\varepsilon|\delta}z^{\operatorname{crk}(\varepsilon)}
\end{equation}
for $z\neq1$, %where ``$\alpha | \beta$'' means $\alpha\in\mathcal P$ is a {\it subpartition} of $\beta\in\mathcal P$ (all the parts of $\alpha$ are also parts of $\beta$), 
and ``$\operatorname{crk}$'' is the {\it crank} statistic of Andrews--Garvan \cite{AndrewsGarvan} from Definition \ref{ch8crk}. %defined in Chapter 3. % below.%\footnote{The methods of \cite{Robert_bracket} will yield a slightly more complicated formula for the coefficients of $j(z;q)$ itself.} % (originally conjectured by Dyson \cite{Dyson})
%, defined as follows. 

%\begin{definition}\label{ch8crk}
%The {\it crank} $\operatorname{crk}(\lambda)$ of a partition $\lambda$ is equal to its largest part %$\operatorname{lg}(\lambda)$ of partition $\lambda$ 
%if the multiplicity $m_1(\lambda)$ of 1 as a part of $\lambda$ is $=0$ (that is, there are no $1$'s), and if $m_1(\lambda)> 0$ then $\operatorname{crk}(\lambda)=\#\{\text{parts of $\lambda$ that are larger than $m_1(\lambda)$\}}-m_1(\lambda).$
%\end{definition}

\begin{remark}
The {crank generating function} \eqref{ch8crk} can be written %is given by
\begin{equation*}
C(z;q)=\frac{(q;q)_{\infty}}{(zq;q)_{\infty} (z^{-1}q;q)_{\infty}}=(1-z)(q;q)_{\infty}\left<j_z\right>_q.
\end{equation*} 
%We recall from Chapter 3 that expanding $C(z;q)$ as a power series in $q$, the $n$th coefficient is given by $\sum_{\lambda \vdash n}z^{\operatorname{crk}(\lambda)}$. 
\end{remark}

 %, to highlight that the series is indexed by partitions $\lambda$; in fact $z\in \mathbb C$ is a variable and $j_z(\lambda)$ is a polynomial in $z$.) 

In Chapter 3 we used the $q$-bracket operator to find the coefficients % $j_z(\lambda)$, as well as the coefficients 
of $\left<j_z\right>_q$ explicitly in terms of sums over subpartitions and the crank statistic, as well. Now we take a different approach, and look at $\left<j_z\right>_q$ from the point-of-view of $q$-hypergeometric relations. It turns out the odd-order universal mock theta function $g_3$ (in an ``inverted'' form) and the unimodal rank generating function $\widetilde{U}$ naturally arise together as components of $\left<j_z\right>_q$.

\begin{theorem}\label{ch8theorem1} 
For $0<|q|<1, z\neq 0,z\neq 1$, the following statements are true:
\begin{enumerate}[(i)]
            \item \label{ch8thm1.1}
            We have the $q$-bracket formula% relating the universal mock theta function $g_3$ and the unimodal rank generating function $U$:
            \begin{equation*}\left<j_z\right>_q=1\  +\  \left[z(1-q)+z^{-1}q\right] g_3(z^{-1},q^{-1})\  +\  \frac{zq^2}{1-z}\widetilde{U}(z,q).\end{equation*}
 
            \item The ``inverted'' mock theta function component in part (\ref{ch8thm1.1}) converges, and can be written in the form
\begin{equation*}\label{ch8umtfsum}
g_3(z^{-1},q^{-1})=\sum_{n=1}^{\infty}\frac{q^n}{(z;q)_n (z^{-1}q;q)_n}.
\end{equation*}
        \end{enumerate}

\end{theorem}
\  

By considering the factor $z(1-q)+z^{-1}q$ as $|z|\to \infty$ and as $|z|\to 0$ in part (\ref{ch8thm1.1}) of Theorem \ref{ch8theorem1}, we get the following asymptotics.

\begin{corollary}\label{ch8cor1}
We have the asymptotic estimates: 
\begin{enumerate}[(i)]
\item For $0<|q|<1 \ll |z|$, we have
\begin{equation*}
\left<j_z\right>_q\sim z(1-q) g_3(z^{-1},q^{-1})\  \text{as}\  |z|\to \infty.
\end{equation*}
\item For $0< |q|<1,\  0<|z|\ll 1$, we have 
\begin{equation*}
\left<j_z\right>_q\sim z^{-1}q \  g_3(z^{-1},q^{-1})\  \text{as}\  |z|\to 0.
\end{equation*}
\end{enumerate}
\end{corollary}

Thus the inverted mock theta function component %as a function of $z$ 
dominates the behavior of the $q$-bracket for $z$ not close to the unit circle (which is ``most'' of the complex plane). 

\begin{remark}
So the universal mock theta function is the main influence on these expected values for large and small $|z|$, with appropriate choice of $q$.%\footnote{For instance, take $q=e^{-1/kT}$ where $k$ is Boltzmann's constant and $T$ is the temperature of a system. %We note that if we set $q=e^{-ct}$ for some real constant $c>0$ and let $t\to \infty$, the term $z(1-e^{-ct})+z^{-1}e^{-ct}$ in Theorem 1(\ref{ch8thm1.1}) behaves similarly to Newton's law of cooling.
%} 
\end{remark}

Conversely, if we write 
\begin{equation*}
\left<j_z\right>_q=:\sum_{n=0}^{\infty}c_n q^n,\  \  \  \  g_3(z^{-1},q^{-1})=:\sum_{n=0}^{\infty}\gamma_n q^n,
\end{equation*} 
where the coefficients $c_n=c_n(z),\  \gamma_n=\gamma_n(z)$ also depend on $z$, % (of course, $c_0$ identically $=1$ from Theorem \ref{ch8theorem1} (\ref{ch8thm1.1})), 
then we proved an explicit combinatorial formula for the $c_n$ in Chapter 3 using nested sums over subpartitions of $n$, viz.
\begin{equation}\label{ch8coeffs}
c_n(z)={(1-z)^{-1}}\sum_{\lambda\vdash n}\sum_{\delta | \lambda}\sum_{\epsilon | \delta} \sum_{\varphi | \epsilon}\mu(\lambda / \delta)z^{\operatorname{crk}(\varphi)}.
\end{equation}

With (\ref{ch8coeffs}) in hand, it follows from Corollary \ref{ch8cor1} that the coefficients of $g_3(z^{-1},q^{-1})$ satisfy the asymptotic 
\begin{equation}
\gamma_n(z) \sim \left\{
        \begin{array}{ll}
            z^{-1} (c_1+c_2+...+c_n)\  \text{as}\ |z|\to \infty\\
            \  \\  
            z c_{n-1}\   \text{as}\  |z|\to 0,\  n\geq 1
        \end{array}
    \right.
\end{equation} 
(which depends entirely on the growth of $z$, not $n$), as the coefficients enjoy the recursion
\begin{equation*}
\gamma_n-\gamma_{n-1}\sim z^{-1}c_n\  \text{for}\  |z|\gg 1.
\end{equation*}
%By equation (\ref{ch8coeffs}), these asymptotics are explicit. 

%\begin{remark} %Another observation: i
It is a well-known fact (see, for instance, \cite{Ono_web}) that if $\zeta_*\neq 1$ is a root of unity, then $$(\zeta_* q;q)_{\infty} (\zeta_*^{-1} q;q)_{\infty} $$ is, up to multiplication by a rational power of $q$, a modular function; % (for instance, see \cite{Ono_web}); 
but this product is the reciprocal of $$(1-\zeta_*)\cdot \left<j_z\right>_q \bigr|_{z=\zeta_*}.$$ This is another example of the intersection of the $q$-bracket with modularity phenomena, and at the same time gives a feeling for the obstruction to the inverted mock theta function's sharing in this modularity at $z=\zeta_*$; for $g_3(z^{-1},q^{-1})$ is not necessarily a dominating aspect of $\left<j_z\right>_q$ for $z\neq 1$ near the unit circle, whereas the unimodal rank generating aspect $\widetilde{U}(z,q)$ makes a more noticeable contribution, and the two pieces work together to produce modular behavior.% --- somewhat like the balance between the holomorphic and non-holomorphic parts of a harmonic Maass form \cite{OnoVisions}.  
%\end{remark}

Going a little farther in this direction, there is a close relation between $g_3$ and the more general class of $k$-fold unimodal rank generating functions. Let us define the {\it rank generating function $\widetilde{U}_k(z,q)$ for unimodal sequences with a $k$-fold peak} by the series
\begin{equation}\label{ch8U_r}
\widetilde{U}_k(z,q):=\sum_{n=0}^{\infty}\frac{q^{kn}}{(zq;q)_n (z^{-1}q;q)_{n}}= \sum_{n=0}^{\infty}\sum_{m=-\infty}^{\infty}\widetilde{u}_k(m,n)z^m q^n,
\end{equation}
where $\widetilde{u}_k(m,n)$ is the number of $k$-fold peak unimodal sequences of rank $m$ and weight $n$. This identity follows directly from the combinatorial definition of $\widetilde{U}_k$, as Lovejoy noted to the author\footnote{J. Lovejoy, Private communication, August 3, 2016.}: the $(z^{-1}q;q)_n^{-1} $ and $(zq;q)_n^{-1}$ generate the $a_i,b_j$ just as in (\ref{ch8U}), and $q^{kn}$ inserts $k$ copies of $n$ as the $k$-fold peak.  %(The $k=1$ case equals $\widetilde{U}(z,q)$ and the $k=2$ case is $W(z,q)$ from \cite{KimLovejoy}.) %, which Lovejoy confirms \cite{Lovejoy_private}. %We note that $U_r(z,q)=U_r(z^{-1},q)$ for $r\geq1$.

Then it is not hard to find (see Theorem \ref{ch41.1}) relations like %{\bf RPS: Need to check unimodal equations from here.}
\begin{equation}\label{ch8U_1, U_2 modular}
\frac{1}{(zq;q)_{\infty}(z^{-1}q;q)_{\infty}}=2-z-z^{-1}+(z+z^{-1})\widetilde{U}_1(z,q)-\widetilde{U}_2(z,q),
\end{equation}
which of course is equal to $(1-z)\left<j_z\right>_q$ and is modular for $z=\zeta_*$, up to multiplication by a power of $q$. %In this case, at $z=\zeta_*$ it is $U_1,U_2$ that together produce modularity, as with $g_3$ and $U=U_1$ previously. Exponentiating both sides of (\ref{ch8U_1, U_2 modular}) will lead to more complicated relations involving $U_1, U_2$, and $U_r$ for other $r>2$, %by \cite{Robert_zeta}, Thm. 1.1, 
%including modular relations at $z$ a root of unity. 
%\begin{remark}
For example, noting that $z+z^{-1}=0$ when $z=i$, then (\ref{ch8U_1, U_2 modular}) yields  
\begin{equation}\label{ch8strongU_1, U_2 modular.rmk1}
2-\widetilde{U}_2(i,q)=(iq;q)_{\infty}^{-1}(-iq;q)_{\infty}^{-1}=(-q^2;q^2)_{\infty}^{-1},
\end{equation}
where $(-q^2;q^2)_{\infty}$ is essentially a modular function.
%The right-hand side of (\ref{ch8strongU_1, U_2 modular.rmk1}) is modular, up to multiplication by a power of $q$, and the constant $2$ is a very ``nice'' function; thus $\widetilde{U}_2(i,q)$ is a quantum modular form. %which is a modular form. 
%\end{remark}

At this point we can compare (\ref{ch8U_1, U_2 modular}) to Theorem \ref{ch8theorem1}(\ref{ch8thm1.1}) %if we wish 
to solve for $g_3(z^{-1},q^{-1})$ in terms of $\widetilde{U}_1,\widetilde{U}_2$, but it is a little messy. 
%\begin{equation*}
%g_3(z^{-1},q^{-1})=\frac{1-z^{-1}+\left(z(1-q^2)+z^{-1}\right)\widetilde{U}_1(z,q)-\widetilde{U}_2(z,q)}{(1-z) \left(z(1-q)+z^{-1}q\right)}
%\end{equation*}
However, it follows from a convenient rewriting of the right-hand side of Theorem \ref{ch8theorem1}(\ref{ch8umtfsum}) using geometric series
\begin{equation*}
\sum_{n=0}^{\infty}\frac{z}{(z;q)_{n+1} (z^{-1}q;q)_{n}}\left( \frac{z^{-1}q^{n+1}}{1-z^{-1}q^{n+1}}\right)=\frac{z}{1-z} \sum_{n=0}^{\infty}\sum_{k=1}^{\infty}\frac{z^{-k}q^{k(n+1)}}{(zq;q)_n (z^{-1}q;q)_n}
\end{equation*}
which converges absolutely for $|q|<|z|$, and then swapping order of summation, that in fact $g_3(z^{-1},q^{-1})$ can be written nicely in terms of the $\widetilde{U}_k$.

\begin{corollary}\label{ch8another_cor}
For $|q|<1<|z|$, we have
\begin{equation*}
g_3(z^{-1},q^{-1})=\frac{z}{1-z}\sum_{k=1}^{\infty}\widetilde{U}_{k}(z,q) z^{-k}q^k.
\end{equation*}
\end{corollary}

%\begin{remark}
%Replacing $z$ with $z^{-1}$ throughout (and noting $\widetilde{U}_k(z^{-1},q)=\widetilde{U}_k(z,q)$), then the resulting formula for $g_3(z,q^{-1})$ converges when $0<|z|<1$.
%%\begin{equation*}
%%g_3(z,q^{-1})=(z-1)^{-1}\sum_{r=1}^{\infty}U_{r}(z,q) z^{r}.
%%\end{equation*} 
%\end{remark}

Thus the inverted universal mock theta function leads to a type of two-variable generating function for the sequence of rank generating functions for unimodal sequences with $k$-fold peaks, $k=1,2,3,...$. %, and also presumably (but we have not proved this) inherits analytic properties of the $U_k$ with respect to the variable $q$, which only appears inside $U_k(z,q)$ on the right-hand side.% (and also inherits to an approximate extent --- especially when $|z|$ and $|q|$ are fairly close to $1$ --- some of their properties).

%%%%Proof

\begin{proof} [Proof of Theorem \ref{ch8theorem1}]
We begin by noting for $|q|<1, z\neq0$, 
\begin{equation*}%\label{ch8prodeq}
\left<j_z\right>_q  = (z;q)_{\infty}^{-1} (z^{-1} q; q)_{\infty}^{-1} =\prod_{n=0}^{\infty}\left(1-q^n (z+z^{-1}q-q^{n+1})\right)^{-1},
\end{equation*} 
where in the final step we multiplied together the $n$th terms from each $q$-Pochhammer symbol. Thus we have 
\begin{equation}\label{ch8equivalent}
\prod_{n=0}^{\infty}\left(1-q^n (z+z^{-1}q-q^{n+1})\right)^{-1}=1+\sum_{n=1}^{\infty}\frac{q^n (z+z^{-1}q-q^{n+1})}{\prod_{j=0}^{n-1}\left(1-q^j (z+z^{-1}q-q^{j+1})\right)},
\end{equation}
which is easily seen to be absolutely convergent, and can be shown by expanding the product on the left as the telescoping series %{\bf DOUBLE XHECK INDICES AND POWERS OF Q IN DENOMINATOR} 
\begin{equation}\label{ch8telescoping}
1+\sum_{n=1}^{\infty}\left(\frac{1}{\prod_{i=0}^{n}\left(1-q^i (z+z^{-1}q-q^i)\right)}-\frac{1}{\prod_{i=0}^{n-1}\left(1-q^{i-1} (z+z^{-1}q-q^{i-1})\right)}  \right)
\end{equation}
with a little arithmetic. % (for more details see the proof of Theorem 1.1 (2) in Chapter 4). 
Now, by the above considerations, (\ref{ch8equivalent}) is equivalent to the following relation.

\begin{lemma}\label{ch8otherlemma} 
For $|q|<1, z\neq 0$, we have
\begin{equation*}%\label{ch8other}
\left<j_z\right>_q=1+(z+z^{-1}q)\sum_{n=1}^{\infty}\frac{q^n}{(z;q)_n (z^{-1}q;q)_n}-q\sum_{n=1}^{\infty}\frac{q^{2n}}{(z;q)_n (z^{-1}q;q)_n}.
\end{equation*}
\end{lemma}

We cannot help but notice how %strikingly 
both series on the right-hand side of Lemma \ref{ch8otherlemma} resemble the right-hand summation of identity (\ref{ch8U}) for $\widetilde{U}(z,q)$. This is not a coincidence; it follows right away from the simple observation %{\bf RPS: Need to verify $q^{n}$ or $q^{n+1}$...} 
\begin{equation*}%\label{ch8Usplit}
\widetilde{U}(z,q)=\sum_{n=0}^{\infty}\frac{q^n}{(zq;q)_n (z^{-1}q;q)_n}=q^{-1}(1-z)\sum_{n=0}^{\infty}\frac{q^{n+1}(1-z^{-1}q^{n+1})}{(z;q)_{n+1} (z^{-1}q;q)_{n+1}},
\end{equation*}
that $\widetilde{U}$ splits off in a very similar fashion %\footnote{By consideration of Lemma \ref{ch8otherlemma} together with equation (\ref{ch8other2}) we wonder: do combinations of sums of the shape $\sum_{n\geq 1} q^{a n}(z;q)_n^{-1} (z^{-1}q;q)_n^{-1}$, with $a\geq 1$, generate other interesting objects? %Work of Kim--Lovejoy \cite{KimLovejoy} connects sums like this to partial theta functions.
%}
to $\left<j_z\right>_q$ in Lemma \ref{ch8otherlemma}, after taking into account $q\neq 0$:
\begin{equation}\label{ch8other2}
\widetilde{U}(z,q)=q^{-1}(1-z)\sum_{n=1}^{\infty}\frac{q^n}{(z;q)_n (z^{-1}q;q)_n}-(zq)^{-1}(1-z)\sum_{n=1}^{\infty}\frac{q^{2n}}{(z;q)_n (z^{-1}q;q)_n}.
\end{equation}
Comparing Lemma \ref{ch8otherlemma} and (\ref{ch8other2}), % --- another striking resemblance --- 
plus a little bit of algebra, then gives 
\begin{equation}\label{ch8midway}
\left<j_z\right>_q=1\  +\  \left[z(1-q)+z^{-1}q\right]\sum_{n=1}^{\infty}\frac{q^n}{(z;q)_n (z^{-1}q;q)_n}\  +\  \frac{zq^2}{1-z}\widetilde{U}(z,q).
\end{equation}

%\begin{remark} 
%As noted in the remark following Theorem \ref{ch8theorem1}, the summation on the far right is very similar to $q^3(1-z)^{-1}W(z,q)$ for the partial theta function $W$ in \cite{KimLovejoy2}, Eq. (2.1):
%\begin{equation}
%W(z,q):=\sum_{n=0}^{\infty}\frac{q^{2n}}{(zq;q)_n(z^{-1}q;q)_n}=q^{-3}(1-z)\sum_{n=1}^{\infty}\frac{q^{2n+1}}{(z;q)_n(z^{-1}q;q)_{n-1}}
%\end{equation}
%\end{remark}
%
%Now, the right-most summation in Lemma \ref{ch8otherlemma} is obviously $\mathcal O(q^3)$, by inspection of the leading term, proving that point of the theorem. 

Now, to connect the remaining summation in (\ref{ch8midway}) to the universal mock theta function $g_3$, we apply a somewhat clever factorization strategy 
in the $q$-Pochhammer symbols to arrive at a useful identity (see \cite{GR}, Appendix 1 (I.3)): % used by L. Rolen and the author in \cite{RolenSchneider}, eq. (4.1):%, that the author and his collaborator used to study quantum modularity in \cite{RolenSchneider}:
\begin{flalign}\label{ch8invertmtf}
\begin{split}
(z;q)_n (z^{-1}q;q)_n%&:=\prod_{j=0}^{n-1} (1-zq^j)(1-z^{-1}q^{j+1})\\
&=\prod_{j=0}^{n-1} \left[(-zq^j)(1-z^{-1}(q^{-1})^{j})\right] \left[(-z^{-1}q^{j+1})(1-z(q^{-1})^{j+1})\right]\\
&=q^{n^2}(z^{-1};q^{-1})_n (zq^{-1};q^{-1})_n.\\
\end{split}
\end{flalign}
Thus 
\begin{flalign}\label{ch8invert2}
\sum_{n=1}^{\infty}\frac{q^n}{(z;q)_n (z^{-1}q;q)_n}&=\sum_{n=1}^{\infty}\frac{q^n}{q^{n^2}(z^{-1};q^{-1})_n (zq^{-1};q^{-1})_n}\\&=\sum_{n=1}^{\infty}\frac{(q^{-1})^{n(n-1)}}{(z^{-1};q^{-1})_n (zq^{-1};q^{-1})_n}.
\end{flalign}
The right-hand side of (\ref{ch8invert2}) is $g_3(z^{-1},q^{-1})$, noting that it converges under the same conditions as the left side (being merely a term-wise rewriting), but with $q=0$ omitted from the domain.%, completes the proof of Theorem \ref{ch8theorem1}. 

\begin{remark}
Equivalently, identities like these result from the observation that
\begin{equation*}
(1-zq^i)(1-z^{-1}q^{-i})^{-1}=-zq^i.
\end{equation*}
Taking the product over $0\leq i \leq n-1$ gives 
\begin{equation*}
(z;q)_n (z^{-1};q^{-1})_n^{-1}=(-1)^n z^n q^{n(n-1)/2}
\end{equation*} and, proceeding in this manner, a variety of $q$-series summand forms can be produced (and inverted as above) by creative manipulation.  
\end{remark}
\end{proof}

\begin{remark}
We note in passing that, using (7.2) and (8.2) of Fine \cite{Fine}, Ch. 1, together with Theorem \ref{ch8theorem1}(\ref{ch8umtfsum}), we can also write
\begin{flalign}
\begin{split}
g_3(z^{-1},q^{-1})\  =\  (z^{-1}q;q)_{\infty}^{-1}(-z;q)_{\infty}^{-1}\sum_{n=0}^{\infty} (-1)^nz^{-2n}q^{\frac{n(n+1)}{2}}-\sum_{n=0}^{\infty} z^{-n+1}(z^{-1};q)_{n}.\\
\end{split}
\end{flalign}
\end{remark}

Recall that many modular forms arise as specializations of $j(z;q)$ (because $j(z;q)$ is essentially a Jacobi form, see \cite{BFOR}), % and these forms have this property \cite{BFOR}, 
and that $g_3(z,q)$ is the prototype for the class of mock modular forms that (using Ramanujan's language) 
``enter into mathematics as beautifully'' as the modular cases \cite{Hardy}. It is interesting that these important number-theoretic objects which are speculatively associated in the literature %\cite{Cameron,DGO} 
to opposite extremes of the universe --- subatomic and supermassive --- are themselves intertwined via the $q$-bracket from statistical physics, which applies to phenomena at every scale.%, and the theory of $q$-hypergeometric functions. 
%{\bf RPS: Connection between unimodal sequences and Bell curves in statistics?}
\  
\  
\

\section{Approaching roots of unity radially from within (and without)}\label{ch8Sect2}

One point that arises in (\ref{ch8invertmtf}) and (\ref{ch8invert2}) above is that, evidently, one can construct pairs of $q$-series $\varphi_1(q), \varphi_2(q)$, convergent for $|q|<1$, with the property   
\begin{equation}\label{ch8inversion}
\varphi_1(q)=\varphi_2(q^{-1})
\end{equation} 
({thus $\varphi_1(q)+\varphi_2(q),\  \varphi_1(q)\varphi_2(q)$ are self-reciprocal. %\footnote{See \cite{LectureHall}}}). 
This type of phenomenon, relating functions inside and outside the unit disk, is studied in \cite{BFR1,Folsom}. In particular, the universal mock theta function $g_3$ can be written %(at least formally) 
as a piecewise function
\begin{equation}\label{ch8umtfquantum}
g_3(z,q)= \left\{
        \begin{array}{ll}
            \sum_{n=1}^{\infty}\frac{q^{n(n-1)}}{(z;q)_n (z^{-1}q;q)_n} & \text{if}\  |q|<1,\\
            \  \\  
            \sum_{n=1}^{\infty}\frac{(q^{-1})^n}{(z;q^{-1})_n (z^{-1}q^{-1};q^{-1})_n} & \text{if}\  |q|>1,
        \end{array}
    \right.
\end{equation} 
for $q$ inside or outside the unit circle, respectively, and $z\neq 0\  \text{or}\  1$.\footnote{We note this does not yield analytic continuation as the unit circle presents a wall of singularities.} What of $g_3(z,q)$ for $q$ lying {\it on} the circle? Generically one expects this question to be somewhat dicey. %, as the unit circle often presents an essentially singular boundary for $q$-series. %problematic, of course.

To be precise in what follows, for $\zeta$ on the unit circle we define $g_3(z,\zeta)$ to mean the limit of $g_3(z,q)$ as $q\to \zeta$  radially from within (or without if the context allows), when the limit exists. Recalling the notation $\zeta_m:=e^{2\pi \text{i}/m}$, it turns out that for $\zeta=\zeta_*$ an appropriate root of unity, $g_3(z,\zeta_*)$ is finite, both in value and length of the sum. %(We do not ask whether $g_3$ generally satisfies the requirements of quantum modularity here; this probably depends on $z$ as well as $q$.)

\begin{theorem}\label{ch8theorem2}
For $q=\zeta_m$ a primitive $m$th root of unity, $z\neq 0, 1$, or a rational power of %\  \text{or a rational power of}\  
$\zeta_m$, and $z^m+z^{-m}\neq 1$, the odd-order universal mock theta function is given by the finite formula
\begin{equation*}
g_3(z,\zeta_m)\  =\  (1-z^m-z^{-m})^{-1}\  \sum_{n=0}^{m-1}\zeta_m^{n}\  (z;\zeta_m)_{n}(z^{-1}\zeta_m;\zeta_m)_{n}.
\end{equation*}
%so long as $|(z^{-1};\zeta_m^{-1})_{r}(z\zeta_m^{-1};\zeta_m^{-1})_{r}|>1$ and $(z^{-1};\zeta_m^{-1})_{r}(z\zeta_m^{-1};\zeta_m^{-1})_{r}\neq 0$ for any $r\geq 1$.
\end{theorem}

\begin{remark}
Bringmann--Rolen \cite{BringmannRolen} and Jang--L\"{o}brich \cite{JangLobrich} have studied radial limits of universal mock theta functions from other perspectives. 
\end{remark}

Thus, under the right conditions, (\ref{ch8umtfquantum}) together with Theorem \ref{ch8theorem2} suggest $g_3(z,q)$ can, in a certain sense, ``pass through'' the unit circle at roots of unity  (as a function of $q$ following a radial path) into the complex plane beyond, and vice versa, while always remaining finite. %Now, this ``passing through'' in the variable $q$ is not smooth; for at $q$ inside the unit circle the function is equal to its ``dual'' function outside the circle evaluated at $q^{-1}$, thus $g_3$ jumps from 

In the theory of quantum modular forms, one encounters functions that exhibit this renormalization behavior (see \cite{BFOR, RolenSchneider}). 
%To record a somewhat precise (yet still somewhat loose) description of quantum modularity, we follow Zagier \cite{Zagierquantum} in giving the following.
 % that feels something like quantum tunneling in physics. 
We see that $g_3$ exhibits this type of 
%``quantum'' 
%renormalization 
behavior.% described above.

Some mock theta functions are closely related to quantum modular forms. As we noted in Chapter 1, Ramanujan's mock theta function $f(q)$ (from (\ref{ch1f(q)def})) is, at even-order roots of unity, essentially a quantum modular form plus a modular form\footnote{We give examples of similar cases in Appendix E.}, through its relation to another rank generating function, the {\it rank generating function $U(z,q)$ for strongly unimodal sequences} \cite{BOPR, FOR}, defined by  
\begin{equation}\label{ch8stronglyU}
U(z,q):=\sum_{n=0}^{\infty}q^{n+1}(-zq;q)_n (-z^{-1}q;q)_n = \sum_{n=0}^{\infty}\sum_{m=-\infty}^{\infty}{u}(m,n)z^m q^n,
\end{equation}
where ${u}(m,n)$ is the number of strongly unimodal sequences of rank $m$ and weight $n$. As with $\widetilde{U}, \widetilde{U}_k$ previously, the identity follows directly from the combinatorial definition: here, the $(-z^{-1}q;q)_n^{-1}$ and $(-zq;q)_n^{-1}$ generate {\it distinct} $a_i\leq n,b_j \leq n$, respectively, and $q^{n+1}$ inserts $n+1$ as the peak term. 

This $U(z,q)$ is a function that strikes deep: up to multiplication by rational powers of $q$, $U(i,q)$ is mock modular, $U(1,q)$ is mixed mock modular, and $U(-1,q)$ is a quantum modular form that can be completed to yield a weight $3/2$ non-holomorphic modular form \cite{BFOR}; in fact, mock and quantum modular properties of $U(z,q)$ are proved in generality for $z$ in an infinite set of roots of unity in \cite{FKTVY}.

%(The reader might notice the formal closeness of (\ref{ch8stronglyU}) to the summation on the right in Theorem \ref{ch8theorem2} if we take $q=\zeta_m$.) %; we will return to this observation in a moment.

Of course, $U$ is the $k=1$ case of the {\it rank generating function $U_k(z,q)$ for strongly unimodal sequences with $k$-fold peak}, given by

\begin{equation}%\label{ch8strongU_r}
U_k(z,q):=\sum_{n=0}^{\infty}q^{k(n+1)}(-zq;q)_n (-z^{-1}q;q)_{n}= \sum_{n=0}^{\infty}\sum_{m=-\infty}^{\infty}{u}_k(m,n)z^m q^n,
\end{equation}
where ${u}_k(m,n)$ counts $k$-fold peak strongly unimodal sequences of rank $m$ and weight $n$, as above. 
Once again, we note the symmetry $U_k(z^{-1},q)=U_k(z,q)$. As with $\widetilde{U}_k$ in (\ref{ch8U_1, U_2 modular}), we can find (see Theorem \ref{ch41.11}) nice relations like
\begin{equation}\label{ch8strongU_1, U_2 modular2}
(zq;q)_{\infty}(z^{-1}q;q)_{\infty}=1-(z+z^{-1}){U}_1(z,q)+{U}_2(z,q),
\end{equation}
which is modular for $z=\zeta_*$ a root of unity, up to multiplication by a power of $q$. 
%\begin{remark}
For instance, at $z=i$, %$z+z^{-1}$ vanishes and 
equation (\ref{ch8strongU_1, U_2 modular2}) gives %the modular form 
\begin{equation}\label{ch8strongU_1, U_2 modular.rmk2}
1+{U}_2(i,q)=(iq;q)_{\infty}(-iq;q)_{\infty}=(-q^2;q^2)_{\infty}.
\end{equation}

\begin{remark}
Multiplying (\ref{ch8strongU_1, U_2 modular.rmk1}) and (\ref{ch8strongU_1, U_2 modular.rmk2}) leads to a nice pair of identities relating $U_2$ and $\widetilde{U}_2$:
%the rank generating functions for both unimodal and strongly unimodal sequences with double peaks:
\begin{equation}
U_2(i,q)=\frac{1-\widetilde{U}_2(i,q)}{\widetilde{U}_2(i,q)-2},\  \  \  \widetilde{U}_2(i,q)=\frac{1+2U_2(i,q)}{1+U_2(i,q)}.
\end{equation}
\end{remark}

Now, taking a similar approach to that in Section \ref{ch8Sect1} with regard to $\widetilde{U}_k$, we can find from Theorem \ref{ch8theorem2}, using an evaluation of $U_k(-z,q)$ at $q=\zeta_m$ much like the theorem\footnote{We note for $k=1, z=1$, $m$ even, that the summation in (\ref{ch8strongU_rfinite}) appears on the right side of (\ref{ch1Watson2}).}
\begin{equation}\label{ch8strongU_rfinite}
U_k(-z,\zeta_m)=\frac{-1}{1-z^m-z^{-m}}\sum_{n=0}^{m-1}\zeta_m^{k(n+1)}(z\zeta_m;\zeta_m)_n (z^{-1}\zeta_m;\zeta_m)_{n},
\end{equation}
that the universal mock theta function $g_3$ also connects to these rank generating functions $U_k$ at roots of unity, through a similar relation to Corollary \ref{ch8another_cor}. 

\begin{corollary}\label{ch8cor2}
For $|z|<1$, we have
\begin{equation*}
g_3(z,\zeta_m)=\frac{z-1}{z}\sum_{k=1}^{\infty} {U}_{k}(-z,\zeta_m)z^{k}\zeta_m^{-k}.
\end{equation*}
\end{corollary}

%\begin{remark}

%We can find finite formulas like (\ref{ch8strongU_rfinite}) for $U_k(z,\zeta_m)$ and $\widetilde{U}_k(z,\zeta_m)$ as well, by the proof below; in particular, then, we have finite formulas for both $U_2(i,\zeta_m)$ and $\widetilde{U}_2(i,\zeta_m)$. Moreover, we know that $(-q^2;q^2)_{\infty}$ is essentially modular of weight $1/2$, and the additive constants $1$ and $2$ on the left-hand sides of (\ref{ch8strongU_1, U_2 modular.rmk2}) and (\ref{ch8strongU_1, U_2 modular.rmk1}), respectively, fit the bill as ``suitably nice'' functions. Then by Definition \ref{ch8qmfdef} of quantum modularity, the following fact is immediate from (\ref{ch8strongU_1, U_2 modular.rmk1}) and (\ref{ch8strongU_1, U_2 modular.rmk2}).
%
%\begin{proposition}\label{ch8unimodal_quantum}
%Up to multiplication by rational powers of $q$, the rank generating functions $U_2(i,q)=U_2(-i,q)$ and $\widetilde{U}_2(i,q)=\widetilde{U}_2(-i,q)$ for unimodal sequences with double peaks 
%are both quantum modular forms.
%\end{proposition}  
%

How suggestive it is, in light of the %quantum modular connections 
relationship between $f(q)$ and $U(-1,q)$ \cite{FOR}, to see specializations of $g_3$ giving rise to both forms of $k$-fold unimodal rank generating functions in Corollaries \ref{ch8another_cor} and \ref{ch8cor2}. % attractive combinatorial interpretation of Ramanujan's mock theta functions. % with $k$-fold peak. 

\begin{proof}[Proof of Theorem \ref{ch8theorem2} and Corollary \ref{ch8cor2}]
We start with an elementary observation. For an arbitrary $q$-series with %arbitrary 
coefficients $d_n$, then in the limit as $q$ approaches an $m$th root of unity $\zeta_m$ radially from within the unit circle, we have
\begin{equation}\label{ch8finite}
\lim_{q\to \zeta_m} \sum_{n=1}^{\infty}d_n q^n=\sum_{n=1}^{m}D_n \zeta_m^n\  \text{where}\  D_n:=\sum_{j=0}^{\infty}d_{n+mj},
\end{equation}
so long as $\sum_j d_{n+mj}$ converges. The moral of this example: $q$-series {\it want} to be finite at roots of unity. %, under the right conditions. 

In a similar direction, Theorem \ref{ch8theorem2} arises from the following, 
very general lemma. % proposition. %lemma. %, which the author has spoken on %during various projects 
%at Emory University since 2012 and used for heuristics, % in projects with collaborators (such as \cite{ClemmSchneider, RolenSchneider}), % and A. Clemm \cite{}, %to study quantum modular forms and mock theta functions 
%with collaborators \cite{ClemmSchneider,RolenSchneider}, 
%but has not published previously. 
It is really Lemma \ref{ch8lemma} that is the pivotal result of Section \ref{ch8Sect2}; the applications to $g_3(z,\zeta_m)$ form an interesting exercise. %and lectured about privately at the time at Emory University, 

\begin{lemma}\label{ch8lemma}
Suppose $\phi\colon \mathbb Z^+ \to \mathbb C$ is a periodic function of period $m\in\mathbb Z^+$, i.e., $\phi(r+mk)=\phi(r)$ for all $k\in \mathbb Z$. Define $f\colon \mathbb Z^+ \to \mathbb C$ by %$f(0):=0$ and, for $j>0$, by 
the product
\begin{equation*}
f(j):=\prod_{i=1}^{j}\phi(i),
\end{equation*}
and its summatory function $F(n)$ by $F(0):=0$ and, for $n\geq 1$,
\begin{equation*}%\label{ch8periodic1}
F(n):=\sum_{j=1}^{n}f(j),\  \  \  \  F(\infty):=\lim_{n\to \infty}F(n)\  \text{if the limit exists}.
\end{equation*}%\label{ch8periodic2}
Then the following statements are true:
\begin{enumerate}[(i)]
\item \label{ch8periodic3} For $0\leq r <m$ we have
\begin{equation*}
F(mk+r)=\frac{1-f(m)^k}{1-f(m)}F(m)+f(m)^k F(r).
\end{equation*}
\item \label{ch8periodic4}
For $|f(m)|<1$ we have the finite formula
\begin{equation*}
F(\infty)=\frac{F(m)}{1-f(m)}.
\end{equation*}
\end{enumerate}
\end{lemma}  
  
\begin{proof}[Proof of Lemma \ref{ch8lemma}]
First we observe that
\begin{equation}\label{ch8periodic5}
f(mk)=\prod_{i=1}^{mk}\phi(i)=\left(\prod_{i=1}^{m}\phi(i)\right)^k=f(m)^k,
\end{equation}
by the periodicity of $\phi$. Then by the definition of $F(n)$ in Lemma \ref{ch8lemma} together with (\ref{ch8periodic5}) we can rewrite\\
\begin{flalign}
\begin{split}
&F(mk+r)\\&=\sum_{j=1}^{m}f(j)+\sum_{j=m+1}^{2m}f(j)+\sum_{j=2m+1}^{3m}f(j)+...+\sum_{j=m(k-1)+1}^{mk}f(j)+\sum_{j=mk+1}^{mk+r}f(j)\\
%&=\sum_{j=1}^{m}f(j)+f(m)\sum_{j=1}^{m}f(j)+f(m)^2\sum_{j=1}^{m}f(j)+...+f(m)^{k-1}\sum_{j=1}^{m}f(j)+f(m)^k\sum_{j=1}^{r}f(j)\\
&=\left(1+f(m)+f(m)^2+...+f(m)^{k-1}\right)\sum_{j=1}^{m}f(j)+f(m)^k\sum_{j=1}^{r}f(j).\\
\end{split}
\end{flalign}
Recognizing the sum $1+f(m)+f(m)^2+...$ as a finite geometric series completes the proof of (\ref{ch8periodic3}). If $|f(m)|<1$, the infinite case gives (\ref{ch8periodic4}).
\end{proof}

\begin{remark}
Euler's continued fraction formula \cite{Euler} allows one to rewrite any hypergeometric sum as a continued fraction, and vice versa. Then we get another finite formula for $F(\infty)$, which holds for any convergent continued fraction of the following shape with periodic coefficients, including $q$-hypergeometric series when $q$ is replaced by appropriate $\zeta_m$: %, viz. % that we can write%, using the notations in Lemma \ref{ch8lemma}, we have
\begin{equation}
F(\infty)=
\frac{{\phi(1)}}{1-\frac{\phi(2)}{1+\phi(2)-\frac{\phi(3)}{1+\phi(3)-\frac{\phi(4)}{1+\phi(4)-...}}}}\  =\  \frac{1}{1-f(m)}\left(\frac{\phi(1)}{1-\frac{\phi(2)}{1+\phi(2)-\frac{\phi(3)}{1+...-\frac{\phi(m)}{1+\phi(m)}}}}\right).
\end{equation}
Therefore, the finiteness and renormalization considerations in this section also apply to $q$-continued fractions.% versions of the $q$-series, given by Euler's formula.
\end{remark}

Clearly if we take $\phi$ to be sine, cosine, etc. in Lemma \ref{ch8lemma}, we can produce a variety of trigonometric identities. More pertinently, if we replace $\phi(i)$ with $\widetilde{\phi}(t,i):=  t \phi(i)$, this $\widetilde{\phi}$ also has period $m$; then we see $\widetilde{f}(j):=\prod_{i=1}^{j}\widetilde{\phi}(t,i)= t^j f(j)$. Thus the summatory functions $\widetilde{F}(n)=\widetilde{F}(t,n)$ and $\widetilde{F}(\infty)=\widetilde{F}(t,\infty)$ represent a polynomial and a power series in $t$, respectively --- which are, respectively, subject to (\ref{ch8periodic3}) and (\ref{ch8periodic4}) of Lemma \ref{ch8lemma}. Then for $\phi$ with period $k$ and the product $f$ as defined above, we get identities like
\begin{equation}
\sum_{n=1}^{\infty}f(n)t^n=\frac{1}{1-f(k)t^k}\sum_{n=1}^{k}f(n)t^n.
\end{equation} 
(We could also take $\widetilde{\phi}(t,i)$ equal to $t^i\phi(i)$ or $t^{2i}\phi(i)$ or $t^{2i-1}\phi(i)$, to lead to power series of other familiar shapes; however, such $\widetilde{\phi}$ are not generally periodic.)

Thinking along these lines, if  we set  
\begin{equation*}
\phi(i)=  t \frac{(1-a_1 q^{i-1})(1-a_2 q^{i-1})...(1-a_r q^{i-1})}{(1-b_1 q^{i-1})(1-b_2 q^{i-1})...(1-b_s q^{i-1})}
\end{equation*}
for $a_*, b_* \in \mathbb C$, the product $f(j)$ becomes a quotient of $q$-Pochhammer symbols, 
%\begin{equation}
%f(j)=  t^j \frac{(a_1;q)_j (a_2;q)_j...(a_r;q)_j}{(b_1;q)_j(b_2;q)_j...(b_s;q)_j},
%\end{equation}
producing the $q$-hypergeometric series
$$F(t,\infty)=\  _{r}F_s (a_1,...,a_r;b_1,...,b_s; t:q).$$ If $q\to \zeta_m$ an $m$th root of unity, then $\phi$ is also cyclic of period $m$,
and in the radial limit $_{r}F_s (a_1,...,a_r;b_1,...,b_s; t:\zeta_m)$ is subject to Lemma \ref{ch8lemma}(\ref{ch8periodic4}), %as well as (\ref{ch8finite}), 
so long as in the denominator $(1-b_* \zeta_m^{i})\neq 0$ for any $i$.

Remembering the ``moral'' of equation (\ref{ch8finite}), then similar considerations apply to almost all $q$-series and Eulerian series, for $q=\zeta_m$ a root of unity that does not produce singularities. In particular, so long as the choice of $z$ also does not lead to singularities, it is immediate from Lemma \ref{ch8lemma} by the definition (\ref{ch8umtf}) of $g_3$ that 
\begin{flalign}\label{ch8simple}
\begin{split}
g_3(z,\zeta_m)&=\frac{1}{1-(z;\zeta_m)_m^{-1}(z^{-1}\zeta_m;\zeta_m)_m^{-1}}\sum_{n=1}^{m} \frac{\zeta_m^{n(n-1)}}{(z;\zeta_m)_{n}(z^{-1}\zeta_m;\zeta_m)_{n}}\\
&=\frac{2-z^m-z^{-m}}{1-z^m-z^{-m}}\sum_{n=1}^{m} \frac{\zeta_m^{n(n-1)}}{(z;\zeta_m)_{n}(z^{-1}\zeta_m;\zeta_m)_{n}},\\
\end{split}
\end{flalign}
where for the final equation we used the elementary fact that $$(X;\zeta_m)_m=1-X^m$$ in the leading factor. For a slightly simpler formula, we apply Lemma \ref{ch8lemma} to the identity for $g_3(z^{-1};\zeta_m^{-1})$ in Theorem \ref{ch8theorem1} instead, then take $z\mapsto z^{-1}$ and $\zeta_m\mapsto \zeta_m^{-1}$, to see
\begin{equation}\label{ch8simpler}
g_3(z,\zeta_m)=\frac{2-z^m-z^{-m}}{1-z^m-z^{-m} }\sum_{n=1}^{m} \frac{\zeta_m^{-n}}{(z^{-1};\zeta_m^{-1})_{n}(z\zeta_m^{-1};\zeta_m^{-1})_{n}}.
\end{equation}
We note that the leading factor is symmetric under inversion of $z$.

\begin{remark} 
Jang--L\"{o}brich prove finite formulas similar %to Theorem \ref{ch8theorem2} 
to (\ref{ch8simple}) and (\ref{ch8simpler}) for $g_3(z,\zeta_m)$ \cite{JangLobrich}, by different methods.% following up on work of Bringmann--Rolen  \cite{BringmannRolen}. % on the even-order mock theta function $g_2(z,q)$ of Gordon and McIntosh (see \cite{GordonMcIntosh}). %Formulas for $g_2(z,\zeta_m)$ also follow from the considerations in this study, but would take us afield of our theme of the $q$-bracket and the Jacobi triple product. 
\end{remark}

A particularly lovely aspect of $q$-series such as these is that they transform into an infinite menagerie of shapes, limited only by the curiosity of the analyst. (For instance, see Fine \cite{Fine} for a stunning exploration of $q$-hypergeometric series.\footnote{Fine writes: 
%
%\begin{quote}
``The beauty and surprising nearness to the surface of some of the results could easily lead one to embark on an almost uncharted investigation of [one's] own.'' (\cite{Fine}, p. xi)}) 
%\end{quote}
Then a form like $g_3$ might have a number of different finite formulas. %, some simpler than others. 

To derive Theorem \ref{ch8theorem2}, which is simpler than the preceding expressions for $g_3$, we use another factorization strategy in the $q$-Pochhammer symbols. Again we exploit that $$(X;\zeta_m)_m=1-X^m=(X;\zeta_m^{-1})_{m};$$ thus for $0\leq n\leq m$, since $\zeta_m^{-j}=\zeta_m^{m-j}$ we have 
\begin{flalign}\label{ch8factorization2}
\begin{split}
(X;\zeta_m^{-1})_n\  &=\  (1-X)(1-X\zeta_m^{m-1})(1-X\zeta_m^{m-2})...(1-X\zeta_m^{m-(n-1)})\\
&=\  \frac{(1-X)(X;\zeta_m)_m}{(X;\zeta_m)_{m-n+1}}\  =\  \frac{(1-X)(1-X^m)}{(X;\zeta_m)_{m-n+1}}.\\
\end{split}
\end{flalign}
Making the change of indices $n\mapsto m-n+1$ in the summation in (\ref{ch8simpler}) then yields
\begin{equation*}
\sum_{n=1}^{m} \frac{\zeta_m^{-(m-n+1)}}{(z^{-1};\zeta_m^{-1})_{m-n+1}(z\zeta_m^{-1};\zeta_m^{-1})_{m-n+1}}=\sum_{n=1}^{m}\frac{\zeta_m^{n-1}(z^{-1};\zeta_m)_{n}(z\zeta_m^{-1};\zeta_m)_{n}}{(1-z^{-1})(1-z\zeta_m^{-1})(2-z^m-z^{-m})}.
\end{equation*}
Substituting this final expression back into (\ref{ch8simpler}), with a little algebra and adjusting of indices, gives Theorem \ref{ch8theorem2}. 

To prove Corollary \ref{ch8cor2}, we use geometric series, along with an order-of-summation swap and index change, to rewrite Theorem \ref{ch8theorem2} in the form
\begin{flalign}
\begin{split}
g_3(z,\zeta_m)&=\frac{1-z}{(1-z^m-z^{-m})}\sum_{n=0}^{m-1}\zeta_m^{n}\frac{(z\zeta_m;\zeta_m)_{n}(z^{-1}\zeta_m;\zeta_m)_{n}}{1-z\zeta_m^n}\\
&=\frac{z^{-1}(1-z)}{z(1-z^m-z^{-m})}\sum_{k=1}^{\infty}z^k\zeta_m^{-k}\sum_{n=0}^{m-1}\zeta_m^{k(n+1)}(\zeta_m;\zeta_m)_{n}(z^{-1}\zeta_m;\zeta_m)_{n}.\\
\end{split}
\end{flalign}
Comparing this with the formula (\ref{ch8strongU_rfinite}) for $U(-z,\zeta_m)$, which follows easily from Lemma \ref{ch8lemma}, gives the corollary. (The sum on the right might be simplified further using (\ref{ch8finite}).)

\begin{remark}
Convergence in these formulas depends on one's choice of substitutions; %and is not automatic; 
for a particular choice, careful analysis may be required to show boundedness %t the natural boundary of convergence %
as $q$ approaches the %unit circle 
natural boundary of a $q$-series %, such as the unit circle forms here 
(see Watson \cite{Watson} for examples). %, %which is the natural boundary of convergence %
%which forms the natural 
%boundary %of convergence %in general presents a boundary of essential singularity 
%for many 
%$q$-hypergeometric series like these. %, %depending on one's choice of substitutions. %; for instance, see the reference to Watson's analysis in the proof of Example \ref{ch8example} below.  %G. Andrews referred the author 
%%The reader is referred to Watson's careful treatment of convergence in \cite{Watson}, Sec. 6, as a model. %: ``There you will see how carefully Watson makes his analytic arguments to guarantee that the limiting values on the unit circle for the fifth order mock theta functions are correct.'' 
%%%\cite{Andrews_private}. 
\end{remark}
\end{proof}

%
%Theorem \ref{ch8theorem3} follows immediately, then, from the study \cite{FolsomOnoRhoades} by Folsom--Ono--Rhoades; the authors show that the {\it rank generating function $U_1(z,q)$ for strongly unimodal sequences}, given by
%\begin{equation}
%U_1(z,q):=\sum_{n=0}^{\infty}q^{n+1}(-z^{-1}q;q)_{n}(-zq;q)_{n}
%\end{equation}
%(which is a mock modular form), INCOMPLETE.   

We note that a slight variation on the proof above leads to finite formulas at applicable roots of unity for the even-order universal mock theta function $g_2(z,q)$ of Gordon--McIntosh \cite{GordonMcIntosh} as well, by an alternative approach to that of Bringmann--Rolen \cite{BringmannRolen}. Using transformations from Andrews \cite{Andrews}, Fine \cite{Fine}, and other authors, still simpler formulas might be found for particular specializations of $g_3$ at roots of unity. We demonstrate this point below.

\begin{example}\label{ch8example}
The limit of the mock theta function $f(q)$ at $\zeta_m$ an odd-order root of unity is given by
\begin{equation*}\label{ch8f_finite}
f(\zeta_m)=1-\frac{2}{3}\sum_{n=1}^{m}(-1)^n \zeta_m^{-(n+1)} (-\zeta_m^{-1};\zeta_m^{-1})_n.
\end{equation*}
\end{example}

\begin{proof}[Proof of Example \ref{ch8example}] The function $f(q)$ is convergent at odd roots of unity; however, for the reader's convenience, we will sketch a proof of convergence to the given value for just the case $q \to \zeta_m$ along a radial path. By (26.22) in \cite{Fine}, Ch. 3, Ramanujan's mock theta function $f(q)$ defined in (\ref{ch1f(q)def}) 
can be rewritten
\begin{equation}\label{ch8f(q)Fine}
f(q)=1-\sum_{n=1}^{\infty}\frac{(-1)^{n}q^n }{(-q;q)_n}.
\end{equation}
To show the summation on the right %in (\ref{ch8f(q)Fine}) 
is bounded as $q$ approaches an odd-order root of unity radially, we exactly follow the steps of Watson's analysis of the mock theta function $f_0(q)$ in \cite{Watson}, Sec. 6. In Watson's nomenclature, take $M=2$, $N$ odd, to write $q=e^{2\pi \text{i}/N}=\zeta_N$. Then by replacing $q^{(nN+m)^2}$ with $(-1)^{nN+m}q^{nN+m}$ in the numerators of the $n\geq1$ terms of the series $f_0(q)$ %in \cite{Watson} 
(we note that Watson's $m$ is not the same as the subscript of $\zeta_m$ we use throughout this paper, which corresponds to $N$ in this proof), one sees %shows the ; for as  the right-hand side 
%is bounded when $q\to \zeta_m$ radially, as 
$$\left|1- \sum_{n=1}^{\infty}\frac{(-1)^{n}q^n }{(-q;q)_n}\right|\leq 2\sum_{n=0}^{N-1}\left|\frac{q^n }{(-q;q)_n}\right|<\infty$$
when $q=\rho \zeta_N$ with $0\leq \rho \leq 1$. %The series converges radially at odd-order roots of unity. %, a\in \mathbb Z$. 
To see the value the series converges to, consider the $(Nk+r)$th partial sum, with $r<N$, of the right-hand side of (\ref{ch8f(q)Fine}) as $\rho\to 1^-$, in light of Lemma \ref{ch8lemma} (i). In fact, as $|(-1)^{N}\zeta_N^N/(-\zeta_N,\zeta_N)_N|=1/2<1$, then part (ii) of Lemma \ref{ch8lemma} applies as $Nk+r\to\infty$ and (also taking into account that $f(\rho\zeta_N)$ converges uniformly for $\rho<1$) we may write % and we have % through any increasing sequence of values:
\begin{equation}\label{ch8rho_lim}
\lim_{\rho\to 1^-} \left(1- \sum_{n=1}^{\infty}\frac{(-1)^{n}(\rho\zeta_N)^n }{(-\rho\zeta_N;\rho\zeta_N)_n}\right)=1-\frac{2}{3}\sum_{n=1}^{N}\frac{(-1)^{n}\zeta_N^n }{(-\zeta_N;\zeta_N)_n}.
\end{equation}
Now, observe that (\ref{ch8factorization2}) gives
\begin{equation}\label{ch8factorization3}
(-\zeta_N;\zeta_N)_n^{-1}=(-\zeta_N^{-1};\zeta_N^{-1})_{N-n-1}.
\end{equation}
%Thus by applying Lemma \ref{ch8lemma} %directly 
%to the right-hand side of (\ref{ch8f(q)Fine}) at $\zeta_N$, 
Applying (\ref{ch8factorization3}) to the right-hand side of (\ref{ch8rho_lim}), then making the change $n\mapsto N-n-1$ in the indices, % of the resulting %finite 
%sum, %and simplifying, %and now taking $N=m$, 
we arrive at the desired result.
\end{proof}
%
%Continuing in this fashion, we can find a formula for this radial limit that is even easier to compute.%\footnote{In a 2013 study \cite{ClemmSchneider} of $f(q)$ at roots of unity using Example \ref{ch8example2}, A. Clemm and the author found a number of elegant relations in SageMath, but without formal proof. For instance, at fifth-order roots of unity $\zeta_5^i$, one computes
%$f(\zeta_5)f(\zeta_5^2)f(\zeta_5^3)f(\zeta_5^4)=256/81$.
%Moreover, one computes 
%\begin{equation*}
%\zeta_5=9\slash 16\  f(\zeta_5)f(\zeta_5^3),\  \  \    
%\zeta_5^2=9\slash 16\  f(\zeta_5)f(\zeta_5^2),\  \  \  
%\zeta_5^3=9\slash 16\  f(\zeta_5^3)f(\zeta_5^4),\  \  \  
%\zeta_5^4=9\slash 16\  f(\zeta_5^2)f(\zeta_5^4),
%\end{equation*}
%which are equivalent to % to the elegant
%$\zeta_5^i f(\zeta_5^i) = \zeta_5^{-i} f(\zeta_5^{-i}).$
%Do similar relations hold for other roots of unity?} %$\zeta_m^i$ at other $m$?}%at other roots of unity?}

\begin{example}\label{ch8example2}
For $\zeta_m$ an odd-order root of unity we have
\begin{equation*}\label{ch8f_finite4}
f(\zeta_m)=\frac{4}{3}\sum_{n=1}^{m}(-1)^n (-\zeta_m^{-1};\zeta_m^{-1})_n.
\end{equation*}
\end{example}

\begin{proof}[Proof of Example \ref{ch8example2}]
Here we use only finite sums, so we do not need to justify convergence. Let us define an auxiliary series
\begin{equation*}%\label{ch8f_finite}
h(\zeta_m)=\frac{2}{3}\sum_{n=1}^{m}(-1)^n (-\zeta_m^{-1};\zeta_m^{-1})_n.
\end{equation*}
Then using Example \ref{ch8f_finite}, with a little arithmetic and adjusting of indices, gives
\begin{equation*}%\label{ch8f_finite3}
f(\zeta_m)-h(\zeta_m)\  =\  1-\frac{2}{3}\sum_{n=1}^{m}(-1)^n (-\zeta_m^{-1};\zeta_m^{-1})_{n+1}\  =\  h(\zeta_m).
\end{equation*}
Comparing the left- and right-hand sides above implies our claim.
\end{proof}

\section{The ``feel'' of quantum theory}
By the considerations here we can find both finite formulas at roots of unity, and inverted versions using factorizations such as in (\ref{ch8invertmtf}) leading to forms such as (\ref{ch8inversion}) and (\ref{ch8umtfquantum}), for a great many $q$-hypergeometric series. Whether or not they enjoy modularity properties, these can display very interesting behaviors, % (we do not expect them to be almost modular generally), 
emerging outside the unit circle radially from an entirely different point $\zeta_m^{-1}$ than the point on the circle $\zeta_m$ approached from within, and likewise when entering the circle radially at roots of unity from without. Moreover, the map inside the unit circle in the variable $q$ looks like an ``upside-down hyperbolic mirror-image'' of the function's behavior on the outside. 
(Taking $q \mapsto \overline{q}$ in either the $|q|<1$ or $|q|>1$ piece of (\ref{ch8umtfquantum}) turns the map ``right-side up'', but at the expense of holomorphicity\footnote{As Tyler Smith, Emory University Department of Physics, noted to the author.}.) 
%if we look at such a function outside the unit circle in the variable $\overline{q}$

This imagery reminds the author of depictions of white holes and wormholes in science fiction\footnote{Not to mention phenomena like quantum tunneling and wall crossing in physics}. Do there exist ``points-of-exit'' (and entry) analogous in some way to roots of unity, at the event horizon of a black hole? Is there a %an M. C. Escher-like 
mirror-image universe %(including all the other singularities) 
contained within? 
We won't take these fantastical analogies too seriously, yet one is led to wonder: %mysterious 
how deep is the connection of $q$-series and partition theory, to phenomena at nature's fringe?
\\

\begin{remark}
{See Appendix \ref{app:G} for further notes on Chapter 8.} 
\end{remark}

 	\clearpage%

\appendix
\chapter{Notes on Chapter 1: Counting partitions}\label{app:A}%{\bf Adapted from \cite{strange}}

\section{Elementary considerations}

%Let us assume standard notations and terminology; in particular, we write $\lambda\vdash n$ to indicate $\lambda$ is a partition of $n$, $\ell(\lambda)$ for the length of the partition (i.e., the number of parts) and $|\lambda|$ for the size (i.e., the sum of all the parts). 
Here we point out some simple but useful relations. We adopt the conventions $p(0):=1$, and $p(n):=0$ for $n<0$. %, and $|\emptyset|:=0,\ell(\emptyset):=0$.
Then we have %We begin with 
the following elementary fact\footnote{Wakhare exploits similar ideas in \cite{Tanay}.}. %by recording the following elementary observation. %Now, let us consider the following obvious fact.

\begin{proposition}\label{Aprop1}
The number of partitions of $n$ with $k$ appearing as a part at least once, is equal to $p(n-k)$.
\end{proposition}

\begin{proof}
There is a bijective correspondence between the set of partitions of $n$ having $k$ as a part, and the partitions of $n-k$. Take any partition of $n$ with $k$ appearing at least once, and delete one part of size $k$ to produce a unique partition of $n-k$. Conversely, take any partition of $n-k$ and adjoin a part $k$ to produce a unique partition of $n$.
\end{proof}

For example, consider the seven partitions of $5$:
$$(5),(4,1),(3,2),(3,1,1),(2,2,1),(2,1,1,1), (1,1,1,1,1).$$\vfill
The number $k=2$ shows up in three partitions of $5$, so $p(3)=p(5-2)$ must be equal to three, which is of course correct.

Proceeding further in this direction, %, using the notation of \cite{Schneider_arithmetic}, we call $\delta\in \mathcal P$ a {\it subpartition} (or sometimes a ``divisor partition'') of $\lambda\in \mathcal P$, and we write $\delta|\lambda$, % (``$\delta$ divides $\lambda$''), 
%if all of the parts of $\delta$ are also parts of $\lambda$. Then 
a natural generalization of Proposition \ref{Aprop1} is the following statement.

\begin{proposition}\label{Athm1}
For any $\delta\in\mathcal P$ with $|\delta| < |\lambda|$, the number of partitions $\lambda\vdash n$ such that $\delta | \lambda$ is equal to $p(n-|\delta|)$.
\end{proposition}

So for any $m<n$, we can recover $p(m)$ by a quick inspection of the partitions of $n$. For example, consider again the partitions of $5$ listed above. The subpartition $(2,1)$ shows up in two partitions of $5$, giving the correct value of two for $p(2)=p(5-|(2,1)|)$. %Clearly, Proposition \ref{prop1} is the special case $\delta=(k)$, the subpartition consisting of just the part $k$.

\begin{proof}
To prove Proposition \ref{Athm1} we show basically the same bijection as above. Take any partition $\lambda \vdash n$ such that $\delta|\lambda$, and delete the parts of $\delta$ to arrive at a partition of $n-|\delta|$. Conversely, setting $k=|\delta|$, then adjoin the parts of $\delta$ to any partition of $n-k$ to arrive at a partition of $n$. %shown previously, by repeatedly applying Proposition \ref{prop1} to any partition $\lambda$ of $n$, at each stage taking $k$ to be a part of partition $\delta|\lambda$ (in any order), until all the parts of $\delta$ have been used and we have produced a partition of $n-|\delta|$. Conversely, we can adjoin the parts of $\delta$ to any partition of $n-|\delta|$ to arrive at a partition of $n$. 
\end{proof}

%{\bf RPS Below we will lead up to and prove the following statements too.}

%Of course, it would be more useful if we could work in the other direction, that is, if we could use the partitions of integers up to some given number to say something about partitions of larger numbers, which rapidly become more difficult to count, such as is the case with Euler's pentagonal number theorem \cite{Andrews}. 

One immediate corollary of the propositions above is the following.

\begin{proposition}\label{Acor2}
For any $a,b,c\geq 1$, we have that $p(a)$ is equal to the number of partitions of $a+bc$ in which $b$ occurs as a part with multiplicity $\geq c$. 
\end{proposition}

%\begin{remark}
We point out the above statement is symmetric in $b,c$.
%\end{remark}

%\begin{proof}
%Corollary \ref{cor1} 
%\end{proof}

\begin{remark}
By the same token, the total number $M_k(n):=\sum_{\lambda\vdash n}m_k(\lambda)$ of $k$'s appearing as parts over all partitions of $n$ is given by $M_k(n-k)+p(n-k)$: adjoin $k$ to partitions of $n-k$, including partitions already containing $k$, to yield partitions of $n$ containing $k$.\footnote{One can prove the well-known generating function formula $\sum_{n=1}^{\infty}M_k(n)q^n=\frac{q^k}{(q;q)_{\infty}(1-q^k)}$ from ideas in Chapter 3. Write the right-hand side of the claimed identity as $(q;q)_{\infty}^{-1}\sum_{n=1}^{\infty}q^{nk}=(q;q)_{\infty}^{-1}\sum_{\lambda\in\mathcal P}\chi_{k}(\lambda)q^{|\lambda|}=\sum_{\lambda\in\mathcal P}q^{|\lambda|}\sum_{\delta |\lambda}\chi_{k}(\delta)$, where we set $\chi_{k}(\lambda)=1$ if $\lambda$ is a partition into all $k$'s and $=0$ otherwise. Then observe the number of subpartitions of $\lambda$ into all $k$'s is exactly $\sum_{\delta|\lambda}\chi_k(\delta)=m_k(\lambda)$.} 
Then by recursion,
\begin{flalign}
M_k(n)%&=p(n-k)+M_k(n-k)\\
=p(n-k)+p(n-2k)+p(n-3k)+...+p\left(n-\left \lfloor {\frac{n}{k}}\right \rfloor k \right).
\end{flalign}
We note that Ramanujan-like congruences yield congruences for $M_k$, too. For instance, \begin{equation} M_5(5n+4)\equiv 0\  (\text{mod}\  5)\end{equation} 
follows from $p(5n+4-5j)\equiv 0\  (\text{mod}\  5)$ for $1\leq j \leq n$. By the same argument, 
\begin{equation}
M_7(7n+5)\equiv 0\  (\text{mod}\  7),\  \  \  \  \  \ M_{11}(11n+6)\equiv 0\  (\text{mod}\  11).
\end{equation}
\end{remark}

\section{Easy formula for $p(n)$}
Here we count partitions of $n$ %i%n various ways %examine some simple, useful patterns within the %partitions of $n\in\mathbb Z^+$; in particular, we identify a subset of . 
%We assume standard notations and terminology with regard to the 
%set $\mathcal P$ of integer partitions $\lambda=(\lambda_1, \lambda_2,...,\lambda_r),\  \lambda_1\geq \lambda_2\geq ...\geq \lambda_r\geq 1$ (including the empty partition $\emptyset$), that will allow us to count partitions in a variety of ways. We 
%identify 
via a natural subclass of partitions we will refer to as {\it nuclear partitions}, which are partitions having no part equal to one. %\footnote{For a partition $\lambda=(1^{m_1}2^{m_2}3^{m_3}4^{m_4}...)$ the author refers to the partition formed by deleting 1's, viz. $(2^{m_2}3^{m_3}4^{m_4}...)$, as the ``nucleus'' of $\lambda$. A nuclear partition is its own nucleus.}. 
In Chapter 4 we call this set $\mathcal P_{\geq 2}$; % and note that it is problematic with respect to convergence of partition zeta functions. H
here we will denote the ``nuclear'' partitions by $\mathcal N\subset \mathcal P$ %(the author refers to $\mathcal N$ as the ``nucleus of $\mathcal P$'') 
and let $\mathcal N_n$ denote nuclear partitions of $n\geq 0$. % the {\it nucleus} of $\mathcal P$ (containing ``nuclear'' partitions), which encodes information about the formation of partitions of every integer $n\geq 0$, and leads to a relatively easy calculation of the number $p(n)$ of partitions of $n$. 
Set $\nu(0):=1$ and for $n\geq 1$, let $\nu(n)$ count the number of nuclear partitions of $n$ (noting $\nu(1)=0$). Clearly we have the recursive relation $p(n)=\nu(n)+p(n-1)$; thus $\nu(n)$ has the generating function ${(q^2;q)_{\infty}^{-1}}.$ By recursion, % yields the identity % from the above ideas (and is likely a classical fact) that 
\begin{equation}\label{Achain}
p(n)=\nu(0)+\nu(1)+\nu(2)+\nu(3)+...+\nu(n).
\end{equation}

So to count partitions of $n$, we need only keep track of nuclear partitions, a much sparser set. For instance, here is an algorithm to compute $p(n)$ from the nuclear partitions of $n$ aside from the partition $(n)$ itself, i.e., the set $\mathcal N_n\backslash (n)$, which is a considerably smaller set than $\mathcal P_n$. We let $\mu=(\mu_1,\mu_2,...,\mu_r),\  \mu_1\geq\mu_2\geq ...\geq  \mu_r \geq 2$, denote a nuclear partition. %of partitions. %, the set $\mathcal N_n\subset \mathcal P_n$. % with all parts greater than $1$. 

\begin{theorem}\label{Acountingthm}
%For partitions $\lambda=(\lambda_1,\lambda_2,...,\lambda_r)\in\mathcal N_n$, 
We have that
$$p(n)=n+\nu(n)-1+\sum_{\mu\in \mathcal N_n\backslash(n)}(\mu_1-\mu_2),$$
with the right-hand sum taken over nuclear partitions of $n$ excluding the partition $(n)$.
\end{theorem}

Then to compute $p(n)$ one can follow these steps: \begin{enumerate}
\item Write down the partitions of $n$ containing no 1's aside from $(n)$ itself, that is, the subset $\mathcal N_n \backslash (n)$. For example, to find $p(6)$ we use $\mathcal N_6\backslash(6)=\{(4,2), (3,3), (2,2,2)\}$.
\item Write down the difference $\mu_1-\mu_2\geq 0$ between the first part and the second part of each partition from the preceding step. In the present example, we write down $$4-2=2,\  \  \  \  3-3=0,\  \  \  \  2-2=0.$$
\item Add together the differences obtained in the previous step, then add the result to $n+\nu(n)-1$ to arrive at $p(n)$. In this example, we add $2+0+0=2$ from the previous step to $6+\nu(6)-1=6+4-1$, arriving at $p(6)=6+4-1+2=11$, which of course is correct.  
\end{enumerate}

\begin{proof}
Observe that every nuclear partition of $n$ can be formed by adding $m_1(\lambda)$ to the largest part $\lambda_1$ of a ``non-nuclear'' partition $\lambda\vdash n$, and deleting all the 1's from $\lambda$, e.g., $(3,2,1,1)\to (5,2)$. Conversely, every nuclear partition $\mu \vdash n$ can be turned into a non-nuclear partition of $n$ by decreasing the largest part $\mu_1$ by some positive integer $j\leq \mu_1-\mu_2$, and adjoining $j$ 1's to form the non-nuclear partition. So the nuclear partitions of $n$ ``decay'' (by giving up 1's from the largest part) into non-nuclear partitions of $n$, e.g., $(5,2)\to(4,2,1)\to (3,2,1,1)\to(2,2,1,1,1)$, of which the total number is $p(n)-\nu(n)$. Each nuclear partition $\mu$ decays into $\mu_1-\mu_2$ different non-nuclear partitions except the partition $(n)$, which decays into $n-1$ non-nuclear partitions, viz. $(n)\to(n-1,1)\to (n-2,1,1)\to...\to (1,1,...,1)$, so the number of non-nuclear partitions of $n$ is $p(n)-\nu(n)=(n-1)+\sum_{\mu\in \mathcal N_n\backslash(n)}(\mu_1-\mu_2).$
\end{proof}

It is interesting to see how the subset $\mathcal N\subset \mathcal P$ produces the entire set $\mathcal P$ by this simple ``decay'' process\footnote{To prove this assume otherwise, that for some $n\geq 0$ there is a non-nuclear partition $\phi$ of $n$ {\it not} produced by decay of some partition in $\mathcal N_n$. Then deleting all the $1$'s from $\phi$ and adding them to the largest part $\phi_1$ produces a nuclear partition of $n$ that decays into $\phi$, a contradiction.}. Now, let $\gamma(n)$ denote the number of nuclear partitions $\mu$ of $n$ such that %either $\mu$ has exactly one part, or 
$\mu_1=\mu_2$ (the first two parts are equal), %\footnote{The author thinks of these partitions as in a ``ground state'', on analogy to atoms, as they are not subject to decay into non-nuclear partitions.}, 
setting $\gamma(0):=0$ and noting $\gamma(1)=\gamma(2)=\gamma(3)=0$, as well. Then for $n\geq 3$ the recursion $\nu(n)=\gamma(n)+\nu(n-1)$ holds (adding 1 to the largest part of every nuclear partition of $n-1$ gives the nuclear partitions $\mu$ of $n$ with $\mu_1>\mu_2$)\footnote{In fact, much as nuclear partitions ``control'' the growth of $p(n)$, these nuclear partitions with first two parts equal --- which the author thinks of as being in their ``ground state'' --- control the growth of $\nu(n)$, thus appearing to fundamentally control $p(n)$.}, thus the generating function for $\gamma(n)$ is $\frac{1}{(1+q)(q^3;q)_{\infty}}-1+q-q^2$. Moreover, noting $\nu(2)=1$, we have for $n\geq 3$ that  \begin{equation}\label{nu_recurrence} \nu(n)=1+\gamma(3)+\gamma(4)+...+\gamma(n).\end{equation} % (the author thinks of these as ``ground state'' partitions). 

For $m\geq 1$, let $\nu(n,m)$ denote the number of nuclear partitions of $n$ whose parts are all $\leq m$. Then it is easily verified that we can also compute $\nu(n)$ as follows\footnote{See Cor. 1.5 of \cite{MercaSchmidt} for a formula for $\nu(n)$ (there written as a backward difference $\nabla [p] (n)$) involving the classical M\"{o}bius function.}. %, as the reader can easily verify. %without much difficulty. % with a moment's consideration. %; for the sake of convenience, let us refer to this set as the {\it partial nuclear partitions of $n$}. Then we have the following formula to compute $\nu(n)$ using the partial nuclear partitions of the integers $\leq n-2$.

\begin{theorem}
%For partitions $\lambda=(\lambda_1,\lambda_2,...,\lambda_r)\in\mathcal N_n$, 
We have $n \geq 4$ that
$$\nu(n)=\sum_{k=2}^{n-2}\nu(k,n-k).$$
\end{theorem}
%Of course, this indicates a bijection between the pertinent subsets of partitions. 
Combining this identity with Theorem \ref{Acountingthm} and \eqref{nu_recurrence} above, and making further simplifications, the task of computing $p(n)$ can be reduced to counting much smaller subsets of partitions of integers $\leq n-2$. These small subsets of partitions of integers up to $n-2$ completely encode the value of $p(n)$.% than $\mathcal N_n$. 
%{\bf CONSTRUCTION MATERIALS BELOW}

More generally, we might let $\nu_k(n)$ denote the number of partitions of $n$ having no part equal to $k$ --- let us refer to these as ``$k$-nuclear'' partitions --- % (we could call these ``$k$-nuclear'' partitions), 
setting $\nu_k(0):=1$ for all $k\geq 1$; thus $\nu(n)=\nu_1(n)$. Let $\mathcal N^k$ denote the set of all $k$-nuclear partitions, and let $ \mathcal N_n^k$ be $k$-nuclear partitions of $n$; thus $\mathcal N=\mathcal N^1, \  \mathcal N_n=\mathcal N_n^1$. Clearly we have $p(n)=\nu_k(n)+p(n-k)$, so $\nu_k(n)$ has the generating function $\frac{1-q^k}{(q;q)_{\infty}}$ and is subject to essentially the same treatment as $\nu(n)$ above. %\footnote{It seems likely a $k$-nuclear analog to $\gamma(n)$ exists, too, but the author has not checked this point.}. 
Then by recursion, as previously, %we can write 
\begin{equation} p(n)=p\left(n-\floor{\frac{n}{k}}k\right)+\sum_{j=1}^{\floor{n/k}}\nu_k(n-jk),\end{equation}
where $\floor{x}$ is the floor function, and by a similar proof (decay into $k$'s instead of 1's) we generalize Theorem \ref{Acountingthm}, which represents the $k=1$ case of the following identity. % from the first two parts.%, which in this case has the following form.

\begin{theorem} We have that
 \begin{equation*}p(n)=\floor{\frac{n}{k}}^*+\nu_k(n)+\sum_{\mu\in \mathcal N_n^k\backslash(n)} \floor{\frac{\mu_1-\mu_2}{k}},
\end{equation*} 
where we set $\floor{\frac{a}{b}}^*:=\floor{\frac{a}{b}}-1$ if $b|a$ and $:=\floor{\frac{a}{b}}$ otherwise, and the right-hand sum is taken over $k$-nuclear partitions of $n$ excluding the partition $(n)$.% when $n\neq k$.
% with the sum taken over partitions $\mu$ of $n$ having no part equal to $k$, excluding $(n)$ itself. 
\end{theorem}
 
As in the remark at the end of the previous section, the Ramanujan congruences imply, for instance, that since $p(5n+4)-p\left(5(n-1)+4 \right)\equiv 0\  (\text{mod}\  5)$, then 
\begin{equation}\nu_5(5n+4)\equiv 0\  (\text{mod}\  5).\end{equation}
Similarly, we also have \begin{equation}\nu_7(7n+5)\equiv 0\  (\text{mod}\  7),\  \  \  \  \  \  \nu_{11}(11n+6)\equiv 0\  (\text{mod}\  11).\end{equation}
%Noting $\nu_a(n+a)\equiv p(n+a)-p(n)\  (\text{mod}\  a)$, i
If these congruences could be proved directly, it seems likely one could run this kind of argument (perhaps using Proposition \ref{Acor2}, as well) in reverse to prove Ramanujan-like congruences by induction.%\end{remark}

%\begin{corollary}\label{cor1}
%\end{corollary}

%-----------------------------------------------------------------------------------------------%
\chapter{Notes on Chapter 3: Applications and algebraic considerations}\label{app:B}
%-----------------------------------------------------------------------------------------------%

\section{Ramanujan's tau function and $k$-color partitions}
Here we give two immediate applications in number theory of the principle at the heart of Proposition \ref{ch3cauchyprod} (the partition Cauchy product formula), extended to products of more than two series. The first example gives a formula for Ramanujan's tau function, an arithmetic function which appears as the coefficients of a weight-$12$, level 1 cusp form (see \cite{Ono_web}). As previously, let $m_i=m_i(\lambda)$ denote the multiplicity of $i$ as a part of $\lambda$.%\footnote{Another partition-theoretic formula for $\tau(n)$ is given in Appendix D.} 

\begin{example}\label{Btau}%[Unpublished]
Ramanujan's tau function $\tau(n)$, defined as the $n$th coefficient of $q(q;q)_{\infty}^{24}$, can be written  
\begin{flalign*}
\tau(n)=\sum_{\lambda \vdash (n-1)}(-1)^{\ell(\lambda)} \binom{24}{m_1(\lambda)}\binom{24}{m_2(\lambda)}\binom{24}{m_3(\lambda)}\cdots.
\end{flalign*}
\end{example}

\begin{proof}
This follows by applying the binomial theorem to each factor $(1-q^i)^{24}$ of the $q$-Pochhammer symbol to give $$q(q;q)_{\infty}^{24}=q\prod_{n=0}^{\infty}\sum_{k=0}^{24}(-1)^k q^k \binom{24}{k},$$
then expanding the product and collecting coefficients of $q^n$ as sums of the shape $\sum_{\lambda \vdash n}$. The extra factor of $q$ produces the shift $\sum_{\lambda \vdash n}\mapsto \sum_{\lambda \vdash n-1}$ in the coefficients.\end{proof}

Next we find a formula for the number $P_k(n)$ of $k$-color partitions of $n$, as studied by Agarwal and Andrews \cite{AgarwalAndrews, Agarwal} and other authors. % as studied, for instance, in \cite{kcolor}.%\footnote{Another partition-theoretic formula for $P_k(n)$ is given in Appendix D.} 

\begin{example}\label{Bkcolor}%[Unpublished]
The number $P_k(n)$ of $k$-color partitions of $n$, which is equal to the $n$th coefficient of $(q;q)_{\infty}^{-k}$ for $k\geq 1$, can be written
\begin{flalign*}
P_k(n)=\sum_{\lambda \vdash n} \binom{k+m_{1}(\lambda)-1}{m_1(\lambda)}\binom{k+m_{2}(\lambda)-1}{m_2(\lambda)}\binom{k+m_{3}(\lambda)-1}{m_3(\lambda)}\cdots.
\end{flalign*}
\end{example}

\begin{proof}
Just like the previous example, this follows by writing $$(q;q)_{\infty}^{-k}=\prod_{n=0}^{\infty}\sum_{k=0}^{24}q^k \binom{n+k-1}{k},$$
expanding the product, and collecting coefficients.\end{proof}

More generally, the same proofs extend to absolutely convergent products of the form $\prod_{n=1}^{\infty}(1-\phi(n)q^n)^{k}$ for any $k\in\mathbb Z$. 
\begin{theorem}\label{Bgeneral}For $\phi\colon \mathbb N \to \mathbb C$ and $q\in\mathbb C$ such that both sides converge, $k\geq 0,$ we have the identities
\begin{flalign*} 
\prod_{n=1}^{\infty}(1-\phi(n)q^n)^{k}&=\sum_{\lambda \in \mathcal P}q^{|\lambda|}(-1)^{\ell(\lambda)} \prod_{i=1}^{\infty}\phi(i)^{m_i(\lambda)}\binom{k}{m_i(\lambda)},\\ \\
\prod_{n=1}^{\infty}(1-\phi(n)q^n)^{-k}&=\sum_{\lambda \in \mathcal P}q^{|\lambda|}\prod_{i=1}^{\infty}\phi(i)^{m_i(\lambda)}\binom{k+m_{i}(\lambda)-1}{m_i(\lambda)}.
\end{flalign*}
\end{theorem}

Combined with Theorems \ref{ch41.1} and \ref{ch41.11} in Chapter 4, and also with Fa\`{a} di Bruno's formula as in Chapter 5 and Appendix D below, arbitrarily complicated products and quotients of $q$-Pochhammer symbols (and other product forms) can be evaluated similarly. Additional formulas for $\tau(n), P_k(n)$ are given in Appendix D.

\section{\texorpdfstring{$q$}{Lg}-bracket arithmetic}%{\bf Adapted from \cite{qbracket}}

The $q$-bracket operator is reasonably well-behaved as an algebraic object; here we give a few formulas that may be useful for computation. From Definition \ref{ch1qbracket} we have $q$-bracket addition
\[
\left<f\right>_q+\left<g\right>_q=\left<f+g\right>_q
,\]
which is commutative, of course, and also associative:
\[
\left<f+g\right>_q+\left<h\right>_q=\left<f\right>_q+\left<g+h\right>_q.
\]
We have for a constant $c\in\C$ that $c\left<f\right>_q=\left<cf\right>_q$; other basic arithmetic relations such as $\left<0\right>_q=0$ and $\left<f\right>_q+\left<0\right>_q=\left<f\right>_q$ follow easily as well.

Now, let
\[
\widetilde{f}(n):=\sum_{\lambda\vdash n}f(\lambda)
.\]
We will define a convolution ``$*$'' of two such functions $\widetilde{f},\widetilde{g}$ by
\begin{equation}\label{Bconvolution}
(\widetilde{f}*\widetilde{g})(\lambda):=\frac{1}{p(|\lambda|)}\sum_{k=0}^{|\lambda|}\widetilde{f}(k)\widetilde{g}(|\lambda|-k)
.
\end{equation}
Note that, by symmetry, $\widetilde{f}*\widetilde{g}=\widetilde{g}*\widetilde{f}$.\footnote{The author is grateful to Alex Rice for a discussion about convolution that informed this section.} 

Let us also define a multiplication ``$\star$'' between $q$-brackets by
\begin{equation}\label{Btimes}
\left<f\right>_q\star\left<g\right>_q:=\frac{\left<f\right>_q\left<g\right>_q}{(q;q)_{\infty}}
,
\end{equation}
where the product and quotient on the right are taken in $\C[[q]]$. It follows from \ref{Bconvolution} and \ref{Btimes} above that
\[
\left<f\right>_q\star\left<g\right>_q=\left<\widetilde{f}*\widetilde{g}\right>_q 
.\]
From here it is easy to establish a $q$-bracket arithmetic yielding a commutative ring structure, with familiar-looking relations such as
\[\left<f\right>_q\star\left<\widetilde{g}*\widetilde{h}\right>_q=\left<\widetilde{f}*\widetilde{g}\right>_q\star\left<h\right>_q
,\]
\[
\left<f\right>_q\star\left<g+h\right>_q=\left<\widetilde{f}*\widetilde{g}\right>+\left<\widetilde{f}*\widetilde{h}\right>_q,
\]
and so on. 

It is trivial to see that $\left<1\right>_q=1$; however, $\left<1\right>_q\star\left<f\right>_q=\frac{\left<f\right>_q}{(q;q)_{\infty}}\neq\left<f\right>_q$, so $\left<1\right>_q$ is not the multiplicative identity in this arithmetic. In fact, as we note in Chapter 3, Section 1, the $q$-bracket of Dyson's rank function (with $\text{rk}(\emptyset):=1$ as in Chapter 3) is 
$$\left<\operatorname{rk}\right>_q=(q;q)_{\infty}
.$$
Then $\left<\operatorname{rk}\right>_q$ may serve as the multiplicative identity in the $q$-bracket arithmetic above, by Equation (\ref{Btimes}).

\newpage

\section{Group theory and ring theory in $\mathcal P$}

\noindent {\bf Based on joint work with Ian Wagner}

\subsection{Antipartitions and group theory}

Here we will define the set $\mathcal P^-$ of \textit{antipartitions}, in analogy with antiparticles in physics: partitions and antipartitions annihilate one another. Then we show that the set $\mathcal P \cup \mathcal P^-$ naturally forms a group structure. 

\begin{definition}\label{Bantipartitiondef}
For $\lambda=(\lambda_1,\lambda_2,...,\lambda_r)\in \mathcal P$, we define an antipartition $\lambda^-=$ $(\lambda_1^-, \lambda_{2}^-,...,$ $\lambda_r^-)$ $\in \mathcal P^-$ such that 
$$\lambda \lambda^-=\emptyset.$$
We refer to the $\lambda_i^-\in\lambda^-$ as ``antiparts''. %, and write them in ascending order instead.
%Furthermore, we will write $\lambda^-=(\lambda_1^-, \lambda_2^-,...,\lambda_r^-)$ and 
\end{definition}

Let us adopt the convention $\lambda^{-a}:=\lambda^-\lambda^-\lambda^-\cdots \lambda^-$ ($a$ repetitions). Clearly we have that $(\lambda^{-})^-=\lambda$; every partition is the antipartition of its own antipartition. We also have right away that $\emptyset^-=\emptyset$. For $a\in\mathbb Z^+$, %here is how %we use the following convention to show how 
corresponding parts and antiparts annihilate each other pair-wise in partitions (we adopt the convention of separating parts and antiparts with a semicolon, and putting the antiparts to the right in a partition):
$$(a;a^-)=(a)(a)^-=\emptyset.$$
But it is not necessary that partitions and antipartitions should cancel; %; only  parts and antiparts indexed by the same integer interact. 
in fact, we might have ``mixed'' partitions containing both parts and antiparts. We can compute, for example, that
$$(5,4,3,3)(4,3,1,1)^-=(5,4,3,3;4^-, 3^-, 1^-, 1^-)=(5,3;1^-,1^-).$$
%Multiplying $(5,3,1^-)$ with another mixed partition goes like this:
%$$(5,3,1^-)(7^-,5,3^-,2,1^-)=(7^-,5,5,3,3^-,2,1^-,1^-)=(7^-,5,5,2,1^-,1^-)$$
Note that parts and antiparts indexed by the same integer cancel. %We adopt the convention of putting the antiparts to the right of the parts within a mixed partition. % (although in the end, paired parts and antiparts will vanish). 
Mixed partitions may also be written in ``rational'' form, e.g.,
$$(5,3;1^-,1^-)=(5,3)(1,1)^-=(5,3)/(1,1).$$ 
%$$(7^-,5,5,2,1^-,1^-)=(5,5,2)/(7,1,1).$$
Then we might refer to the set \begin{equation}\label{BQdef}
\mathcal Q:=\mathcal P \cup \mathcal P^-
\end{equation}
as {\it rational partitions}. (In usage, however, we still refer to elements of $\mathcal Q$ as ``partitions''.) 

%\begin{remark}
%These antiparts are like negative numbers crossed with reciprocals; for $\lambda_i\in\Z^+$, they obey the relations: 
%$$\lambda_i + \lambda_i^-=0,\  \  \  \lambda_i \lambda_i^-=1$$
%\end{remark}

A few other relations are immediate.

\begin{proposition}\label{Bprop1}
We have the following identities:
\begin{equation*}\label{Blog relations2}
\ell(\lambda^-)=-\ell(\lambda),\  \  \  \  
|\lambda^-|=-|\lambda|,\  \  \  \  
n_{\lambda^-}=\frac{1}{n_{\lambda}},\  \  \  \  
m_{k}(\lambda^-)=-m_{k}(\lambda).
\end{equation*}
\end{proposition}

\begin{proof}
The first identity follows from
$$\ell(\lambda)+\ell(\lambda^-)=\ell(\lambda\lambda^-)=\ell(\emptyset)=0.$$
The second identity follows from
$$|\lambda|+|\lambda^-|=|\lambda\lambda^-|=|\emptyset|=0.$$
The third identity follows from
$$n_{\lambda}n_{\lambda^-}=n_{\lambda\lambda^-}=n_{\emptyset}=1.$$
The fourth identity is formally necessary if we want
$$\ell(\lambda^-)=-\ell(\lambda)=-(m_1(\lambda)+m_2(\lambda)+m_3(\lambda)+...).$$%=-\ell(\lambda)$$
\end{proof}

%For the time being, 
%We think of negative length as coming from the reversal of order of the antiparts, on analogy to change of direction on the number line, and 
For the time being we can take negative lengths and multiplicities, as well as fractional norms, %we just take 
as just formal artifacts; % (perhaps an indication of direction of orientation, as with vectors); 
but the second equation in Proposition \ref{Bprop1} admits the following interpretation: {\it the antipartitions $\mathcal P^-$ are partitions of negative integers}. %, and are multiplicative partitions (see \ref{multpart}) of unit fractions.   

%We note that, for readability, we have chosen to write rational partitions in the form displayed above, with the antiparts on the right. But we might write the mixed, or rational, antipartition in this form instead:
%$$(1^-,1^-,3^-,7^-,5,5,3,2)$$
%Written in this way, there is an obvious bijection between the rational partitions and the set of {\it unimodal sequences} in combinatorics: delete the minuses from the superscripts of the antiparts of any rational partition to arrive at a unimodal sequence, and adjoin a minus superscript to the increasing side of a unimodal sequence (with or without using the peak) to produce a rational partition.NEED TO CONSIDER DOUBLED PEAK

At this point the climax of the section will not be too surprising to the reader. 
\begin{theorem}\label{Bgroup}
Rational partitions $\mathcal Q$ form an abelian group under partition multiplication.
\end{theorem}

\begin{proof}
Clearly under our multiplication operation ``$\  \cdot\  $'' on $\mathcal Q$ we have the identity element $\emptyset$, the elements $\lambda,\lambda^-$ are multiplicative inverses, associativity and commutativity are automatic from set-theoretic considerations, and $\mathcal Q$ is closed under multiplication, verifying $(\mathcal Q,\cdot)$ has the claimed group structure. 
\end{proof}

Then $(\mathcal Q,\cdot)$ looks a lot like $\mathbb Q \backslash \{0\}$ as a multiplicative group. 
We hope that classical techniques of group theory may lead to new identities, congruences and bijections in the theory of partitions. %In the next section we will explore some consequences of this group structure on partitions and antipartitions.

%\section{Group theory in \texorpdfstring{$\mathcal Q$}{$Lg$}}
%
%{\bf RPS: In this tentative section (which can be absorbed into the previous section if we don't have much to say), we might give any ``textbook'' consequences we can think of in the group-theoretic setting. For instance, are there equivalence relations, congruence relations, etc.? We know partitions index the equivalence classes of the symmetric group $\mathcal S$; but all groups (including presumably $\mathcal Q$) are isomorphic to some subgroup of $\mathcal S$: can we identify which subgroup $\mathcal Q$ looks like? Are there isomorphism theorems or other basic group theoretic connections? These are not vital for the paper but are natural questions; if they are low-hanging fruits, maybe we record a couple of such results and leave the rest to the reader as open questions. Otherwise, we can just restate this bold-face comment as an open question and move on...}  

\subsection{Partitions and diagonal matrices}

For the sake of pointing toward future work in the algebraic vein, we also note a few connections to matrix algebra, a gold-mine of structural archetypes, although our study in this direction is incomplete. There is an obvious way to associate nonempty partitions to diagonal matrices, which are well known to enjoy beautiful algebraic properties. 

\begin{definition}
For a nonempty %(possibly rational) 
partition $\lambda=(\lambda_1, \lambda_2,...,\lambda_r)\in \mathcal P, \lambda_1\geq \lambda_2\geq...\geq \lambda_r\geq 1$, we define the diagonal matrix
$$M_{\lambda}:=\begin{pmatrix} 
\lambda_{1}&0&0&\ldots&0
\\
0 & \lambda_{2} &0 &\ldots & 0
\\
0&0& \lambda_{3}&\ldots & 0
\\
\vdots &\vdots&\vdots &\ddots &\vdots
\\
0&0&\ldots&0&\lambda_{r} 
\end{pmatrix},$$
which we might refer to as the ``matrix of $\lambda$''. %We define $M_{\emptyset}$ to be any square matrix whose entries all equal zero.
\end{definition}

%\begin{remark}
%We address matrices of antipartitions and partitions in $\mathcal Q$ in the next section.
%\end{remark}

We have immediately an interpretation of dimension ``$\operatorname{dim}$'', determinant ``$\operatorname{det}$'', and trace ``$\operatorname{tr}$'' in terms of partition-theoretic statistics:
\begin{equation}\label{Bmatrix}
\operatorname{dim}(M_{\lambda})=\ell(\lambda),\  \  \  \operatorname{tr}(M_{\lambda})=|\lambda|,\  \  \  \operatorname{det}(M_{\lambda})=n_{\lambda}.
\end{equation}
Then there are a natural addition and multiplication we might define on partitions of a fixed length $r$, as an extension of matrix operations. For $\lambda, \lambda'\in \mathcal P$ with $\ell(\lambda)=\ell(\lambda')=r$, we define $\lambda+\lambda',\  \lambda\times \lambda'$ to be the partitions whose parts are the diagonal entries of the matrices $M_{\lambda}+M_{\lambda'}, M_{\lambda}M_{\lambda'}$, respectively:
$$M_{\lambda}+M_{\lambda'}=M_{\lambda+\lambda'},\  \  \  M_{\lambda}M_{\lambda'}=M_{\lambda\times \lambda'}.$$ 
Our operations are given explicitly by
$$\lambda+\lambda':=(\lambda_1+\lambda_1',\lambda_2+\lambda_2',...,\lambda_r+\lambda_r'),\  \  \  \lambda\times\lambda':=(\lambda_1\lambda_1',\lambda_2\lambda_2',...,\lambda_r\lambda_r').$$  

Qualitatively, these operations are quite different from the partition multiplication introduced in Chapter 3, which is purely a set-theoretic operation and does not depend on any arithmetic taking place between the parts themselves, aside from putting them in weakly decreasing order. Now we see the parts adding and multiplying, to produce the parts of new partitions. (We discuss matrices involving antipartitions, as well, below.) % and partitions in $\mathcal Q$ until the next section, as we have not yet addressed the meaning of multiplication and addition between parts and antiparts, which turns out to be nontrivial. 
%\end{remark}

%We refrained from defining these matrices, intended for heuristics, as arithmetic between parts and antiparts is not well defined. Here is the authors' working hypothesis: {\it antiparts behave like negative numbers when adding, and like reciprocals when multiplying.}

In any event, the ``$\geq$'' ordering on the entries on the diagonal ensure that the entries of $M_{\lambda}+M_{\lambda'},$ $\  M_{\lambda}M_{\lambda'}$ also obey the same ordering, so these entries do indeed comprise partitions of length $r$. 
Clearly the $r\times r$ zero matrix, which we identify with the empty partition, $\emptyset$, and $r \times r$ identity matrix, which we identify with the length-$r$ partition into all 1's, viz. $$I_0:=\emptyset,\  \  \  \  I_r:=(1,1,...,1)\  \  (\text{$r$ repetitions}),$$ 
are, respectively, the additive and multiplicative identities: %, which observations translate in partition terms to
\begin{equation}\label{Bmatrixop1}
\lambda+\emptyset=\lambda,\  \  \  \lambda\times I_{r}=\lambda,\  \  \  \lambda\times \emptyset=\emptyset.
\end{equation}
Then all the machinery of linear algebra of diagonal matrices can be extended to partitions of length $r$ (for any fixed $r$) under these operations. 

We may also include partitions in $\mathcal Q$ if we define an arithmetic relating parts and antiparts. %, and the partitions of length $r$ form a ring. 
%Let $\mathcal M_r$ denote the set of $r\times r$ matrices of partitions of length $r$. Then we also have the following nice relations between these matrices:
%$$M_{\lambda}+ M_{\delta}= M_{\lambda+\delta}\in \mathcal M_r,\  \  \  M_{\lambda} M_{\delta}=M_{\lambda\times \delta}\in \mathcal M_r$$
However, the arithmetic between parts and antiparts turns out to be a nontrivial question. Of course, antiparts are just positive integers decorated with minus signs, so we expect something like the usual arithmetic in $\mathbb Z$; for instance, if $a,b\in\mathbb Z^+$ we expect 
$$ab^-=b^-+b^-+...+b^-\  (\text{$a$ repetitions})=(ab)^-,$$ 
because the antiparts should add together. On the other hand, considering the relations in Proposition \ref{Blog relations2}, we see that the antiparts sometimes act like negative numbers and sometimes act like fractions. These antiparts are, in fact, formal entities that arise naturally from partition-theoretic (as opposed to matrix-theoretic) considerations, and their arithmetic properties may well depend on context --- indeed, this is what the author assumes to be the case. 

In the case of the matrix-based operations above, a workable rule of thumb for the arithmetic between a part and an antipart is:
{\it Antiparts act like negative integers under addition, and reciprocals under multiplication.} 

In symbols, for $a,b\in \mathbb Z^+$ we might set
\begin{equation}\label{Bmatrixop}
a\cdot b^-=ab^{-1}\in\mathbb Q,\  \  \  a+b^-=a-b\in \mathbb Z.
\end{equation}
Note that, when writing the partitions resulting from these operations, we will follow the convention of converting reciprocals and negative numbers back into the ``minus'' notation of antipartitions. (For this sketch of matrix-based ideas, we assume that in fact $b|a$ so the resulting part $ab^-$ is still an integer; the question of partition-like objects whose parts come from other sets such as $\mathbb Q$, $\mathbb F_{p}$, ring ideals, etc., %is of much interest to the author, but 
is beyond the scope of this thesis.)

The relations (\ref{Bmatrixop}) fit intuitively with (\ref{Bmatrix}) above, and give natural-looking identities like
\begin{equation}\label{Bniceid}
\lambda+\lambda^-=\emptyset,\  \  \  \lambda\times\lambda^-=I_{r},
\end{equation}
where $r=\ell(\lambda)$. Encouraging as this matrix-analog structure is, there are points we have not followed through; we have not even proved the demands in (\ref{Bmatrixop}) to be consistent. %\footnote{Ian Wagner and the author have looked at matrix-like partition operations of a different nature, proving a ring theory of partitions and related results we have yet to publish.}.%, %seem to suggest parts and antiparts all are square roots of $-1$), 
%as well as some overarching algebraic issues. %(beyond the yet-to-be-discussed details of part--antipart arithmetic). 
%First, the nice structure in this subsection does not extend to describe interactions between partitions of arbitrary lengths, as matrix algebra requires the matrices be of the same size $r\times r$; and there is no unique matrix associated to $\emptyset$ or to the multiplicative identity. %, but rather we must use a different zero matrix for every length $r$. 
%Furthermore, these matrix-based operations do not seem to extend the classical-styled identities such as those in the body of Chapter 3. %; partition multiplication, which we hoped to build on, does not preserve length so is excluded from this scheme. %is not even involved.%\footnote{The author and Ian Wagner are engaged in ongoing work \cite{SchneiderWagner} giving a ring theory of partitions, which eliminates these issues.}

%Finally, we have
%$$|\lambda+\lambda'|=|\lambda|+|\lambda'|\  \text{but}\  |\lambda\times \lambda'|\neq |\lambda| |\lambda'|.$$
%If we want partition theory to look like relations in algebra, we would like the size to be multiplicative. However, we proposed this matrix ``toy model'' chiefly to get a feel for the terrain, and won't get hung up on %working through its 
%fine details. By examining other properties of matrices---and adjusting our hypotheses on the arithmetic of the antiparts---we can indeed define a new multiplication operation to satisfy the second, mutliplicative relation. 

\subsection{Partition tensor product and ring theory}

In this section we introduce a direct sum $\oplus$ and tensor product $\otimes$ between partitions, and prove that $(\mathcal Q,\oplus,\otimes)$ forms a commutative ring with identity. 

\begin{definition}\label{newsum}
For $\lambda,\lambda'\in \mathcal P$ we define the direct sum $\lambda\oplus \lambda'\in {\mathcal P}$ to be a rewriting of multiplication from Chapter 3:
$$\lambda\oplus\lambda':=\lambda\lambda'.$$
Then we will also write $\lambda\oplus\lambda\oplus...\oplus\lambda\  (\text{$n$ repetitions})=:n\lambda$.%\footnote{This exponential notation reminds us this is also $\lambda^n$ in the product notation of \cite{Schneider_arithmetic}}
\end{definition}

In fact, in \cite{Andrews}, Andrews uses the symbol $\oplus$ to define this exact operation, although he expresses the direct sum %operation % (and his whole theory of partition ideals) 
in terms of a sum of the multiplicities $m_k$ (or ``frequencies'' $f_k$ in his terminology).

%Then it is easy to see that 
%\begin{equation}\label{sum}
%|\lambda\oplus\lambda'|=|\lambda|+|\lambda'|.
%\end{equation}
\begin{remark} So, for example, we write (in a few alternative ways):
$$(3,1,1)=(3)\oplus(1,1)=(3,1)\oplus(1)=(3)\oplus(1)\oplus(1)=(3)\oplus 2(1).$$ 
\end{remark}
%It is the relation $|\lambda*\lambda'|=|\lambda| |\lambda'|$ in (\ref{wish}) that is more subtle; as the size of a partition is an additive object, multiplicative properties are rare, and rather surprising. 
In exploring operations between partitions, the author's collaborator Ian Wagner discovered --- along the lines of the matrix analogy in the previous section --- that the tensor product of two partition matrices suggests a very well-behaved ``times'' operation between partitions.

\begin{definition} \label{newprod}
For $\lambda, \lambda' \in \mathcal P$ with $\ell(\lambda)=r, \ell(\lambda')=s$, we define the tensor product $\lambda\otimes\lambda'\in \mathcal P$ to be the partition whose parts are exactly the set
$$
\{\lambda_1\lambda_1', \lambda_1\lambda_2',...,\lambda_1\lambda_s',\lambda_2\lambda_1',\lambda_2\lambda_2',...,\lambda_2\lambda_s',...,\lambda_r\lambda_1',\lambda_r\lambda_2',...,\lambda_r \lambda_s'\}\subset\mathbb Z^+,
$$
reorganized to be in canonical weakly decreasing order. Then we will also write $\lambda\otimes \lambda \otimes ...\otimes \lambda\  (\text{$n$ repetitions})=:\lambda^{\otimes n}$.
\end{definition}
%Under this tensor product $\otimes$ as well as the sum $\oplus$ defined in (\ref{newsum}) above 

Of course, the empty partition $\emptyset$ acts as the identity under $\oplus$, and in this setting, the length one partition $I_1=(1)$ is the multiplicative identity. Thus the direct sum and tensor product of partitions lead to elementary identities like those in (\ref{Bmatrixop1}): 
\begin{equation}\label{tensorop1}
\lambda\oplus\emptyset=\lambda,\  \  \  \lambda\otimes (1)=\lambda,\  \  \  \lambda\otimes \emptyset=\emptyset.
\end{equation}
\begin{remark}
Generalizing the middle equation in (\ref{tensorop1}), we actually have that $$\lambda\otimes I_n=\lambda\oplus \lambda \oplus \ ... \oplus \lambda\  (\text{$n$ repetitions})=n\lambda.$$
%or $\lambda^n$ in the product notation of \cite{Schneider_arithmetic}. 
Thus we might write any partition $\lambda$ in a ``split'' form, i.e., some reordering of 
$$\left[(1)\otimes I_{m_1(\lambda)}\right]\oplus\left[(2)\otimes I_{m_2(\lambda)}\right]\oplus\left[(3)\otimes I_{m_3(\lambda)}\right]\oplus...$$
%for example:
like
$$(4,4,3,2,2,2)=\left[(4)\otimes(1,1)\right]  \oplus(3) \oplus \left[(2)\otimes(1,1,1)\right].$$
\end{remark}

This product $\otimes$ is in fact very similar to the Kronecker product in matrix algebra, a well-known case of the tensor product; and the trace of a Kronecker product is multiplicative.  %this leads to % Together with the sum ``$\oplus$'' defined in the previous section, this leads to 
Analogous considerations give a very natural-looking pair of %multiplication 
relations. %, which we collect with an additive relation in the following statement. %, so we have the following identities.

\begin{proposition}\label{mult}
For $\lambda,\lambda'\in \mathcal P$ we have that
$$|\lambda \oplus \lambda'|=|\lambda| + |\lambda'|,\  \  \  |\lambda \otimes \lambda'|=|\lambda| |\lambda'|.$$
\end{proposition}

\begin{proof}
The sum identity is immediate from partition multiplication and the definition of $\oplus$. For the product identity in Proposition \ref{mult}, we simply rewrite the right-hand side as
$$|\lambda| |\lambda'|=(\lambda_1+\lambda_2+...+\lambda_r) (\lambda_1'+\lambda_2'+...+\lambda_s').$$
Directly expanding the product on the right and inspecting the resulting summands, we see term-by-term that they are the parts of $\lambda\otimes \lambda'$.
\end{proof}

As in Chapter 3, we define $\operatorname{lg}(\lambda)$ and $\operatorname{sm}(\lambda)$ to denote the largest part and the smallest part of $\lambda$, respectively; %. Then %Furthermore, 
these complementary identities follow from the definitions of $\oplus,\otimes$.

\begin{proposition}\label{mult2}
For $\lambda,\lambda'\in \mathcal P$ we have the relations:
\begin{align*}
&\ell(\lambda\oplus\lambda')=\ell(\lambda)+\ell(\lambda'),&&\ell(\lambda\otimes\lambda')=\ell(\lambda)\ell(\lambda'),\\
&\operatorname{lg}(\lambda\otimes\lambda')=\operatorname{lg}(\lambda)\operatorname{lg}(\lambda'), &&\operatorname{sm}(\lambda\otimes\lambda')=\operatorname{sm}(\lambda)\operatorname{sm}(\lambda'),\\
&n_{\lambda\oplus\lambda'}=n_{\lambda}n_{\lambda'}, &&n_{\lambda\otimes\lambda'}=n_{\lambda}^{\ell(\lambda')}n_{\lambda'}^{\ell(\lambda)}.
\end{align*}
We also have that $$m_{k}(\lambda\oplus \lambda')=m_k(\lambda)+m_k(\lambda'),\  \  \  \  \  \  \  \  m_k(\lambda\otimes\lambda')=\sum_{d|k}m_d(\lambda)m_{{k}/{d}}(\lambda'), $$
where the final summation is taken over the divisors of $k$.
\end{proposition}

\begin{remark}
The next-to-last identity in Proposition \ref{mult2}, giving $m_{k}(\lambda\oplus \lambda')$, is equivalent to the definition of the operation $\oplus$ given in Andrews \cite{Andrews}. 
\end{remark}

\begin{proof}
All of these identities but the last one are immediate. The final summation is clear if we write Definition \ref{newprod} in the alternative notation % Andrews's ``frequency'' notation (his term for multiplicity \cite{Andrews}):
\begin{flalign*}
\lambda\otimes\lambda'&=(1^{m_1(\lambda)}2^{m_2(\lambda)}3^{m_3(\lambda)}4^{m_4(\lambda)}...)\otimes(1^{m_1(\lambda')}2^{m_2(\lambda')}3^{m_3(\lambda')}4^{m_4(\lambda')}...)\\
&=(1^{m_1(\lambda\otimes\lambda')}2^{m_2(\lambda\otimes\lambda')}3^{m_3(\lambda\otimes\lambda')}4^{m_4(\lambda\otimes\lambda')}...).
\end{flalign*}
For every pair of divisors $d, d'$ of a given part $k\in\lambda\otimes\lambda'$, the number of repetitions of $k$ in $\lambda\otimes \lambda'$ produced by the pairing $d\in \lambda,d'\in\lambda'$ is $m_{d}(\lambda)m_{d'}(\lambda')$. Noting in the definition of $\otimes$ that for each $k$ we sum over all pairings $d,d'$ with $dd'=k$, finishes the proof.
\end{proof}

%\begin{remark}

We note that the final identity in Proposition \ref{mult2} gives the tensor product of the $m_k$'s essentially %a rewriting of 
as Dirichlet convolution (see \cite{HardyWright}); to some extent, the arithmetic of partition multiplicities inherits the convenient algebra of convolutions. This observation, together with standard facts about convolutions, connects the tensor product to the algebra of classical Dirichlet series as well. %generating function of $m_k(\cdot)$ in this multiplication formula: %$k=1,2,3,...$, viz.  

\begin{corollary}\label{Dirichlet}
For $\lambda, \lambda' \in \mathcal P,\  s\in\mathbb C$, we have the multiplication identity
\begin{equation*}
\left(\sum_{k=1}^{\infty} \frac{m_k(\lambda)}{k^s} \right)\left(\sum_{k=1}^{\infty} \frac{m_k(\lambda')}{k^s} \right)=\sum_{k=1}^{\infty} \frac{m_k(\lambda\otimes \lambda')}{k^s}.
\end{equation*}
\end{corollary}
\begin{remark}
As these series have only finitely many nonzero $m_k$, they are well-defined for any $s\in \mathbb C$. For $s=0$, equation (\ref{Dirichlet}) reduces to $\ell(\lambda)\ell(\lambda')=\ell(\lambda\otimes \lambda')$.
%If $\lambda, \lambda'$ are partitions into distinct, coprime parts, and furthermore the two partitions have no parts in common, as $m_k=1$ for all parts $k$ we obtain a finite variation on a classical zeta function theorem (with $\sigma_0(\cdot)$ the number of divisors): 
%$$\left(\sum_{k\in \lambda} \frac{1}{k^s} \right) \left(\sum_{k\in \lambda'} \frac{1}{k^s} \right)=\sum_{k\in \lambda \otimes \lambda'} \frac{\sigma_0(k)}{k^s}$$
%We will see that the function $m_k(\lambda)$ is important in the theory we develop. 
\end{remark}

We would like to extend the preceding definitions and relations involving the operations $\oplus, \otimes$ to the larger class $\mathcal Q \supset \mathcal P$ of rational partitions. Of course, we can rewrite Definition \ref{Bantipartitiondef} in the form 
\begin{equation}\label{plusform}
\lambda\oplus \lambda^-=\emptyset.
\end{equation} 
But to extend all of the preceding relations to include antipartitions and partitions in $\mathcal Q$, we need to decide on an arithmetic for interactions between parts and antiparts. 
The bad news is, the part-antipart arithmetic that worked so well for our matrix analogy in the previous section is incompatible with Proposition \ref{mult}. Happily, in the present setting, we can in fact impose an even simpler rule than the relations in (\ref{Bmatrixop}): {\it Antiparts act like negative integers in both addition and multiplication.}

\begin{definition}\label{tensorop} 
For $a,b\in \mathbb Z^+$, in the context of the operations $\oplus, \otimes$ defined above, we set
\begin{equation*}
a b^-:=-ab\in\mathbb Z,\  \  \  \  a+b^-:=a-b\in \mathbb Z.
\end{equation*}
\end{definition}

\begin{remark}
We note that these imply $a^- b=(ab)^-=-ab$ as well, and also $a^-b^-=ab$.
\end{remark}

As in the previous section, we will translate negative numbers back to antiparts with the minus signs in the upper indices, when we write them inside partitions. Using Definition \ref{tensorop} to give meaning to the products of parts and antiparts, we can immediately generalize the structure we have built in this section to rational partitions $\mathcal Q$, just by inserting antipartitions and antiparts appropriately.

\begin{proposition}
%Definitions \ref{newsum}, \ref{newprod} and Theorems \ref{mult}, \ref{mult2} all 
All of the definitions and relations given above in this section may be extended to hold for $\lambda,\lambda'\in\mathcal Q$.
\end{proposition}

Now the real goal of this section is an easy deduction.

\begin{theorem}\label{ring}
The set $\mathcal Q$ of rational partitions is a commutative ring under the sum $\oplus$ and product $\otimes$.
\end{theorem}

\begin{proof}
Theorem \ref{Bgroup} plus Definition \ref{newsum} already give that $(\mathcal Q,\oplus)$ is an abelian group with identity element $\emptyset$. That $\otimes$ is associative and commutative are automatic from its set-theoretic definition, and we noted above that $I_1=(1)\in \mathcal Q$ is the tensor product identity. To see that associativity holds, we observe that for $\lambda,\alpha,\beta\in Q$ with $\ell(\lambda)=r,\ell(\alpha)=s,\ell(\beta)=t$, by the definitions of $\oplus$ and $\otimes$, 
$$\lambda\otimes(\alpha\oplus\beta)=(\lambda_1,\lambda_2,...,\lambda_r)\otimes(\alpha_1,\alpha_2,...,\alpha_s, \beta_1, \beta_2,...,\beta_t)$$
produces the following set of parts (which we then reorder to look like a partition):
$$
\lbrace \lambda_1\alpha_1,\lambda_2\alpha_2,...,\lambda_1\alpha_s,\lambda_2\alpha_1,\lambda_2\alpha_2...,\lambda_2\alpha_s,...,\lambda_r\alpha_1,\lambda_r\alpha_2...,\lambda_r \alpha_s,$$ $$\lambda_1 \beta_1, \lambda_1\beta_2,...,\lambda_1\beta_t,\lambda_2\beta_1,\lambda_2\beta_2,...,\lambda_2\beta_t,...,\lambda_r\beta_1,\lambda_r\beta_2...,\lambda_r \beta_t\rbrace.$$
By noting that this set of parts (reordered) is also identically equal to 
$$(\lambda\otimes\alpha)\oplus(\lambda\otimes\beta),$$
we prove the distributive property, verifying $(\mathcal Q,\oplus,\otimes)$ has the claimed ring structure. 
\end{proof}

Thus both the size $| \cdot |$ and length $\ell(\cdot)$ represent ring homomorphisms in $(\mathcal Q,\oplus,\otimes)$. 

We note that if we %recover arithmetic in the integers if we 
restrict our attention to partitions of the shapes
$$I_n=(1,1,...,1)\  \  \text{and}\  \  I_n^-=(1^-,1^-,...,1^-)\  \  (\text{$n$ repetitions in both cases}),$$
the set of all such partitions is isomorphic to the integers. Size and length are equal in these cases, and uniquely associate each partition $I_n$ (resp. $I_n^-$) to the integer $n$ (resp. $-n$). Moreover, from the $\oplus, \otimes$ operations we recover addition and multiplication on the integers %in two different ways 
via the size map ``$|\cdot|$''. If we restrict the operations to partitions into all 1's and anti-1's as above, % $+,\times$, respectively, %in the integers these quantities reduce to   then 
%e.g. %for $a,b\in \mathbb Z$ 
then we can write
\begin{flalign}\label{usual1}
a+b=|I_a\oplus I_b|,\  \  \  \  ab=|I_a\otimes I_b|,
\end{flalign}
and so on, with negative numbers coming into play when we involve antipartitions. (Note that we also have 
$a+b=\ell(I_a\oplus I_b),\  ab=\ell(I_a\otimes I_b).$) %We can also restrict $\oplus, \otimes$ 
%to the set of partitions of length one  
%and instantly recover elementary arithmetic:
%\begin{equation}\label{usual2}
%a+b=|(a)\oplus (b)|,\  \  \  \  ab=|(a)\otimes (b)|
%\end{equation}
Evidently, the partitions into all $1$'s and anti-1's %and the partitions into exactly one part are 
is a subring of $\mathcal Q$. 

In the next section we will explore further ring-theoretic aspects of these ideas.

\subsection{Ring theory in \texorpdfstring{$\mathcal Q$}{$Lg$}}%in $\mathcal Q$}

%{\bf RPS: In this tentative section (which can be absorbed into the previous section if we don't have much to say), we might give any ``textbook'' consequences we can think of in the ring-theoretic setting. Perhaps a discussion of ideals. Does this relate to Andrews's theory of partition ideals \cite{Andrews} in any way? (I'm guessing that is a reach but it would be awesome if so.)} 

Certainly many classical techniques from ring theory can be brought to bear on algebraic questions in $(\mathcal Q,\oplus,\otimes)$. In this section, we consider factorization %and the nature of ideals 
in the ring of partitions. %, before turning our attention to applications in the classical theory of partitions. %Here is an application of these algebraic ideas. 

%
%\section{Applications in $q$-series}
%{\bf RPS: This section will show classical-styled results about appropriately defined M\"{o}bius function, phi functions, etc. We will emphasize that the outputs of these functions are partitions, not numbers, so this is a purely set-theoretic theory---even though we have to allow arithmetic to take place between the parts of the partitions.}

%\section{Questions of factorization and irreducibility} 

%In this section 

First we address the question of which partitions in $\mathcal Q$ can in fact be written as a tensor product. 

\begin{definition}\label{divided}
For $\delta, \lambda\in \mathcal P$, let us write $\delta \| \lambda$ to mean we can write $\lambda$ as a tensor product of the form
$$\lambda =\delta \otimes \delta'$$ 
for some partition $\delta'$. If $\lambda$ has a nontrivial factorization (i.e., both $\delta,\delta'\neq (1)$), we say it is {\it reducible}; otherwise we call it {\it irreducible}.
\end{definition}

Note that $(1)\| \lambda$ and $\lambda\| \lambda$ for all partitions $\lambda$. %, and we can temporarily write division this way:
%$$\lambda \tiny{\  \textcircled{${\div} $}\  } \delta' = \delta$$ 
Then if $\delta \| \lambda$, it must obey all of the divisibility relations
\begin{equation}\label{divisibility}
|\delta|\  |\  |\lambda|,\  \  \  n_{\delta} | n_{\lambda},\  \  \  \ell(\delta) | \ell(\lambda),\  \  \  \operatorname{lg}(\delta) |\operatorname{lg}(\lambda),\  \  \  \operatorname{sm}(\delta) |\operatorname{sm}(\lambda).
\end{equation}
%where ``$\operatorname{lg}$'' and ``$\operatorname{sm}$'' denote the largest and smallest parts of a partition, respectively. 
That is %a lot of %restrictions%, so it seems these tensor product divisors must be somewhat scarce
%---and 
a lot of information about what a ``tensor divisor'' of $\lambda$ must look like. (These are just easy-to-check consequences of Proposition \ref{mult2}: %; %they are 
necessary, but not sufficient.) 

Of course, there is the trivial factorization for every partition:
$$\lambda=\lambda\otimes (1).$$
%If $\lambda$ can be factored in any nontrivial way as a tensor product, we will say it is {\it reducible}, and otherwise we will call it {\it irreducible}. 
It is also clear that if $ab$ is composite ($a,b\neq 1$), there is the factorization 
$$(ab)=(a)\otimes(b).$$
Continuing in this direction, from the relation $|\lambda\otimes\lambda'|=|\lambda| |\lambda'|$ in Proposition \ref{mult} we can fully characterize the reducible partitions of any integer $n$, as well as their tensor divisors.

Let $\mathcal R_n\subset \mathcal P$ denote the set of reducible partitions of $n$. It turns out we can construct the set $\mathcal R_n$ by applying the tensor product to the partitions of pairs $d,d'$ of nontrivial divisors of $n$.

\begin{theorem}\label{reduciblethm1}
The reducible partitions of $n>1$ are given by %identically equal to the set 
$$\mathcal R_n=\{\lambda\otimes\lambda' : \lambda\vdash d, \lambda'\vdash d'\  \text{for all}\  dd'=n,\  d,d'\neq 1\}.$$
\end{theorem}
 
As usual, let $p(n)$ denote the classical partition function, i.e., the number of partitions of $n\geq 1$, and set $r (n)$ %:=\#\mathcal R_n$ 
equal to the number of reducible partitions of $n$; then Theorem \ref{reduciblethm1} gives a natural upper bound on $r(n)$. 

\begin{corollary}
The number of reducible partitions of size $n$ obeys the inequality 
$$r(n)\leq \sum_{\substack{dd'=n\\d,d'\neq 1}} p(d)p(d').$$
\end{corollary}

In lieu of using Theorem \ref{reduciblethm1} to generate a list of reducible partitions of $n$ and then checking all the entries, in order to determine the reducibility or irreducibility of a given partition of size $n$, %the reducibility of a given partition seems to need to be proved case-by-case, like checking irreducible polynomials in abstract algebra. Here 
here are a few easy criteria for irreducibility that follow from Proposition \ref{mult2}. %, and in the next section, we show that the multiplicities $m_k(\cdot)$ are a useful tool for showing irreducibility as well.

\begin{proposition}\label{irred1}
We have the following rules-of-thumb for irreducibility: %he partition $\lambda \in \mathcal P$ is irreducible  under the following conditions:
\begin{enumerate}
\item If $|\lambda|$ is prime then $\lambda$ is irreducible. 
\item If all its parts are mutually coprime, then $\lambda$ is irreducible. 
\item If $\operatorname{lg}(\lambda)$ is prime with multiplicity $1$, then $\lambda$ is irreducible.  
\item If $\ell(\lambda)$ is prime and there is no integer $d>1$ dividing all the parts of $\lambda$, then  %either all the parts of $\lambda$ have a common divisor $d>1$ and $\lambda$ can be factored, e.g. $\lambda=(d)\otimes(\lambda_1/d, \lambda_2/d,...,\lambda_r/d)$, or 
$\lambda$ is irreducible. 
\end{enumerate}  
\end{proposition}

\begin{proof}
All of these items are obvious from Proposition \ref{mult2} together with the definition of the tensor product $\otimes$.
%Part (1) of Proposition \ref{irred1} is immediate from the identity $|\lambda\otimes \lambda'|=|\lambda|\  |\lambda'|$ of Proposition \ref{mult}; part (2) follows from $\ell(\lambda\otimes \lambda')=\ell(\lambda)\ell(\lambda')$ in Proposition \ref{mult2}; and parts (3), (4) and (5) are obvious from the definition of $\otimes$.
\end{proof}

There are also two easy rules in the affirmative direction, which the reader can easily check.

\begin{proposition}\label{reducprop1}
We have the following rules-of-thumb for reducibility:
\begin{enumerate}
\item If all the parts of $\lambda\in\mathcal P$ are divisible by some $k>1$, %with no part having multiplicity $\geq k$, 
then $\lambda$ is reducible and $(k)\|\lambda$.
\item If $k>1$ divides the multiplicity of every part of $\lambda$, %but no part is itself divisible by $k$, 
then $\lambda$ is reducible and $I_k\| \lambda$.
\end{enumerate}
\end{proposition}
%
%\begin{proof}
%{\bf Insert proof here.}
%\end{proof}

Of course, the %Consider that the 
subsets of partitions focused on in Proposition \ref{reducprop1}, the set $\mathcal P_{k\mathbb Z}$ of partitions into parts divisible by $k>1$, and the set $\mathcal P_{k|m_*}$ of partitions with multiplicities all divisible by $k$, are familiar to students of partition theory. Interestingly, both subsets are closed under the operations $\oplus$ and $\otimes$. In fact, $\mathcal P_{k\mathbb Z}$ and $\mathcal P_{k|m_*}$ are subgroups of $(\mathcal Q,\oplus)$ and subrings of $(\mathcal Q,\oplus,\otimes)$, and are also two-sided ideals in $(\mathcal Q,\oplus,\otimes)$ according to the standard usage in ring theory.

The author and his collaborator would like to see if standard theorems from classical ring theory extend to this partition-theoretic scheme. Moreover, it is our goal to use this ring structure to seek alternative proofs of partition bijections, Ramanujan-like congruences and other classical partition theorems, as well as to seek applications in Andrews's theory of partition ideals (see \cite{Andrews}).

\chapter{Notes on Chapter 4: Further observations}\label{app:C}

\label{app:C}{\bf Based on joint work with Maxwell Schneider}

%\section{End of chapter notes: Other directions}

%\section{Approximating zeroes of partition zeta functions} 
%Also, box partition zeta functions? golden ratio partition Dirichlet series?

\section{Sequentially congruent partitions}

We consider a somewhat exotic subset $\mathcal S\subset \mathcal P$ suggested by the indices of Corollary \ref{ch4etaquotient}, which we refer to as ``sequentially congruent partitions'', the parts of which obey abnormally strict congruence conditions. We find sequentially congruent partitions are in bijective correspondence with the set of all partitions, and yield explicit expressions for the coefficients of the expansions of a broad class of infinite products. Somehow these complicated-looking objects are embedded in a  natural way in partition theory. %Let $\mathcal P$ denote the set of all integer partitions of the shape $\lambda=(\lambda_1,\lambda_2,...,\lambda_r),\  \lambda_1 \geq \lambda_2\geq ... \geq \lambda_r \geq 1$, and for a given partition $\lambda$, let $|\lambda|$ denote its {\it size} (i.e., the sum of the parts) and $\ell(\lambda)=r$ denote its {\it length}. We define a complicated-looking set $\mathcal S \subset \mathcal P$ of ``{\it sequentially congruent partitions}'', which turn out to fit naturally in partition theory. 

\begin{definition}
We define a partition $\lambda$ to be {\it sequentially congruent} if 
the following congruences between the parts are all satisfied:
$$\lambda_1 \equiv \lambda_2\  (\operatorname{mod}\  1),\  
\lambda_2 \equiv \lambda_3\  (\operatorname{mod}\  2),\  \lambda_3 \equiv \lambda_4\  (\operatorname{mod}\  3),\  ...\  ,\  
\lambda_{r-1} \equiv \lambda_r\  (\operatorname{mod}\  r-1),$$
and for the smallest part, $\lambda_r \equiv 0\  (\operatorname{mod}\  r).$ 
\end{definition}

For example, the partition $(20,17,15,9,5)$ 
is sequentially congruent, because $20 \equiv 17\  (\operatorname{mod}\  1)$ trivially, $17\equiv 15\  (\operatorname{mod}\  2)$, $15\equiv 9\  (\operatorname{mod}\  3)$, $9\equiv 5\  (\operatorname{mod}\ 4)$, and $5\equiv 0\  (\operatorname{mod}\  5)$. On the other hand, $(21,18,16,10,6)$ is {\it not} sequentially congruent, for while the first four congruences still hold, clearly $6\not\equiv 0\  (\operatorname{mod}\  5)$. Note that increasing the largest part $\lambda_1$ of any $\lambda \in \mathcal S$ yields another partition in $\mathcal S$, as does adding or subtracting a fixed integer multiple of the length $r$ to all its parts, so long as the resulting parts are still positive.%\footnote{We might call this partition {\it weakly} sequentially congruent.}

No doubt, this congruence restriction on the parts hardly appears natural. However, it turns out sequentially congruent partitions are in one-to-one correspondence with the entire set of partitions. 

Let $\mathcal P_n$ denote the set of partitions of $n$, let $\mathcal S_{\operatorname{lg}=n}$ denote sequential partitions $\lambda$ whose largest part $\lambda_{1}$ is equal to $n$, and let $\#Q$ be the cardinality of a set $Q$ as usual.

\begin{theorem}\label{Cthm1}
There is an explicit bijection $\pi$ between the set $\mathcal S$ and the set $\mathcal P$. Furthermore, it is the case that
$$\pi(\mathcal S_{\operatorname{lg}=n})=\mathcal P_n,\  \  \  \  \  \  \  \pi^{-1}(\mathcal P_n)=\mathcal S_{\operatorname{lg}=n}.$$
%the set $\mathcal P_n$ of partitions of $n$ is in bijective correspondence with the set $\mathcal S_{\ell g=n}$ of sequentially congruent partitions having largest part equal to $n$. 
\end{theorem}

We prove these bijections directly, by construction. %, the second using generating functions.

\begin{proof}%[Proof 1]
For any partition $\lambda=(\lambda_1, \lambda_2,...,\lambda_i,...,\lambda_r)$, one can construct its sequentially congruent dual $$\lambda'=(\lambda_1',\lambda_2',...,\lambda_i',...,\lambda_r')$$ %--- noting that $\lambda$ itself may be hypercongruent as well --- 
by taking the parts equal to
\begin{equation}\label{Cconstruction1}
\lambda_i'=i\lambda_i+\sum_{j=i+1}^{r}\lambda_j.
\end{equation}
Note that $\lambda_r'\equiv 0\  (\operatorname{mod}\  r)$ as $\sum_{j=r+1}^{r}$ is empty; the other congruences between successive parts of $\lambda'$ are also immediate from equation (\ref{Cconstruction1}). 

Let us call this map $\pi$: \begin{flalign*}\pi\colon \mathcal P &\to \mathcal S\\ \lambda &\mapsto \lambda'\end{flalign*} The above argument establishes, in fact, that we have more strongly $\pi\colon \mathcal P_n \to \mathcal S_{\text{lg}=n}$. 

Conversely, given a sequentially congruent partition $\lambda'$, one can recover the dual partition $\lambda$ by working from right-to-left. Begin by computing the smallest part 
\begin{equation}
\lambda_r=\frac{\lambda_r'}{r},\end{equation}
then compute $\lambda_{r-1},\lambda_{r-2},...,\lambda_1$ in this order by taking
\begin{equation}\label{Cconstruction2}
\lambda_i=\frac{1}{r}\left( \lambda_i'-\sum_{j=i+1}^{r}\lambda_j \right).\end{equation}

We let this construction define the inverse map $\pi^{-1}\colon \mathcal S \to \mathcal P$. Noting that the uniqueness of $\lambda$ implies the uniqueness of $\lambda'$, and vice versa, the bijection between $\mathcal S$ and $\mathcal P$ follows from this two-way construction. Furthermore, observe that $\lambda_1'=|\lambda|$, thus every partition $\lambda$ of $n$ corresponds to a sequentially congruent partition $\lambda'$ with largest part $n$, and vice versa. 
\end{proof}

 %\begin{noindent}
We see by construction that $\pi(\lambda)=\lambda'$ ``looks similar'' to $\lambda$ in terms of length and distribution of the parts. For example, taking $\lambda = (2)$, $(3,1,1)$, $(2,2,2,2)$, respectively, and writing $\pi(\lambda)=\pi(\lambda_1,\lambda_2,...,\lambda_r)$ instead of $\pi\left( (\lambda_1,\lambda_2,...,\lambda_r)\right)$ for notational ease: $$\pi \left( 2 \right)=(2),\  \  \pi \left( 3,1,1\right)=(5,3,3),\  \  \pi \left( 2,2,2,2 \right)=(8,8,8,8).$$%\  \text{(See what I mean?)}$$
Thus the set $\pi(\mathcal P_n)=\mathcal S_{\text{lg}=n}$ ``looks like'' the set $\mathcal P_n$, even up to ``similar-looking'' partitions being in the same positions if we consider their ordering within each set (the map $\pi$ does not permute the sequentially congruent images within the set). Of course, the same is true for $\pi^{-1}(\mathcal S_{\text{lg}=n})=\mathcal P_n$ ``looking like'' $S_{\text{lg}=n}$.%This is ha

The sets $\mathcal P$ and $\mathcal S$ enjoy another interrelation that can be used to compute the coefficients of infinite products. Now, it is the first equality of Theorem \ref{ch41.1} in Chapter 4 (and equivalent to Equation 22.16 in Fine \cite{Fine}) that for a function $f\colon \mathbb N \to \mathbb C$ and $q\in \mathbb C$ with $f,q$ chosen such that the product converges absolutely, we can write
\begin{equation}\label{Czeta_thm}
\prod_{n=1}^{\infty}(1-f(n)q^n)^{-1}=\sum_{\lambda \in \mathcal P}q^{|\lambda|}\prod_{i\geq 1}f(i)^{m_i},
\end{equation}
where $m_i=m_i(\lambda)$ is the multiplicity of $i$ as a part of partition $\lambda$. In fact, it follows from another formula in Chapter 4 that the product on the left-hand side of (\ref{Czeta_thm}) can also be written as a sum over sequentially congruent partitions. 

Let $\operatorname{lg}(\lambda)=\lambda_1$ denote the {\it largest part} of partition $\lambda$, and set $\lambda_k=0$ if $k>\ell(\lambda)$.

\begin{theorem}\label{Cthm2} For $f\colon \mathbb N \to \mathbb C, q\in \mathbb C$ such that the product converges absolutely, we have
\begin{equation*}%\label{Czeta_thm}
\prod_{n=1}^{\infty}(1-f(n)q^n)^{-1}=\sum_{\lambda \in \mathcal S}q^{\operatorname{lg}(\lambda)}\prod_{i\geq 1}f(i)^{(\lambda_i-\lambda_{i+1})/i}.
\end{equation*}

\end{theorem}

\begin{proof}
Theorem \ref{Cthm2} results from Corollary \ref{ch4etaquotient} in Chapter 4. %(and can be proved from (\ref{zeta_thm}) 
%by repeated application of the Cauchy product formula) 
%by setting $f_j= f$ and $\mathbb X_j=\{j\}$ for all $j\in \mathbb N$, taking every $\pm$ equal to a minus sign, and letting $n \to \infty$. Now, f
For every $j\in \mathbb N$ take $\mathbb X_j=\{j\}$, fix $f_j= f$, and set $\pm$ equal to a minus sign,. In this case, $\lambda \in \mathcal P_{\mathbb X_j}$ means if $\lambda \neq \emptyset$ that $\lambda=(j,j,...,j)$, so we must have $j|(k_j - k_{j+1})$ in any nonempty partition sum on the right side above. Then every summand comprising $c_k$ %on the right side %of Cor. 2.9 %in \cite{Schneider_zeta} 
vanishes unless all the $k_i\leq k$ are parts of a sequentially congruent partition having length $\leq n$: %, as  %which follows from the definition $X_j=\{j\}$ in the conditions beneath the partition sums comprising the summands: 
each sum over partitions is empty (i.e., equal to zero) if $j$ does not divide $k_j-k_{j+1}$; is equal to 1 if $k_j - k_{j+1}=0$ as then $\lambda=\emptyset$ and $\prod_{\lambda_i\in \emptyset}$ is an empty product; or else has one term $f(j)^{m_j}=f(j)^{(k_j-k_{j+1})/j}$ as there is exactly one $\lambda = (j,j,...,j)$ with $|\lambda|=m_j j=k_j-k_{j+1}> 0$. Finally, let $n \to \infty$ so this argument encompasses partitions in $\mathcal S$ of unrestricted length.
\end{proof}

\begin{remark}
We note that setting $f=1$, then comparing equation (\ref{Czeta_thm}) to Theorem \ref{Cthm2}, gives another proof of Theorem \ref{Cthm1}: the sets $\mathcal S_{\operatorname{lg}=n}$ and $\mathcal P_n$ (and thus, the sets $\mathcal S$ and $\mathcal P$) have the same product generating function. \end{remark}

\begin{remark}If we instead take every $\pm$ equal to plus in Corollary \ref{ch4etaquotient}, we see there is also a bijection between partitions into {\it distinct} parts and a subset of $\mathcal S$, viz. partitions into parts with differences $\lambda_i - \lambda_{i+1}=i$ exactly.
\end{remark}

Comparing Theorem \ref{Cthm2} with (\ref{Czeta_thm}) above, we have two quite different-looking decompositions of the coefficients of $\prod_{n\geq1}(1-f(n)q^n)^{-1}$ as sums over partitions. One observes that these decompositions of $\sum_{\lambda \in \mathcal P_n}$ and $\sum_{\lambda \in \mathcal S_{\text{lg}=n}}$ have identical summands, that is, the sums do not just involve different numbers that add up to the same coefficient of $q^n$, but rather involve the same set of terms in seemingly a different order. Then one wonders: for given $\phi \in \mathcal S_{\text{lg}=n}$, precisely {\it which} partition $\gamma \in \mathcal P_n$ is such that $$\prod_{i \geq 1}f(i)^{(\phi_i-\phi_{i+1})/i}=\prod_{j \geq 1}f(j)^{m_j(\gamma)} ?$$
This $\gamma$ is generally {not} the same partition $\lambda=\pi^{-1}(\phi)$ as above. Then the set $\mathcal S$ evidently enjoys a second map (beside $\pi^{-1}$) to $\mathcal P$, which we will call $\sigma$:
\begin{flalign*}\sigma\colon \mathcal S &\to \mathcal P\\ \phi &\mapsto \gamma\end{flalign*}
We can easily write down this map by comparing the forms of the products above, using parts-multiplicity notation: % (instead of the weakly decreasing list of parts) for partitions:
$$\sigma(\phi):=(1^{\phi_1-\phi_{2}}\  2^{(\phi_2-\phi_{3})/2}\  3^{(\phi_3-\phi_{4})/3}...)=\gamma.$$
For example, $\sigma(5,3,3)=(1^{5-3}\  2^{(3-3)/2}\  3^{(3-0)/3})=(3,1,1)$, which in this case turns out to be the pre-image of $(5,3,3)$ over $\pi$ (but this is not generally the case, as we will see).

In fact, under this map we also have $\sigma(\mathcal S_{\text{lg}=n})=\mathcal P_n$, but a fact hidden by the preceding example and differing from the map $\pi$ is that $\sigma$ {\it does} permute the images in $\mathcal P_n$, so $\sigma(\mathcal S_{\text{lg}=n})$ and $\mathcal P_n$ do not ``look similar''. Then we also have that the composite map $$\sigma \pi\colon \mathcal P_n \to \mathcal P_n$$ permutes the set of partitions of $n$. %, thus is actually some permutation of $\{1,2,3,..., p(n)\}$. 
Similarly, the map $\pi \sigma\colon \mathcal S_{\text{lg}=n} \to \mathcal S_{\text{lg}=n}$ permutes the elements of $\mathcal S_{\text{lg}=n}$. %, thus represents a permutation of $\{1,2,3,...,p(n)\}$. Which permutations are they? %
A natural question one might ask is: %what kind of permutation structure might be happening as we alternately compose $\pi, \sigma$, that is, 
what if we apply $\sigma\pi\sigma\pi ... \sigma\pi$ to a partition of $n$? %Of course, we are just bouncing between partitions of $n$ and sequentially congruent partitions with largest part $n$, which are equinumerous. % but, more pertinently, finite. 
%What kind of cycle structure occurs? %Does this give us new information about the partitions of $n$? 
Let's see some examples for $n=1,2,3,4$. For $n=1$,
$$(1)\overset{\pi}\longmapsto (1)\overset{\sigma}\longmapsto  (1)$$
stabilizes right away as there is only one such partition. For $n=2$:
$$(2)\overset{\pi}\longmapsto (2)\overset{\sigma}\longmapsto (1,1)\overset{\pi}\longmapsto (2,2)\overset{\sigma}\longmapsto (2),$$
$$(1,1)\overset{\pi}\longmapsto (2,2)\overset{\sigma}\longmapsto (2)\overset{\pi}\longmapsto (2)\overset{\sigma}\longmapsto (1,1).$$%\overset{\pi}\to(1,1)\overset{\sigma}\to(2)$$
For $n=3$:
$$(3)\overset{\pi}\longmapsto (3)\overset{\sigma}\longmapsto (1,1,1)\overset{\pi}\longmapsto (3,3,3)\overset{\sigma}\longmapsto (3),$$
$$(2,1)\overset{\pi}\longmapsto (3,2)\overset{\sigma}\longmapsto (2,1),$$
$$(1,1,1)\overset{\pi}\longmapsto (3,3,3)\overset{\sigma}\longmapsto (3)\overset{\pi}\longmapsto (3)\overset{\sigma}\longmapsto (1,1,1).$$
Finally, for $n=4$:
$$(4)\overset{\pi}\longmapsto (4)\overset{\sigma}\longmapsto (1,1,1,1)\overset{\pi}\longmapsto (4,4,4,4)\overset{\sigma}\longmapsto (4),$$
$$(3,1)\overset{\pi}\longmapsto (4,2)\overset{\sigma}\longmapsto (2,1,1)\overset{\pi}\longmapsto (4,3,3)\overset{\sigma}\longmapsto (3,1),$$
$$(2,2)\overset{\pi}\longmapsto (4,4)\overset{\sigma}\longmapsto (2,2),$$
$$(2,1,1)\overset{\pi}\longmapsto (4,3,3)\overset{\sigma}\longmapsto (3,1)\overset{\pi}\longmapsto (4,2)\overset{\sigma}\longmapsto (2,1,1),$$
$$(1,1,1,1)\overset{\pi}\longmapsto (4,4,4,4)\overset{\sigma}\longmapsto (4)\overset{\pi}\longmapsto (4)\overset{\sigma}\longmapsto (1,1,1,1).$$

At this point the following fact is apparent.

\begin{theorem}\label{Cthm3}
The composite map $\sigma\pi\colon \mathcal P_n \to \mathcal P_n$ takes partitions to their conjugates. %Moreover, 
%$$\sigma\pi(\mathcal S_{\operatorname{lg}=n})=\mathcal P_n,\  \  \  \  \  \  \  \pi^{-1}(\mathcal P_n)=\mathcal S_{\operatorname{lg}=n}.$$
%the set $\mathcal P_n$ of partitions of $n$ is in bijective correspondence with the set $\mathcal S_{\ell g=n}$ of sequentially congruent partitions having largest part equal to $n$. 
\end{theorem}

\begin{proof} % together in light of the following lemma. % Proposition \ref{conjugate1} in Chapter 1.
If we write %a partition as
\begin{equation*}\label{Ceq1}
\lambda=(a_1^{m_{a_1}}a_2^{m_{a_2}}a_3^{m_{a_3}}...\  a_r^{m_{a_r}}),\  a_1>a_2>...>a_r\geq 1,
\end{equation*}
%where $m_{a_i}=m_{a_i}(\lambda)\geq 1$ denotes the multiplicity of $a_i \in \mathbb N$ as a part of $\lambda$. T
then we can compute the parts and multiplicities of the conjugate partition $$\lambda^*=(b_1^{m_{b_1}}b_2^{m_{b_2}}b_3^{m_{b_3}}...\  b_s^{m_{b_s}}),\  b_1>b_2>...>b_s\geq 1,$$ directly from the parts and multiplicities of $\lambda$ %as follows. 
%These statements are immediate 
by comparing the Ferrers-Young diagrams of $\lambda,\lambda^*$.

\begin{lemma}\label{Cconjugate1}
The conjugate $\lambda^*$ of partition $\lambda$ has  largest part $b_1$ given by
$$b_1=\ell(\lambda)=m_{a_1}+m_{a_2}+...+m_{a_r},\  \  \text{with}\  \  m_{b_1}(\lambda^*)=a_r,
$$
and for $1<i\leq s$, the parts and their multiplicities are given by
$$b_i=m_{a_1}+m_{a_2}+...+m_{a_{r-i+1}},\  \  \  \  m_{b_i}(\lambda^*)=a_{r-i+1}-a_{r-i+2}.
$$
%Furthermore, $\lambda^*$ has the same number of unique parts as does $\lambda$, that is, 
Moreover, we have that $s=r$. % ($\lambda^*$ has the same number of unique parts as $\lambda$).
\end{lemma}
The theorem results from using the definitions of the maps $\pi$ and $\sigma$, keeping track of the parts in the transformation $\lambda \mapsto \sigma \pi(\lambda)$, then comparing the parts of $\sigma \pi (\lambda)$ with the parts of $\lambda^*$ in Lemma \ref{Cconjugate1} above to see they are the same.
\end{proof}

%Thus $\sigma \pi$ is a bijection; m
Thus $\sigma \pi(\lambda)=\lambda$ if $\lambda$ is self-conjugate, and $(\sigma  \pi)^2(\lambda)=\lambda$ for all $\lambda\in\mathcal P$, as we can see in the examples above. For $\phi$ sequentially congruent, we also have  $\pi \sigma(\phi)=\phi$ if $\sigma(\phi)$ is self-conjugate, and $(\pi \sigma)^2(\phi)=\phi$ for all $\phi$. Interestingly, the composite map $\pi \sigma$ defines a duality analogous to conjugation of partitions in $\mathcal P_n$, that instead connects partitions $\phi $ and $\pi\sigma(\phi)$ in $S_{\operatorname{lg}=n}$. For instance, from the examples above, we have that $(2,1,1)$ and $(3,1)=\sigma \pi (2,1,1)$ are conjugates in $\mathcal P_4$, while $(4,3,3)$ and $(4,2)=\pi \sigma (4,3,3)$ are paired under the analogous duality in $\mathcal S_{\operatorname{lg}=4}$.

These phenomena give further partition-theoretic examples resembling structures in abstract algebra. One more fact is also evident by considering Ferrers-Young diagrams. 

\begin{theorem}
A sequentially congruent partition $\phi$ %with largest part $n$ 
is mapped by conjugation to a partition $\phi^*$ whose multiplicities $m_i=m_i(\phi^*)$ obey the congruence condition \begin{equation*}%\label{congr} 
m_i\equiv 0\  (\operatorname{mod}\  i). \end{equation*}
Conversely, any partition with parts obeying this congruence condition has a sequentially congruent partition as its conjugate. 
\end{theorem}

%I don't know if it depends on distinct parts or not and I have not proved this, but in a few examples I worked by hand, the {\it conjugates} of partitions obeying the above frequencies-congruence were always in $\mathcal S$: is there a bijection, or even just a one-way map, between these partitions and $\mathcal S$? I think this is either obvious, or false (isn't everything?), but in any event there may be further connections to partitions whose $m_i$ obey this congruence. 

\chapter{Notes on Chapter 5: Fa\`{a} di Bruno's formula in partition theory}\label{app:D}

%\section{End of chapter notes}
\section{Fa\`{a} di Bruno's formula with product version}%representation}

Francesco Fa\`{a} di Bruno was an Italian priest and mathematician active in the mid-nineteenth century. For the convenience of the reader, we record an easy proof of a useful variant of the formula that bears his name \cite{diBruno}, and also adjoin an infinite product representation to the usual statement of the identity, based on elementary ideas. We follow up with a few examples related to topics studied in this thesis\footnote{See \cite{Andrews}, Chapter 12, for more about this useful formula.}.
%
%First we must fix some notations. Let $\mathcal P$ denote the set of all integer partitions, including the empty partition $\emptyset$, and let $\lambda=(1^{m_1}\  2^{m_2}\  3^{m_3}  ...\  k^{m_k}  ...)$ denote a given partition with $m_k$ representing the {\it multiplicity} of $k$ as a part of $\lambda\in \mathcal P$, noting that $\lambda$ has only finitely many nonzero %multiplicities 
%$m_k$. Then we take $|\lambda|$ to be the {\it size} of $\lambda$ (i.e., the number being partitioned) with the convention $|\emptyset|:=0$, take $\ell(\lambda)=m_1+m_2+m_3+...$ to be its {\it length} (i.e., the number of parts) with $\ell(\emptyset):=0$, and we write ``$\lambda \vdash n$'' to mean $\lambda$ is a partition of $n$.  We let $q\in\mathbb C$, $|q|<1$ unless noted otherwise, and define the usual exponential function $\operatorname{exp}(x):=e^x$.
%
%We must first fix some notations. %is as follows ; %, although Fa\`{a} di Bruno was neither the first nor the last to re-discover and publish this result \ref{?}. 

%Using the partition notations incorporated in this work, w
We will write Fa\`{a} di Bruno's identity in a slightly simplified, equivalent form to that given in \eqref{ch5determinant}, as a sum over all partitions $\lambda$ --- making it amenable to techniques developed in this dissertation such as application of the $q$-bracket --- %, where it looks much cleaner than in the usual power series form\footnote{We also eliminate the factorials present in the summands inside the exponential function in \cite{diBruno}.}, 
and add in a product representation as well, using other classical facts. %This flexible result is unchanged, though, as we will see from the proof that follows.

\begin{proposition}[Fa\`{a} di Bruno's formula with product representation]
For $a(n)$ an arithmetic function and $a_n \in \mathbb C$, we have 
\begin{equation*}\label{DFaa2}
\operatorname{exp}\left({\sum_{n=1}^{\infty}a_n q^n}\right)=\sum_{\lambda\in \mathcal P}q^{|\lambda|}\frac{a_1^{m_1}a_2^{m_2}a_3^{m_3}...}{m_1!\  m_2!\  m_3!\  ...}=\prod_{n=1}^{\infty}(1-q^n)^{a(n)},
\end{equation*}
where $a_n$ and $a(n)$ are related by
$$a_n = -\frac{1}{n}\sum_{d|n}  a(d)d,\  \  \  \  a(n)\  =\  -\frac{1}{n}\sum_{d|n}\mu(n/d) a_d d$$
with the sums taken over divisors of $n$, and $\mu$ being the classical M\"{o}bius function.
\end{proposition}

\begin{proof}
To prove the first equality, we begin with the well-known {\it multinomial theorem}, re-written as a sum over partitions $\lambda$ in the set $\mathcal P_{[k]}\subset \mathcal P$ whose parts are all $\leq k$, having length $\ell(\lambda)=n$:
\begin{equation}\label{Dmultinomial}
(a_1+a_2+a_3+...+a_k)^n=n! \sum_{\substack{\lambda\in\mathcal P_{[k]}\\\ell(\lambda)=n}} \frac{a_1^{m_1}a_2^{m_2}a_3^{m_3}...a_k^{m_k}}{m_1!\  m_2!\  m_3!\  ...\  m_k!}.
\end{equation}
If we let $k$ tend to infinity, assuming the infinite sum $a_1+a_2+a_3+...$ converges, the series on the right becomes a sum over all partitions of length $n$. Then dividing both sides of (\ref{Dmultinomial}) by $n!$ and summing over $n\geq 0$, % yields  
%where we might take $k\geq n$ (noting then that $k>n$ does not contribute to any partition of $n$). For $N \leq k$, dividing both sides of (\ref{multinomial}) by $n!$ and summing over $0\leq n\leq N$, then %yields
%%\begin{equation*}%\label{multinomial2}
%%\sum_{n=0}^{N} \frac{(a_1+a_2+a_3+...+a_k)^n}{n!}= \sum_{n=0}^{N} \sum_{\lambda \vdash n} \frac{a_1^{m_1}a_2^{m_2}a_3^{m_3}...a_k^{m_k}}{m_1!\  m_2!\  m_3!\  ...\  m_k!}.
%%\end{equation*}
%letting both $N$ and $k\geq N$ tend to infinity under the condition of uniform convergence, % in both sums, %condition that $a_1+a_2+a_3+...$ converges uniformly, 
%we arrive at 
%\begin{equation}\label{multinomial3}
%\sum_{n=0}^{\infty} \frac{(a_1+a_2+a_3+...)^n}{n!}= \sum_{n=0}^{\infty} \sum_{\substack{\lambda\in\mathcal P \\ \ell(\lambda)=n}} \frac{a_1^{m_1}a_2^{m_2}a_3^{m_3}...}{m_1!\  m_2!\  m_3!\  ...}.
%\end{equation}
the left-hand side yields the Maclaurin series expansion for $\operatorname{exp}(a_1+a_2+a_3+...)$, and the right side can be rewritten as a sum over all partitions: %Then %so long as $a_1+a_2+a_3+...$ is convergent, 
%we have 
\begin{equation}\label{Dmultinomial4}\operatorname{exp}(a_1+a_2+a_3+...)=\sum_{\lambda\in \mathcal P} \frac{a_1^{m_1}a_2^{m_2}a_3^{m_3}...}{m_1!\  m_2!\  m_3!\  ...}.
\end{equation}
Now, taking $a_k\mapsto a_k q^k$ in (\ref{Dmultinomial4}), %\footnote{Taking $a_k\mapsto a_k q^k/k!$ instead, implies the more common version of Fa\`{a} di Bruno's formula connected to Bell polynomials.}, 
we can write  
\begin{flalign*}\label{DFaa3}
a_1^{m_1}a_2^{m_2}a_3^{m_3}...\mapsto&\  \  (a_1 q)^{m_1}(a_2 q^2)^{m_2}(a_3 q^3)^{m_3}...\\=&q^{m_1+2m_2+3m_3+...}a_1^{m_1}a_2^{m_2}a_3^{m_3}...=q^{|\lambda|}a_1^{m_1}a_2^{m_2}a_3^{m_3}...
\end{flalign*}
in the summands on the right-hand side, which completes the series aspect of the proof. 

The product representation follows from Bruinier, Kohnen and Ono \cite{BKO}, and also is immediate from the proofs in \cite{Schneider_golden} if we replace $-f(n)/n$ with an arithmetic function $a(n)$ in the final equation of that paper. For a given $a(n)$, if we set \begin{equation}\label{Dcoeff}
a_n\  =\  -\frac{1}{n}\sum_{d|n} d\cdot a(d),
\end{equation}
we have 
\begin{equation*}%\label{prod}
\prod_{n=1}^{\infty}(1-q^n)^{{a(n)}}=\operatorname{exp}\left({\sum_{n=1}^{\infty}{a_n}  q^n}\right).
\end{equation*}
Applying M\"{o}bius inversion to (\ref{Dcoeff}) gives the converse divisor sum identity for $a(n)$.
 
%Then we can expand the right-hand side of this equation using (\ref{multinomial4}): 
%\begin{flalign*}
%=\sum_{\lambda\in \mathcal P} q^{|\lambda|}\frac{(\frac{a_1}{1})^{m_1}(\frac{a_2}{2})^{m_2}(\frac{a_3}{3})^{m_3}...}{m_1!\  m_2!\  m_3!\  ...}\\&=\sum_{\lambda\in \mathcal P}q^{|\lambda|} \frac{a_1^{m_1}a_2^{m_2}a_3^{m_3}...}{1^{m_1}2^{m_2}3^{m_3}\cdots m_1!\  m_2!\  m_3!\  ...}
%\end{flalign*}
%which is .
%\begin{remark}
%Formulas like these are related to modular forms in \cite{BKO}. 
%\end{remark}
\end{proof}

\section{Further examples}% applications}% to partition Dirichlet series}

So we can view Fa\`{a} di Bruno's formula as a generating function for coefficients of certain partition-theoretic sums involving the form $(a_1^{m_1}a_2^{m_2}a_3^{m_3}...)/(m_1!\  m_2!\  m_3!\  ...)$. As examples, we give a few simple substitutions that lead to interesting partition sum identities. %In what follows, let $(a;q)_{\infty}:=\prod_{n\geq 0}(1-aq^n)$ denote the infinite {\it $q$-Pochhammer symbol}, and $n_{\lambda}:=\prod_{k=1}^{\infty}k^{m_k}$ denote the so-called {\it norm} of $\lambda$ (the product of the parts) in the notation of \cite{Schneider_arithmetic, Schneider_zeta}, with $n_{\emptyset}:=1$. %, connected to the Riemann zeta function and classical geometric series.
%

%
%\begin{example}
%Setting $a_i\equiv 1$ in Proposition \ref{Faa2} gives immediately 
%$$\operatorname{exp}\left(\frac{q}{1-q}\right)=\sum_{\lambda\in \mathcal P} \frac{q^{|\lambda|}}{m_1!\  m_2!\  m_3!\  ...}.$$
%\end{example}

\begin{example}
Setting $a_i=i^{-s},\operatorname{Re}(s)>1$, and $q=1$ in Proposition \ref{DFaa2}, %after a little simplification, 
gives
$$\operatorname{exp}\left({\zeta(s)}\right)=\sum_{\lambda\in \mathcal P} \frac{1}{n_{\lambda}^{s}\  m_1!\  m_2!\  m_3!\  ...}$$
with $\zeta(s):=\sum_{n\geq 1}n^{-s}$ the Riemann zeta function.%ooks like a different variety of partition zeta functions.%' in \cite{ORS_zeta, Schneider_zeta}. 
\end{example}

%\begin{remark}
We note that the right-hand side of this example is a type of partition Dirichlet series. More generally, if we exponentiate a convergent classical Dirichlet series $\sum_{n=1}^{\infty}a(n)n^{-s}$ we arrive at a partition Dirichlet series of the form introduced at the end of Chapter 5, viz. \begin{equation}\sum_{\lambda \in \mathcal P}A(\lambda)n_{\lambda}^{-s}=\text{exp}\left(\sum_{n=1}^{\infty}a(n)n^{-s}\right),\end{equation} where \begin{equation}\label{appAdef} A(\lambda):=\frac{a(1)^{m_1}a(2)^{m_2}a(3)^{m_3}\cdots a(i)^{m_i}\cdots}{m_1!\  m_2!\  m_3!\  \cdots\  m_i!\  \cdots}.\end{equation} In fact, $A(\lambda)$ is multiplicative in the partition sense (see Definition \ref{ch3multfctn}), that is, $A(\lambda\gamma)=A(\lambda)A(\gamma)$ when $\operatorname{gcd}(\lambda,\gamma)=\emptyset$.\footnote{We note for the subset $\mathcal P^*$ of partitions into distinct parts there is the simpler Euler product generating function $\prod_{n=1}^{\infty}(1+a(n)n^{-s})=\sum_{\lambda \in \mathcal P^*}a(1) a(2) a(3)\cdots n_{\lambda}^{-s}=\sum_{\lambda \in \mathcal P^*}A(\lambda)n_{\lambda}^{-s}$.}

Next, we give alternative evaluations of functions evaluated in Appendix B, (\ref{Btau}) and (\ref{Bkcolor}). 

\begin{example}
Setting $a\equiv 24$ in Proposition \ref{DFaa2} yields $a_i=-24\sigma(i)$, where $\sigma(i)=\sum_{d|i}d$ as usual. Then Ramanujan's tau function $\tau(n)$ %, which is the $n$th coefficient of $q(q;q)_{\infty}^{24}$, 
can be written
$$\tau(n)=\sum_{\lambda \vdash (n-1)}(-24)^{\ell(\lambda)}\frac{\sigma(1)^{m_1}\sigma(2)^{m_2}\sigma(3)^{m_3}...}{n_{\lambda}\  m_1!\  m_2!\  m_3!\  ...}.$$
\end{example}

\begin{example}
Setting $f\equiv -k$ with $k\geq 1$ in Proposition \ref{DFaa2} yields $a_i= k \sigma(i)$. Then the number $P_k(n)$ of $k$-color partitions of $n$ %, which is the $n$th coefficient of $(q;q)_{\infty}^{-k}$, 
can be written
$$P_k(n)=\sum_{\lambda \vdash n} k^{\ell(\lambda)}\frac{\sigma(1)^{m_1}\sigma(2)^{m_2}\sigma(3)^{m_3}...}{n_{\lambda}\  m_1!\  m_2!\  m_3!\  ...}.$$
\end{example}

Let $\varphi=\frac{1+\sqrt{5}}{2}$ denote the {\it golden ratio}, a number that makes connections throughout the sciences, nature and the arts. The reciprocal of the golden ratio is similarly ``golden'': the two constants are intertwined in classical relations like 
\begin{equation}\label{Dphi}
\varphi=1+\frac{1}{\varphi}.
\end{equation} 
Then we can write down formulas to compute $\varphi$ and $1/\varphi$ in terms of $\pi$ and $\zeta(s)$.% via sums over partitions. 

\begin{example}\label{Dthm1}
We have the following identities for the golden ratio and its reciprocal:
\begin{equation}\label{Dthm1.2}
\varphi\  =\  \frac{5}{\pi}\sum_{\lambda\in \mathcal P } \frac{\zeta(2)^{m_1}\zeta(4)^{m_2}\zeta(6)^{m_3}\zeta(8)^{m_4}\cdots}{\  n_{\lambda}\   100^{|\lambda|}\  m_1!\  m_2!\   m_3!\   m_4!\  \cdots},
\end{equation}
\begin{equation}\label{Dthm2.2}
\frac{1}{\varphi}\  =\  \frac{\pi}{5}\sum_{\lambda\in \mathcal P  } \frac{(-1)^{\ell(\lambda)}\zeta(2)^{m_1}\zeta(4)^{m_2}\zeta(6)^{m_3}\zeta(8)^{m_4}\cdots}{n_{\lambda}\   100^{|\lambda|}\  m_1! \  m_2!\  m_3!\  m_4!\  \cdots}.
\end{equation}
\end{example}

Set $b_{2j}:={(-1)^{j+1}B_{2j}2^{2j-1}}/{(2j)!}$ with $B_k\in \mathbb Q$ the $k$th Bernoulli number. Then %\begin{equation}\label{beta}
$\zeta(2j)=\pi^{2j} b_{2j}$ %\end{equation*} 
for $j\in \mathbb Z^+$, by Euler. Comparing this fact with equations (\ref{Dthm1.2}) and (\ref{Dthm2.2}) implies additional expressions giving $\varphi$ and $1/\varphi$ in terms of $\pi$.

\begin{example}\label{Dcor}
We have the following identities for the golden ratio and its reciprocal:
\begin{equation}\varphi\  =\  {5}\sum_{\lambda\in \mathcal P } \frac{\pi^{2|\lambda|-1}b_2^{m_1}b_4^{m_2}b_6^{m_3}b_8^{m_4}\cdots}{\  n_{\lambda}\   100^{|\lambda|}\  m_1!\  m_2!\   m_3!\   m_4!\  \cdots},\end{equation}
\begin{equation}\label{Dcor3.2}
\frac{1}{\varphi}\  =\  \frac{1}{5}\sum_{\lambda\in \mathcal P } \frac{(-1)^{\ell(\lambda)}\pi^{2|\lambda|+1}b_2^{m_1}b_4^{m_2}b_6^{m_3}b_8^{m_4}\cdots}{\  n_{\lambda}\   100^{|\lambda|}\  m_1!\  m_2!\   m_3!\   m_4!\  \cdots}.
\end{equation}
\end{example}

Then by the classical relation (\ref{Dphi}), further formulas for $\varphi$ may be obtained from adding $1$ to both sides of equations (\ref{Dthm2.2}) and (\ref{Dcor3.2}).% and (\ref{pf2.2}). 

\begin{proof}
It is a straightforward deduction from geometry (see \cite{phi}) that we can write $$\frac{1}{\varphi}:=\frac{-1+\sqrt{5}}{2}=2\  \text{sin}\left(\frac{\pi}{10} \right).$$ 
Comparing this result to Euler's formula % for the sine function %setting $x=\frac{\pi}{10}$ in Euler's formula 
$\text{sin}(x)=x\prod_{n=1}^{\infty}\left(1-\frac{x^2}{\pi^2 n^2}\right)$ with $x=\pi/10$, then gives %, multiplying by 2 and taking reciprocals, we have
\begin{equation}\label{Deq}\varphi = \frac{1}{2\  \text{sin}(\frac{\pi}{10})}=\frac{5}{\pi}\prod_{n=1}^{\infty}\left(1-\frac{1}{100\  n^2}\right)^{-1}.\end{equation}
At this stage, we note it follows immediately from Theorems \ref{ch41.1} and \ref{ch41.11} in Chapter 4 that %for $f\colon \mathbb N \to \mathbb C$ such that the following products converge absolutely, we have 
%$$\prod_{n=1}^{\infty}(1-f(n))^{-1}=\sum_{\lambda\in \mathcal P}\prod_{n \in \lambda}f(n),\  \  \  \  \  \  \prod_{n=1}^{\infty}(1-f(n))=\sum_{\lambda\in \mathcal P}\mu_{\mathcal P}(\lambda)\prod_{n \in \lambda}f(n),$$ 
%where $j\in \lambda$ means $j\geq 1$ is a part of partition $\lambda$, thus the right-hand products are taken over the parts of $\lambda$ (we set $\prod_{j \in \emptyset}$ equal to 1). Taking $f(n)={1}/{100n^2}$ (thus we have absolute convergence) and combining the above equations yields (\ref{thm1.1}) and (\ref{thm2.1}), respectively.\footnote{Expressions (\ref{thm1.1}) and (\ref{thm2.1}) represent ``partition Dirichlet series'' in the sense of \cite{ORS}.
%}
%
\begin{equation}%\label{thm1.1}
\varphi\  =\  \frac{5}{\pi}\sum_{\lambda\in \mathcal P} \frac{1}{\  n_{\lambda}^2\  100^{\ell(\lambda)}},\  \  \  \  %\end{equation}
%
%\begin{equation}\label{pf2.1}
\frac{1}{\varphi}\  =\  \frac{\pi}{5}\sum_{\lambda\in \mathcal P } \frac{\mu_{\mathcal P}(\lambda)}{\  n_{\lambda}^2\  100^{\ell(\lambda)}},\end{equation}
which are  further examples of partition Dirichlet series. Now, to prove (\ref{Dthm1.2}) and (\ref{Dthm2.2}), begin with the Maclaurin expansion of the natural logarithm $$-\text{ln}(1-x)=\sum_{j=1}^{\infty}\frac{x^j}{j},\  \  \  \  |x|<1.$$ 
Setting $\text{exp}(x):=e^x$, we then take $x=1/100n^2<1$ for each $n=1,2,3,...$ to write
\begin{flalign*}
\prod_{n=1}^{\infty}\left(1-\frac{1}{100\  n^2}\right)^{-1}&=\prod_{n=1}^{\infty}\text{exp}\left(-\text{ln}\left(1-\frac{1}{100n^2}\right)\right) =\text{exp}\left(\sum_{n=1}^{\infty}\sum_{j=1}^{\infty}\frac{1}{n^{2j}100^j j }\right)\\
&=\text{exp}\left(\sum_{j=1}^{\infty}\frac{1}{100^j j }\sum_{n=1}^{\infty}\frac{1}{n^{2j} }\right)=\text{exp}\left(\sum_{j=1}^{\infty}\frac{\zeta(2j)}{100^j j }\right).
\end{flalign*}
Then setting $q=1,a_j=\zeta(2j)/j 100^j$ in Proposition \label{DFaa2} and comparing all this to (\ref{Deq}), plus some algebra, proves (\ref{Dthm1.2}). Setting $a_j=-\zeta(2j)/j 100^j$ (a minus sign is introduced by taking reciprocals) gives (\ref{Dthm2.2}).

Example \ref{Dcor} follows easily from (\ref{Dthm1.2}) and (\ref{Dthm2.2}) by making the substitution $\zeta(2j)\mapsto \pi^{2j}b_{2j}$ in each summand.%, then observing $2m_1 + 4m_2+6m_3+8m_4+...=2(m_1 + 2m_2+3m_3+4m_4+...)=2|\lambda|$ in the exponent of $\pi$.
\end{proof}

Using the Maclaurin expansion of the natural logarithm from the above proof plus a little algebra using a summation swap and geometric series in the exponential, we have %it is easy to see
$$(z;q)_{\infty}^{-1}=\prod_{k=0}^{\infty}\text{exp}\left(\sum_{n=1}^{\infty}\frac{z^n q^{nk}}{n}\right)=\text{exp}\left(\sum_{n=1}^{\infty}\frac{z^n}{n(1-q^n)}\right).$$
It is a case of the $q$-binomial theorem (see Lemma \ref{ch6lemma1}) that $(z;q)_{\infty}^{-1}=\sum_{n=0}^{\infty}\frac{z^n}{(q;q)_n}$. Combining these formulas with Fa\`{as} di Bruno's formula gives our next example.
\begin{example}\label{Dqbinomial} We have that 
$$(z;q)_{\infty}^{-1}=\sum_{n=0}^{\infty}\frac{z^n}{(q;q)_n}=\sum_{\lambda\in \mathcal P}\frac{z^{|\lambda|}}{n_{\lambda}\  m_1!\  m_2!\  m_3!\   \cdots\  (1-q)^{m_1}(1-q^2)^{m_2}(1-q^3)^{m_3}\  \cdots}.$$
\end{example}

Comparing coefficients on both sides of this identity gives Chapter 12, Example 1 of \cite{Andrews}, which Andrews attributes to Cayley, but Sills argues in  \cite{SillsMM} is due to MacMahon \cite{MacMahon}. Noting from Example \ref{Dqbinomial} (and Lambert series) that 
\begin{equation}\label{Faa_q}
(q;q)_{\infty}^{-1}=\operatorname{exp}\left(\sum_{n=1}^{\infty}\frac{q^n}{n(1-q^n)}\right)=\operatorname{exp}\left(\sum_{n=1}^{\infty}\frac{\sigma(n)q^n}{n}\right),\end{equation}
as a final example we show the $q$-bracket of $A(\lambda)$ from \eqref{appAdef} above takes a nice form.

\begin{example}\label{Faabracket}
We have that 
\begin{flalign}\left< A \right>_q = \sum_{\lambda \in \mathcal P}q^{|\lambda|}\frac{(a(1)-\sigma(1))^{m_1}(a(2)-\sigma(2)/2)^{m_2}\cdots (a(i)-\sigma(i)/i)^{m_i}\cdots}{m_1!\  m_2!\  \cdots\  m_i!\  \cdots}.\end{flalign}
More generally, using the notation of Definitions \ref{ch3diffndefantidef} and \ref{ch3antidef}, it is the case that
\begin{flalign}\left< A \right>_q^{(\pm k)} = \sum_{\lambda \in \mathcal P}q^{|\lambda|}\frac{(a(1)\mp k\sigma(1))^{m_1}(a(2)\mp k\sigma(2)/2)^{m_2}\cdots (a(i)\mp k\sigma(i)/i)^{m_i}\cdots}{m_1!\  m_2!\  \cdots\  m_i!\  \cdots},\end{flalign}
where with regard to ``$\pm,\  \mp$'', a plus on the left gives minus on the right, and vice versa.
\end{example}

It is interesting that multiplication and division by $(q;q)_{\infty}$ produces this homogenous shift in the values of the coefficients in the numerator, by terms involving $\sigma(n)$.

\begin{proof}
This follows from writing $(q;q)_{\infty}$ as the reciprocal of \eqref{Faa_q} (noting a minus sign is introduced inside the exponential) and using $\operatorname{exp}\left(\sum_{n=1}^{\infty} a(n)q^n  \right)=\sum_{\lambda \in \mathcal P}A(\lambda)q^{|\lambda|}$.
\end{proof}

In addition to applications in number theory, the author and his collaborators in the Emory Working Group in Number Theory and Molecular Simulation (an interdisciplinary research group run by Professor James Kindt in Emory's Chemistry Department) make extensive use of Fa\`{a} di Bruno's formula in theoretical chemistry to develop simulation algorithms and probe classical laws like the Law of Mass Action from partition-theoretic first principles (see, for example, \cite{Kindt}).

%\end{example}

%\subsection{Expected values}

\chapter{Notes on Chapter 6: Further relations involving  $F_{\mathbb S_{r,t}}$}\label{app:E}

\section{Classical series and arithmetic functions}
In this note we essentially use the left-hand side of Theorem \ref{ch6rt}, viz. the limit %from inside the unit circle
%, the identity %If  $0\leq r<t$ are integers and $\gcd(m,t)=1$, then we have that
\begin{equation}\label{Elimiting}\lim_{q \to \zeta} F_{\mathbb S_{r,t}}(q) = -\lim_{q\to \zeta}\sum_{\substack{ \lambda \in \mathcal{P}\\ \operatorname{sm}(\lambda) \in \mathbb S_{r,t}}} \mu_{\mathcal{P}}(\lambda) q^{| \lambda|}= \lim_{q\to \zeta}\  (q;q)_{\infty}\sum_{\substack{ \lambda \in \mathcal{P}\\ \operatorname{lg}(\lambda) \in \mathbb S_{r,t}}} q^{| \lambda|}\end{equation}%= \frac{1}{t}
%with $q\to \zeta$ a primitive $m$th root of unity 
from inside the unit circle, %where $r,t\in\mathbb Z$, $0\leq r <t$ and $\zeta$ is a primitive $m$th root of unity, 
as an elaborate way to write $1/t$. Then trivially, we can rewrite many classical series as limits of this type. % (and by \eqref{Elimiting}, give partition-theoretic interpretations). 
For instance, 
%, for instance:
%%like %finite 
%%geometric series
%\begin{equation}
%\lim_{q \to \zeta} F_{\mathbb S_{r,t}}(q) + F_{\mathbb S_{r,t}}^2(q) + F_{\mathbb S_{r,t}}^3(q) +...+F_{\mathbb S_{r,t}}^N(q)=\frac{1}{t}+\frac{1}{t^2}+\frac{1}{t^3}+...+\frac{1}{t^N}=\frac{t^N-1}{t(t-1)}\end{equation}
%% including the limiting case as $N\to\infty$. 
if we set $r=0$ to satisfy $r<t$ for all $t\geq 1$, we can rewrite the zeta function as the limit of a Dirichlet  series  %generating function % Riemann zeta function
\begin{equation}
\zeta(s)= \lim_{q \to \zeta} \sum_{t=1}^{\infty}F_{\mathbb S_{0,t}}(q) t^{s-1}\  \  \  (\text{Re}(s)>1).\end{equation}%=\lim_{q \to \zeta} \sum_{t=1}^{\infty}F_{\mathbb S_{0,t}}^2(q) t^{s-2}=...=\lim_{q \to \zeta} \sum_{t=1}^{\infty}F_{\mathbb S_{0,t}}^k(q)t^{s-k}\end{equation}
% the MacLaurin expansion of the natural logarithm as the limit of a generating function 
%\begin{equation}\label{log}
%-{\operatorname{log}(1-x)}=\lim_{q \to \zeta}\sum_{t=1}^{\infty}{F}_{\mathbb S_{0,t}}(q) x^t
%\end{equation}
%for $|x|<1$, 
%and indeed this is the case anywhere factors of $1/t$ appear. %the $n$th harmonic number $H_n$ and other classical constants, viz.
%\begin{equation}\label{harmonic}
%H_n=\lim_{q\to \zeta}\sum_{t=1}^{n}{F}_{\mathbb S_{0,t}}(q),\  \  \  \operatorname{log}2=\lim_{q\to \zeta}\sum_{t=1}^{\infty}(-1)^{t-1} {F}_{\mathbb S_{0,t}}(q),\  \  \  \frac{\pi}{4}=\lim_{q\to \zeta}\sum_{j=1}^{\infty}(-1)^{j-1}{F}_{\mathbb S_{0,2j-1}}(q).\end{equation}
%Many examples of this sort are available, wherever a factor of $1/t$ appears.
%; and there are divisor sum relations such as 
%\begin{equation}\label{divisor}
%\lim_{u \to 0^+}\sum_{d|t}\mu(d)\widetilde{F}_{0,d}(q)=\frac{\varphi(t)}{t},
%\end{equation}
%where $\mu$ is the classical M\"{o}bius function and $\varphi$ is the Euler phi function. 
Another elementary observation is that if $A(t):=\sum_{d|t}a(d)$ for $a\colon \mathbb N \to \mathbb C$, we have
\begin{flalign*}
\sum_{t=1}^{\infty}A(t)\frac{q^t}{(q;q)_t}  =\sum_{t=1}^{\infty} a(t)\sum_{k=1}^{\infty}\frac{q^{tk}}{(q;q)_{tk}}&=\sum_{t=1}^{\infty}a(t)\left(\frac{{F}_{\mathbb S_{0,t}}(q)}{(q;q)_{\infty}}+1\right).%\\ &=-(q;q)_{\infty}^{-1}\sum_{t=1}^{\infty}a(t){{F}_{\mathbb S_{0,t}}(q)}-\sum_{t=1}^{\infty}a(t).
\end{flalign*}
Reorganizing gives the following identity.
\begin{proposition}\label{EFthm}
Using the above notation, if $\sum_{t=1}^{\infty}a(t)$ converges we have
$$(q;q)_{\infty}^{-1}\sum_{t=1}^{\infty}a(t){{F}_{\mathbb S_{0,t}}(q)} = \sum_{t=1}^{\infty}a(t)+\sum_{t=1}^{\infty}A(t)\frac{q^t}{(q;q)_t}.$$
In terms of partitions we can write
$$\sum_{t=1}^{\infty}\sum_{ \substack{\emptyset  \neq \lambda \in \mathcal{P}\\ \operatorname{lg}(\lambda) \in \mathbb S_{0,t}}} a(t) q^{| \lambda|}= \sum_{\emptyset\neq \lambda \in \mathcal P} A(\operatorname{lg}(\lambda))  q^{|\lambda|}.$$
\end{proposition}

\begin{remark}
We note by conjugation that $\text{lg}(\lambda)=\ell(\lambda^*)$ and $|\lambda|=|\lambda^*|$, thus for any $f\colon \mathbb N \to \mathbb C$ we have %that
$\sum_{\lambda \in \mathcal P}f(\operatorname{lg}(\lambda))q^{|\lambda|}=\sum_{\lambda \in \mathcal P}f(\ell(\lambda))q^{|\lambda|}$ (which also holds for sums $\sum_{\emptyset \neq \lambda\in \mathcal P}$ above).
\end{remark}

Proposition \ref{EFthm} is useful in further applying the ideas of Chapter 6. Here is an example that gives another $q$-series relation to arithmetic densities.
\begin{example}
Set $a(n)=\mu(n)/n$ in Proposition \ref{EFthm} with $\mu$ the M\"{o}bius function. Then as $A(n)=\sum_{d|n}\mu(d)/d=\varphi(n)/n$, using the classical facts $\sum_{n\geq 1}\mu(n)/n=0$ and $\sum_{n\geq 1}\mu(n)/n^2=\zeta(2)^{-1}$ together with Theorem \ref{ch6rt} and a little algebra, we have
\begin{equation}
\lim_{q\to\zeta}\  (q;q)_{\infty} \sum_{n=1}^{\infty}\frac{\varphi(n)q^n}{n(q;q)_n} =\frac{6}{\pi^2},
\end{equation}
which is well known to be $\lim_{n\to \infty} \frac{1}{n}\sum_{k=1}^{n} \varphi(k)/k$.
\end{example}
\begin{remark}One wonders if there are more general classes of arithmetic functions $f(n)$ with

\begin{equation*}
\lim_{q\to\zeta}\  (q;q)_{\infty} \sum_{n=1}^{\infty}\frac{f(n)}{n}\cdot\frac{q^n}{(q;q)_n} =\lim_{n\to \infty} \frac{1}{n}\sum_{k=1}^{n} \frac{f(k)}{k}.
\end{equation*}\end{remark}

With $(q;q)_{\infty}$ floating around in these formulas, we could apply $q$-bracket ideas from Chapter 3 for further relations. Moreover, a finite version of the above order-of-summation swapping holds with respect to partial sums.
Let $F_{\mathbb S_{r,t}}(q,N)$ denote the following partial sum, with $F_{\mathbb S_{r,t}}(q,N)\to F_{\mathbb S_{r,t}}(q)$ as $N\to\infty$ per the proof of Lemma \ref{ch6lemma3}:
\begin{equation}F_{\mathbb S_{r,t}}(q,N):=(q;q)_{\infty}\sum_{n=0}^{N}\frac{q^{nt+r}}{(q;q)_{nt+r}}= (q;q)_{\infty}\sum_{\substack{ \lambda \in \mathcal{P}\\  \operatorname{lg}(\lambda) \in \mathbb S_{r,t}\\\operatorname{lg}(\lambda)\leq Nt+r}} q^{| \lambda|}.\end{equation}

\begin{proposition}\label{EFthm2}
Using the above notation, we have that
$$(q;q)_{\infty}^{-1}\sum_{t=1}^{N}a(n){{F}_{\mathbb S_{0,t}}\left(q,\floor{\frac{N}{t}}\right)}=\sum_{t=1}^{N}a(t)+\sum_{t=1}^{N}A(t)\frac{q^t}{(q;q)_t}.$$
In terms of partitions we can write
$$\sum_{t=1}^{N}\sum_{\substack{ \emptyset \neq \lambda \in \mathcal{P}\\  \operatorname{lg}(\lambda) \in \mathbb S_{0,t}\\\operatorname{lg}(\lambda)\leq \floor{N/t}t}}  a(t) q^{| \lambda|}= \sum_{ \substack{\emptyset\neq \lambda \in \mathcal{P}\\ \operatorname{lg}(\lambda)\leq N}}  A(\operatorname{lg}(\lambda)) q^{| \lambda|}.$$

\end{proposition}

Here is an example involving Mertens's function, the summatory function of the M\"{o}bius function\footnote{It is a famous fact that the statement $M(x)=\mathcal O(x^{1/2+\epsilon})$ is equivalent to the Riemann Hypothesis.}, viz. $M(x):=\sum_{1\leq n \leq x}\mu(n)$.%, the summatory function of the M\"{o}bius function.

\begin{example}\label{EFex2}
Set $a(n)=\mu(n)$ in Proposition \ref{EFthm2}. Then as $A(n)=\sum_{d|n}\mu(d)=1$ if $n=1$ and $=0$ otherwise, %using again the fact $\lim_{N\to \infty}\sum_{n=1}^{N}\mu(n)/n=0$, % together with Theorem \ref{ch6rt}, 
a little algebra gives
\begin{equation}\label{EFex3}
(q;q)_{\infty}M(N)\  =\  q(q^2;q)_{\infty}+\sum_{n=1}^{N} \mu(n) F_{\mathbb S_{0,n}}\left(q,\floor{\frac{N}{n}}\right).\end{equation}
\end{example}

%\begin{remark}
One notes heuristically in the double limit $q\to\zeta, N\to\infty$ (for instance, %for $z$ in the upper half-plane $\mathbb H$ 
consider $q=e^{2\pi \text{i}z}, z=\text{i}/N$ as $N\to \infty$), the right-hand side of \eqref{EFex3} appears to vanish (both $(\zeta^2;\zeta)_{\infty}$ and $\sum_{n=1}^{\infty}\mu(n)/n$ equal zero) while the left side is indeterminate ($(\zeta;\zeta)_{\infty}=0$ and $M(N)$ oscillates but grows asymptotically without bound in absolute value). Can facts about $(q;q)_{\infty}$ tell us something about the growth of Mertens's function?\footnote{For instance, for $q=e^{2\pi \text{i}z}, z\in\mathbb H$ the upper half-plane, the modularity relation $\eta(z):=q^{1/24}(q;q)_{\infty}=\sqrt{-\text{i}z}\cdot\eta(-1/z)$ yields $\eta(\text{i}/N)=\eta(\text{i}N)/\sqrt{N}$ in the case $z=\text{i}/N$ suggested above.}

\chapter{Notes on Chapter 7: Alternating ``strange'' functions}\label{app:F}{\bf Adapted from \cite{strange}}%, a joint work with Ken Ono and Larry Rolen}

%\begin{abstract}
%We consider infinite series similar to the ``strange'' function $F(q)$ of Kontsevich studied by Zagier, Bryson-Ono-Pitman-Rhoades, Bringmann-Folsom-Rhoades, Rolen-Schneider, and others in connection to quantum modular forms. Here we show that a class of ``strange'' alternating series that are well-defined almost nowhere in the complex plane can be added (using a modified definition of limits) to familiar infinite products to produce convergent $q$-hypergeometric series, of a shape that specializes to Ramanujan's mock theta function $f(q)$, Zagier's quantum modular form $\sigma(q)$, and other interesting number-theoretic objects. We also  give Ces\`{a}ro sums for these ``strange'' series. 
%\end{abstract}
%
%\maketitle

\section{Further ``strange'' connections to quantum and mock modular forms}\label{Sect1}
Recall the ``strange'' function $F(q)$ of Kontsevich (see Definition \ref{ch1Fdef})
%\begin{equation}\label{EF(q)def}
%F(q):=\sum_{n=0}^{\infty}(q;q)_n
%\end{equation}  
studied in Chapter 7, %i.e., %, where the {\it $q$-Pochhammer symbol} is defined by $(a;q)_0:=1$, $(a;q)_n:=\prod_{j=0}^{n-1}(1-aq^j)$, and $(a;q)_{\infty}:=\lim_{n\to \infty}(a;q)_n$ for $a,q\in \mathbb C, |q|<1$. This series $F(q)$ is often referred to in the literature as 
%Kontsevich's ``strange'' function}, 
which has been studied deeply by Zagier \cite{Zagier_Vassiliev} --- it was one of his prototypes for quantum modular forms --- %, which enjoy beautiful transformations similar to classical modular forms, and also resemble objects in quantum theory \cite{Zagier_quantum} --- 
as well as by other authors \cite{BFR,BOPR} in connection to quantum modularity, unimodal sequences, and other topics. 

%In \cite{Zagier_Vassiliev}, Zagier provides many good %shows in here are many 
%reasons to say 
%the series (\ref{F(q)def}) is ``strange'' (see \cite{Zagier_Vassiliev}). %For brevity, 
For the sake of this appendix, we remind the reader that %We remind the reader that as $n\to \infty$, then $(q;q)_{n}$ converges on the unit disk, is essentially singular on the unit circle (except at roots of unity, where it vanishes),  
%and diverges when $|q|>1$. Thus 
$\sum_{n\geq 0} (q;q)_n$ converges {almost nowhere} 
in the complex plane. However, at $q=\zeta_m$ an $m$th order root of unity, $F$ is suddenly very well-behaved: because $(\zeta_m;\zeta_m)_{n}=0$ for $n\geq m$, then as $q\to \zeta_m$ radially from within the unit disk, $F(\zeta_m):=\lim_{q\to\zeta_m} F(q)$ is just a polynomial in $\mathbb Z[\zeta_m]$. (We generalize this phenomenon in Chapter 8.)%the following chapter. 

Now let us turn our attention to the alternating case of this series, viz. 
\begin{equation} 
\widetilde{F}(q):=\sum_{n=0}^{\infty}(-1)^n(q;q)_{n},
\end{equation}
a summation that has been studied by Cohen \cite{BOPR}, which is similarly ``strange'': it doesn't converge anywhere in $\mathbb C$ except at roots of unity, where it is a polynomial. 
In fact, 
computational examples suggest the odd and even partial sums of $\widetilde{F}(q)$ oscillate asymptotically between two convergent $q$-series. 

To capture this oscillatory behavior, let us adopt a 
notation we will use throughout this appendix. If $S$ is an infinite series, we will write $S_+$ to denote the limit of the sequence of odd partial sums, and $S_-$ for the limit of the even partial sums, if these limits exist (clearly if $S$ converges, then $S_+=S_-=S$). 

Interestingly, like $F(q)$, the ``strange'' series $\widetilde{F}(q)$ is closely connected to a sum Zagier provided as another prototype for quantum modularity (when multiplied by $q^{1/24}$) \cite{Zagier_quantum}, the function \begin{equation}\label{Esigma}
\sigma(q):=\sum_{n=0}^{\infty}\frac{q^{n(n+1)/2}}{(-q;q)_n}=1+\sum_{n=0}^{\infty}(-1)^n q^{n+1}(q;q)_n
\end{equation}
from Ramanujan's ``lost'' notebook, with the right-hand equality due to Andrews \cite{A-J-O}. If we use the convention introduced above and write $\widetilde{F}_{+}(q)$ (resp. $\widetilde{F}_{-}(q)$) to denote the limit of the odd (resp. even) partial sums of $\widetilde{F}$, we can state this connection explicitly, depending on the choice of ``$+$'' or ``$-$''.

\begin{theorem}\label{Ethm1}
For $0<|q| < 1$ we have
\begin{equation*}
\sigma(q)=2\widetilde{F}_{\pm}(q)\pm (q;q)_{\infty}.
\end{equation*}
\end{theorem}
 
We can make further sense of alternating ``strange'' series such as this using Ces\`{a}ro summation, a well-known alternative definition of the limits of infinite series (see \cite{Hardy_divergent}). 

\begin{definition} The Ces\`{a}ro sum of an infinite series is the limit of the arithmetic mean of successive partial sums, if the limit exists.
\end{definition}

In particular, it follows immediately that the Ces\`{a}ro sum of the series $S$ is the average $\frac{1}{2}(S_+ + S_-)$ if the limits $S_+,S_-$ exist. Then Theorem \ref{Ethm1} leads to the following fact.

\begin{corollary}
We have that $\frac{1}{2}\sigma(q)$ is the Ces\`{a}ro sum of the ``strange'' function $\widetilde{F}(q)$.
\end{corollary}

A similar relation to Theorem \ref{Ethm1} involves Ramanujan's prototype $f(q)$ for a mock theta function 
\begin{equation}\label{Ef(q)def}
f(q):=\sum_{n=0}^{\infty}\frac{q^{n^2}}{(-q;q)_n^{2}}=1-\sum_{n=1}^{\infty}\frac{(-1)^{n}q^n }{(-q;q)_n},
\end{equation} 
the right-hand side of which is due to Fine (see (26.22) in \cite{Fine}, Ch. 3). Now, if we define 
\begin{equation}
\widetilde{\phi}(q):=\sum_{n=0}^{\infty}\frac{(-1)^n}{(-q;q)_n},
\end{equation}
which is easily seen to be ``strange'' like the previous cases, and write  $\widetilde{\phi}_{\pm}$ for limits of the odd/even partial sums as above, we can write $f(q)$ in terms of the ``strange'' series and an infinite product.

\begin{theorem}\label{Ethm2}
For $0<|q|<1$ 
we have 
\begin{equation*}
f(q)=2\widetilde{\phi}_{\pm}(q)\pm \frac{1}{(-q;q)_{\infty}}.
\end{equation*}
\end{theorem}

Again, the Ces\`{a}ro sum results easily from this theorem.

\begin{corollary}
We have that $\frac{1}{2}f(q)$ is the Ces\`{a}ro sum of the ``strange'' function $\widetilde{\phi}(q)$.
\end{corollary}

Theorems \ref{Ethm1} and \ref{Ethm2} typify a general phenomenon: 
the combination of an alternating Kontsevich-style ``strange'' function 
with a related infinite product is a convergent $q$-series when we fix the $\pm$ sign in this modified definition of limits. 
Let us fix a few more notations in order to discuss this succinctly. As usual, for $n$ a non-negative integer, define 
\begin{equation*}
(a_1,a_2,...,a_r;q)_n:=(a_1;q)_n (a_2;q)_n \cdots (a_r;q)_n,
\end{equation*} 
along with the limiting case $(a_1,a_2,...,a_r;q)_{\infty}$ as $n\to\infty$. 
Associated to the sequence $a_1,a_2,...,a_r$ of complex coefficients, we will define a polynomial $\alpha_r(X)$ 
by the relation
\begin{equation}
(1-a_1 X)(1-a_2 X)\cdots (1-a_r X)=: 1-\alpha_r(X) X,
\end{equation}
thus
\begin{equation}
(a_1q,a_2q,...,a_rq;q)_n=\prod_{j=1}^{n}(1-\alpha_r(q^j)q^j), 
\end{equation}
and we follow this convention in also writing $(1-b_1 X)(1-b_2 X)\cdots (1-b_s X)=: 1-\beta_s(X) X$ for complex coefficients $b_1,b_2,...,b_s$. Moreover, we define a generalized alternating ``strange'' series:
\begin{equation}
\widetilde{\Phi}(a_1,a_2,...,a_r;b_1,b_2,...,b_s; q):=\sum_{n=0}^{\infty} (-1)^n \frac{(a_1q,a_2q,...,a_rq;q)_{n}}{(b_1q,b_2q,...,b_sq;q)_{n}}.
\end{equation}
Thus $\widetilde{F}(q)$ is the case $\widetilde{\Phi}(1;0; q)$, and $\widetilde{\phi}(q)$ is the case $\widetilde{\Phi}(0;-1; q)$. We note that if $q$ is a $k$th root of $1/a_i$ for some $i$, then $\widetilde{\Phi}$ truncates after $k$ terms like $F$ and $\widetilde{F}$. As above, let $\widetilde{\Phi}_{\pm}$ denote the limit of the odd/even partial sums; then we can encapsulate the preceding theorems in the following statement.

\begin{theorem}\label{Ethm3}
For $0<|q|<1$ we have \begin{flalign*}
2\widetilde{\Phi}_{\pm}&(a_1,a_2,...,a_r;b_1,b_2,...,b_s; q) \pm\frac{(a_1q,a_2q,...,a_rq;q)_{\infty}}{(b_1q,b_2q,...,b_sq;q)_{\infty}}\\
&=1-\sum_{n=1}^{\infty} \frac{(-1)^n q^{n}\left(\alpha_r(q^n)-\beta_s(q^n)\right) (a_1q,a_2q,...,a_rq;q)_{n-1}}{(b_1q,b_2q,...,b_s q;q)_{n}}.
\end{flalign*}
\end{theorem}

From this identity we can fully generalize the previous corollaries.

\begin{corollary}
We have that $1/2$ times the right-hand side of Theorem \ref{Ethm3} is the Ces\`{a}ro sum of the ``strange'' function $\widetilde{\Phi}(a_1,...,a_r;b_1,...,b_s; q)$.
\end{corollary}

The takeaway is that the $N$th partial sum of an alternating ``strange'' series oscillates asymptotically as $N\to \infty$ between $\frac{1}{2}(S(q)+(-1)^N P(q))$, where $S$ is an Eulerian infinite series and $P$ is an infinite product as given in Theorem \ref{Ethm3}. 
We recover Theorem \ref{Ethm1} from Theorem \ref{Ethm3} as the case $a_1=1,\  a_i=b_j=0$ for all $i>1,j\geq 1$. Theorem \ref{Ethm2} is the case $b_1=-1,\  a_i=b_j=0$ for all $i\geq1, j>1$. 

Considering these connections together with diverse connections made by Kontsevich's $F(q)$ to important objects of study \cite{BFR, BOPR,Zagier_Vassiliev}, it seems the ephemeral ``strange'' functions almost ``enter into mathematics as beautifully''\footnote{To redirect Ramanujan's words} as their 
convergent relatives, mock theta functions. We note that considerations of finiteness at roots of unity and renormalization phenomena studied in Chapter 8 apply to Theorem \ref{Ethm3} as well.

\section{Proofs of results}

In this section we quickly prove the preceding theorems, and justify the corollaries.

\begin{proof}[Proof of Theorem \ref{Ethm1}]
Using telescoping series to find for $|q|<1$ that
\begin{flalign*}
(q;q)_{\infty}
= 1-\sum_{n=0}^{\infty}(q;q)_n \left( 1 - (1-q^{n+1})\right)
= 1-\sum_{n=0}^{\infty}q^{n+1}(q;q)_n,
\end{flalign*}
and combining this functional equation with the right side of (\ref{Esigma}) above, easily gives
\begin{equation*}
\sigma(q)-(q;q)_{\infty}=2\sum_{n=0}^{\infty}q^{2n+1}(q;q)_{2n}.
\end{equation*}
On the other hand, manipulating symbols heuristically (for we are working with a divergent series $\widetilde{F}$) suggests we can rewrite 
\begin{flalign*} 
\widetilde{F}(q)=\sum_{n=0}^{\infty}\left( (q;q)_{2n}-(q;q)_{2n+1}\right)&=\sum_{n=0}^{\infty}(q;q)_{2n}\left(1- (1-q^{2n+1})\right)=\sum_{n=0}^{\infty}q^{2n+1}(q;q)_{2n},
\end{flalign*}
which is a rigorous statement if by convergence on the left we mean the limit as $N\to\infty$ of partial sums $\sum_{n=0}^{2N-1}(-1)^n (q;q)_n$. We can also choose the alternate coupling of summands to similar effect, e.g. considering here the partial sums $1+\sum_{n=1}^{N-1}[(q;q)_{2n}-$ $(q;q)_{2n-1}]-(q;q)_{2N-1}$ as $N\to \infty$. 
Combining the above considerations 
proves the 
theorem for $|q|<1$, which one finds to agree with computational examples.
\end{proof}

\begin{proof}[Proof of Theorem \ref{Ethm2}]
Following the 
formal steps that prove Theorem \ref{Ethm1} above, we can use 
\begin{equation*}
\frac{1}{(-q;q)_{\infty}}=1-\sum_{n=0}^{\infty}\frac{1}{(-q;q)_n}\left( 1 - \frac{1}{1+q^{n+1}}\right)=1-\sum_{n=1}^{\infty}\frac{q^n}{(-q;q)_n}
\end{equation*}
and rewrite the related ``strange'' series 
\begin{equation*}
\widetilde{\phi}(q)=\sum_{n=0}^{\infty} \frac{1}{(-q;q)_{2n}}\left( 1-\frac{1}{1+q^{2n+1}}\right)=\sum_{n=0}^{\infty} \frac{q^{2n+1}}{(-q;q)_{2n+1}},
\end{equation*}
which of course fails to converge for $0<|q|<1$ on the left-hand side 
but makes sense if we use the modified definition of convergence used above, 
%in Section \ref{ESect1}, 
to yield the identity in the theorem (which is, again, borne out by computational examples).
\end{proof}

\begin{proof}[Proof of Theorem \ref{Ethm3}]
Using the definitions of the polynomials $\alpha_r(X),\beta_s(X)$, then following the exact steps that yield Theorems \ref{Ethm1} and \ref{Ethm2}, i.e., manipulating and comparing telescoping-type series with the same modified definition of convergence, gives the theorem.
\end{proof}

\begin{proof}[Proof of Corollaries]
Clearly, for an alternating ``strange'' series in which the odd and even partial sums each approach a different limit, the average of these two limits will equal the Ces\`{a}ro sum of the series.
\end{proof}

\begin{remark}
It follows from Euler's continued fraction formula \cite{Euler} that 
alternating ``strange'' functions have 
representations such as 
\begin{equation*}
\widetilde{F}(q)=\frac{{1}}{1+\frac{1-q}{q+\frac{1-q^2}{q^2+\frac{1-q^3}{q^3+...}}}},\  \  \  \  \  \  \  \  \widetilde{\phi}(q) =  \frac{{1}}{1+\frac{1}{q+\frac{1+q}{q^2+\frac{1+q^2}{q^3+...}}}}.
\end{equation*}
These ``strange'' continued fractions diverge on $0<|q|<1$ with successive convergents equal to the corresponding partial sums of their series representations. Then we can substitute continued fractions for the Kontsevich-style summations in the theorems above using a similarly modified definition of convergence: we take 
the $\pm$ sign to be positive when we define the limit of the continued fraction to be the limit of the even convergents, and negative if instead we use odd convergents.
\end{remark}

\chapter{Notes on Chapter 8: Results from a computational study of $f(q)$}
\label{app:G}{\bf Based on joint work with Amanda Clemm}

\section{Cyclotomic-type structures at certain roots of unity}

Here we record some relations the author and Amanda Clemm observed computationally during a study at Emory University (September--December, 2013) of the mock theta function $f(q)$\footnote{Recall from \eqref{ch1f(q)def}} at roots of unity. 
%Following up on patterns we noticed in both numerical examples and algebraic formulas, we computed a list of identities indicative of strong algebraic structure connecting values of $f(q)$ at certain roots of unity, proved by direct calculation. We did not attempt algebraic proofs, which would be required to understand more general phenomena, but they are certain to follow from well-known facts about polynomials evaluated at roots of unity. %; then we had to put aside our investigation due to the workload of graduate studies and exams. 
%Still, these relations and structures---and the possibility they point to more general phenomena---are intriguing enough that we feel we should compile them. 
In our study, we programmed SageMath using the finite formula for $f(\zeta_m)$ given in Example \ref{ch8example2} and we looked for patterns in our numerics. We saw traces of cyclic group theory related to the values $f(\zeta_m^i)$ for odd $m$. The algebraic structure appears most transparently if we use the normalized version
\begin{equation}\label{Fnormalized}
\widetilde{f}(\zeta_m^i):=\frac{3}{4}f(\zeta_m^i)
\end{equation}
for $m$ an odd number, which is just the summation on the right-hand side of Example \ref{ch8example2}. %We see %then %by expanding the terms of this summation 
%that 
We note $\widetilde{f}(1)=1$. These evaluations of $\widetilde{f}$ enjoy surprisingly nice combinations. We observe computationally (without proof) that the coefficients of the cyclotomic-type polynomial
\begin{equation}\label{FFcyclotomic}
\widetilde{F}_m(X):=\prod_{\substack{1\leq i< m\\ \operatorname{gcd}(m,i)=1\\}} \left(X-\widetilde{f}(\zeta_m^i)\right)
\end{equation}
are integers; %Then %(we might think of $\widetilde{F}_m(X)$ as a ``mock cyclotomic'' polynomial). 
%by standard facts about coefficients \cite{DummitFoote} 
in other words, %our computations suggest
\begin{flalign}\label{Fcoeff1}
\sum_{\operatorname{gcd}(m,i)=1} \widetilde{f}(\zeta_m^i),\  \sum_{\substack{i\neq j\\ \operatorname{gcd}(m,i)=\operatorname{gcd}(m,j)=1\\}} \widetilde{f}(\zeta_m^i)\widetilde{f}(\zeta_m^j),\  \sum_{\substack{i\neq j \neq k\\ \operatorname{gcd}(m,i)=\operatorname{gcd}(m,j)=\operatorname{gcd}(m,k)=1\\}} \widetilde{f}(\zeta_m^i)\widetilde{f}(\zeta_m^j)\widetilde{f}(\zeta_m^k),
\end{flalign}
and so on up to 
\begin{equation}\label{Fcoeff2} \prod_{\substack{1\leq i < m \\ \operatorname{gcd}(m,i)=1\\}}\widetilde{f}(\zeta_m^i),\end{equation} are all integers. %In our study, we looked for more subtle relations.%, which is surprising given the usual divergence of mock theta functions at the unit circle. 
To simplify calculations with respect to the $\operatorname{gcd}$, 
%let us 
take $m=p$ a prime number. We observe computationally %(without proof, however) 
that for the first few primes $p$, the coefficients indicated in (\ref{Fcoeff1}), (\ref{Fcoeff2}) appear to be congruent to 1 modulo $p$, though we do not know if this holds for all primes. It also appears %One also observes %computationally %(without having proved this) 
that the $\widetilde{f}(\zeta_p^i)$ are cyclic of order $p$, modulo $p$: 
\begin{equation}\label{Ffcyclic}
\widetilde{f}(\zeta_p^i)^n\equiv \widetilde{f}(\zeta_p^i)^{n+pk}\  (\operatorname{mod} p)\  \text{for all}\  i,k,n\in\mathbb Z.
\end{equation}
%\begin{remark}
%Classical theorems from algebra \cite{DummitFoote} such as 
%\begin{equation}
%F(pq)\equiv F(q)^p\  (\operatorname{mod}\  p)
%\end{equation}
%for $F(q)$ a power series, lead to other facts in this direction.
%\end{remark}

Following up on these observations, we computed examples for $p=5$ and found many linear combinations of the forms $\widetilde{f}(\zeta_5^i)$ yield nice evaluations, % (so many we did not think to record them all at the time), 
such as this infinite system, which %we found by experimentation but 
is not hard to prove from facts about polynomials at roots of unity \cite{DummitFoote}:% from standard facts about polynomials (see \cite{DummitFoote}): %follow by applying ideas from polynomial theory to $\widetilde{F}_5(X)$; direct calculation .
\begin{flalign}\label{Fdiscrim}
\begin{split}
&\widetilde{f}(\zeta_5)^{\  }+\widetilde{f}(\zeta_5^2)^{\  }+\widetilde{f}(\zeta_5^3)^{\  }+\widetilde{f}(\zeta_5^4)^{\  }=4,\\
&\widetilde{f}(\zeta_5)^2+\widetilde{f}(\zeta_5^2)^2+\widetilde{f}(\zeta_5^3)^2+\widetilde{f}(\zeta_5^4)^2=4,\\
&\widetilde{f}(\zeta_5)^3+\widetilde{f}(\zeta_5^2)^3+\widetilde{f}(\zeta_5^3)^3+\widetilde{f}(\zeta_5^4)^3=-11,\\
&\widetilde{f}(\zeta_5)^4+\widetilde{f}(\zeta_5^2)^4+\widetilde{f}(\zeta_5^3)^4+\widetilde{f}(\zeta_5^4)^4=-76,...\\
%&\widetilde{f}(\zeta_5)^5+\widetilde{f}(\zeta_5^2)^5+\widetilde{f}(\zeta_5^3)^5+\widetilde{f}(\zeta_5^4)^5=-246,\\
%&\widetilde{f}(\zeta_5)^6+\widetilde{f}(\zeta_5^2)^6+\widetilde{f}(\zeta_5^3)^6+\widetilde{f}(\zeta_5^4)^6=-521,...\\
\end{split}
\end{flalign}
%In general, for $n\geq 1$ we have that $\widetilde{f}(\zeta_5)^n+\widetilde{f}(\zeta_5^2)^n+\widetilde{f}(\zeta_5^3)^n+\widetilde{f}(\zeta_5^4)^n$ is an integer.
%
%\begin{remark}
%We can prove identities like these with explicit numerics by using the following fact about polynomials in clever ways. For a degree-$(n-1)$ polynomial $g(x)=\sum_{j=0}^{n-1}a_j x^j$, we have of course that
%\begin{equation}
%G_m(x):=\sum_{i=0}^{m-1}g(x^i)=\sum_{j=0}^{n-1}\sum_{i=0}^{m-1} a_j x^{ij}=\sum_{j=0}^{n-1}a_j \frac{1-x^{jm}}{1-x^j}
%\end{equation}
%by geometric series. But it is an elementary fact that, letting $x=\zeta_m$ an order-$m$ root of unity, then $\sum_{0\leq i\leq m-1}\zeta_m^i=0$. Therefore $G_m(\zeta_m)=a_0 n$. 
%\end{remark}

More strikingly, we see these $\widetilde{f}(\zeta_5^i)$ involved in cyclotomic-like structures. Noting $\widetilde{f}(1)=1$, by direct computation we find this simple identity as the $m=5$ case of \eqref{Fcoeff2}:
%\begin{proposition}
%We have that
\begin{equation}\label{Fthm0}
\widetilde{f}(\zeta_5)\widetilde{f}(\zeta_5^2)\widetilde{f}(\zeta_5^3)\widetilde{f}(\zeta_5^4)=1.
\end{equation} 
%\end{proposition}

%Then we get the obvious relations
%\begin{equation}\label{cor0}
%\widetilde{f}(\zeta_5^i)\widetilde{f}(\zeta_5^j)=\frac{1}{\widetilde{f}(\zeta_5^k)\widetilde{f}(\zeta_5^l)},\  \  \widetilde{f}(\zeta_5^i)=\frac{1}{\widetilde{f}(\zeta_5^j)\widetilde{f}(\zeta_5^k)\widetilde{f}(\zeta_5^l)},  
%\end{equation} 
%where $i,j,k,l\in\{1,2,3,4\},\  i\neq j\neq k\neq l$. %, which are still surprising if we remember these $\widetilde{f}$ are limiting values of a highly enigmatic infinite series. 
%%Other nice identities arise using ideas from algebra
%Combining all the above formulas, and also using well-known facts from algebra \cite{DummitFoote} like 
%\begin{equation}
%1+\zeta_m+\zeta_m^2+\zeta_m^3+...+\zeta_m^{m-1}=0\  \text{for}\  m>1,
%\end{equation}
%will lead to further relations between these $\widetilde{f}(\zeta_5^i)$. %We feel certain that some combination of identities like these, possibly together with manipulation of (\ref{MTFfinite}), will yield the next theorem (which we have not, however, been able to prove algebraically nor to generalize). 
Direct calculation verifies further identities, which we did not prove formally.%, which are quite startling.

\begin{proposition}\label{Fthm1}
Certain products $\widetilde{f}(\zeta_5^i)\widetilde{f}(\zeta_5^j), i\neq j,$ are equal to roots of unity:
\begin{flalign*}
\begin{split}
\widetilde{f}(\zeta_5)\widetilde{f}(\zeta_5^3)&=\zeta_5,\\
\widetilde{f}(\zeta_5)\widetilde{f}(\zeta_5^2)&=\zeta_5^2,\\
\widetilde{f}(\zeta_5^3)\widetilde{f}(\zeta_5^4)&=\zeta_5^3,\\
\widetilde{f}(\zeta_5^2)\widetilde{f}(\zeta_5^4)&=\zeta_5^4.\\
\end{split}
\end{flalign*}
%There are not similar identities for $\widetilde{f}(\zeta_5^2)\widetilde{f}(\zeta_5^3)$ or $\widetilde{f}(\zeta_5)\widetilde{f}(\zeta_5^4)$ (which are reciprocals).
\end{proposition}

At this point it is easy to derive any number of identities algebraically, e.g., %for example
\begin{equation*}
\widetilde{f}(\zeta_5)^3\widetilde{f}(\zeta_5^2)^2\widetilde{f}(\zeta_5^3)=1,\  \  \  \left(\widetilde{f}(\zeta_5)+\widetilde{f}(\zeta_5^4)\right)\left(\widetilde{f}(\zeta_5^2)+\widetilde{f}(\zeta_5^3)\right)=-1.
\end{equation*}
%and so on. Of course, we can rewrite these equations for $\widetilde{f}$ in terms of the mock theta function $f=3\slash 4\  \widetilde{f}$. For instance, by (\ref{normalized}) and (\ref{discrim}), for $n\geq 1$ we get identities of the shape
% \begin{equation}
%f(\zeta_5)^n+f(\zeta_5^2)^n+f(\zeta_5^3)^n+f(\zeta_5^4)^n=(4/3)^n \beta_n,
%\end{equation}
%where $\beta_n\in \mathbb Z$ denotes the right-hand side of the corresponding identity in (\ref{discrim}) (e.g. $\beta_1=4,\beta_2=4,\beta_3=-11$, etc.). 
%
%
%Then using (\ref{Fnormalized}) we arrive at the following nice relations.
%
%\begin{proposition}\label{Fthm2}
%For the mock theta function $f(q)$ at fifth-order roots of unity $\zeta_5^i$, we have
%\begin{flalign*}
%\begin{split}
%&9\slash 16\  f(\zeta_5)f(\zeta_5^3)=\zeta_5\\
%&9\slash 16\  f(\zeta_5)f(\zeta_5^2)=\zeta_5^2\\
%&9\slash 16\  f(\zeta_5^3)f(\zeta_5^4)= \zeta_5^3\\
%&9\slash 16\  f(\zeta_5^2)f(\zeta_5^4)= \zeta_5^4\\
%\end{split}
%\end{flalign*}
%and also
%\begin{equation*}
%f(\zeta_5)f(\zeta_5^2)f(\zeta_5^3)f(\zeta_5^4)=256/81.
%\end{equation*}
%\end{proposition}
%\  \  
From the preceding formulas and (\ref{Fnormalized}), we also derive a very tidy relation for $f(\zeta_5^i)$. %we arrive at an elegant identity. 

\begin{proposition}\label{Fthm3}
At fifth-order roots of unity, we have the symmetric relation
\begin{equation*}
\zeta_5^i f(\zeta_5^i)\  =\  \zeta_5^{-i} f(\zeta_5^{-i}).
\end{equation*}
\end{proposition}

The empirical conjecture that \eqref{Fcoeff2} is an integer\footnote{In fact, computations suggest \eqref{Fcoeff1}, \eqref{Fcoeff2} may be integers even without conditions on the $\operatorname{gcd}$'s.} suggests an equation like \eqref{Fthm0} exists for every odd-order root of unity ($f(q)$ diverges at even-order roots of unity, thus the finite formula in Example \ref{ch8example2} does not represent its limit). 
%, possibly with some integer aside from 1 (and, by our conjecture above, that is $\equiv 1\  \text{modulo}\  p$ for $m=p$ prime). As further examples, for $m=6,7,8$ we compute \eqref{Fcoeff2}:
%$$
%\widetilde{f}(\zeta_6)\widetilde{f}(\zeta_6^5)=1,\  \  \  \  \  \  \widetilde{f}(\zeta_7)\widetilde{f}(\zeta_7^2)\widetilde{f}(\zeta_7^3)\widetilde{f}(\zeta_7^4)\widetilde{f}(\zeta_7^5)\widetilde{f}(\zeta_7^6)=337\equiv 1\  (\operatorname{mod}\  7)$$
%$$\widetilde{f}(\zeta_8)\widetilde{f}(\zeta_8^3)\widetilde{f}(\zeta_8^5)\widetilde{f}(\zeta_8^7)=9\equiv 1\  (\operatorname{mod}\  8).
%$$
%%, e.g.,
%% 
%%\begin{flalign}\label{Ffinaleqs}
%%...
%%\end{flalign} 
%%%\begin{remark}
%%%Theorem \ref{thm3} holds if we replace $f$ with $\widetilde{f}$ as the functions differ only by a constant. %Thus we can write $\widetilde{F}_5(X)=\prod_{1\leq i\leq 4} \left(X-\zeta_5^{2i}\widetilde{f}(\zeta_5^i)\right)$, which really does look ``mock cyclotomic''. 
%%\end{remark}
%
Now, %as noted above, 
we computed (\ref{Fthm0}) and Proposition \ref{Fthm1} directly from the formula in Example \ref{ch8example2}, letting $m$ be the prime $p=5$; we have not proved these by algebraic methods, so we don't have a clear intuition as to how the propositions generalize. One expects there to be analogs of Propositions \ref{Fthm1} and \ref{Fthm3} above (but perhaps more complicated) for $f(q)$ at other odd-order roots of unity $\zeta_m$, as presumably these propositions depend in the end on properties of Example \ref{ch8example2} and facts about polynomials at roots of unity, not on the choice of the order $m$. 

Are there cyclotomic-type relations like \eqref{Fthm0} and Propositions \ref{Fthm1} and \ref{Fthm3} for other mock theta functions --- or other $q$-hypergeometric series --- at roots of unity?%, in light of Lemma \ref{ch8lemma}?

			\clearpage %
\clearpage
\addcontentsline{toc}{chapter}{Bibliography}

% % % % % % % % % %
% % The bibliography % %
% % % % % % % % % %
\bibliography{references}
\bibliographystyle{alpha}

\end{document}